\documentclass[11pt]{amsart}
\usepackage{amsmath, amssymb, amscd, cancel, graphicx, soul, stmaryrd, accents, bbm}
\pdfoutput=1
\usepackage[colorlinks=true, urlcolor=NavyBlue, linkcolor=NavyBlue, citecolor=NavyBlue, bookmarks=false]{hyperref}

\usepackage[margin=1in]{geometry}
\usepackage[dvipsnames,usenames]{xcolor}
\usepackage{mathdots}
\usepackage{float}
\usepackage{amsthm}
\usepackage{amsfonts}
\usepackage[parfill]{parskip} 
\usepackage{paralist}
\usepackage{marginnote}
\usepackage[all]{xy}
\numberwithin{equation}{section}
\usepackage{tikz}
\usepackage{wrapfig}
\usepackage{pinlabel}
\usepackage{multirow}
\usepackage{blkarray}
\usepackage{mathtools}
\usepackage{minitoc}
\usepackage{setspace}
\usepackage{tikz}
\usetikzlibrary{shapes, arrows.meta, positioning}
\usetikzlibrary{arrows.meta}

\usepackage{tcolorbox}
\tcbuselibrary{theorems}
\tcbuselibrary{skins}
\usepackage{lmodern}
\usepackage[T2A,T1]{fontenc} % T2A for Cyrillic letters
\usepackage[utf8]{inputenc}
\usepackage{inputenc}
\usepackage[UKenglish]{babel}
\usepackage{yfonts}
\usepackage{bm}
\usepackage{afterpage}
\usepackage{ifthen}
\usepackage{tabularx}
\usepackage{nicefrac}
\usepackage{picins}
\usepackage{soul}
\usepackage{calc}
\usepackage{float}
\usepackage{wrapfig}

\usepackage{enumitem}
\setlist[itemize]{leftmargin=2em}%{label=$\boldsymbol{\circ}$}

\usepackage{picins}
\usepackage[font=footnotesize,labelfont=bf, format=hang, width=.8\linewidth]{caption}
\usepackage{wrapfig}
\DeclareMathAlphabet{\mathpzc}{OT1}{pzc}{m}{it}

\usepackage[labelfont=bf]{caption} 
\DeclareCaptionFormat{myformat}{\fontsize{9}{10}\selectfont#1#2#3}
\usepackage{placeins}
\usepackage{pinlabel}

\usepackage{graphicx}
\usepackage{framed}

\usepackage[nameinlink]{cleveref}
\definecolor{cadetblue}{rgb}{0.37, 0.62, 0.63}

\hypersetup{colorlinks=true,linkcolor=NavyBlue,citecolor=NavyBlue}

\newtheorem{thm}{Theorem}[section]
\newtheorem{conj}[thm]{Conjecture}
\newtheorem{cor}[thm]{Corollary}
\newtheorem{lemma}[thm]{Lemma}
\newtheorem{prop}[thm]{Proposition}
\newtheorem{question}[thm]{Question}
\newtheorem{defn}[thm]{Definition}
\newtheorem{rmk}[thm]{Remark}

\newcommand{\Q}{\mathbb Q}
\newcommand{\Z}{\mathbb Z}
\newcommand{\R}{\mathbb R}
\newcommand{\Cone}{\mathcal{C}_{1}}
\newcommand{\Ctwo}{\mathcal{C}_{2}}
\newcommand{\Cthree}{\mathcal{C}_{3}}
\newcommand{\Cmin}{\mathcal{C}_{\text{min}}}
\newcommand{\Cmax}{\mathcal{C}_{\text{max}}}
\newcommand{\Codd}{\mathcal{C}_{\text{odd}}}
\newcommand{\Ceven}{\mathcal{C}_{\text{even}}}
\newcommand{\TauSup}{\tau_{\text{sup}}}
\newcommand{\Seifert}{\textbf{Seifert disk sector analysis: }} 
\newcommand{\Polygon}{\textbf{Polygon sector analysis: }} 
\newcommand{\Horizontal}{\textbf{Horizontal sector analysis: }} 
\newcommand{\one}{\sigma_1}
\newcommand{\two}{\sigma_2}
\newcommand{\three}{\sigma_3}

\begin{document}
\renewcommand\labelitemi{$\boldsymbol{\circ}$}

\title[Taut foliations, braid positivity, and unknot detection]
{Taut foliations, braid positivity, and unknot detection}

\author{Siddhi Krishna}
\address{Department of Mathematics, Columbia University\\ New York, NY 10027}
\email{sk5026@columbia.edu}

\def\subjclassname{\textup{2020} Mathematics Subject Classification}
\expandafter\let\csname subjclassname@1991\endcsname=\subjclassname
%\expandafter\let\csname subjclassname@2000\endcsname=\subjclassname
\subjclass{
57K35, %Other geometric structures on 3-manifolds
57K30 (primary); %general topology of 3-manifolds
57K18, %homology theories in knot theory
57K10 %knot theory
(secondary). 
}

\begin{abstract}
We study \textit{positive braid knots} (the knots in the three--sphere realized as positive braid closures) through the lens of the L-space conjecture. This conjecture predicts that if $K$ is a non-trivial positive braid knot, then for all $r < 2g(K)-1$, the 3-manifold obtained via $r$-framed Dehn surgery along $K$ admits a taut foliation. Our main result provides some positive evidence towards this conjecture: we construct taut foliations in such manifolds whenever $r<g(K)+1$.
%Our main result constructs taut foliations in the 3--manifolds obtained by Dehn surgery along such knots whenever the surgery coefficient $r < g(K)+1$; here, $g(K)$ denotes the Seifert genus of the knot. This provides some positive evidence towards the L-space conjecture, which predicts that every non-L-space obtained by Dehn surgery along a non-trivial positive braid knot admits a taut foliation. 
As an application, we produce a novel braid positivity obstruction for cable knots by proving that \textit{the $(n,\pm 1)$--cable of a knot $K$ is braid positive if and only if $K$ is the unknot}. We also present some curious examples demonstrating the limitations of our construction; these examples can also be viewed as providing some negative evidence towards the L-space conjecture. 
%Namely, we produce an infinite family of L-space knots where $r$-framed Dehn surgery has a taut foliation whenever $r < 2g(K)-2$, but our construction does not appear to yield taut foliations when $r$ approaches the non-L-space threshold slope of $2g(K)-1$. 
Finally, we apply our main result to produce taut foliations in some splicings of knot exteriors. 
\end{abstract}

\maketitle

\section{Introduction} \label{section:intro}

\subsection{Motivation.} \label{subsection:motivation} 

Tools for studying 3--manifolds come in a variety of flavors. A fruitful approach to studying 3--manifolds comes from studying the geometric structures they admit. One such geometric structure is a \textbf{taut foliation}, a particular type of decomposition of a 3--manifold $Y$ into (typically non--compact) surfaces (called \textit{leaves}), such that there is a simple closed curve meeting every leaf transversely. Historically, taut foliations have been useful for studying Dehn surgery theoretic problems. While the most prominent application is Gabai's seminal proof of the Property R conjecture \cite{Gabai:FoliationsIII}, taut foliations were also used to probe the Property P conjecture \cite{DasbachLi}, as well as Thurston's geometrization \cite{RobertsShareshianStein}. In a different vein, Floer homological invariants provide a powerful approach for studying 3--manifolds; they, too, are often well suited for studying problems within the Dehn surgery realm. For example, Kronheimer--Mrowka--Ozsv\'ath--Szab\'o used monopole Floer homology to resolve Gordon's conjecture \cite{KMOSz}, and Greene used Heegaard Floer homology to address the lens space realization problem \cite{Greene:LensSpaceRealization}. \textbf{L--spaces} (3--manifolds with ``simple'' Floer homology) played an essential role in both results. 

L--spaces cannot admit taut foliations \cite{OSz:HolDisks, Bowden:Approx, KazezRoberts}. Given this result, one may wonder: \textit{can Floer homology detect the existence of a taut foliation?} Answering this question and its relatives has culminated in the bold L--space conjecture. If true, it neatly organizes 3--manifolds based on their Floer homological, geometric, and algebraic data:

\begin{conj}[The L-space Conjecture \cite{BoyerGordonWatson, Juhasz:Survey}] \label{conj:LSpace} 
Suppose $Y$ is an irreducible rational homology 3-sphere. Then the following are equivalent: 
\begin{enumerate}
\item $Y$ is a non-L-space $($i.e.~ the Heegaard Floer homology of $Y$ is not ``small''$)$,
\item $\pi_1(Y)$ admits a total ordering which is invariant under left multiplication, and
\item $Y$ admits a taut foliation.
\end{enumerate}
\end{conj}

The L--space conjecture is true for \textit{graph manifolds}, i.e. the 3--manifolds whose JSJ decomposition consists only of Seifert fibered pieces \cite{BoyerClay2, BoyerClay, BoyerGordonWatson, BrittenhamNaimiRoberts, ClayLidmanWatson, EisenbudHirschNeumann, HRRW, LiscaStipsicz}. 

There are three primary constructions of non-L-spaces in the literature: branched covers of knots in $S^3$, splicings of knot exteriors, and Dehn surgery along knots in the $S^3$. The branched covers perspective has been investigated by many
\cite{Peters:Thesis, BoileauBoyerGordon:SQP1, BoileauBoyerGordon:SQP2}; in this paper, we focus on Dehn surgery along knots, though we also provide some applications towards splicings.

In particular, this work primarily focuses on the rational homology spheres obtained by surgery along knots in $S^3$. We investigate: \textit{does every non--L--space obtained by Dehn surgery along a knot in $S^3$ admit a taut foliation?} A first step to addressing this question is to understand when surgery along a non--trivial knot $K \subset S^3$ yields a non--L--space.

\begin{defn}
A non--trivial knot $K \subset S^3$ is an \textbf{L--space knot} if there exists some $r > 0$ such that $S^3_r(K)$ is an L--space.
\end{defn}

Concrete examples of L--space knots include torus knots \cite{Moser:TorusKnots}, Berge knots \cite{Berge}, and 1--bridge braids \cite{GLV:11Lspace}. In some sense, L--space knots are not well understood; there is no characterization or classification of these knots. However, it is known that if a knot admits a surgery to a single L--space, it admits infinitely many:

\begin{thm}[\cite{KMOSz, Rasmussen2}]
Let $K$ be a non--trivial knot in $S^3$. Then $S^3_r(K)$ is an L-space if and only if $K$ is an L-space knot and $r \geq 2g(K)-1$, or $m(K)$ is an L-space knot and $r \leq 1 - 2g(K)$. 
\end{thm}

A braid is \textit{positive} if it is the product of positive Artin generators.  
If $K$ is realized as a positive braid closure (i.e. $K$ is \textit{braid positive}), then $S^3_r(K)$ is a non--L--space for all $r<2g(K)-1$ (regardless of whether $K$ is or is not an L-space knot). \Cref{conj:LSpace} predicts these manifolds should all admit taut foliations. The author confirmed this prediction for knots realized as positive 3--braid closures \cite{Krishna:3Braids}.
%\begin{thm}[\cite{Krishna:3Braids}] \label{thm:Positive3Braids}
%If $K \subset S^3$ is realized as a positive 3--braid closure, then $S^3_r(K)$, admits a taut foliation whenever $r < 2g(K)-1$. 
%\end{thm}
This was the first example confirming $(1) \iff (3)$ in \Cref{conj:LSpace} for every non--L--space obtained by Dehn surgery along an infinite family of hyperbolic L--space knots.

\subsection{Main results.} \label{subsection:results}

In this work, we study positive braid knots of any braid index.

%The main result of this work regards positive braid knots of any braid index. \textcolor{blue}{Our primary motivation for studying positive braid knots is that every known example of a hyperbolic L-space knot is braid positive 

\begin{thm}\label{thm:main}
Suppose the knot $K \subset S^3$ is realized as the closure of a positive braid, and $g(K) \geq 2$. Then $S^3_r(K)$ admits a taut foliation whenever $r < g(K)+1$.
\end{thm}

We note: the only positive braid knot of genus one is the right handed trefoil; it is already known that if $K$ is the right handed trefoil, $S^3_r(K)$ has a taut foliation if and only $r < 2g(K)-1$, as predicted by the L-space conjecture. Thus, the assumption that $g(K) \geq 2$ is not restrictive.

\Cref{thm:main} and \cite{Krishna:3Braids} are the only examples in the literature producing taut foliations in manifolds obtained by Dehn surgery along \textit{hyperbolic} knots, where the supremal slope is a function of $g(K)$, as predicted by \Cref{conj:LSpace}. Studying braids of ``big'' strand number (i.e. four or more strands) presents significant technical challenges, especially compared to braids on at most three strands; this is a typical dichotomy within the world of braids. 
For example,
Murasugi classified all 3--braids up to conjugation \cite{Murasugi:3Braids}, but no such classification exists for higher braid index, even amongst positive braids. 
While we did not use Murasugi's classification in \cite{Krishna:3Braids}, the significant increase in complexity for large strand number braids partially explains the challenges encountered while proving \Cref{thm:main}. 

%one (of the many) challenges in proving the analogous statement as \cite{Krishna:3Braids} for all positive braids partially reflects the lack of a classification for positive braids of index at least four. 

\Cref{thm:main} can be used to provide a novel obstruction to braid positivity. Throughout, we let $K_{p,q}$ denote the $(p,q)$-cable of a knot $K$, where the cable traverses the longitudinal direction of the knot $p$ times ($p \geq 2$), and the meridional direction $q$ times. 

\begin{thm} \label{UnknotDetection} The cable knot $K_{n, \pm 1}$ is braid positive $\iff$ $K$ is the unknot.
%The $(n,1)$--cable of a knot $K$ is braid positive $\iff K$ is the unknot.
\end{thm}

\Cref{UnknotDetection} is reminiscent of previously known unknot detection results, namely Grigsby--Wehrli's proof that \textit{Khovanov's categorification of the $n$--colored Jones polynomial detects the unknot} \cite{GrigsbyWehrli}, and Hedden's result that \textit{Khovanov homology of the 2--cable detects the unknot} \cite{Hedden:KhovanovHomology2Cable}. We pursue generalizations of \Cref{UnknotDetection} in \cite{HeddenKrishna}.  
Assuming \Cref{thm:main}, the proof of \Cref{UnknotDetection} is both short and self-contained. As the result may be of independent interest outside of \Cref{conj:LSpace}, it is presented in \Cref{section:Cabling}.

We observe a novelty of \Cref{UnknotDetection}: in general, it is difficult to distinguish \textit{braid positivity} from the closely related notions of \textit{knot positivity} (when a knot admits a diagram with only positive crossings) and \textit{strong quasipositivity} (see \cite{Rudolph:QPsliceness, Hedden:Positivity} for a definition). We recall that
$$\{\text{positive braid knots}\} \subsetneq \{\text{positive knots}\} \subsetneq \{\text{strongly quasipositive knots}\}.$$
Many invariants cannot readily distinguish these three notions. For example, for all strongly quasipositive knots, $\tau(K) = g_3(K) = g_4(K)$ and $s(K) = 2\tau(K)$ (where $\tau(K)$ and $s(K)$ denote the Ozsv\'ath--Szab\'o and Rasmussen concordance invariants, respectively, and $g_4(K)$ denotes the smooth 4--ball genus of $K$) \cite{Livingston:TauInvariant, HeddenOrding}. Thus, $\tau$ and $s$ cannot detect braid positivity. Similarly,  many classical invariants cannot distinguish positive braids amongst positive knots: for example, all positive knots have negative signature \cite{Przytycki:PositiveKnotsNegativeSignature} and positive Conway polynomials \cite{Cromwell:HomogeneousLinks}. Some properties of the coefficients and degrees of the HOMFLY and Jones polynomials are able to differentiate positive braid knots amongst positive knots on an ad--hoc basis; see \cite[Section 8]{Stoimenow:PositiveKnotsJonesPoly} and \cite{Ito:BraidPositiveHOMFLY}. However, no general formula computing either polynomial for cables of knots in terms of the polynomials of $T(p,q)$ and $K$ can exist \cite{StoimenowTanaka}, making it difficult to distinguish braid positivity from knot positivity for infinite families. 

%As a concrete example, any positive cable of any strongly quasipositive knot will be strongly quasipositive \cite{Hayden:LegendrianRibbons}, yet our results show they are often not braid positive. Our obstruction fundamentally differs in spirit from these results using knot polynomials, by deferring to a geometric structure on manifolds obtained by Dehn surgery.

%\begin{thm}
%Let $\beta$ be the positive braid in \textcolor{red}{Figure XX}, and let $K_m = \widehat{\beta^m}$. Then $S^3_r(K_m)$ has a taut foliation when $r < \frac{22m-3}{13m-3}g(K)$, but our techniques cannot build taut foliations in $S^3_r(K_m)$ when $r \in \left[\frac{22m-3}{13m-3}g(K), 2g(K_m)-1\right)$.
%\end{thm}

%\begin{wrapfigure}{r}{0.3\textwidth}
%\labellist %\tiny
%\pinlabel {$(\sigma_3 \sigma_3 \sigma_2)^{2m}$} at 20 20
%\endlabellist
%  \begin{center}
%    \includegraphics[scale=0.1]{4BraidExample_Intro_Vertical}
%  \end{center}
%  \caption{Birds}
%  \label{fig:4Braid_Intro}
%\end{wrapfigure}

%\begin{wrapfigure}{L}{.45\linewidth}
%\labellist
%\pinlabel {$(\sigma_3 \sigma_3 \sigma_2)^{2m}$} at 20 20
%\endlabellist
%\begin{minipage}[c][.52\paperheight]{64mm}
%\begin{raggedright}
%\vspace{1cm}
%\includegraphics[scale=.1]{4BraidExample_Intro_Vertical}
%\end{raggedright}
%\vspace*{8mm}
%\captionof{figure}{There are two \\ product disks identified, $D_j$ and $D_{j+1}$. \\ $\partial D_j \subset S_1 \cup S_2 \cup \mathbbm{b}_j \cup \mathbbm{b}_{j+3}$, and \\ $\partial D_{j+1} \subset S_2 \cup S_3 \cup \mathbbm{b}_{j+1} \cup \mathbbm{b}_{j+2}$.}
%\label{fig:2productdisks}
%\end{minipage}
%\end{wrapfigure}

On an ad--hoc basis, one can often improve upon \Cref{thm:main} for individual positive braid knots, or families thereof, and further close the gap towards $2g(K)-1$.

\begin{figure}[h!]\center
\labellist %\tiny
\pinlabel \rotatebox{90}{$(\sigma_3 \sigma_3 \sigma_2)^{2m}$} at 4355 470
\endlabellist
\includegraphics[scale=0.1]{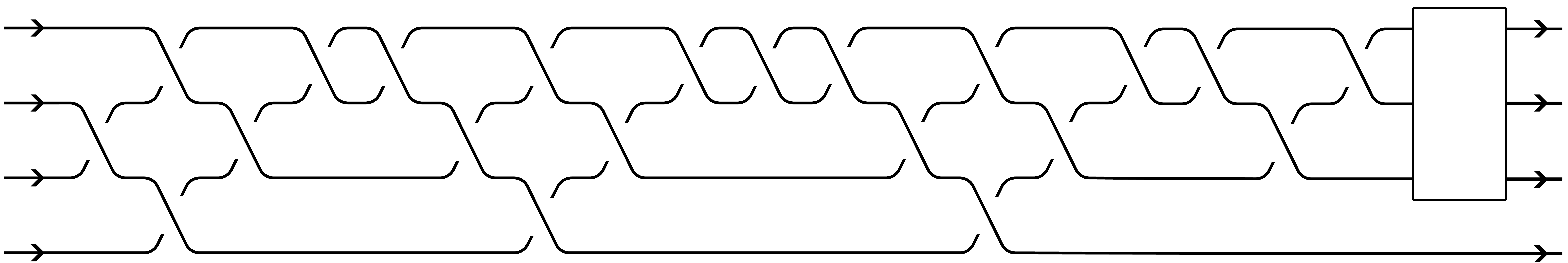}
\caption{Let $m \geq 1$. Define $\{\beta_m\}$ to be 4-braids pictured above. Taking the braid closure of $\beta_m$ yields the knot $K_m$. The infinite family $\{K_m\}$ are the knots of interest in \Cref{thm:examples}.}
\label{fig:4Braid_Intro}
\end{figure}

\begin{thm}\label{thm:examples}
There are infinitely many hyperbolic 4-braid L-space knots $\{K_m\}$ such that $S^3_r(K_m)$ admits a taut foliation whenever $r < 2g(K_m)-2$.
\end{thm}

Our concrete examples are seen in \Cref{fig:4Braid_Intro}. In \Cref{section:WeirdExamples}, we describe a surprising phenomenon for these knots -- namely, it does not appear that \textit{our construction} produces taut foliations in $S^3_r(K)$ for the remaining unit interval of Dehn surgery slopes which conjecturally admit taut foliations. For the taut foliations constructed in both \cite{Krishna:3Braids} and \Cref{thm:main}, the core of the Dehn surgery solid torus is a closed transversal for the taut foliation. This leads us to pose a qualitative question, which we plan to pursue in future work:

\begin{question} \label{TransversalQuestion}
Fix any hyperbolic knot $K_m$ from the family shown in \Cref{fig:4Braid_Intro} and fix some $r$ such that $r \in [2g(K_m)-2, 2g(K_m)-1)$. 
Suppose $S^3_r(K_m)$ has a taut foliation. Is the core of the Dehn surgery solid torus isotopic to a transversal for that taut foliation?
\end{question}

Finally, we provide an application towards the \textit{splicing of knot exteriors}. We begin by recalling the construction/definition: let $K_1$ and $K_2$ be non-trivial knots in $S^3$, and let $X_1$ and $X_2$ be their respective exteriors. Both $X_1$ and $X_2$ have torus boundary, and each torus is equipped with the standard homology basis $\mu_1, \lambda_1$ and $\mu_2, \lambda_2$, where $\lambda_1$ and $\lambda_2$ are the Seifert longitudes for $K_1$ and $K_2$, respectively. The 3-manifold $M(K_1, K_2)$ is obtained by gluing the two tori together via an orientation-reversing homeomorphism which identifies $\lambda_1$ with $\mu_2$, and $\mu_1$ with $\lambda_2$; this manifold is called the \textit{splicing of the the exteriors of $K_1$ and $K_2$}. A standard Meyer-Vietoris argument verifies that $M(K_1, K_2)$ is an integer homology 3--sphere (see \cite{Gordon:HomologySpheres} or \cite{Saveliev:Homology3Spheres} for details on this construction and its variations). Splicings have received considerable interest from the Floer homological community: Hedden-Levine showed that the splicing of non-trivial knot exteriors always yields a non-L-space \cite{HeddenLevine:Splicing}, and Baldwin-Sivek \cite{BaldwinSivek:Splicing} proved the analogous statement in the instanton Floer homology setting. Other Floer homological investigations of splicings include \cite{Eftekhary, HanselmanRasmussenWatson, Zemke:LinkSurgery, KarakurtLidmanTweedy}.

In light of \Cref{conj:LSpace} and the Hedden-Levine result, one expects splicings to both admit taut foliations and have left-orderable fundamental groups. Work of Boileau-Boyer \cite[Theorem 0.2]{BoileauBoyer:HomologySpheres} confirms both when $M(K_1,K_2)$ is a graph manifold that is not $S^3$ or the Poincar\'e homology sphere, and subsequent work of Boyer-Gordon-Hu \cite[Theorem 2.1]{BoyerGordonHu:Toroidal} does the same if one of $K_1$ or $K_2$ is fibered. We can apply the proof of \Cref{thm:main} to partially recover both:

\begin{cor} \label{cor:splicing}
Suppose $K_1$ and $K_2$ are positive braid knots of genus at least two, and let $M(K_1, K_2)$ denote the splicing of the knot exteriors of $K_1$ and $K_2$. Then $M(K_1, K_2)$ is a non-L-space if and only if it admits a taut foliation. Moreover, in this case, $\pi_1(M(K_1, K_2))$ is left-orderable.
%For $M(K_1, K_2)$, $(1) \iff (3) \implies (2)$ in \Cref{conj:LSpace} holds.
\end{cor}

\subsection{Organization}
The techniques required to deduce \Cref{UnknotDetection} from \Cref{thm:main} (a result potentially of independent interest) require no background in taut foliations or constructions thereof, so we present the proof in \Cref{section:Cabling}.
%As \Cref{UnknotDetection} may be of general interest, we present the proof in \Cref{section:Cabling}, and remark that the techniques therein require no background in taut foliations or constructions thereof; in particular, the techniques used in the proof of \Cref{UnknotDetection} are independent of the proof of \Cref{thm:main}. 
%The remainder of the paper is devoted to proving \Cref{thm:main}. 
\Cref{section:background} reviews the background material on building taut foliations via branched surfaces, as well as factorizations of the monodromies of positive braid links. \Cref{section:example} demonstrates the general proof strategy of \Cref{thm:main} with an example. \Cref{thm:main} is proved for 4-braids in \Cref{section:Proof4Braids}, and for braids on $n \geq 5$ strands in \Cref{section:proof}. We prove \Cref{thm:examples} in \Cref{section:WeirdExamples}. The proof of \Cref{cor:splicing} appears in \Cref{section:Splicing}.

\subsection{Acknowledgements} We thank John Baldwin, Peter Feller, Josh Greene, Kyle Hayden, Matt Hedden, Jen Hom, Tao Li, Francesco Lin, Marissa Loving, and Liam Watson for helpful conversations. Special thanks to Tye Lidman for asking about cabling obstructions, leading us to prove \Cref{UnknotDetection}. Additional gratitude goes to John and Francesco for their encouragement during various phases of the writing process. We thank Chuck Livingston and Allison H. Moore for maintaining \textit{KnotInfo}, which we utilized at various points throughout this project. Finally, we thank the referee for their careful reading and helpful comments. This work was partially supported by NSF grants DMS--1745583 and DMS--2103325.

% \begin{figure}
% \labellist
% \small
% \pinlabel $v$ at 58 246
% \pinlabel $w$ at 246 57 
% \pinlabel {$\langle v, w \rangle = 1$} at 180 140
% \endlabellist
% \includegraphics[scale=.30]{intersection_number}
% \label{fig:intersection_number}
% \end{figure}

\subsection{Conventions} 
\begin{itemize}
\item Throughout, we let $X_K$ denote a knot exterior, and $\lambda, \mu \in H_1(\partial X_K)$ denote the standard Seifert longitude and meridian for a knot.
\item Let $\langle \alpha, \beta \rangle$ denote the algebraic intersection number. Using the sign convention established in \Cref{fig:intersection_number}, we set $\langle \lambda, \mu \rangle = 1$. The slope of an essential simple closed curve $\gamma$ on $T^2 \approx \partial X_K$ is determined by $\displaystyle \frac{\langle \gamma, \lambda \rangle}{\langle \mu, \gamma \rangle}$.
\item We use $\sigma_1, \sigma_2, \ldots, \sigma_{n-1}$ to represent the standard Artin generators for the $n$-stranded braid group. When strands are drawn vertically, they are oriented ``from north to south'', and when they are drawn horizontally, they are oriented ``from west to east''. 
\item The surface $F$ will always be orientable; in all figures of Seifert surfaces, only $F^{+}$ is visible.
\item Blue arcs lie on $F^-$, and pink arcs lie on $F^+$. 
\item When $K$ is a fibered knot in $S^3$, we denote the monodromy of the mapping torus by $\varphi$, and we assume that $\varphi \approx \mathbbm{1}$ on $\partial X_K$. 
\item Unless stated otherwise, if we say that we are ``conjugating'' a braid $\beta$, we are performing a cyclic conjugation thereof. 
\item When possible, we shade different sectors of a branched surface in different colors. 
\end{itemize}

\bigskip

%\begin{wrapfigure}[1]{R}{.18\linewidth}
\begin{figure}[h]
\labellist \small
\pinlabel $v$ at 58 246
\pinlabel $w$ at 246 57 
\pinlabel {$\langle v, w \rangle = 1$} at 180 140
\endlabellist
%\begin{minipage}[c][.1\linewidth]{20mm}
\centering
\includegraphics[scale=.25]{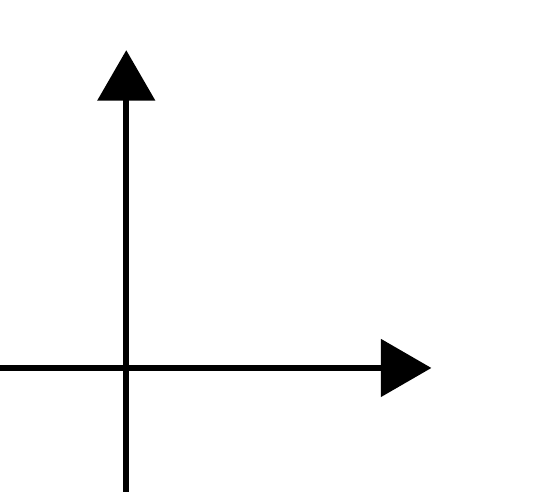}
\caption{Our conventions for computing the intersection number.} 
%\end{minipage}
\label{fig:intersection_number}
\end{figure}
%\end{wrapfigure}

\newpage
\section{Detecting the unknot} \label{section:Cabling}

In this section, we prove \Cref{UnknotDetection}. Throughout, we denote the $(p,q)$--cable of $K$ by $K_{p,q}$, where the cable traverses the longitudinal direction of the knot $p$ times (where $p \geq 2$), and the meridional direction $q$ times. We begin by proving:

\noindent \textbf{\Cref{UnknotDetection}.} 
\textit{$K_{n,\pm 1}$ is braid positive $\iff K$ is the unknot.} 

\begin{proof}
Throughout, we assume that $n \geq 2$. 

The ``if'' direction is straightforward: if $K$ is the unknot, then the $K_{n,\pm 1} \approx T(n,\pm 1)$, a positive braid knot (which is isotopic to the unknot in $S^3$). 

To obtain the converse, we prove the contrapositive: when $g(K)\geq 1$, the $(n,\pm1)$--cable of $K$ is never braid positive. 

Suppose $g(K) \geq 1$. The genus formula for cables (\cite{Schubert}, c.f. \cite{Schultens:SatelliteKnots}) tells us that 
\begin{align*}
g(K_{n,\pm 1})&=ng(K) + g(T(n,\pm 1)) \\
&= ng(K) + 0 \\
&= ng(K)
\end{align*}

Note: since $g(K) \geq 1$ and $n \geq 2$, then $g(K_{n, \pm 1}) \geq 2$. 

If $K_{n, \pm 1}$ is braid positive, then we could apply \Cref{thm:main} to deduce that $S^3_r(K_{n, \pm 1})$ has a taut foliation for all $r < g(K_{n, \pm 1}) +1$. Thus, we have taut foliations in $S^3_r(K_{n,\pm 1})$ for all $r$ such that
\begin{align*}
r < g(K_{n, \pm 1}) +1 &= ng(K)+1
\end{align*}

Since $g(K) \geq 1$, it follows that $\pm n < n +1 \leq ng(K) + 1$. We deduce that $S^3_{\pm n}(K_{n,\pm 1})$ admits a taut foliation. 

On the other hand, the cabling construction tells us that $$S^3_{pq}(K_{p,q}) \approx L(p,q) \# S^3_{q/p}(K)$$ (see \cite[Section 2.4]{Hedden:CablingII} for more details). In this case, we have $p = n$ and $q = \pm 1$, and so $$S^3_{\pm n}(K_{n,\pm 1}) \approx L(n, \pm 1) \# S^3_{\pm 1/n}(K)$$ Since $g(K) \geq 1$, $K$ is non-trivial; the Gordon-Luecke theorem tells us that $S^3_{\pm 1/n}(K) \not \approx S^3$ \cite{GordonLuecke:KnotComplements}. Therefore, $S^3_{\pm n}(K_{n, \pm 1})$ has two non-trivial summands, hence it is a reducible manifold. However, reducible manifolds cannot admit taut foliations (if $M$ admits a taut foliation, the universal cover $\widetilde M$ is homeomorphic to $\R^3$ \cite{GabaiOertel}; in contrast, if $M$ is reducible, then $\widetilde{M} \not \approx \R^3$ \cite{Hatcher:3Manifolds}). This yields a contradiction, and we deduce that $K_{n, \pm 1}$ is not braid positive. 
\end{proof}

The rest of this work is focused on proving \Cref{thm:main}.

\newpage
%%%%%%%%BACKGROUND SECTION

\section{Background} \label{section:background}

%%%%%Monodromies of positive braids
\subsection{Factorizing the monodromy of the mapping torus of positive braid closures} \label{section:monodromy}

It is well known that positive braid knots are fibered knots in $S^3$; see \cite{Stallings:Fibered} or \cite{Misev:Thesis}. In this section, we will prove something stronger: that any positive braid word $\beta$ for a positive braid knot $K$ provides an algorithm for building the fiber surface $F$ for $K$, and for identifying a factorization of the monodromy of the mapping torus. Our proof relies on realizing the surface $F$ as a plumbing of positive Hopf bands. This algorithm is well known to experts -- however, as we heavily rely on the construction throughout, we include a proof here for completeness.

We recall: a positive Hopf band is a fibered link in $S^3$. The fiber surface is an annulus $A$, and the monodromy $\varphi$ is a positive Dehn twist about the core curve $c$. Let $H_1$ and $H_2$ denote two positive Hopf links, with monodromies $\varphi_1$ and $\varphi_2$ respectively. The simplest way to produce a new fibered link from $H_1$ and $H_2$ is via the \textit{plumbing} operation, which requires specifying an arc $\alpha$ on the fiber surface of $H_1$; $\alpha$ is often called the \textit{plumbing arc}. Plumbing $H_2$ onto $H_1$ yields a new fibered link \cite{Stallings:Fibered}. As explained in \cite{Gabai:MurasugiSumII}, the resulting monodromy $\varphi$ is obtained via precomposition: that is, $\varphi \approx \varphi_1 \circ \varphi_2$ (and we read this from right--to--left). Next, consider a sequence of plumbings of Hopf links $H_1, \ldots, H_k$ (where $H_i$ is plumbed onto $H_{i-1}$). The result is a fibered link $L$ with monodromy $\varphi = \varphi_1 \circ \ldots \circ \varphi_k$, where $\varphi_i$ is a positive Dehn twist about the core curve of $H_i$. We emphasize: as will be illustrated explicitly in our proof, the choice of plumbing arc is key.

\subsubsection{Preliminaries}

By classical work of Neuwirth, the fiber surface of a fibered knot $K$ in $S^3$ is unique up to isotopy. 
If $\beta$ is a positive braid on $n$ strands, we claim (and show below) that the fiber surface is just the \textbf{Bennequin surface} $F$, which is built by attaching positively twisted bands to $n$ disks, as specified by a \textit{particular} braid word.

\begin{defn}
Suppose $\beta$ is a braid on $n$ strands. The Bennequin surface $F$ is built from $n$ disks, which we call \textbf{Seifert disks}, and reading from left-to-right, we label them by $S_i$, for $1 \leq i \leq n$. 
\end{defn}

\begin{defn}
A \textbf{presentation} for a braid is a braid word to which we choose not to apply braid relations. 
\end{defn}

For example, suppose $\beta_1 = \omega_1 \sigma_i \sigma_{i+1} \sigma_i \omega_2$ and $\beta_2 = \omega_1 \sigma_{i+1} \sigma_{i} \sigma_{i+1} \omega_2$. We see that $\beta_1$ and $\beta_2$ are related by a single braid relation, hence the represent the same word in the braid group. Nevertheless, we want to consider them as different presentations of the same braid. 

\begin{rmk}
%While the Bennequin surface is unique it is specified by a \underline{particular} braid word. 
\textup{Throughout this work, when we refer to \underline{the} Bennequin surface $F$, there is an implicit presentation of the braid in mind. We specify presentations of a braid (and cyclic conjugations thereof) only when there may be ambiguity about the assumed presentation. }
\end{rmk}

\begin{defn}
Fix a presentation of a braid. The total number of $\sigma_i$ letters in $\beta$ is $c_i$. 
\end{defn}

\begin{defn}
Fix a presentation of $\beta$. Let $F$ denote the Bennequin surface for $\widehat{\beta}$. 
\begin{itemize}
\item The $j^{th}$ positively twisted band connecting $S_i$ and $S_{i+1}$ is denoted by $\mathbbm{b}_{i,j}$.
\item The $j^{th}$ plumbing arc on $S_i$ is denoted $\alpha_{i,j}$.
\item The $j^{th}$ occurrence of a $\sigma_i$ letter in $\beta$ is denoted $\sigma_{i,j}$.
\end{itemize}
\end{defn}

\begin{defn} \label{defn:column}
The \textbf{$i^{th}$ column of $F$}, denoted $\Gamma_i$, is the union of the Seifert disks $S_{i}$, $S_{i+1}$, and the $($positively twisted$)$ bands $\mathbbm{b}_{i,1}, \ldots, \mathbbm{b}_{i, c_i}$ connecting them.
\end{defn}

\begin{defn}
Fix a presentation of a braid $\beta$. For $s \neq t$, we define the function  $d(j; s, t):$
% records the number of $\sigma_{j+1}$ bands between $\mathbbm{b}_{j,s}$ and $\mathbbm{b}_{j,t}$. That is, 
\begin{align*}
d(j; s, t) = 
%	\begin{cases}
%	\text{the number of $\sigma_{j+1}$ bands between $\mathbbm{b}_{j,s}$ and $\mathbbm{b}_{j,t}$} & $s < t$ \\
%	c_{j+1} - d(j; t, s) & $t < s$
%	\end{cases}
	\begin{cases}
	\text{the number of $\sigma_{j+1}$ letters between $\sigma_{j,s}$ and $\sigma_{j,t}$} & \qquad $s < t$ \\
	c_{j+1} - d(j; t, s) & \qquad $t < s$
	\end{cases}
\end{align*}
Analogously, $d_L(j; s,t)$ records the number of $\sigma_{j-1}$ bands between $\mathbbm{b}_{j,s}$ and $\mathbbm{b}_{j,t}$.
\end{defn}

That is, if we fixed a presentation of $\beta$ and looked at the corresponding Bennequin surface, $d(j; s,t)$ counts the number of bands in $\Gamma_{j+1}$ which are between $\mathbbm{b}_{j,s}$ and $\mathbbm{b}_{j,t}$.

\begin{rmk}
\textup{The quantities $d(j; s,t)$ appear frequently throughout our proofs, while the values $d_L(j; s,t)$ will appear sparingly throughout $($hence the asymmetry in the naming conventions$)$.} 
\end{rmk}

\subsubsection{Identifying fiber surfaces and monodromies}
With this setup in place, we can now build the fiber surface $F$ in stages, one column at a time, from left--to--right. We build the first column, and then continue building the surface $F$ by first stabilizing, and then plumbing Hopf bands based on the relative positions of the $\sigma_i$'s with respect to the $\sigma_{i-1}$'s. Throughout, we present our bands so that we only see the ``positive'' side of the surface, \`a-la Rudolph \cite{Rudolph:QPsliceness}.

\begin{lemma} \label{lemma:monodromy}
Let $\beta$ be a positive braid on $n$ strands. The standard Bennequin surface $F$, obtained by attaching positively twisted bands between $n$ disks, is a fiber surface for $\widehat{\beta}$. Moreover, any conjugated presentation of the braid word determines a factorization of the monodromy of $F$. 
\end{lemma}

%We prove \Cref{lemma:monodromy} and explicitly demonstrate our algorithm for the braid $\beta$ in (\ref{ex:braid}) below:
%
% \begin{align} \label{ex:braid}
% \beta \approx \sigma_1 \sigma_2^2 \sigma_1^2 \sigma_2
% \end{align}

%This algorithm is neatly realized in the positive braid setting: indeed, a plumbing sequence can be read off directly from the braid word. After establishing some preliminaries, 

%We build the fiber surface $F$ in stages, one column at a time. We first build the surface $F_1$, corresponding to the first column. We iteratively build the surface $F_i$ from $F_{i-1}$ by first stabilizing, and then plumbing Hopf bands based on the relative positions of the $\sigma_i$'s with respect to the $\sigma_{i-1}$'s. \\

%We now return to showing that the Bennequin surface for a positive braid knot is fibered. 

%the first $\sigma_2$ in $\beta$ occurs after some number of $\sigma_1$s. 

\begin{proof}

We prove \Cref{lemma:monodromy} alongside an example, where we let $\beta$ be the braid in (\ref{ex:braid}) below:
\begin{align} \label{ex:braid}
\beta \approx \sigma_1 \sigma_2^2 \sigma_1^2 \sigma_2
\end{align}

Our goal is to explicitly read off a Hopf plumbing sequence from a presentation for a braid. 
Given any braid, we can always cyclically conjugate it so that it begins with a $\sigma_1$ letter. So, we fix a presentation of $\beta$ that begins with a $\sigma_1$ letter. \\

\tcbox[size=fbox, colback=gray!30]{Step 0: Build the unknot as a braid closure.}

First, we build the fiber surface for the unknot by attaching a single positively twisted band between two disks, as in \Cref{fig:HopfPlumbing_F1} (left). This presents the unknot as the closure of the braid $\sigma_1 \in \mathcal{B}_2$, which is the braid group on two strands. \\

\begin{figure}[h]\center
\labellist \tiny
\pinlabel {isotopy} at  270 120
\pinlabel {plumb a} at  610 122
\pinlabel {Hopf band} at  610 96
\endlabellist
\includegraphics[scale=.45]{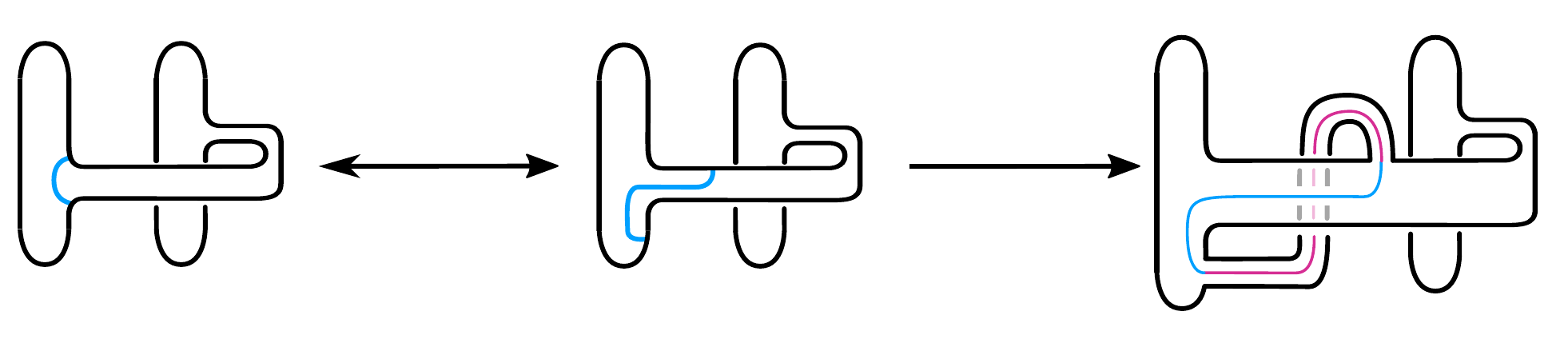} 
\caption{Plumbing a positive Hopf band to the fiber surface for the unknot. The monodromy of the result is a positive Dehn twist about $\gamma$, which is the union of the blue and pink arcs.} 
\label{fig:HopfPlumbing_F1}
\end{figure}

\newpage
%%%%%%%%%%%%
\tcbox[size=fbox, colback=gray!30]{Step 1: Build the first column, $\Gamma_1$.} 

We build the first column of the fiber surface $F$: recall that $c_1$ is the number of $\sigma_1$ letters in $\beta$. Build the torus link $T(2, c_1)$ by plumbing $c_1-1$ Hopf bands in sequence. Each Hopf band is plumbed along an arc which is isotopic to the co-core of a band. See \Cref{fig:HopfPlumbing_F1} (middle/right) to see the plumbing of the first Hopf band. An isotopy producing $T(2, 2)$ is seen in \Cref{fig:HopfPlumbing_F2}. 

\begin{figure}[h]\center
\labellist \tiny
\pinlabel {isotopy} at  378 210
\endlabellist
\includegraphics[scale=.45]{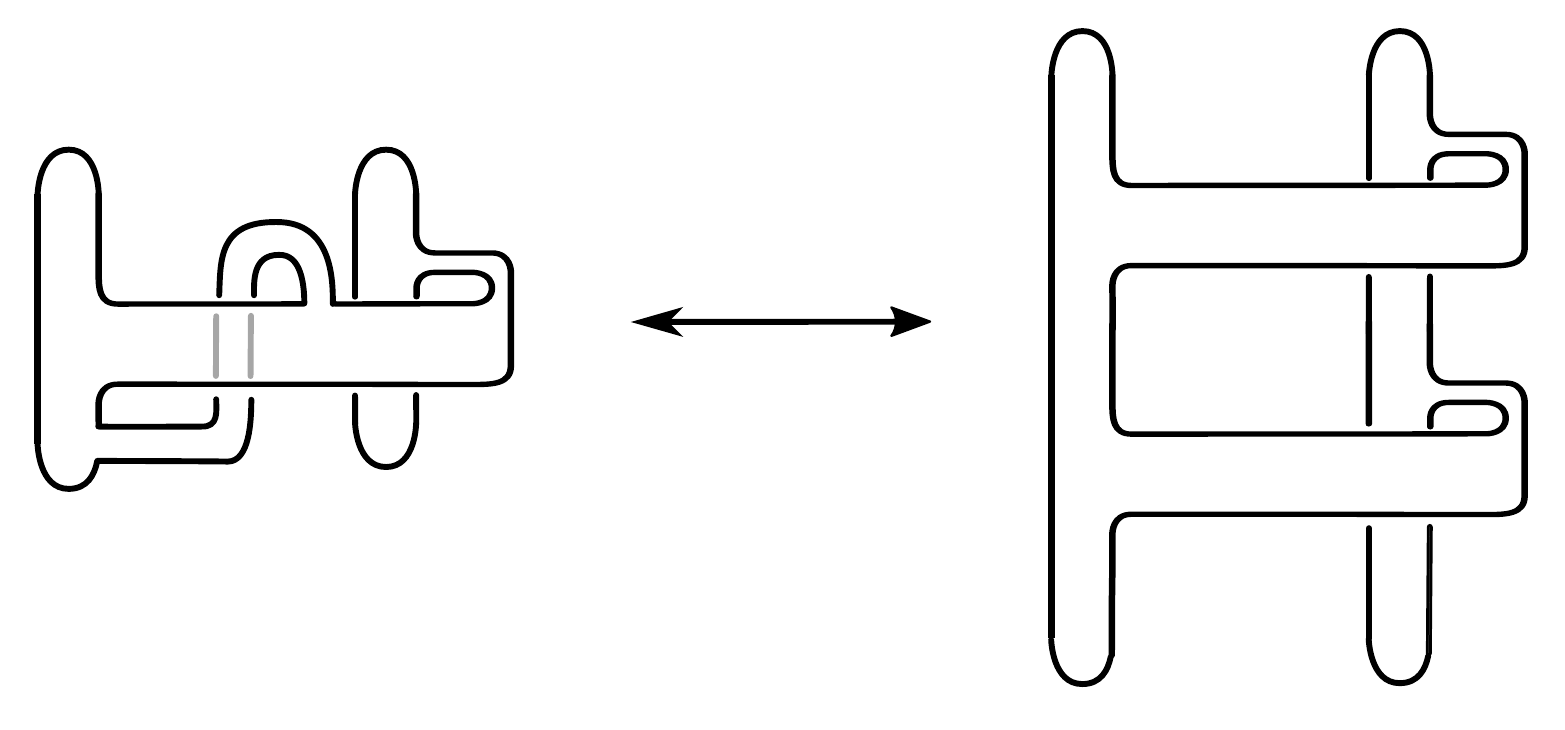} 
\caption{The last frame of \Cref{fig:HopfPlumbing_F1} and the standard fiber surface for the link $T(2,2)$ are related by an isotopy.} 
\label{fig:HopfPlumbing_F2}
\end{figure}

For the braid $\beta$ in (\ref{ex:braid}), $c_1=3$, so we plumb another Hopf band to produce $T(2,3)$. \Cref{fig:HopfPlumbing_F3} contains the result of plumbings and isotopies. 

\begin{figure}[h]\center
\labellist \tiny
\pinlabel {plumb $+$} at  260 192
\pinlabel {isotopy} at  260 165
\endlabellist
\includegraphics[scale=.45]{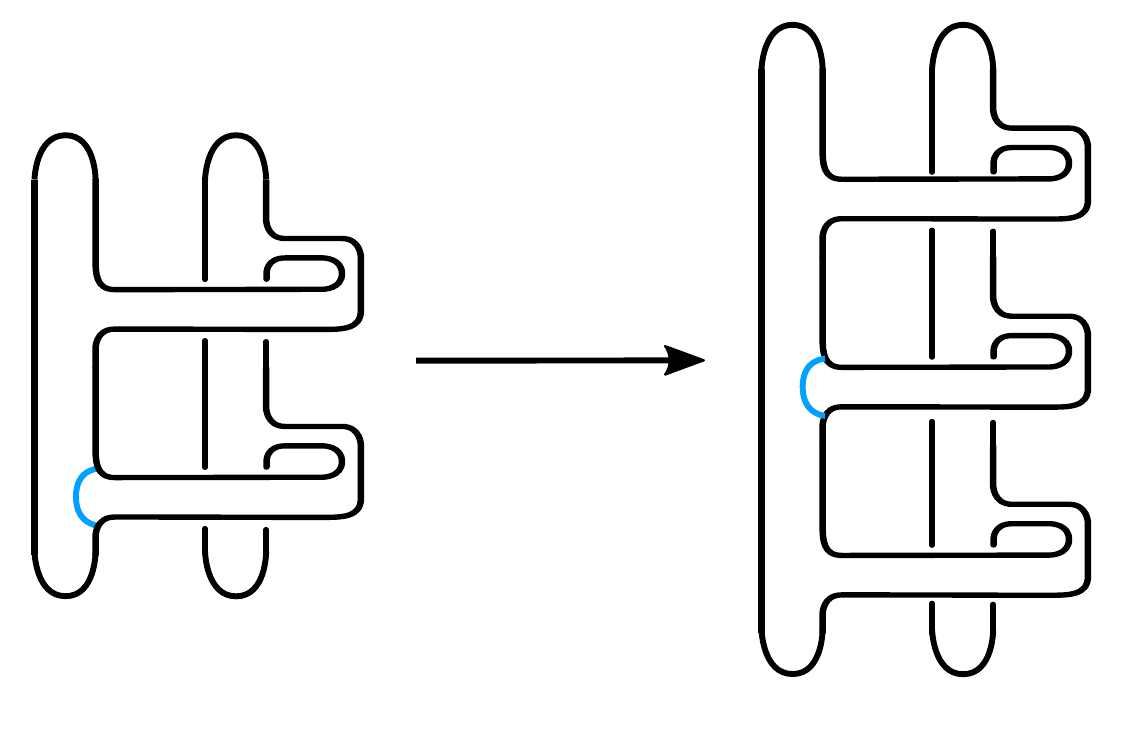} 
\caption{Plumbing along the co-core of the lower band (i.e. along the blue arc) yields, after an isotopy, the fiber surface for $T(2,3)$.}
\label{fig:HopfPlumbing_F3}
\end{figure}

%%%%%%%%%%%%
\tcbox[size=fbox, colback=gray!30]{Step 2: Stabilize and plumb onto $\Gamma_1$ to produce $\Gamma_1 \cup \Gamma_2$} 

Next, we stabilize $\Gamma_1$, such that the stabilization occurs at the site of the first $\sigma_2$ in $\beta$, as in \Cref{fig:HopfPlumbing_F4} (middle/left).  The post-stabilization surface is isotopic to the fiber surface for $T(2, c_1)$; however, this surface is now realized as the Bennequin surface of a braid on three (rather than two) strands. Next, we sequentially plumb $c_2 - 1$ Hopf bands, where all the plumbing arcs lie in the Seifert disk $S_2$. To do this, we need to identify plumbing arcs in $S_2$.

\begin{figure}[h]\center
\labellist \tiny
\pinlabel {isotopy} at  272 260
\pinlabel {stabilize} at  615 261
\endlabellist
\includegraphics[scale=.42]{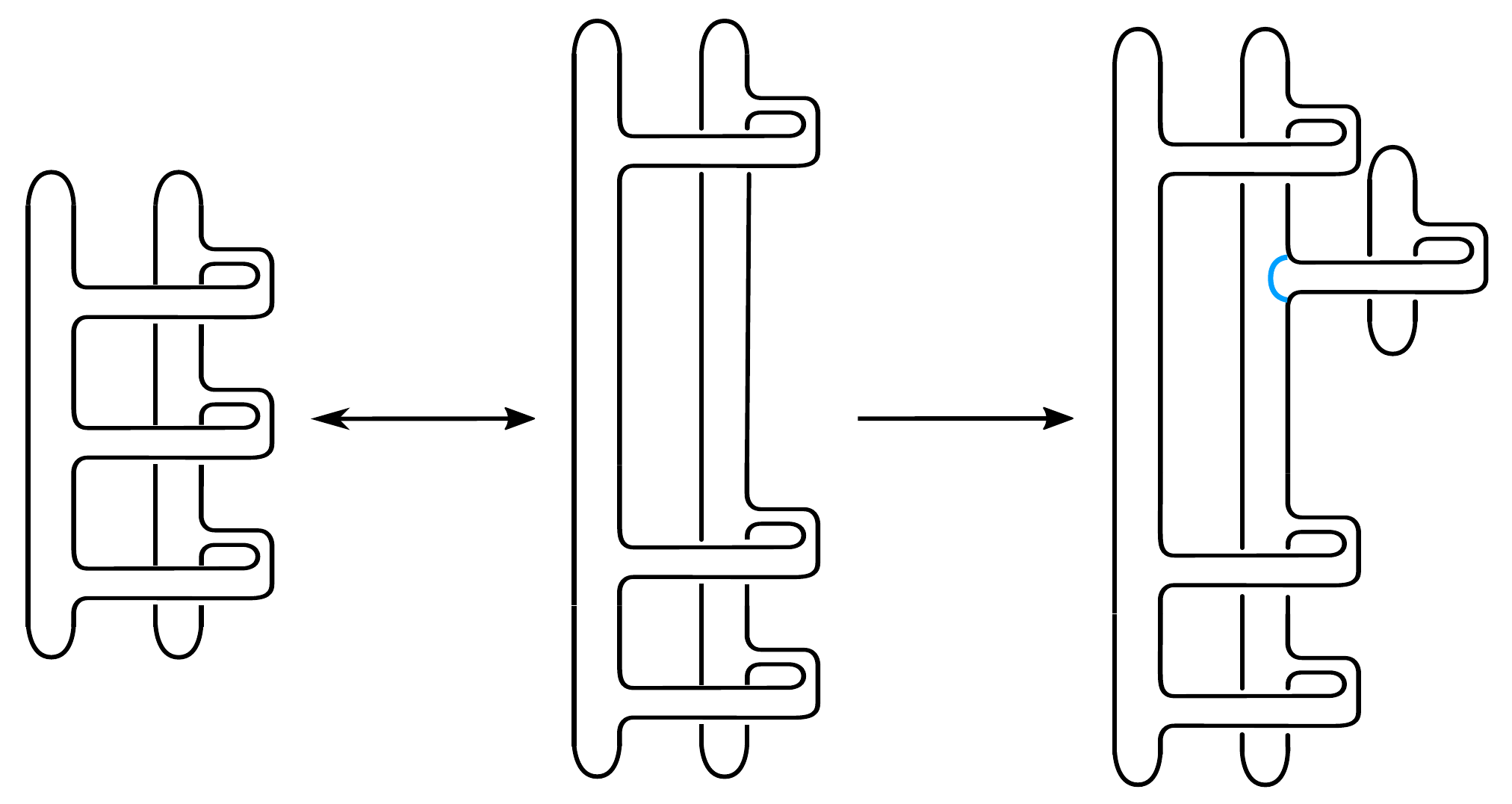} 
\caption{We isotope and stabilize our surface such that the first $\sigma_2$ letter appears in the correct location relative to the $\sigma_1$ letters. The blue arc in the rightmost frame is the first plumbing arc in $S_2$, and we will plumb a postive Hopf band along this arc.} 
\label{fig:HopfPlumbing_F4}
\end{figure}

We already identified the first occurrence of $\sigma_2$ in $\beta$ (this determined the site of the stabilization). Now, identify the second occurrence of $\sigma_2$; it occurs after $d_L(2; 1,2)$ occurrences of $\sigma_1$. Indeed, in Example \ref{ex:braid}, $d_L(2; 1,2) = 0$, which corresponds to the fact that $\mathbbm{b}_{2,1}$ and $\mathbbm{b}_{2,2}$ are consecutive and both fall between $\mathbbm{b}_{1,1}$ and $\mathbbm{b}_{1,2}$. 

We now construct the plumbing arc $\alpha_{2,1}$. The plumbing arc will be a simple arc in $S_2$, so it suffices to determine the locations of the upper and lower endpoints of $\alpha_{2,1}$. We denote these endpoints by $u_{\alpha_{2,1}}$ and $\ell_{\alpha_{2,1}}$, respectively. The upper endpoint $u_{\alpha_{2,1}}$ will lie ``above'' the attachment site of the previous $\sigma_2$ band, and the lower endpoint $\ell_{\alpha_{2,1}}$ will lie ``below'' the right attachments sites of the $d_L(2; 1,2)$ bands in $\Gamma_1$ between $\sigma_{2,1}$ and $\sigma_{2,2}$. Connecting $u_{\alpha_{2,1}}$ and $\ell_{\alpha_{2,1}}$ via a simple geodesic arc in $S_2$ yields $\alpha_{2,1}$. For our example braid in (\ref{ex:braid}), $d_L(2; 1,2)=0$, so $\alpha_{2,1}$ is isotopic to the co-core of $\mathbbm{b}_{2,1}$; see \Cref{fig:HopfPlumbing_F4} (right).

\begin{figure}[h]\center
\labellist \tiny
\pinlabel {isotopy} at  360 315
\pinlabel {plumb} at  810 318
\endlabellist
\includegraphics[scale=.32]{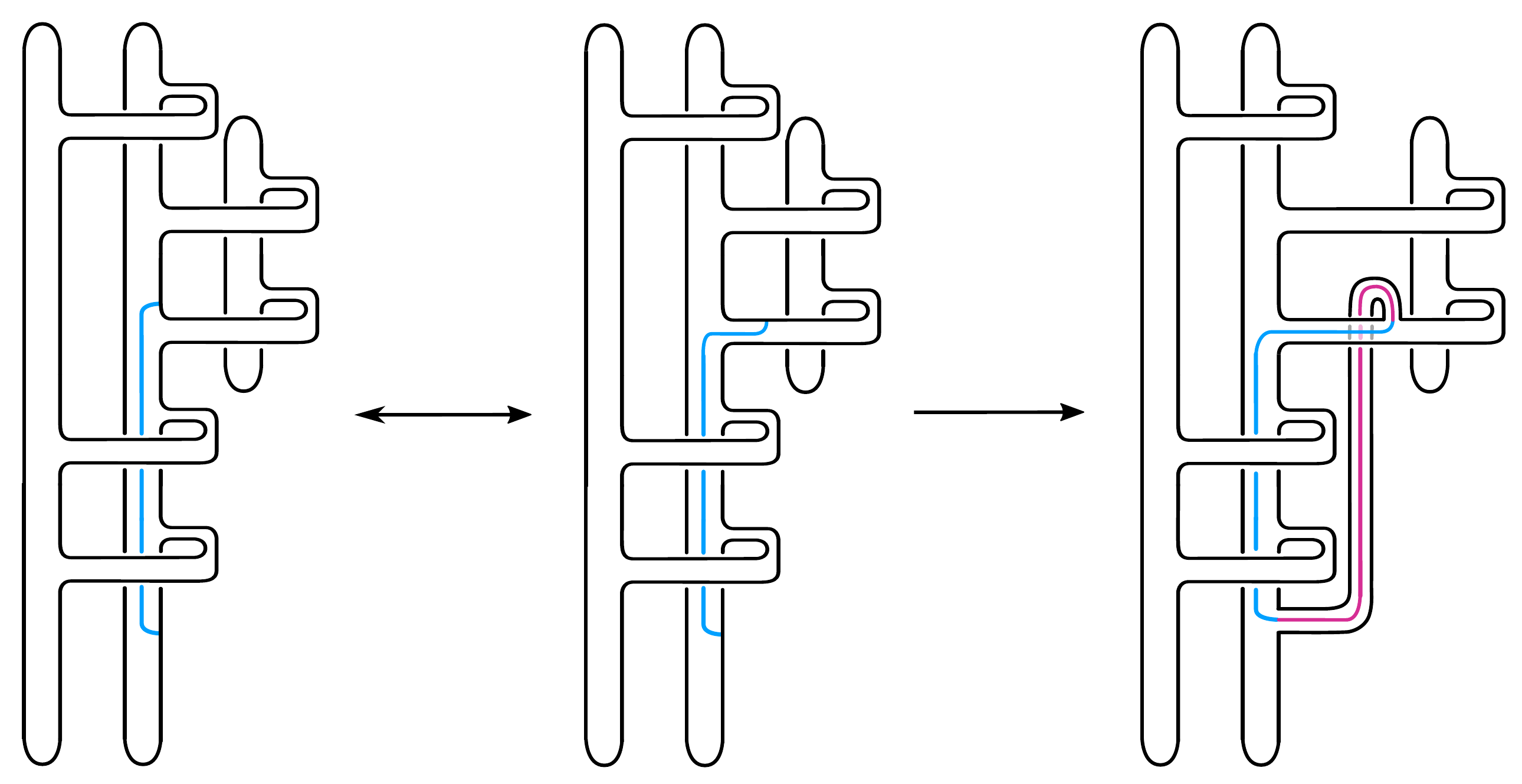} 
\caption{Plumbing along the blue arc in the last frame of \Cref{fig:HopfPlumbing_F4} yields the surface in the left frame. After isotopy, we plumb along the blue arc in the middle frame to produce the rightmost surface.}
\label{fig:HopfPlumbing_F5}
\end{figure}

We repeat this process and exhaust all occurrences of $\sigma_2$ in $\beta$. More precisely, to construct $\alpha_{2,t}$, for $1 \leq t \leq c_2-1$, we: compute $d_L(2; t, t+1)$, place $u_{\alpha_{2,t}}$ ``above'' the left attachment site of $\mathbbm{b}_{2,t}$, place $\ell_{\alpha_{2,t}}$ ``below'' the right attachment sites of the $d_L(2; t, t+1)$ bands to $S_2$, and then draw a simple arc in $S_2$ between them. The result of this process is a realization of the Bennequin surface in $\Gamma_1 \cup \Gamma_2$ as a Hopf plumbing. \Cref{fig:HopfPlumbing_F5} shows this procedure for our braid in (\ref{ex:braid}), where $d_L(2; 2,3) = 2$.

\tcbox[size=fbox, colback=gray!30]{Step 3: Exhaust all $\sigma_i$'s in $\beta$ to build $F$.} 

Repeating this procedure will yield $F$: for each $3 \leq i \leq n-1$, stabilize $S_i$ such that the stabilization occurs at the location of the first occurrence of $\sigma_i$ in $\beta$. Then, use the braid word to identify plumbing arcs in the Seifert disk $S_i$, and then plumb $c_i - 1$ Hopf bands onto the stabilized surface. As our example in (\ref{ex:braid}) is a 3-braid, we are done. \hfill \fbox{End of procedure.} \\

This procedure allows us to read off an explicit sequence of Hopf plumbings for the Bennequin surface $F$ from a presentation of the braid word $\beta$. We deduce that $F$ is a fiber surface for $\widehat{\beta}$. Moreover, each pair of consecutive bands between adjacent Seifert disks specifies a simple closed curve (for simplicity we have only displayed two such curve in these figures; they are in Figures \ref{fig:HopfPlumbing_F1} and \ref{fig:HopfPlumbing_F5}). The knot exterior has the structure of a mapping torus, and the monodromy of the fibration is a sequence of positive Dehn twists about these simple closed curves obtained from the plumbings: from ``bottom--to--top'', we Dehn twist about the simple closed curves in $\Gamma_{n-1}$, then $\Gamma_{n-2}$, then $\Gamma_{n-3}$, until we reach the first column $\Gamma_1$. This yields both the fiber surface $F$ for $\widehat{\beta}$, and an explicit factorization of the monodromy of $F$. 
\end{proof}

\begin{rmk}
\textup{This procedure also works for alternating and homogeneous braids. For these braids, the plumbing involves positive and negative Hopf bands, in accordance with the sign of the $\sigma_i$ in $\beta$.}
\end{rmk}

\begin{figure}[h]\center
%\labellist \tiny
%\pinlabel {isotopy} at  360 325
%\pinlabel {plumb} at  810 325
%\endlabellist
\includegraphics[scale=.25]{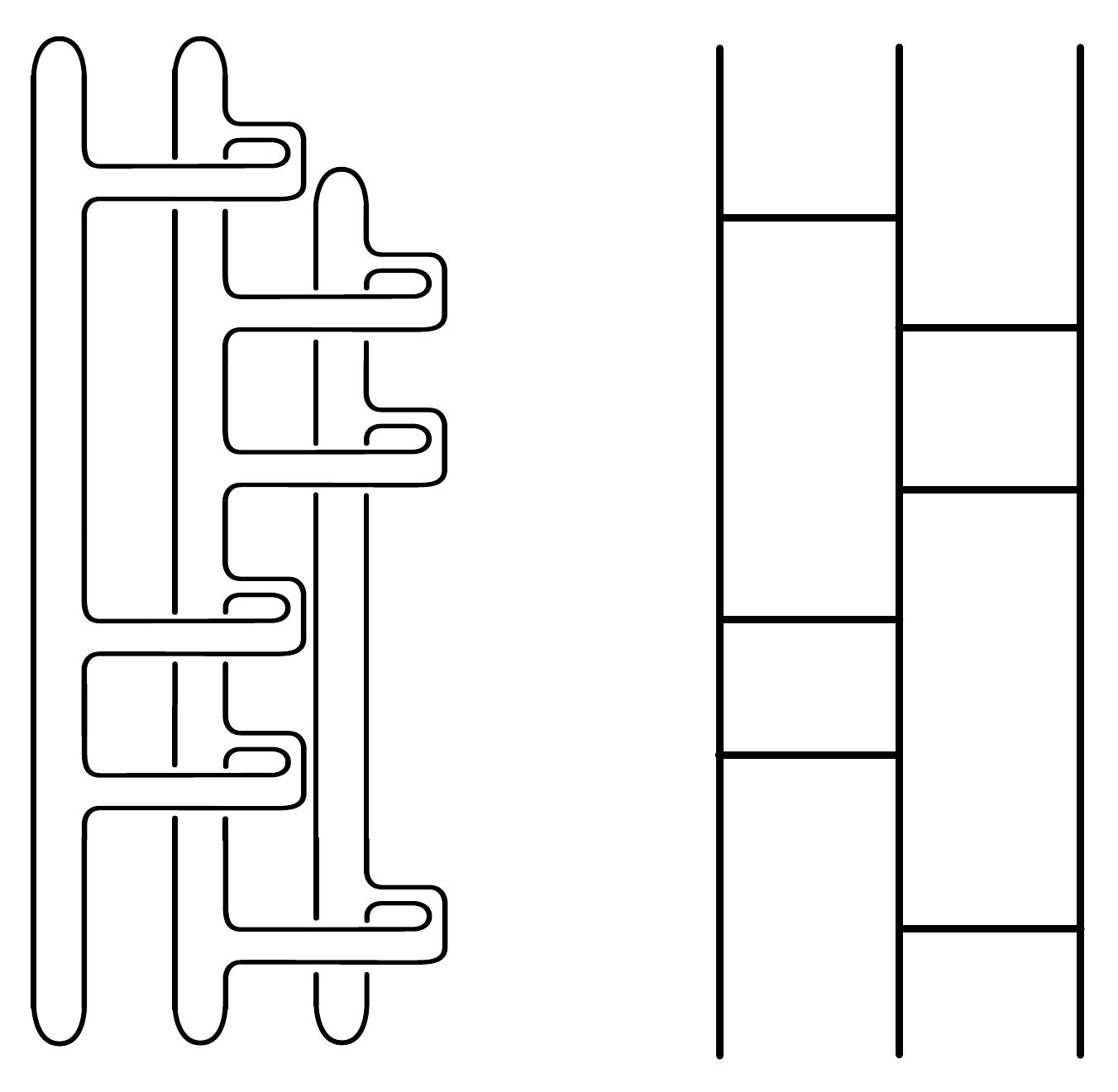} 
\caption{Isotoping the rightmost surface in \Cref{fig:HopfPlumbing_F5} yields the surface on the left; it is the fiber surface for the closure of $\sigma_1 \sigma_2^2 \sigma_1^2 \sigma_2$. The corresponding brick diagram is on the right.}
\label{fig:BrickDiagram}
\end{figure}

\Cref{cor:image} follows from the proof of \Cref{lemma:monodromy}.

\begin{cor} \label{cor:image}
Fix a presentation of a positive $n$-braid $\beta$. 
Suppose $K \approx \widehat{\beta}$  with fiber surface $F$ and monodromy $\varphi$. 
For any plumbing arc $\alpha_{i,j}$, $1 \leq i \leq c_{n-1}$, and $1 \leq j \leq c_i-1$, we completely understand $\varphi(\alpha_{i,j})$: viewed as an arc on $F$, $\varphi(\alpha_{i,j})$ has the same endpoints $u_{\alpha_{i,j}}, \ell_{\alpha_{i,j}}$ as $\alpha_{i,j}$, and travels from the former to the latter by passing through $\mathbbm{b}_{i,j}, S_{i+1}, \mathbbm{b}_{i, j+1}$ and $S_{i}$.
\hfill $\Box$
\end{cor}

%Note: the product disks used in our forthcoming construction are exactly the sweep--outs of the plumbing arcs used to build $F$. 
%\textcolor{red}{Define product disks somewhere, and say how the plumbing structure tells you exactly what the image of a plumbing arc is under the monodromy.}

%\begin{wrapfigure}{R}{.45\linewidth}
%\begin{minipage}[c][.14\paperheight]{80mm}
%\begin{raggedright}
%\vspace{1.5cm}
%\includegraphics[scale=.2]{BrickDiagram} 
%\end{raggedright}
%%\vspace*{8mm}
%\captionof{figure}{}%{The brick diagram for the braid in \ref{ex:braid}.}
%\label{fig:BrickDiagram}
%\end{minipage}
%\end{wrapfigure}

To simplify our diagrams and figures, we will often present positive braids via \textbf{brick diagrams}, where we draw a single vertical line for each strand of the braid, and a horizontal line between adjacent vertical lines, each representing a positive band attached between disks. \Cref{fig:BrickDiagram} presents a brick diagram for the braid in (\ref{ex:braid}). Precedent for exhibiting positive braids in this manner can be found in \cite{BaaderLewarkLiechti, Baader:MaximalSignature}; they, too, use these diagrams as a convenient way to encode a braid and ``see'' the associated Bennequin surface.

\subsection{Branched surfaces and taut foliations} \label{background:foliations} Our technique for constructing taut foliations involves \textit{branched surfaces}. 

\begin{defn}
A \textbf{spine} for a branched surface is a co--oriented 2--complex in a 3--manifold, locally modeled by \Cref{fig:spine}.
\end{defn}

\begin{figure}[h]\center
\includegraphics[scale=.25]{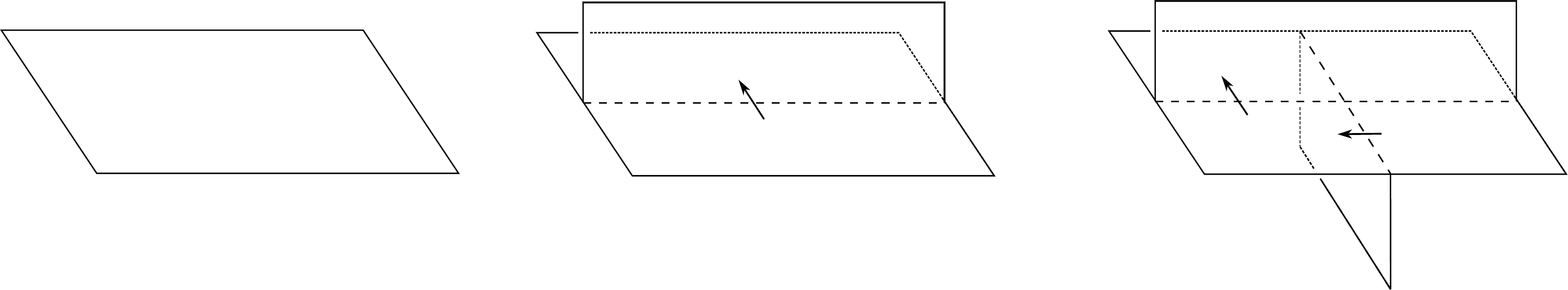} 
\caption[The local model for the spine of a branched surface.]{The local model for the spine of a branched surface.} 
\label{fig:spine}
\end{figure}

\begin{defn}
A \textbf{branched surface $B$} in a 3--manifold $M$ is obtained by providing co--orientations (i.e. smoothing directions) to a spine, and is locally modeled as in \Cref{fig:branched_surface}.
\end{defn}

\begin{figure}[h]\center
\includegraphics[scale=.25]{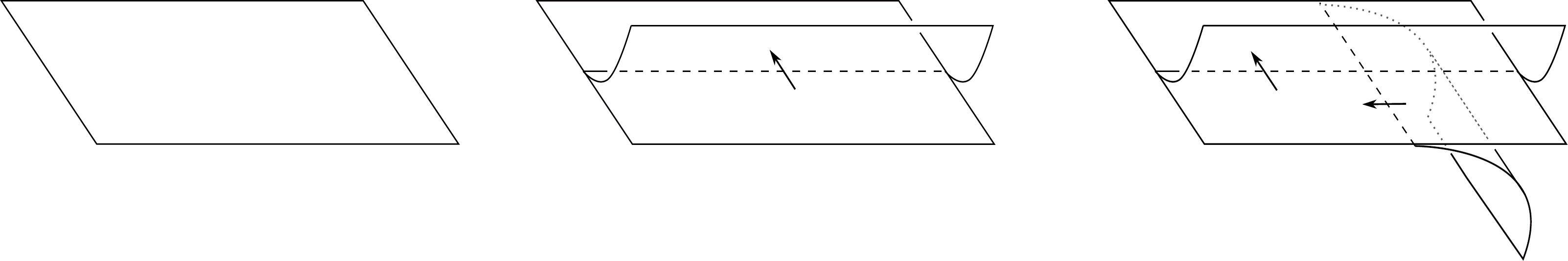} 
\caption{The local model for a branched surface.} 
\label{fig:branched_surface}
\end{figure}

Branched surfaces are used to build essential laminations and taut foliations in 3--manifolds. In \cite{TaoLi:SinkDisk, TaoLi:BoundarySinkDisk}, Li showed that \textbf{laminar branched surfaces} always carry essential laminations: 

\begin{thm}[Theorem 2.5 in \cite{TaoLi:BoundarySinkDisk}]\label{thm:taolisinkdisk}
Let $M$ be an orientable, irreducible 3--manifold, such that the boundary $\partial M$ is an incompressible torus. Suppose $B$ is a properly embedded branched surface in $M$ such that the following properties hold:

\begin{enumerate}
\item[(1a)] $\partial_h(N(B))$ is incompressible and $\partial$-incompressible in $M - \text{int}(N(B))$
\item[(1b)] There is no monogon in $M - \text{int}(N(B))$
\item[(1c)] No component of $\partial_h N(B)$ is a sphere or a disk properly embedded in $M$
\item[(2)] $M - \text{int}(N(B))$ is irreducible and $\partial M - \text{int}(N(B))$ is incompressible in $M - \text{int}(N(B))$
\item[(3)] $B$ contains no Reeb branched surface $($see \cite{GabaiOertel} for more details$)$
\item[(4)] $B$ does not contain any sink disks or half sink disks.
\end{enumerate}
Then $B$ is a \textbf{laminar branched surface}. Let $r$ be any slope in $\Q \cup \{\infty\}$ fully carried by the boundary train track $\tau_B = B \cap \partial X_K$. If $B$ does not carry a torus that bounds a solid torus in $M(r)$ $($the manifold obtained by $r$-framed Dehn filling$)$ then $B$ carries an essential lamination in $M$ meeting the boundary torus in parallel simple closed curves of slope $r$. Moreover, $M(r)$ contains an essential lamination. 
\end{thm}

The key condition in this list is having no sink disks, or being \textbf{sink disk free}:

\begin{defn}
A \textbf{sink disk} \cite{TaoLi:SinkDisk} is a branch sector $S$ of $B$ such that:
\begin{enumerate}
\item $S$ is homeomorphic to a disk, 
\item $\partial S \cap \partial M = \varnothing$, and
\item the branch direction of every smooth arc or curve in its boundary points into the disk.
\end{enumerate}
If $S$ does meet $\partial M$, and conditions (1) and (3) above hold, then $S$ is called a \textbf{half sink disk} \cite{TaoLi:BoundarySinkDisk}. We remark that $\partial S \cap \partial M$ may not be connected. When a branched surface $B$ contains no sink disk or half sink disk, we say $B$ is \textbf{sink disk free}. See \Cref{fig:sinkdiskdefn}.
\end{defn}

\begin{figure}[h]\center
\includegraphics[scale=.5]{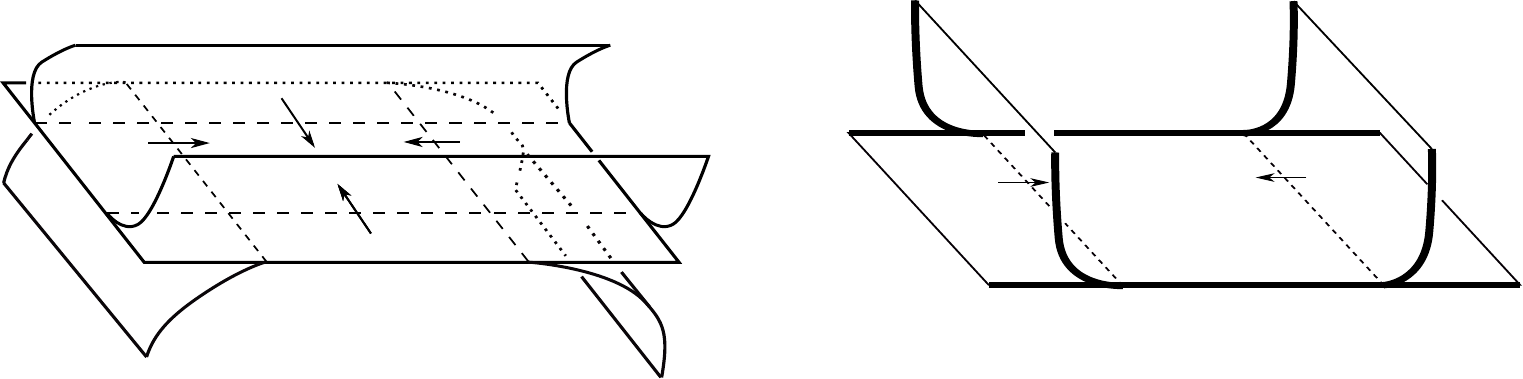}
\caption{On the left, a \textbf{sink disk}. On the right, the \textbf{bolded} lines lie on $\partial M \approx T^2$; this is the abstract model for a \textbf{half sink disk}. }
\label{fig:sinkdiskdefn}
\end{figure}

In this paper, our branched surfaces are built from a copy of the fiber surface $F$ and a collection of co--oriented \textbf{product disks}, in the sense of Gabai \cite{Gabai:Fibered}: in particular, these are the sweep--outs of properly embedded essential arcs on $F$ (for a more detailed summary, see \cite[Section 2]{Krishna:3Braids}). We recall the key details within the context of positive braids: let $K$ be a positive braid knot and $F$ denote its fiber surface. Let $\alpha$ be a plumbing arc on $F$, and let $\mathbb{D}_\alpha$ denote the product disk obtained by pushing $\alpha$ through the mapping torus. We can decompose $\partial \mathbb{D}_\alpha$ into four pieces: $\mathbb{D}_\alpha \cap F^+, \mathbb{D}_\alpha \cap F^{-}$, and $\mathbb{D}_\alpha \cap \partial X_K \approx \partial \alpha \times I$ (i.e. the trace of the endpoints of $\alpha$ through the monodromy). %In particular, our branched surface $B$ meets $\partial X_K$ in a train track formed from the Seifert longitude $\lambda$, and pairs of arcs (one pair for each disk $\mathbb{D}_\alpha$).

Carefully specifying co--orientations for the product disks yields a laminar branched surface:

\begin{prop} \cite[Proposition 3.12]{Krishna:3Braids} \label{prop:BisLaminar}
Let $K$ be a fibered knot in $S^3$. Suppose $B$ is a branched surface built from  a copy of the fiber surface and a collection of co--oriented product disks. If $B$ is sink disk free, then $B$ is laminar. Moreover, $B$ does not carry a torus. \hfill $\Box$
\end{prop}

Branched surfaces built in this way meet the boundary of the knot exterior, $\partial X_K$ in a train track, which we denote $\tau_B$. If $r$ is a slope fully carried by $\tau_B$, then Li's theorem builds an essential lamination $\mathcal{L}_r$ meeting $\partial X_K$ in simple closed curves of slope $r$. These essential laminations can be extended to taut foliations in $S^3_r(K)$:

\begin{thm}
\label{thm:foliations}
Suppose $K$ is a fibered knot in $S^3$, and $B$ is a laminar branched surface built from a copy of the fiber surface and a collection of co--oriented product disks. For any slope $r$ fully carried by $\tau_B$, $S^3_r(K)$ has a taut foliation.
\end{thm}

\begin{proof}
Suppose $B$ is a branched surface satisfying the stated hypotheses. By \Cref{prop:BisLaminar}, $B$ does not carry a torus. Therefore, applying  \Cref{thm:taolisinkdisk}, for all slopes $r$ fully carried by $\tau_B$, the branched surface $B$ carries an essential lamination $\mathcal{L}_r$ which meets $\partial X_K$ in simple closed curves of slope $r$. In \cite[Proposition 3.19]{Krishna:3Braids}, we proved that $\mathcal{L}_r$ can be extended to a taut foliation $\mathcal{F}_r$ of $X_K$ meeting $\partial X_K$ in simple closed curves of slope $r$. Performing $r$--framed Dehn filling extends the foliation of the knot exterior to a foliation of $S^3_r(K)$.
\end{proof}

In particular, this reduces the task of building taut foliations in Dehn surgeries along fibered knots to two combinatorial conditions on the branched surface $B$:
\begin{enumerate}
\item proving $B$ is sink disk free, and
\item computing the slopes fully carried by the induced train track $\tau_B$.
\end{enumerate}

%Indeed, this strategy was employed to prove \Cref{thm:Positive3Braids}. %Proving \Cref{thm:main} requires leveraging features of positive braid knots of index at least 4. 
%we will leverage that if $K$ is a positive braid knot, then the positive braid word identifies a a collection of product disks in the surface exterior. We explore this perspective in \Cref{section:monodromy}. We will leverage other features of positive braid knots in \Cref{section:example} and \Cref{section:proof}.
We emphasize that both are combinatorial conditions. Therefore, to prove that $B$ is sink disk free, it suffices to show that every branch sector contains an outward pointing arc in its boundary.
%abuts to a branch locus with an outward--pointing co--orientation.

%%%%%%%%An Example
\section{An example alongside the construction} \label{section:example}

In this section, we demonstrate the construction used to prove \Cref{thm:main} alongside an example. In addition to establishing some preliminaries, we hope it cultivates some intuition for how to build branched surfaces in fibered knot exteriors. %This example demonstrates the strategy for proving \Cref{thm:main} for a ``generic" positive braid (our definition of ``generic" will become clear in \Cref{section:proof}). 

We build branched surfaces by fusing together the fiber surface $F$ with a well chosen collection of co--oriented product disks. The plumbing structure described in \Cref{section:monodromy} is well suited to selecting product disks: namely, we choose the product disks swept out by plumbing arcs. 

In \cite{Krishna:3Braids}, we developed a strategy for building a branched surface for positive 3--braid knots. Our goal is to extend this strategy to positive braids on any number of strands. Our proof of \Cref{thm:main} is inspired by a technique known as the \textit{amplification method} \cite{TerryTao:Amplification}; Feller referred to it as the \textit{reduction method}, and used it to produce a sharp signature bound for 4-braids \cite{Feller:4Braids}.

We build taut foliations in $S^3_r(K)$, where $K$ is realized as the closure of a positive braid on $n$ strands, $r < g(K)+1$, and $n \geq 4$, by following the outline below:

%%%%Outline
\begin{tabular}{ll}
\textit{\Cref{section:Features}}: & Identify some key features of positive braid words. \\
\textit{\Cref{section:BuildTemplate}:} & Design ``branched surface templates'' for columns of the braid. \\
\textit{\Cref{section:distribution}:} & Study the distribution of crossings in the braid $\beta$. \\
\textit{\Cref{section:column}:} & Apply the templates to build a branched surface $B$. \\
\textit{\Cref{section:SinkDisk}:} & Check that $B$ is sink disk free. \\
\textit{\Cref{section:TrainTrack}:} & Calculate the slopes carried by the train track $\tau_B$. \\
\textit{\Cref{section:Foliations}:} & Construct taut foliations in manifolds obtained by Dehn surgery along $\widehat{\beta}$. \\
\end{tabular}

Throughout this section, we demonstrate the construction alongside the braid 
\begin{align} \label{ex:main}
\beta \approx \two \one \three \two \three^2 \two \one \three \two \three^3 \two \three \one \two \three^2 \two \three (\three \three \two)^2
\end{align}

%%%%%%%%%%
\subsection{Identify some key features of positive braid words.} \label{section:Features}
Before building our templates, we establish some necessary preliminaries about braid words. Recall that $c_i$ denotes the number of occurrences of $\sigma_i$ in $\beta$.

\begin{lemma}
Let $\beta$ be a braid on $n$-strands. If $\widehat{\beta}$ is a prime knot, then $c_i \geq 2$ for all $1 \leq i \leq n-1$. 
\end{lemma}

\begin{proof}
If, for some $i$, $c_i = 0$, then $\widehat{\beta}$ is a split link. If, for some $i$, $c_i = 1$, then $\widehat{\beta}$ can be presented on $n-1$ strands; see \Cref{fig:Braid_FewerStrands} for an isotopy demonstrating how to reduce the strand number. 
\end{proof}

\begin{defn}
The positive braid words $\beta$ and $\beta'$ \textbf{synonyms} if their closures are isotopic in $S^3$. 
\end{defn}

\begin{defn} \label{defn:standard}
A braid $\beta$ is in \textbf{standard form} if, for all $1 \leq i \leq n-2$, there are no occurrences of $\sigma_{i} \sigma_{i+1}\sigma_i$ in $\beta$, even up to cyclic presentations of $\beta$. 
%Suppose a non--trivial knot $K$ is realized as the closure of a positive braid $\beta$. %realizing the braid index of $K$. 
%A \textbf{standard form} for $\beta$ is obtained by replacing \underline{all} occurrences of $\sigma_{i}\sigma_{i+1}\sigma_i$ with $\sigma_{i+1}\sigma_i\sigma_{i+1}$, where $1\leq i \leq n-2$, and destabilizing the braid as necessary.
\end{defn}

\begin{lemma}
Every braid $\beta$ admits at least one synonym which is in standard form. 
\end{lemma}

\begin{proof}
Let $\beta$ denote any positive braid. It necessarily has finite length. Thus, in finite time, one can replace \underline{all} occurrences of $\sigma_{i}\sigma_{i+1}\sigma_i$ with $\sigma_{i+1}\sigma_i\sigma_{i+1}$ whenever $1 \leq i \leq n-2$ (note that the other braid relation $\sigma_{i}\sigma_{j} \approx \sigma_j \sigma_i, |i-j| \geq 2$ may also be required). This eliminates all $\sigma_i \sigma_{i+1} \sigma_i$ subwords, where $1 \leq i \leq n-2$, in the fixed presentation of $\beta$. However, such subwords may exist after cyclically conjugating $\beta$, so those must be eliminated as well. So, for every possible cyclic conjugation of $\beta$, eliminate any $\sigma_i \sigma_{i+1} \sigma_i$ subword that arises (and, as before, use the braid relations $\sigma_{i}\sigma_{j} \approx \sigma_j \sigma_i, |i-j| \geq 2$ as needed). The result is a synonym of $\beta$ in standard form.
\end{proof}

%That is, every non--trivial positive braid admits a braid synonym such that for all $1 \leq i \leq n-2$, there are no occurrences of $\sigma_{i}\sigma_{i+1}\sigma_i$ in $\beta$. Applying these braid relations as much as possible pushes crossings into the right--most columns. 
We remark: if $K$ is a braid positive, it may admit many standard forms. Our construction is not dependent on a particular standard form. 
%
%\begin{defn} \label{defn:Block}
%\textcolor{red}{define a block here.}
%\end{defn}

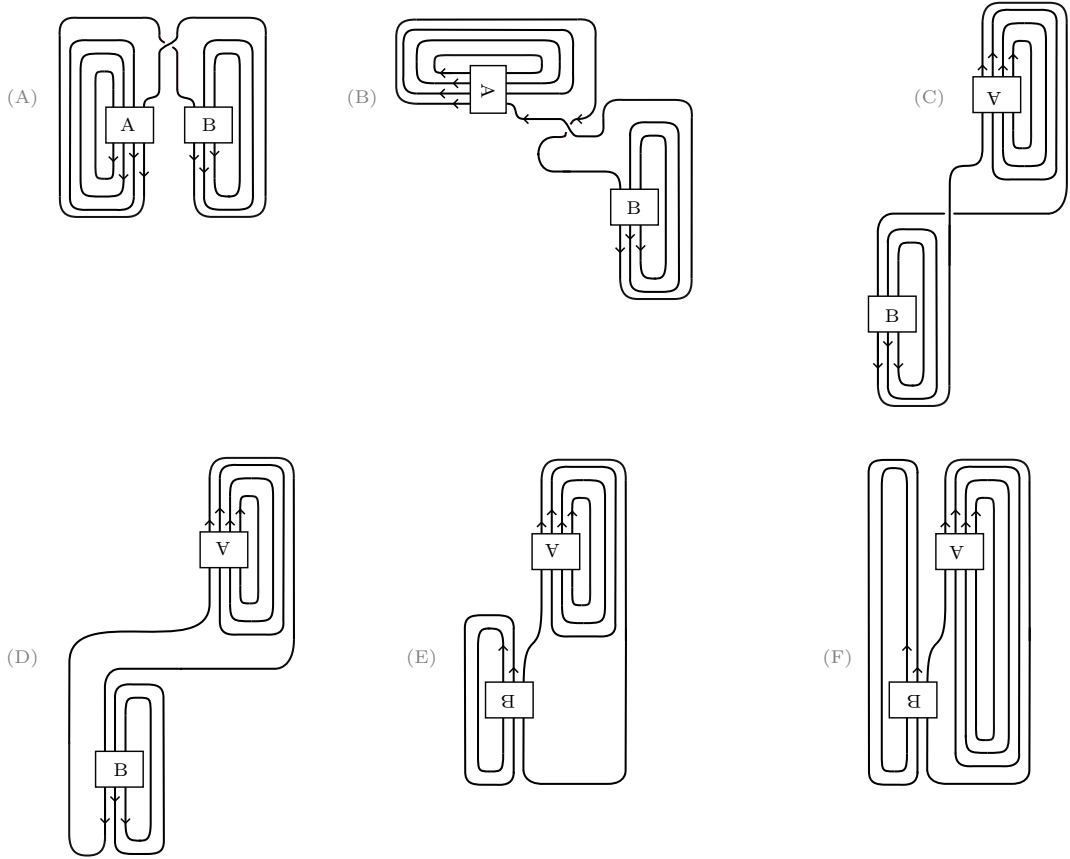
\begin{figure}[h!]\center
\labellist \tiny
%\pinlabel {(A)} at 75 320
\pinlabel {\textcolor{gray}{(A)}} at -10 450
%\pinlabel {$w_1$} at 52 435
%\pinlabel {$w_2$} at 100 435
\pinlabel {A} at 52 435
\pinlabel {B} at 100 435
%
%\pinlabel {(B)} at 300 320
\pinlabel {\textcolor{gray}{(B)}} at 190 450
%\pinlabel \rotatebox{-90}{$w_1$} at 264 455
%\pinlabel {$w_2$} at 350 386
\pinlabel \rotatebox{-90}{A} at 264 455
\pinlabel {B} at 350 386
%
%\pinlabel {(C)} at 550 320
\pinlabel {\textcolor{gray}{(C)}} at 524 450
%\pinlabel \rotatebox{-180}{$w_1$} at 562 450
%\pinlabel {$w_2$} at 502 322
\pinlabel \rotatebox{-180}{A} at 562 450
\pinlabel {B} at 502 322
%
%\pinlabel {(D)} at 75 100
\pinlabel {\textcolor{gray}{(D)}} at -10 120
%\pinlabel \rotatebox{-180}{$w_1$} at 108 185
%\pinlabel {$w_2$} at 48 55
\pinlabel \rotatebox{-180}{A} at 108 185
\pinlabel {B} at 48 55
%
%\pinlabel {(E)} at 300 100
\pinlabel {\textcolor{gray}{(E)}} at 225 120
%\pinlabel \rotatebox{-180}{$w_1$} at 302 184
%\pinlabel \rotatebox{-180}{$w_2$} at 276 95
\pinlabel \rotatebox{-180}{A} at 302 184
\pinlabel \rotatebox{-180}{B} at 276 95
%
%\pinlabel {(F)} at 500 100
\pinlabel {\textcolor{gray}{(F)}} at 470 120
%\pinlabel \rotatebox{-180}{$w_1$} at 540 183
%\pinlabel \rotatebox{-180}{$w_2$} at 515 95
\pinlabel \rotatebox{-180}{A} at 540 183
\pinlabel \rotatebox{-180}{B} at 515 95
\endlabellist
\includegraphics[scale=0.65]{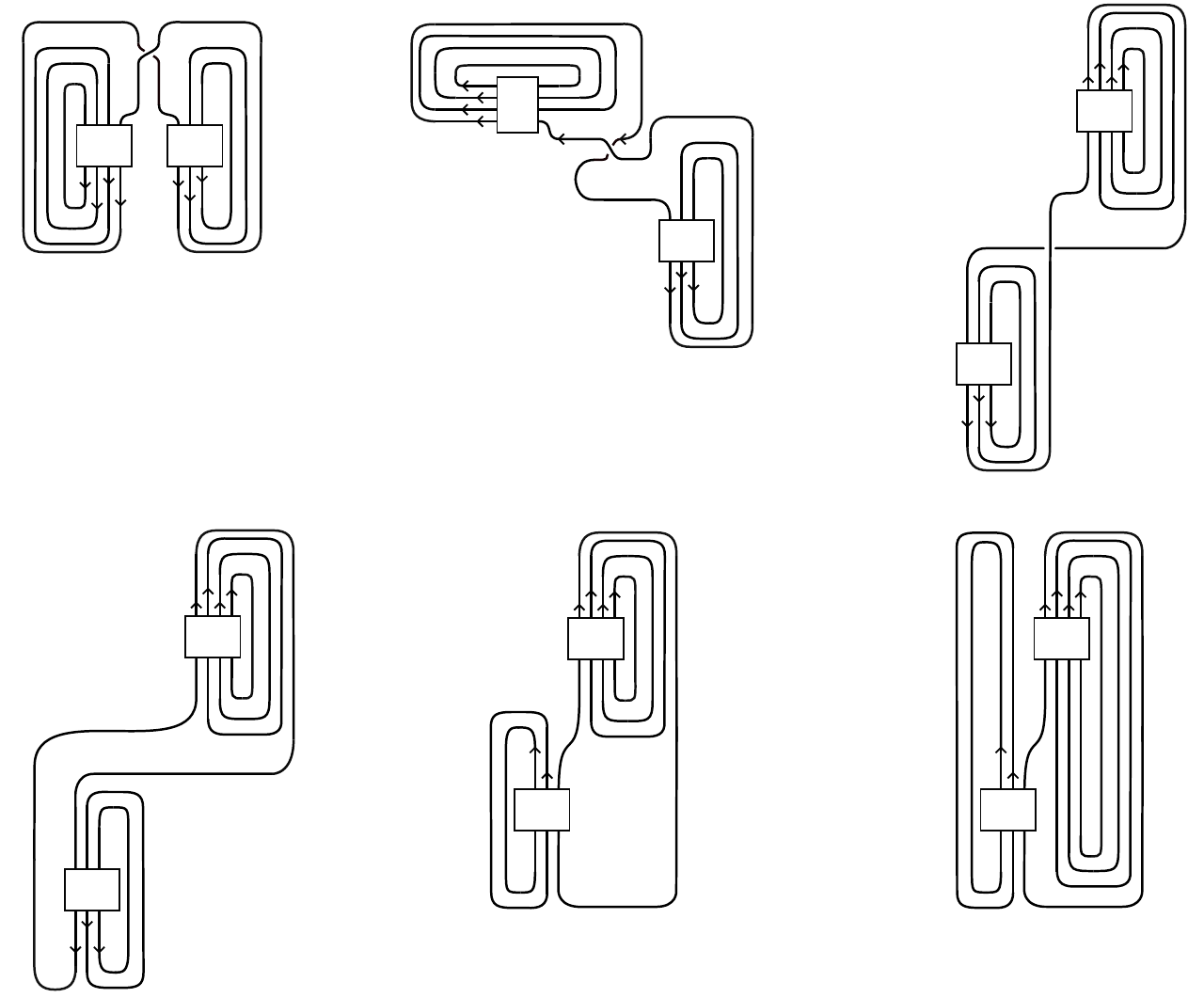}
\caption{An isotopy showing that if $c_i=1$ for some $i$, then $\widehat{\beta}$ can be presented on one less strand.}
\label{fig:Braid_FewerStrands}
\end{figure}

\begin{defn} \label{defn:BlockFunctions}
Let $\mathbb{B}_n^+$ denote the monoid of positive braids on $n$ strands. For $1 \leq i \leq n-1$, define the functions $\mathcal{B}_i: \mathbb{B}_n^+ \to \Z^+$ as follows: fix a presentation of a positive braid $\beta$ in standard form.  Cyclically conjugate the braid so it appears as $\beta \approx \sigma_i^{p_1} w_1  \sigma_i^{p_2} w_2 \ldots \sigma_i^{p_k} w_k$, where for all $1 \leq j \leq k$, $1 \leq p_j$, and $w_j$ is a subword of $\beta$ with no $\sigma_i$ letters. Then $\mathcal{B}_i(\beta) := k$. 
\end{defn}

Note: this function takes as input a \textit{presentation} of a braid in standard form, so the functions $\mathcal{B}_i$ are well--defined.

\begin{lemma} \label{lemma:prime}
If $\widehat{\beta}$ is a prime knot, then for all $1 \leq i \leq n-1$, $\mathcal{B}_i(\beta) \geq 2$.
\end{lemma}

\begin{proof}
We proceed via a proof by contrapositive: if there exists some $i$ such that $\mathcal{B}_i(\beta) = 1$, then there exists some conjugation of $\beta$ such that $\beta \approx \sigma_i^{p_1}w_1$, for $w_1$ is a word spelled without $\sigma_i$ letters. But this means that $\beta$ can be realized a connected sum of two braids. In particular, one can identify a splitting $S^2$ which contains a non-trivial braid on either side. See \Cref{fig:Braid_ConnectedSum} (left).
\end{proof}

\begin{lemma} \label{lemma:prime2}
Suppose $\widehat{\beta}$ is a prime knot, and there exists some $i \geq 2$ such that $c_i=2$. For any conjugated presentation of $\beta$, $d(i; 1,2) \geq 1$ and $d_L(i; 1,2) \geq 1$. 
%Fix any conjugated presentation of $\beta$ beginning with $\sigma_i$. Then, between $\mathbbm{b}_{i,1}$ and $\mathbbm{b}_{i,2}$, there must be at least one $\sigma_{i-1}$ band, \underline{and} at least one $\sigma_{i+1}$ band. %Similarly, $\mathbbm{b}_{i,2}$ must be followed by both $\sigma_{i-1}$ and $\sigma_{i+1}$ bands. 
\end{lemma}

%\begin{lemma} \label{lemma:prime2}
%%Let $\beta$ be a standardized positive braid on $n$ strands. For any $2 \leq i \leq n-2$, there exists some cyclic conjugation of $\beta$ such that for some $t, s < c_i$, $d(i; 1, t)$ and $d(i; t, c_i)$ are both positive. 
%%Suppose $\widehat{\beta}$ is prime, and there exists some $i \geq 2$ such that $c_i=2$. For any conjugated presentation of $\beta$, $d(i; 1,2) \geq 1$ and $d_L(i; 1,2) \geq 1$. 
%%Fix any conjugated presentation of $\beta$ beginning with $\sigma_i$. Then, between $\mathbbm{b}_{i,1}$ and $\mathbbm{b}_{i,2}$, there must be at least one $\sigma_{i-1}$ band, \underline{and} at least one $\sigma_{i+1}$ band. %Similarly, $\mathbbm{b}_{i,2}$ must be followed by both $\sigma_{i-1}$ and $\sigma_{i+1}$ bands. 
%\end{lemma}

\begin{proof}
Suppose there exists some $i$ such that $c_i = 2$. If one of $d(i; 1,2)$ or $d_L(i; 1,2)$ is zero, then as in \Cref{fig:Braid_ConnectedSum}, we produce a splitting $S^2$ for $\beta$, and deduce $\widehat{\beta}$ is not prime.
\end{proof}

\begin{figure}[h!]\center
\labellist
\pinlabel {$w_1$} at 53 68
\pinlabel {$w_2$} at 100 68
\pinlabel {$w_1$} at 302 95
\pinlabel {$w_2$} at 258 53
\pinlabel {$w_3$} at 302 53
\endlabellist
\includegraphics[scale=1.1]{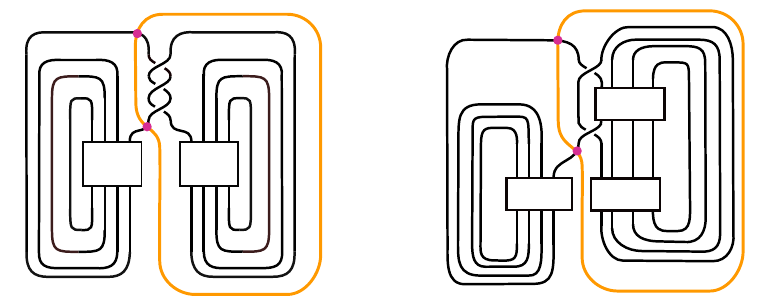}
\caption[Braids as connected sums of knots]{\textbf{Left}: A braid with $\mathcal{B}_4(\beta)=1$; $w_1$ and $w_2$ are braid words not containing $\sigma_4$. The orange unknotted circle is the equator of an $S^2$ realizing $\widehat{\beta}$ as the connected sum of two knots. \\
\textbf{Right}: a braid with some $c_i = 2$ and $d_L(i; 1,2) = 0$. The closure is a connected sum.}
\label{fig:Braid_ConnectedSum}
\end{figure}
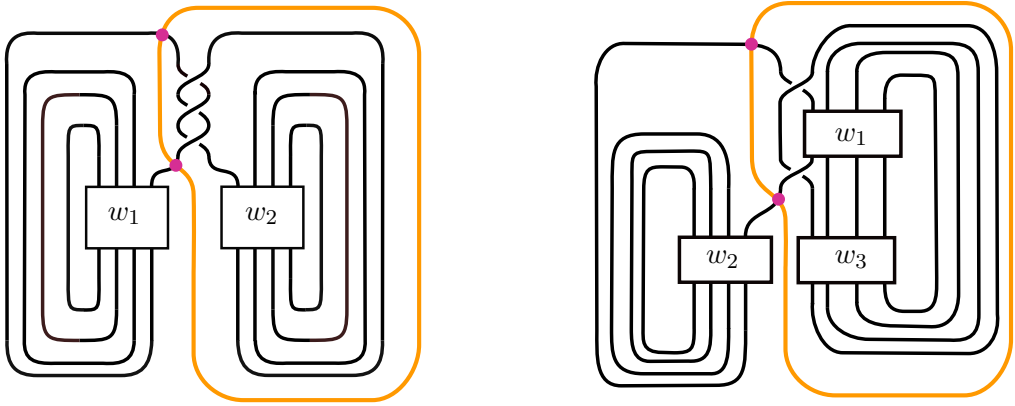

%\begin{lemma} \label{lemma:prime3}
%\textcolor{blue}{more general version of the splinter lemma....must always have some left and right splinters} 
%\end{lemma}

Delman--Roberts proved that if $K$ is the connected sum of two fibered knots, then every manifold obtained by rational Dehn surgery has a taut foliation \cite{DelmanRoberts:composite}. Thus, we restrict to \textit{prime} braid positive knots throughout. Going forward, we assume all positive braids are in standard form, i.e. have been \textit{standardized}.

%%%%%%%%%%%SECTION: BUILD TEMPLATE
\subsection{Design ``branched surface templates'' for columns of the braid} \label{section:BuildTemplate}

To build the branched surface $B$ required for \Cref{thm:main}, we design and apply some standard templates to a subset of the columns of $F$. This requires:
\begin{itemize}
%\item Conjugating $\beta$ into a specified form,
\item identifying the product disks in a column $\Gamma_j$,
\item building the spine for a branched surface supported in $\Gamma_j$, and
\item assigning co--orientations to build the branched surface. 
\end{itemize}

%In particular, if $\Cbig = \Codd$ (resp. $\Ceven$), we apply our template to the odd (resp. even) columns.

%%%%%
\subsubsection{Identify product disks in a column $\Gamma_j$} \label{section:ProductDisk}
The construction in \Cref{lemma:monodromy} identifies a factorization of the monodromy of the fiber surface for $\widehat{\beta}$ by realizing the fiber surface as the plumbing of positive Hopf bands. In particular, it identifies the plumbing arcs as well as their images under the monodromy. %We will use the $c_j - 1$ plumbings in $\Gamma_j$. 
Pushing the $c_j-1$ plumbing arcs in $S_j$ through the mapping torus yields $c_j-1$ product disks supported in $\Gamma_j$. 

% \begin{defn}
% Let $\mathcal{A}_$ denote the set of $c_j-1$ plumbing arcs in Seifert disk $S_j$. Let $\mathbb{D}_j$ denote the set of product disks obtained by pushing the plumbing arcs in $\mathbb{A}_j$ through the monodromy. 
% \end{defn}

%%%%%%%%%%
\subsubsection{Build the spine for a branched surface in a column $\Gamma$} \label{section:spine}

Now that we have identified the product disks in $\Gamma_j$, we can build the spine for a branched surface. 

%\begin{defn}
%The spine $\mathbb{S}_j$ is formed by fusing the fiber surface $F$ with the product disks $\mathbb{D}_j$.
%\end{defn}

Let $\mathbb{S}_j$ denote the spine which is formed by fusing the fiber surface $F$ with the product disks supported in $\Gamma_j$.
The branch locus for $\mathbb{S}_j$ is supported in Seifert disks $S_j, S_{j+1}$, and the bands $\mathbbm{b}_{j,1}, \ldots, \mathbbm{b}_{j, c_j}$ connecting them. In \cite[Section 3.2]{Krishna:3Braids}, the author defined an operation called \textbf{spinal isotopy}, which simplifies the branch locus for a forthcoming branched surface. For the purposes of this paper, it suffices to understand the result of this operation. After performing a spinal isotopy, the fiber surface $F$ has geodesic representatives of the plumbing arcs and their images as elements in $H_1(F, \partial F)$. Mildly abusing notation, we refer to the spine after the spinal isotopy also as $\mathbb{S}_j$, and also refer to plumbing and image arcs after spinal isotopy by $\alpha_{j,s}$ and $\varphi(\alpha_{j,s})$. Note: after a spinal isotopy, the branch locus is now supported in $S_j$ and $S_{j+1}$. See \Cref{fig:Enclosing_Bands} for an example.

\begin{defn} \label{defn:PositiveNBraidsEncloses}
Consider an arc $\gamma \subset S_{j}$, and ignore all other arcs $\gamma' \subset S_j$. Then $\gamma$ \textbf{encloses the band $\mathbbm{b}_{j,k}$ on the left} if, thinking of the band as a 2-dimensional 1-handle, the co--core and left attaching region of $\mathbbm{b}_{j,k}$ is contained in a branch sector $\mathcal{S}$ with $\gamma \subset \partial \mathcal{S}$, where $\mathcal{S}$ does not contain the center of the Seifert disk $S_j$. On the other hand, $\gamma$ \textbf{encloses the band $\mathbbm{b}_{j-1,t}$ on the right} if the co--core and right attaching region of $\mathbbm{b}_{j-1,t}$ is contained in a branch sector $\mathcal{S}$ with $\gamma \subset \partial \mathcal{S}$, where $\mathcal{S}$ does not contain the center of the Seifert disk $S_j$. 
% %
% \textbf{encloses a band $\mathbbm{b}_{j,k}$ on the left} 
% %
% if, ignoring other arcs $\gamma' \subset S_j$, the left attaching sphere of $\mathbbm{b}_{j,k}$ is contained in the branch sector $\mathcal{S}$, where $\gamma \subset \partial \mathcal{S}$, and 
% %
% $\mathcal{S}$ does not contain the 
% %
% $($resp. right$)$} if the left $($resp. right$)$ attaching site of $\mathbbm{b}$ is contained in the branch sector $\mathcal{S}$ with $\gamma \subset \partial \mathcal{S}$. 
\end{defn}

We observe that we have the following behavior for any arc $\alpha_{j,s}$:
\begin{itemize}
\item $\alpha_{j,s}$ encloses exactly $\mathbbm{b}_{j, s}$ on the left,
\item $\alpha_{j,s}$ encloses exactly $d_L(j; s, s+1)$ bands on the right,
\item $\varphi(\alpha_{j,s})$ encloses exactly $\mathbbm{b}_{j, s+1}$ on the right, and
\item $\varphi(\alpha_{j,s})$ encloses $d(j; s, s+1)$ bands on the left.
\end{itemize}

See \Cref{fig:Enclosing_Bands} for an example. This figure also demonstrates why we fix an arc $\gamma \subset S_j$ and ignore the other arcs $\gamma' \subset S_j$: our definition is written so that $\alpha_{j,s}$ encloses two bands from $\Gamma_{j-1}$ on the right (not one band). 

\begin{figure}[h!]\center
\labellist \tiny
\pinlabel $\alpha_{j,s}$ at 55 200
\pinlabel $\varphi(\alpha_{j,s})$ at 200 215
\pinlabel $\Gamma_j$ at 360 260
\pinlabel $\Gamma_{j-1}$ at 600 260
\pinlabel $\Gamma_{j}$ at 665 260
\pinlabel $\Gamma_{j+1}$ at 730 260
\endlabellist
\includegraphics[scale=.5]{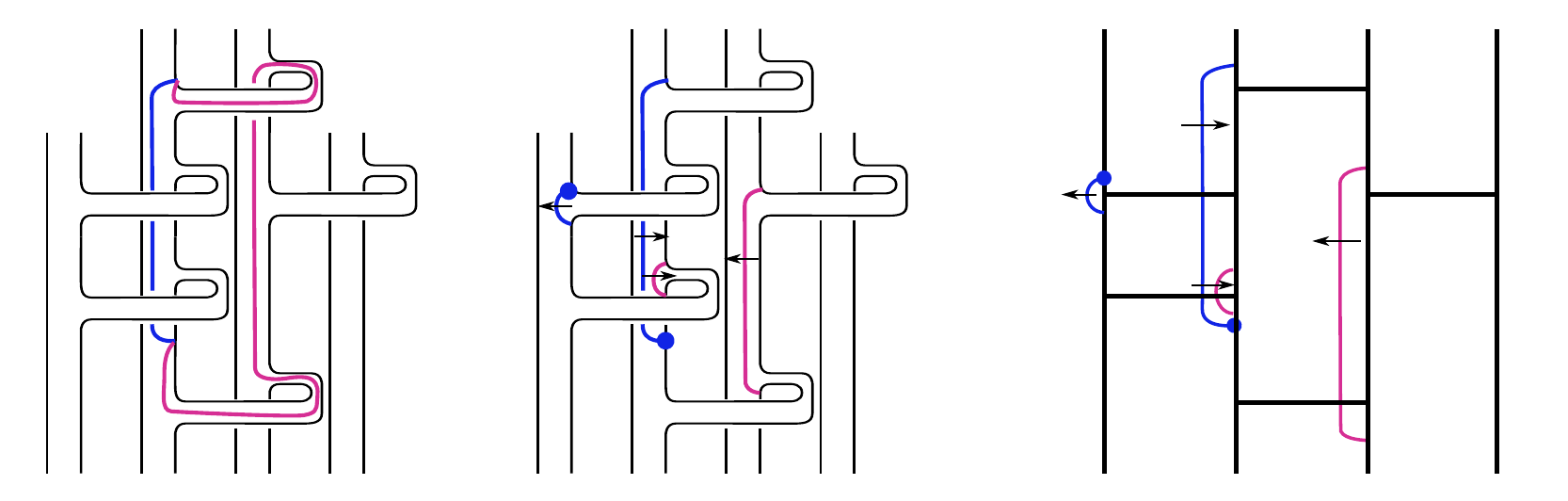}
\caption{ \textbf{Left:} The arc $\alpha_{j,s}$ encloses bands on both the left and on the right. The pink arc is $\varphi(\alpha_{j,s})$, before performing a spinal isotopy. \\
\textbf{Middle:} The result of the spinal isotopy for two disks. If a plumbing arc is co--oriented to the left (resp. right), the image is smoothed to the right (resp. left). \\
\textbf{Right:} The co--oriented disks after a spinal isotopy, captured via a brick diagram.}
\label{fig:Enclosing_Bands}
\end{figure}

Much of the time, when we study branch sectors, we will present them via in a brick diagram presentation as in \Cref{fig:Enclosing_Bands} (right).

%%%%%%%%%%%%%%%%%%%%%%%%%%%%
\subsubsection{Co--orient the product disks} \label{section:Co-Orient}
To produce a branched surface from the spine, we need to specify co--orientations for the product disks. 

\begin{defn}
An arc $\alpha$ on $F$ is co--oriented \textbf{to the left $($resp. right$)$} if, when looking at the positive side of the fiber surface, $F^{+}$, the arc is decorated with a left $($resp. right$)$ pointing arrow. We say the arc $\alpha$ is a \textbf{left $($resp. right$)$ pointer}, or is \textbf{left $($resp. right$)$ pointing}.
\end{defn}

The direction of the arrow indicates the smoothing direction of the locus where a product disk meets the fiber surface.

\begin{lemma} \cite[Section 3.4.1]{Krishna:3Braids}
\label{lemma:smoothing}
Let $\alpha$ be a plumbing arc on $F$ and let $\mathbb{D}_\alpha$ denote the corresponding product disk. 
A choice of co--orientation of $\mathbb{D}_\alpha$ is in one--to--one correspondence with the induced smoothing directions on the arcs $\alpha$ and $\varphi(\alpha)$. In particular, there are two choices of co--orientations for $\mathbb{D}_\alpha$, and each induces one of the following:
\begin{itemize}
\item $\alpha$ is co--oriented to the left, and $\varphi(\alpha)$ is co--oriented to the right, or 
\item $\alpha$ is co--oriented to the right, and $\varphi(\alpha)$ is co--oriented to the left. \hfill $\Box$
\end{itemize}
\end{lemma}

Accordingly, the co--orientations for the disks are encoded by recording smoothing directions for the plumbing arcs used to build the spine; see \Cref{fig:Enclosing_Bands} (middle, right).

We describe a template which provides smoothing directions for plumbing arcs (and therefore product disks) in either a single column $\Gamma_i$, or a pair of consecutive columns, $\Gamma_{i} \cup \Gamma_{i+1}$. The relative application site of these templates is determined by a particular cyclic presentation of the standardized braid. 

\begin{defn} \label{defn:pivot}
Suppose $\beta = \gamma \sigma_i \gamma'$ is a presentation of a positive braid, and $\sigma_i$ is some distinguished letter in $\beta$. We \textbf{pivot $\beta$ about $\sigma_i$} by cyclically conjugating the braid to present is as $\beta = \sigma_i \gamma' \gamma$. That is, we \textbf{pivot $\beta$ with respect to this distinguished $\sigma_i$}. 
\end{defn}

\begin{defn} \label{defn:Templates}
Let $\beta$ be a braid in standard form, and fix a presentation of the braid as $\beta = \gamma \sigma_{i} \gamma'$ for some distinguished $\sigma_i$ letter. 

\begin{itemize}
\item To \textbf{apply our template to $\Gamma_i$}, we:
	\begin{enumerate}
	\item pivot $\beta$ about $\sigma_i$,
	\item co--orient the plumbing arcs $\alpha_{i,\star} \subset S_i$ appearing before the first $\sigma_{i+1}$ letter to the left, and
	\item co--orient all subsequent plumbing arcs to the right.
	\end{enumerate}
%The resulting branched surface contains $c_i - 1$ product disks from $\Gamma_i$. 
%
\item To \textbf{apply our template to $\Gamma_i \cup \Gamma_{i+1}$}, we:
	\begin{enumerate}
	\item pivot $\beta$ about $\sigma_i$,
	\item co--orient the plumbing arcs $\alpha_{i,\star} \subset S_i$ appearing before the first $\sigma_{i+1}$ letter to the left, and all subsequent plumbing arcs to the right, and
        \item co--orient the first plumbing arc of $S_{i+1}$ to the right, and all subsequent plumbing arcs to the left. 
	\end{enumerate}
%The resulting branched surface contains $c_i - 1$ (resp. $c_{i+1}-1$) product disks from $\Gamma_i$ (resp. $\Gamma_{i+1}$). 
\end{itemize}
\end{defn}

We postpone the figures showing these templates.

%%%%%%%%%%%%%%
\subsection{Study the distribution of crossings of the braid} \label{section:distribution}
%\textcolor{red}{As in Section \ref{section:monodromy}, we have an explicit factorization of the monodromy for the fiber surface associated with $\widehat{\beta}$.} 

Now that we have templates for building a branched surface, we choose to which columns the templates are applied. These choices are made based on the distribution of the crossings in the braid word.

\begin{defn}
For $j \in \{1, 2, 3\}$ and $1 \leq i \leq n-1$, let $\displaystyle \mathcal{C}_{j} := \sum_{i \equiv j \mod 3} c_i$.
\end{defn}

\begin{defn} \label{defn:Cmin}
We define $\Cmin$ as follows:
\begin{align*}
\Cmin := \left\{
    \begin{array}{cll}
\textup{min}\{\Cone, \Ctwo, \Cthree\} && \textup{when a unique minimum exists} %\Cone \neq \Ctwo \neq \Cthree 
\\
\Ctwo && \Ctwo = \Cthree < \Cone \\
\Cone && \textup{otherwise} %\Cone = \Ctwo = \Cthree, \ \Cone = \Ctwo < \Cthree, \textup{ or } \Cone = \Cthree < \Ctwo \\
%\Cthree && \Cone = \Ctwo = \Cthree
    \end{array}
\right.
\end{align*}
\end{defn}

For our example in (\ref{ex:main}), we have:
\begin{align*}
 \left.
    \begin{array}{l}
	\Cone =  c_1 =3\\
	\Ctwo = c_2  = 9\\
	\Cthree = c_3 = 15 
    \end{array}
\right \} \implies \Cmin = \Cone
\end{align*}

\begin{figure}[h!]\center
\labellist
\pinlabel {$\Gamma_1$} at 84 980
\pinlabel {$\Gamma_2$} at 150 980
\pinlabel {$\Gamma_3$} at 218 980
\endlabellist
\includegraphics[scale=.55]{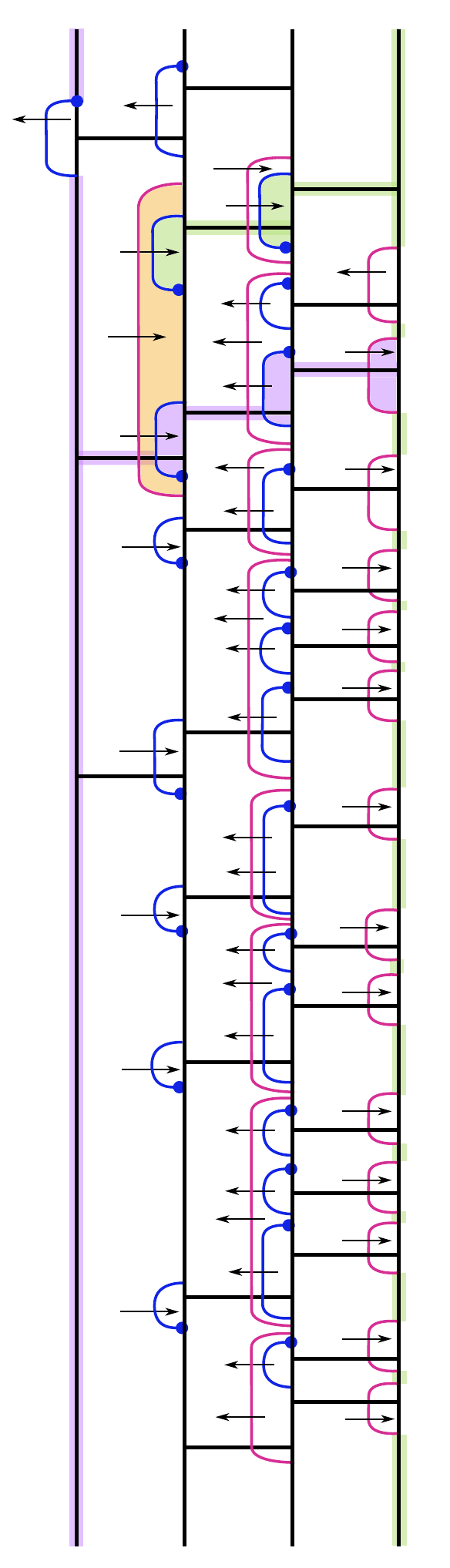}
\caption{The branched surface for $K = \widehat{\beta}$, for $\beta$ as in (\ref{ex:main}).}
\label{fig:Example_Main}
\end{figure}

%%%%%%%%
\subsection{Apply the templates to build a branched surface}\label{section:column}

We apply the templates designed in \Cref{section:BuildTemplate} to the columns of $\beta$, based on the value of $\Cmin$. 

We demonstrate the technique for our ongoing example, where $\Cmin = \Cone$.  With the presentation of $\beta$ in (\ref{ex:main}), we apply our template to $\Gamma_2 \cup \Gamma_3$, and then co-orient $\alpha_{1,1}$ to the left. See \Cref{fig:Example_Main}.

\subsection{Searching for sink disks} \label{section:SinkDisk}

To apply \Cref{thm:foliations}, we require that $B$ is a laminar branched surface. By \Cref{prop:BisLaminar}, it suffices to show that $B$ is sink disk free. This requires an analysis of the branch sectors. We built $B$ from a copy of the fiber surface $F$ and a collection of co--oriented product disks, so we can classify the branch sectors into four types:
\begin{itemize}
\item the (isotoped) product disks
\item the sectors containing the Seifert disks $S_i$
\item the sectors contained in a Seifert disk
\item the remaining sectors 
\end{itemize}

\begin{lemma} \cite[Lemma 3.11]{Krishna:3Braids} \label{lemma:ProductDisksAren'tSinkDisks}
The $($isotoped$)$ product disks are never sink disks. \hfill $\Box$
\end{lemma}

\begin{defn}
The $($connected component$)$ of a branch sector containing the center of a Seifert disk $S_i$ is called the \textbf{$i^{th}$ disk sector}, and is denoted $\mathcal{S}_i$. $($Note the subtle calligraphic difference between $S_i$ and $\mathcal{S}_i$.$)$ If a branch sector is contained in a Seifert disk, it is a \textbf{polygon sector}. The remaining branch sectors are \textbf{horizontal sectors}.
\end{defn}

We note: our construction does not force every branch sector to topologically be a disk. 

\begin{defn}
%A horizontal sector $\mathcal{H}$ \textbf{spans} $m$ columns if $\overline{\partial \mathcal{H} - \partial F}$ is contained in $m+1$ distinct Seifert disks. In particular, 
%The horizontal sector $\mathcal{H}$ \textbf{spans} $\Gamma_j \cup \ldots \cup \ \Gamma_{j+m} \ $ if, for some $t \in \{j, \ldots, j+m\}$, $\partial \mathcal{H} \cap S_t \neq \varnothing$. 
A disk or polygon sector $\Delta$ \textbf{spans} $\Gamma_i$ if $\Delta \cap \{S_i, S_{i+1}, \mathbbm{b}_{i,1}, \ldots , \mathbbm{b}_{i, c_i}\} \neq \varnothing$. 
A horizontal sector $\mathcal{H}$ \textbf{spans} $\Gamma_j$ if $\mathcal{H} \cap S_j \neq \varnothing$ and $\mathcal{H} \cap S_{j+1} \neq \varnothing$. We say that $\mathcal{H}$ \textbf{spans} $\Gamma_j \cup \ldots \cup \Gamma_{j+m}$ if $\mathcal{H}$ spans at least one of $\Gamma_j, \ldots, \Gamma_{j+m}$.
\end{defn}

%Note: with our definition, it is possible that $\mathcal{H}$ spans $\Gamma_j \cup \ldots \cup \Gamma_{j+m}$ but the span of $\mathcal{H}$ exceeds these columns. 

%Note: with our definition, the following are both allowed:
%\begin{itemize}
%\item $\mathcal{H}$ spans $\Gamma_j \cup \ldots \cup \Gamma_{j+m}$ but that the span of $\mathcal{H}$ exceeds these columns, or 
%\item $\mathcal{H}$ spans some strict, non-trivial subset of the columns $\Gamma_j \cup \ldots \cup \Gamma_{j+m}$.
%\end{itemize}

%Note: we may not know that every $t \in \{1, \ldots, j-1, j+m+2, \ldots, n\}, \partial \mathcal{H} \cap S_t = \varnothing$.

The shortest horizontal sectors are the branch sectors exactly containing a single band connecting adjacent Seifert disks. Other horizontal sectors may span more than two Seifert disks. 

\begin{rmk}
\textup{Throughout this work, we will analyze the branch sectors by type. Since \Cref{lemma:ProductDisksAren'tSinkDisks} holds, we will only analyze disk, polygon, and horizontal sectors explicitly. In practice, the horizontal sectors typically span many columns of the braid, so they are the most difficult to analyze. Finally, as shorthand, we will often say ``the sector is not a sink'' (instead of specifying whether it is ``not a sink disk'' or ``not a half sink disk'').
}
\end{rmk}

\begin{lemma} \label{lemma:SeifertDisksNotSinks}
If we apply our template to $\Gamma_j \cup \Gamma_{j+1}$, then $\mathcal{S}_{j+1}$ is not a sink.
\end{lemma}

\begin{proof} 
When we defined our templates in \Cref{defn:Templates}, we first presented our standardized braid as $\sigma_j w_1 \sigma_{j+2}w_2$ (here, the identified $\sigma_{j+2}$ is the first letter of its type in this braid presentation), and then we applied our template to $\Gamma_j \cup \Gamma_{j+1}$. We fix this presentation throughout. We argue that $\mathcal{S}_{j+1}$ is not a sink because some arc always points out of the sector.
\begin{itemize}
\item If $d(j; 1,2) = 0$, then it is immediate that $\varphi(\alpha_{j,1})$ is in $\partial \mathcal{S}_{j+1}$; since $\alpha_{j,1}$ is a left pointer, $\varphi(\alpha_{j,1})$ is a right pointer, and it points out of $\mathcal{S}_{j+1}$, so the sector is not a sink. See \Cref{fig:Template_Polygon} (A).
\item Suppose $d(j; 1,2) = 1$. We have that either $d_L(j+1; 1,2) = 1$ (as in \Cref{fig:Template_Polygon} (B)) or $d_L(j+1; 1,2) \geq 2$ (as in \Cref{fig:Template_Polygon} (C)). If the former occurs, then $\varphi(\alpha_{j,1})$ and $\alpha_{j+1,1}$ are parallel arcs, and we can guarantee that $\varphi(\alpha_{j,1}) \subset \partial \mathcal{S}_{j+1}$; when the latter occurs, these two arcs are not parallel, but we can still ensure that $\varphi(\alpha_{j,1}) \subset \partial \mathcal{S}_{j+1}$, as in as in \Cref{fig:Template_Polygon} (C). Since $\varphi(\alpha_{j,1})$ is a right pointer, $\mathcal{S}_{j+1}$ is not a sink.
\item If $d(j; 1,2) \geq 2$, then $\varphi(\alpha_{j,1}) \subset \partial \mathcal{S}_{j+1}$, and so $\mathcal{S}_{j+1}$ is not a sink. See \Cref{fig:Template_Polygon} (D).
\end{itemize}

We deduce that when we apply our template to $\Gamma_j \cup \Gamma_{j+1}$, the sector $\mathcal{S}_{j+1}$ is never a sink.
\end{proof}

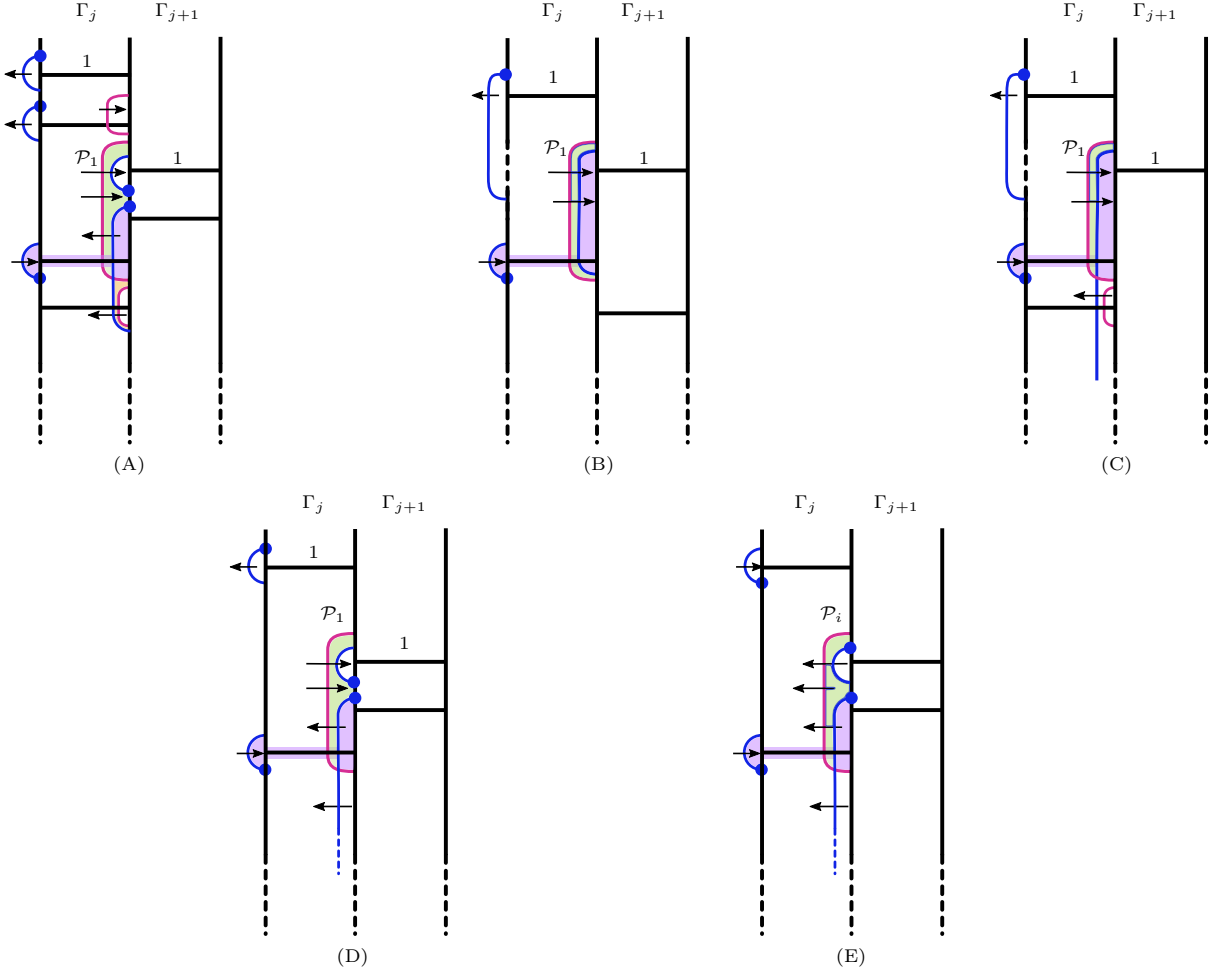
\begin{figure}[h!]\center
\labellist \tiny
\pinlabel {1} at 30 275
\pinlabel {1} at 58 245
\pinlabel {$\Gamma_j$} at 30 290
\pinlabel {$\Gamma_{j+1}$} at 58 290
\pinlabel {$\mathcal{P}_1$} at 30 245
\pinlabel {(A)} at 43 150
\pinlabel {1} at 174 270
\pinlabel {1} at 202 245
\pinlabel {$\Gamma_j$} at 174 290
\pinlabel {$\Gamma_{j+1}$} at 202 290
\pinlabel {$\mathcal{P}_1$} at 175 248
\pinlabel {(B)} at 188 150
\pinlabel {1} at 335 270
\pinlabel {1} at 360 245
\pinlabel {$\Gamma_j$} at 335 290
\pinlabel {$\Gamma_{j+1}$} at 360 290
\pinlabel {$\mathcal{P}_1$} at 335 248
\pinlabel {(C)} at 348 150
\pinlabel {1} at 100 123
\pinlabel {1} at 128 95
\pinlabel {$\Gamma_j$} at 100 138
\pinlabel {$\Gamma_{j+1}$} at 128 138
\pinlabel {$\mathcal{P}_1$} at 106 104
\pinlabel {(D)} at 112 -2
\pinlabel {$\Gamma_j$} at 252 138
\pinlabel {$\Gamma_{j+1}$} at 280 138
\pinlabel {$\mathcal{P}_i$} at 260 104
\pinlabel {(E)} at 266 -2
\endlabellist
\includegraphics[scale=1.23]{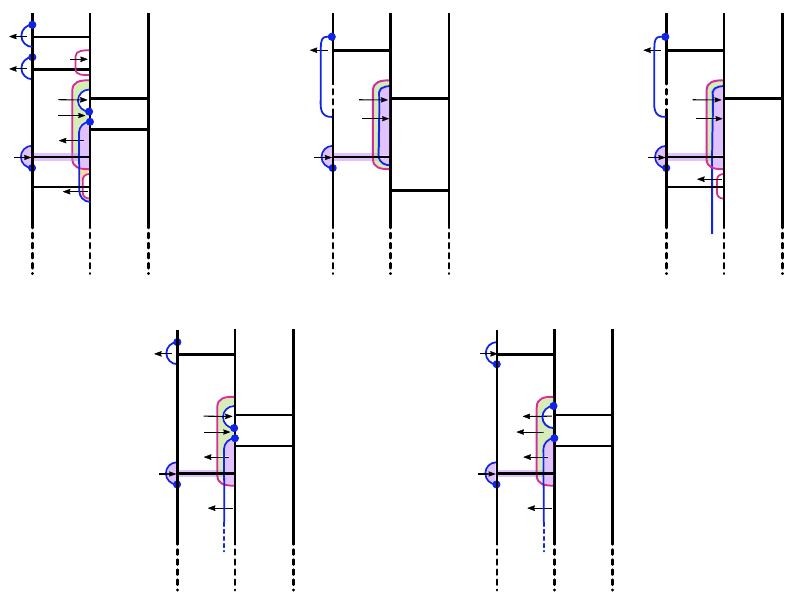}
\vspace{1em}
\caption{We apply our template to $\Gamma_j \cup \Gamma_{j+1}$. In \textbf{(A) -- (D)}, we see that when $d(j; 1,2) = 0$, $d(j; 1,2) =1$, and $d(j; 1,2) \geq 2$, the Seifert disk sector $\mathcal{S}_{j+1}$ is not a sink. These same frames demonstrate why the polygon sector $\mathcal{P}_1$ is never a sink. Frame \textbf{(E)} shows a generic polygon sector $\mathcal{P}_i$, where $2 \leq i \leq k$.}
\label{fig:Template_Polygon}
\end{figure}

\begin{lemma} \label{lemma:PolygonSectorsNotSinks}
Applying the branched surface template to $\Gamma_j \cup \Gamma_{j+1}$ does not create any sink polygon sectors in $S_{j+1}$.
\end{lemma}

\begin{proof}
For every $i$ such that $d(j; i, i+1) \geq 2$, we produce between one and two polygon sectors; see \Cref{fig:Template_Polygon}. Moving from north to south, we label the polygon sectors $\mathcal{P}_1, \ldots, \mathcal{P}_k$. 

By construction, $\partial \mathcal{P}_1$ contains $\alpha_{j+1,1}$; this arc is a right pointer, pointing out of $\mathcal{P}_1$, so this sector is not a sink. Suppose $2 \leq i \leq k$. 

If $\partial \mathcal{P}_1 \cap \mathcal{P}_2 \neq \varnothing$ (as in \Cref{fig:Template_Polygon} (C)), then $\mathcal{P}_2$ is not a sink because $\varphi(\alpha_{j,1})$ points out of the sector. A single argument works for $\mathcal{P}_3, \ldots, \mathcal{P}_k$, or when $\partial \mathcal{P}_1 \cap \mathcal{P}_2 = \varnothing$: in these cases, the polygon sector contains some image arc $\alpha_{j,t}$ which points out of $\mathcal{P}_i$ and into $S_{j+1}$, so the polygon sector is not a sink; see \Cref{fig:Template_Polygon} (D), (E). Therefore, no polygon sectors in $S_{j+1}$ are sinks.
\end{proof}

%For example, consider \Cref{fig:TemplateDiskBranchSectors} (A): the green polygon sector is not a sink, because it contains the right pointing $\alpha_{j+1,1}$ in the boundary. Neither purple sector is a polygon sector, because they are not contained in $S_{j+1}$. The orange sectors both contain a pair of coherently oriented plumbing and image arcs, thus it is not a sink. 

%\textcolor{red}{In our ongoing example, there are polygon sectors in $S_2$ and $S_5$, but also $S_3, S_4$ and $S_6$. Using \textcolor{red}{Figure XX}, we see the sectors not covered by \Cref{lemma:PolygonSectorsNotSinks} are not sinks. }

Finally, we check that the horizontal sectors are not sinks. 

\begin{lemma} \label{lemma:TemplatesHorizontalSectorsNotSinks}
%When we apply our template to $\Gamma_j$, there is at most one horizontal sector spanning $\Gamma_j$ which could be a sink. However, 
If we apply our template to $\Gamma_j \cup \Gamma_{j+1}$, then there is at most one potential horizontal sink disk sector spanning $\Gamma_j \cup \Gamma_{j+1}$, and it is the sector containing the unique band in $\Gamma_{j+1}$ enclosed by a right pointing plumbing arc. 
\end{lemma}

%\begin{wrapfigure}{L}{.28\linewidth}
%\begin{minipage}[c][.2\paperheight]{50mm}
%\begin{raggedright}
%\vspace{1.5cm}
\begin{figure}
\labellist %\tiny
\pinlabel {$\mathbbm{b}_{t, s-1}$} at 144 150
\pinlabel {$\mathbbm{b}_{t, s}$} at 140 76
\endlabellist
\includegraphics[scale=.5]{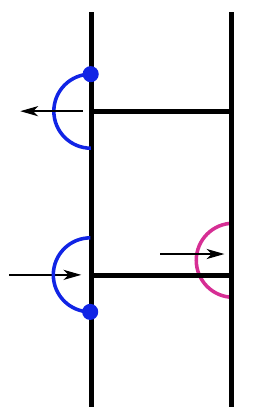}
%\end{raggedright}
%\vspace*{8mm}
\caption{The horizontal sector containing $\mathbbm{b}_{t,s+1}$ is a sink, but we never use these smoothings in our construction.}
\label{fig:SinkDiskInBand}
%\end{minipage}
%\end{wrapfigure}
\end{figure}

\begin{proof}
Suppose $\mathcal{H}$ is a horizontal sector spanning $\Gamma_j \cup \Gamma_{j+1}$. %Then either $\mathcal{H}$ spans exactly $\Gamma_j$, or exactly $\Gamma_{j+1}$, or both $\Gamma_j \cup \Gamma_{j+1}$. We analyze these cases separately.

Recall that every plumbing arc encloses exactly one band on the left, and that every image arc always encloses exactly one band on the right. If $\mathcal{H}$ spans exactly one column, then it exactly contains some band $\mathbbm{b}_{t, s}$ for $t \in \{j, j+1\}$, and $d(t; s-1, s) =0$. In particular, $\mathcal{H} \cap S_t$ in $\alpha_{t, s}$ and $\mathcal{H} \cap S_{t+1}$ in $\varphi(\alpha_{t, s-1})$.
For $\mathcal{H}$ to be a sink, we must have that both $\alpha_{t,s}$ and $\varphi(\alpha_{t,s-1})$ are right pointers (therefore, $\alpha_{t,s-1}$ is a left pointer); see \Cref{fig:SinkDiskInBand}. However, our template is designed to avoid these instructions, so such an $\mathcal{H}$ does not appear.

Alternatively, suppose $\mathcal{H}$ spans both $\Gamma_j \cup \Gamma_{j+1}$. Then, in particular, $\mathcal{H} \cap S_{j+1}$ is non-empty. We claim that $\mathcal{H} \cap S_{j+1}$ is a plumbing arc $\alpha_{j+1, s}$: if $\mathcal{H}$ spans multiple columns, then $\mathcal{H}$ must contain some band $\mathbbm{b}_{j,\ell}$ and some other band $\mathbbm{b}_{j+1, r}$; moreover, when scanning the braid from top-to-bottom, $\mathbbm{b}_{j+1,r}$ must lie between $\mathbbm{b}_{j, \ell-1}$ and $\mathbbm{b}_{j, \ell}$. In fact, for $\mathcal{H}$ to span both $\Gamma_j$ and $\Gamma_{j+1}$, we must have that $\mathbbm{b}_{j+1,r}$ is the last band between $\mathbbm{b}_{j, \ell-1}$ and $\mathbbm{b}_{j, \ell}$. Therefore, we have that $d(j; \ell, \ell+1) \geq 1$. If $d(j; \ell, \ell+1) = 1$, then we are in a configuration as in \Cref{fig:Template_Polygon} (B, C), so we can isotope our arcs so that $\mathcal{H} \cap S_{j+1}$ is a plumbing arc. Otherwise, $d(j; \ell, \ell+1) \geq 2$, and we have a configuration as in \Cref{fig:Template_Polygon} (A, D, E); in these cases, we see that the image arc $\varphi(\alpha_{j, \ell-1})$ at least partially encloses the plumbing arc $\alpha_{j+1,r}$. In particular, $\mathcal{H}$ meets $S_{j+1}$ in $\alpha_{j+1, r}$, and we deduce that $\mathcal{H} \cap S_{j+1}$ is always a plumbing arc $\alpha_{j+1, s} = \alpha_{j+1,r}$.  

If $\mathcal{H} \cap S_{j+1}$ contains $\alpha_{j+1,s}$ for $s \geq 2$, then since $\alpha_{j+1,s}$ is a left pointer, it must point out of $\mathcal{H}$, and $\mathcal{H}$ is not a sink; see \Cref{fig:Template_Polygon} (A, D, E). Alternatively, $\mathcal{H} \cap S_{j+1}$ contains $\alpha_{j+1,1}$, which is a right pointer; \Cref{fig:Template_Polygon} (B, C). We cannot conclude that this sector is not a sink, so this is the unique potential horizontal sink disk spanning $\Gamma_j \cup \Gamma_{j+1}$. 
\end{proof}

With these lemmas established, we now inspect the branch sectors for our ongoing example, as seen in \Cref{fig:Example_Main}. We check that there are no sink disks in $B$.
\begin{itemize}
\item \Seifert We have to check that the Seifert disk sectors $\mathcal{S}_1, \mathcal{S}_2, \mathcal{S}_3, \mathcal{S}_4$ are no sinks. $\mathcal{S}_1$ is not a sink: this sector containing $S_1$ also contains the bands $\mathbbm{b}_{1,2}$ and $\mathbbm{b}_{2,3}$, so in particular, $\partial \mathcal{S}_1$ contains $\alpha_{3,3}$, which points out of the sector. $\partial \mathcal{S}_2$ contains $\varphi(\alpha_{1,1})$, which points out of the region. We applied our template to $\Gamma_2 \cup \Gamma_3$; applying \Cref{lemma:SeifertDisksNotSinks}, we see that $\mathcal{S}_3$ is not a sink. Finally, all but one of the image arcs in $\mathcal{S}_4$ are right pointers, so $\mathcal{S}_4$ is not a sink.
\item \Polygon There are polygon sectors in $S_2$ and $S_3$. There is a unique polygon sector in $S_2$, and it is not a sink because $\alpha_{2,2}$ points out of the sector. We applied our template to $\Gamma_2 \cup \Gamma_3$: by \Cref{lemma:PolygonSectorsNotSinks}, no polygon sectors in $S_3$ are sinks.
\item \Horizontal Suppose $\mathcal{H}$ is a horizontal sector spanning $\Gamma_1$. By inspecting \Cref{fig:Example_Main}, we see that every horizontal sector spanning $\Gamma_1$ also spans $\Gamma_2$. Therefore, it suffices to see why no horizontal sector spanning $\Gamma_2 \cup \Gamma_3$ is a sink. By \Cref{lemma:TemplatesHorizontalSectorsNotSinks}, we see that there is at most one potential horizontal sector spanning $\Gamma_2 \cup \Gamma_3$, and it is the sector containing $\mathbbm{b}_{2,1}$ and $\mathbbm{b}_{3,1}$. However, in \Cref{fig:Example_Main}, we also see that this sector also contains $\mathcal{S}_4$, which we already proved is not a sink. Therefore, we deduce that no horizontal sector in $B$ is a sink.
\end{itemize}

We deduce that the branched surface in \Cref{fig:Example_Main} does not contain any sink disks.

\subsection{Calculate the slopes carried by the train track} \label{section:TrainTrack}
%Explain what slopes are carried by the train track (can cite/rewrite lemmas or propositions from other paper as needed.)
As noted in \Cref{background:foliations}, every product disk $\mathbb{D}_\alpha$ meets $\partial X_K$ in a pair of arcs. Therefore, our branched surface $B$ meets $\partial X_K$ in a train track formed from the Seifert longitude $\lambda$, and one pair of arcs for each disk $\mathbb{D}_\alpha$ used to build $B$. 

\begin{defn}
For a branched surface $B$ in $X_K$, the associated train track on $\partial X_K$ is $\tau_B$.
\end{defn}

The train track $\tau_B$ is a co--oriented trivalent graph on $\partial X_K$.

\begin{defn}
Fix a train track $\tau_B$. Each component of $\mathbb{D}_\alpha \cap \partial X_K$ is a \textbf{sector} of $\tau_B$. %In particular, $\overline{\tau_B - \lambda}$ is a collection of sectors. 
\end{defn}

%here, show an example of the train track induced by the spine, and the result of the co--orientation. 

\begin{defn} \label{defn:contribute}
Let $\Sigma$ denote a sector of $\tau_B$; so, in particular, $\Sigma$ is a co-oriented arc of $\tau_B$. Suppose we begin traversing $\lambda$ from the north-most point of $S_1$ $($the left-most Seifert disk$)$, and travel due south. Let $s_1$ $($resp. $s_2)$ denote the first $($resp. second$)$ endpoint of $\Sigma$ encountered while traveling along $\lambda$. Let $\eta$ denote the normal vector to $\lambda$ at $s_1$. If $\langle \eta, \lambda \rangle = +1$, then we say that \textbf{$\Sigma$ contributes maximally to $\tau_B$}. See \Cref{fig:TrainTrackSetup}.
\end{defn}

Since each product disk $\mathbb{D}_\alpha$ is the sweep-out of a fixed arc $\alpha$, then each disk contributes exactly two sectors to $\tau_B$: these sectors correspond to the two endpoints of $\alpha$. Moreover, exactly one of these sectors contributes maximally to $\tau_B$, and the other does not. The endpoint of an arc $\alpha \subset F$ which contributes maximally to $\tau_B$ is \textbf{bolded} in our figures.

 %Suppose we fix a disk $\mathbb{D}_\alpha$. Let $s, s'$ denote the sectors induced by $\mathbb{D}_\alpha$, and suppose $s$ contributes maximally to $\tau_B$. Then $\lambda \cup s$ carries all slopes $(-\infty, 1)$, and $\lambda \cup s'$ fully carries all slopes  

%we proved that for a fixed co-oriented disk $\mathbb{D}_\alpha$, the sectors $\mathbb{D}_\alpha \cap \partial X_K$ are

\begin{lemma}
Suppose $\alpha$ is a geodesic plumbing arc on $F$. Suppose we begin traversing $\lambda$ from the north-most point of $S_1$ and continue traveling due south. Let $\epsilon_1$ $($resp. $\epsilon_2)$ denote the first (resp. second) endpoint of $\alpha$ encountered while traveling along $\lambda$. If $\alpha$ is co--oriented to the left $($resp. right$)$, then $\epsilon_1$ $($resp. $\epsilon_2)$ contributes maximally to $\tau_B$. \hfill $\Box$
\end{lemma} 

%\begin{proof}
%\textcolor{red}{This is most easily understood by considering \textcolor{red}{Figure XX} below.} \textcolor{red}{See first paper?}
%\end{proof}

Let $\nu$ denote a sector of $\tau_B$ which contributes maximally to $\tau_B$, and let  $\theta$ denote the sector which does not contribute maximally to $\tau_B$. Then $\lambda \cup \nu$ fully carries all slopes in $(0, 1)$, and $\lambda \cup \theta$ fully carries all slopes $(-\infty, 0)$. It is tempting -- but wrong! -- to think that the train track $\tau_B$ carries all slopes $r < C$, where $C$ is the total number of product disks used to build $B$, as this argument excludes the possibility that sectors in $\partial X_K$ can be \textbf{linked}:

\begin{defn}
Let $\alpha$ and $\alpha'$ be co-oriented plumbing arcs on $F$. As described above, each arc contributes exactly one sector to $\tau_B$ which contributes maximally to $\tau_B$ -- let $\nu$ and $\nu'$ denote these sectors of $\tau_B$. Let $s_1, s_2$ denote the endpoints of $\nu$, and $s_1'$ and $s_2'$ denote the endpoints of $\nu'$. If, while traversing $\lambda$, we encounter the endpoints in the order $s_1, s_2, s_1', s_2'$, we say $\alpha$ and $\alpha'$ are \textbf{linked arcs}. Alternatively, if while traversing $\lambda$, we encounter the endpoints in the order $s_1, s_2, s_1', s_2'$, then $\alpha$ and $\alpha'$ are \textbf{unlinked} or \textbf{not linked}. See \Cref{fig:TrainTrackSetup}.
The arcs $\alpha_1, \alpha_2, \alpha_3$ forms a \textbf{linked triple} if $\alpha_2$ is linked with both $\alpha_1$ and $\alpha_3$, but $\alpha_1$ and $\alpha_3$ are not linked with each other.
\end{defn}

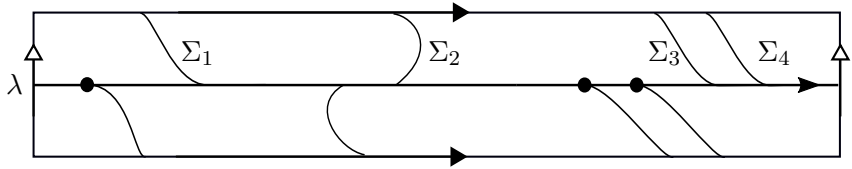
\begin{figure}[h!]\center
\labellist
\pinlabel {$\lambda$} at -5 45
\pinlabel {$\Sigma_1$} at 95 62
\pinlabel {$\Sigma_2$} at 230 62
\pinlabel {$\Sigma_3$} at 350 62
\pinlabel {$\Sigma_4$} at 410 62
\endlabellist
\includegraphics[scale=.7]{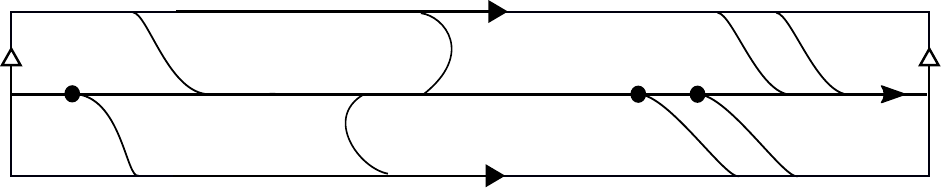}
\caption[]{The sectors $\Sigma_1$, $\Sigma_3$ and $\Sigma_4$ contribute maximally to $\tau_B$. $\Sigma_1 \cup \lambda$ fully carries $(0, 1)$. $\Sigma_2 \cup \lambda$ fully carries $(-\infty, 0)$. The sectors $\Sigma_3$ and $\Sigma_4$ are linked, so $\Sigma_3 \cup \Sigma_4 \cup \lambda$ fully carries $(0, 1)$. }
\label{fig:TrainTrackSetup}
\end{figure}

Thus, to compute the slopes carried by $\tau_B$, we need to first count the total number of product disks used to build $B$, and then deduct the total number of pairs of linked arcs.

\begin{defn}
Let $B$ be a branched surface constructed from a copy of the fiber surface and a collection of co-oriented plumbing disks, and let $\tau_B$ be the associated train track. Then $$\TauSup := \sup \{\ r \ | \ r \text{ is a slope fully carried by $\tau_B$} \}.$$
\end{defn}

That is, $\tau_B$ fully carries all slopes $r$ such that $r < \TauSup$. A priori, it is possible that $\TauSup$ can be any value in $\Q \cup \{\infty\}$. Our construction forces $\TauSup$ to always be a positive integer.

\begin{lemma} \label{lemma:TemplateLinking}
Applying our template to $\Gamma_i$ yields no pairwise linking between any arcs in $S_i$. However, applying our template to $\Gamma_i \cup \Gamma_{i+1}$ yields exactly one pair of linked arcs in $S_i \cup S_{i+1}$.
\end{lemma}

\begin{figure}[h!]\center
\labellist \footnotesize
\pinlabel {$\alpha_{i,t}$} at -8 100
\pinlabel {$\alpha_{i,t+1}$} at -13 25
\pinlabel {$\alpha_{i,t}$} at 180 100
\pinlabel {$\alpha_{i,t+1}$} at 175 25
\pinlabel {$\alpha_{i,t}$} at 355 100
\pinlabel {$\alpha_{i,t+1}$} at 350 25
\pinlabel {$\alpha_{i,t}$} at 540 100
\pinlabel {$\alpha_{i,t+1}$} at 535 25
\endlabellist
\includegraphics[scale=.7]{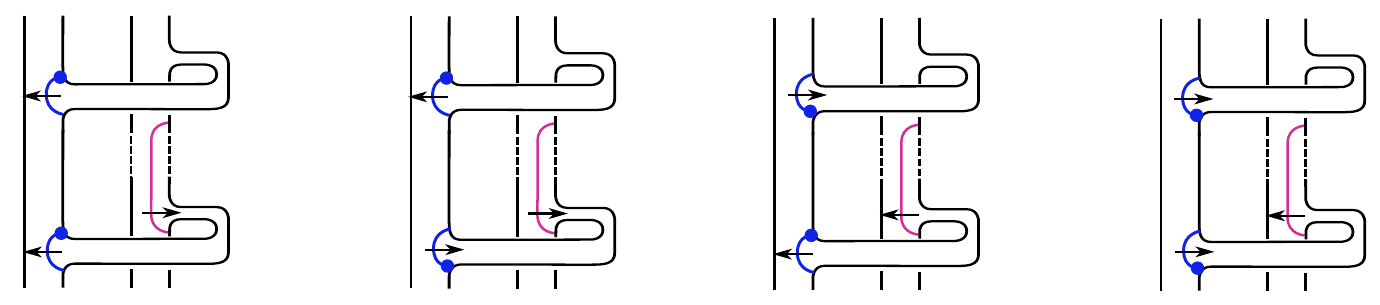}
\caption{If $\alpha_{i,t}$ and $\alpha_{i,t+1}$ are linked, then $\alpha_{i,t}$ is a right pointer and $\alpha_{i,t+1}$ is a left pointer.}
\label{fig:SameDiskLinking}
\end{figure}
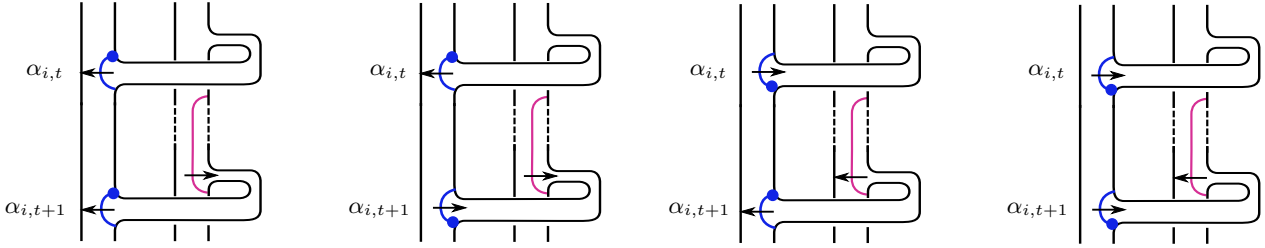

\begin{proof}
%\textcolor{red}{do i need this paragraph?} Fix an arc $\alpha_{s,t}$, and let $\Sigma$ and $\Sigma'$ denote the sectors of $\alpha_{s,t}$ in $\tau_B$; denote the endpoints of $\Sigma$ and $\Sigma'$ by $s_1, s_2$ and $s_1', s_2'$ respectively. As in \Cref{defn:contribute}, $s_1$ (resp. $s_1'$) are the first points of $\Sigma$ (resp. $\Sigma'$) that one encounters when traversing $\lambda$. If $\Sigma$ contributes maximally to $\tau_B$, then as in \Cref{defn:contribute}, $s_1$ is bolded. While traversing along $\lambda$, it is possible that there is a bolded endpoint between $s_1$ and $s_1'$, and thus some arc links with $\alpha_{s,t}$. We analyze when such linking can occur. 

Fix a presentation of $\beta$. As described in \Cref{section:spine}, if a plumbing arc $\alpha$ lies in $S_i$, then after performing the spinal isotopy, $\varphi(\alpha)$ lies in $S_{i+1}$. Therefore, if some arc links with $\alpha_{s,t}$, then the arc either lies in $S_s$ or $S_{s\pm1}$. 

We first study how arcs in the same Seifert disk can be linked. If $\alpha_{i,t}$ links with some other arc $\alpha_{i, r}$, then we must have that $r = t \pm 1$; in particular, they are consecutive plumbing arcs on $S_i$. Without loss of generality, let us assume that we are studying $\alpha_{i,t}$ and $\alpha_{i,t+1}$. As there are two possible co-orientations each for $\alpha_{i,t}$ and $\alpha_{i,t+1}$, there are four possible cases to study. As seen in \Cref{fig:SameDiskLinking}, of these four cases, there is only one case where a linked pair arises: it is when $\alpha_{i,t}$ is a left pointer and $\alpha_{i,t+1}$ is a right pointer. In particular, if we apply our template to a single column $\Gamma_i$, then our design ensures that there are no linked pairs in $S_i$; if we apply our template to $\Gamma_i \cup \Gamma_{i+1}$, then our design of this template produces exactly one pair of linked arcs in $S_{i+1}$.

To finish proving the lemma, we need to check that if we apply our template to $\Gamma_i \cup \Gamma_{i+1}$, there is no linking between arcs in adjacent columns. If $\alpha_{i,t}$ links with some other arc $\alpha_{i+1,r}$, then we must have that $\alpha_{i,t}$ immediately precedes $\alpha_{i+1,r}$, or $\alpha_{i,t}$ immediately follows $\alpha_{i+1,r}$; see \Cref{fig:AdjacentDiskLinking}. For them to be linked, we either have:
\begin{itemize}
\item $\alpha_{i,t}$ is a left pointer immediately preceded by a right pointing $\alpha_{i+1,r}$, as in \Cref{fig:AdjacentDiskLinking} (left), or
%\item $\alpha_{s,t}$ is a left pointer immediately followed by a left pointing $\alpha_{s,q}$, as in \Cref{fig:AdjacentDiskLinking} (left)
\item $\alpha_{i,t}$ is a right pointer; and if $\mathbbm{b}_{i+1, r}$ is the last band between $\alpha_{i,t}$ and $\alpha_{i+1,t}$, then $\alpha_{i+1,r}$ is also a right pointer, as in \Cref{fig:AdjacentDiskLinking} (right).
\end{itemize}

We designed our template for $\Gamma_{i} \cup \Gamma_{i+1}$ to avoid these instructions. We deduce that there is no pairwise linking between arcs in adjacent columns in this case. 
\end{proof}

For the knot $K$ realized as the closure of the braid in (\ref{ex:main}), we see exactly one pair of linked arcs in \Cref{fig:Example_Main}: they are $\alpha_{3,1}$ and $\alpha_{3,2}$. We now calculate the slopes carried by $\tau_B$: in our example, $\Cmin = \Cone$. We compute $\TauSup$:
\begin{align*}
\TauSup &= (\text{\# disks from $\Gamma_1$}) + (\text{\# disks from $\Gamma_2$}) + (\text{\# disks from $\Gamma_3$}) - (\text{\# pairs of linked arcs}) \\
&= 1 + (9-1) + (15-1) - 1 = 22
\end{align*}
Therefore, $\tau_B$ carries all slopes $r < 22$. We now compare this value to the genus $K \approx \widehat{\beta}$: 
$$2g(K)-1 = 2g(F)-1 = \mathcal{C}-n = 27-4 = 23$$

We see that $\TauSup = 2g(K)-2$, and that $\tau_B$ carries all slopes $r < 2g(K)-2$. 

\begin{figure}[h!]\center
\labellist \tiny
\pinlabel {$\alpha_{i,t}$} at -5 98
\pinlabel {$\alpha_{i+1,r}$} at 100 170
%\pinlabel {$\alpha_{s+1,q}$} at 100 27
%
\pinlabel {$\alpha_{i,t}$} at 235 170
\pinlabel {$\alpha_{i+1,r}$} at 340 112
\endlabellist
\includegraphics[scale=.7]{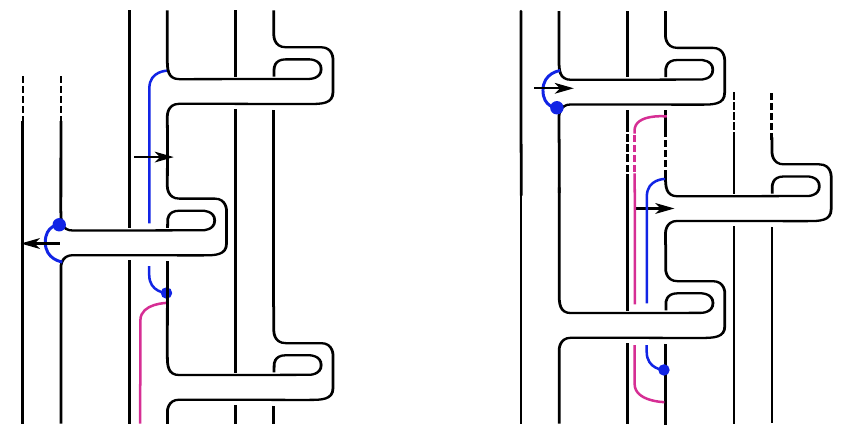}
\caption{How linking can occur between arcs in adjacent Seifert disks.}
\label{fig:AdjacentDiskLinking}
\end{figure}
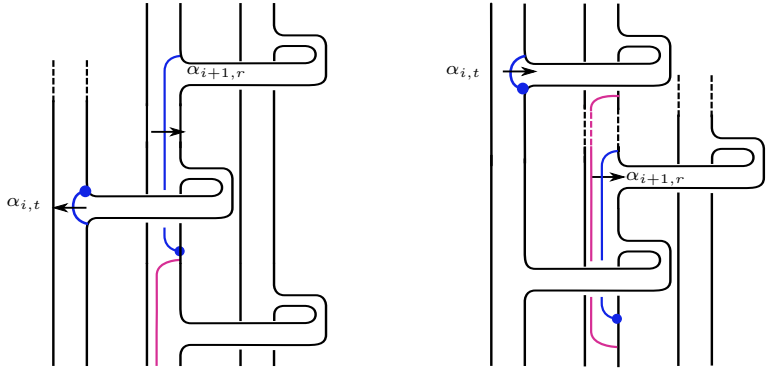

\subsection{Constructing taut foliations} \label{section:Foliations} We are ready to build taut foliations in many manifolds obtained by Dehn surgery along $\widehat{\beta}$, for $\beta$ as in (\ref{ex:main}). By \Cref{section:SinkDisk}, we have constructed a sink disk free branched surface $B$. By \Cref{prop:BisLaminar}, $B$ is laminar. In \Cref{section:TrainTrack}, we showed the train track $\tau_B$ carries all slopes in the interval $(-\infty, 2g(K)-2)$. Applying \Cref{thm:foliations}, we conclude that $S^3_r(K)$ admits a taut foliation for all $r < 2g(K)-2$.

This concludes the example, and the outline of the construction. The key is leveraging the distribution of crossings throughout the braid to build a laminar branched surface. The delicate portions of the argument occur while arguing the branch surface is both sink disk free while attempting to maximize $\TauSup$.

\section{Proof of \Cref{thm:main} for positive $4$--braids}
\label{section:Proof4Braids}

%%%%%%%%
%\section{Proof of \Cref{thm:main}}
% \label{section:proof}
Proving \Cref{thm:main} requires dividing the set of (prime) positive braid knots $\beta$ into two categories, based on the distribution of crossings in the  columns used for the construction. 

\begin{defn}
A positive $n$--braid $\beta$ is \textbf{generic} if for all $1 \leq i \leq n-1$, $c_i \geq 3$. If, for some $i$, $1 \leq i \leq n-1$, $c_i =2$, then $\beta$ is \textbf{sparse}.
\end{defn}

%\begin{enumerate}
%\item the \textbf{``generic"} case: positive braids where $c_i \geq 3$ for all $i \not \equiv m \mod 3$, and 
%\item the \textbf{``edge"} case: positive braids where $c_i =2$ for some $i \not \equiv m \mod 3$. 
%\end{enumerate}

\begin{defn}
Suppose $\mathcal{S}$ and $\mathcal{R}$ are portions of a branched surface $B$. If $\mathcal{S}$ and $\mathcal{R}$ are part of the same branch sector, we write $\mathcal{S} \sim \mathcal{R}$. 
\end{defn}

In Sections \ref{section:4--braids} and \ref{section:Edge4Braids}, we prove \Cref{thm:main} for generic and sparse 4--braids, respectively.  

%%%%%%%
\subsection{Generic positive 4--braids} \label{section:4--braids}
To prove \Cref{thm:main} for generic positive 4--braids, we first note that such braids have only three columns, so there are three possible values for $\Cmin$. This determines how many product disks to take from the three columns, as summarized in \Cref{table:Generic4Braids}.

\begin{table}[H]
\begin{tabular}{|c || c | c | c | c |}
\hline
\textit{Value of} & \multicolumn{3}{c|}{\textit{\# of disks chosen from}} & \textit{total $\#$} \\
\cline{2-4}
$\Cmin$ & $\Gamma_1$ & $\Gamma_2$ & $\Gamma_3$ & \textit{disks added} \\
\hline
$\Cone$ & 1 & $c_2-1$ & $c_3-1$ & $c_2 + c_3-1$ \\
$\Ctwo$ & $c_1-1$ & $1$ & $c_3-1$ & $c_1 + c_3-1$ \\
$\Cthree$ & $c_1-1$ & $c_2-1$ & $1$ & $c_1 + c_2-1$ \\
\hline
\end{tabular} 
\caption{The cases to consider for positive 4--braids.}
\label{table:Generic4Braids}
\end{table}

\begin{prop} \label{prop:Generic4BraidCMinIsCThree}
Suppose $\beta$ is a generic positive 4--braid such that $\Cmin = \Cthree$. There exists a sink disk free branched surface $B$ for $\widehat{\beta}$ such that $\tau_B$ has exactly one linked pair or one linked triple.
\end{prop}

\begin{proof}
For a positive 4-braid as in the hypotheses, we can always find a conjugated presentation of $\beta$ which presents it as $\beta \approx \sigma_1 \omega_1 \sigma_1 \omega_2$, where $\omega_1$ is non-empty, contains no $\sigma_1$ letters, and contains some $\sigma_2$ letters. Since $\beta$ is standardized, with this presentation of the braid, we know that $d(1; 1,2) \geq 2$. We use this presentation of $\beta$ to build the branched surface. 

To build $B$, we first apply our template to $\Gamma_1 \cup \Gamma_2$. Next, we have to choose a single co--oriented product disk from $\Gamma_3$. This choice is dependent on the value of $d(2;1,2)$:

\begin{itemize}
\item \textbf{Case (A):} if $d(2; 1,2) \geq 2$, then co--orient $\alpha_{3,1}$ to the left. 
\item \textbf{Case (B):} if $d(2; 1,2) = 0$, then we will modify the presentation of $\beta$. Based on that forthcoming presentation of $\beta$, we will co--orient $\alpha_{3,1}$ to the right. 
\end{itemize}

Notice that $d(2; 1,2) \neq 1$, because this contradicts the assumption that $\beta$ is standardized.

\tcbox[size=fbox, colback=gray!30]{\textbf{Case (A):} $d(2;1,2) \geq 2$.} 

An example of our branched surface appears in \Cref{fig:Positive4Braid_CminC3_new}. We prove that $B$ has no sink disks.

%\begin{wrapfigure}{l}{.28\linewidth}
%\begin{minipage}[c][.3\paperheight]{50mm}
%\begin{raggedright}
%\vspace{1.5cm}
%\labellist
%\pinlabel {$\Gamma_1$} at 35 140
%\pinlabel {$\Gamma_2$} at 62 140
%\pinlabel {$\Gamma_3$} at 89 140
%\endlabellist
%\includegraphics[scale=1.45]{Positive4Braid_CminC3_new}
%\end{raggedright}
%%\vspace*{8mm}
%\caption[]{}
%\label{fig:SinkDiskInBand}
%\end{minipage}
%\end{wrapfigure}

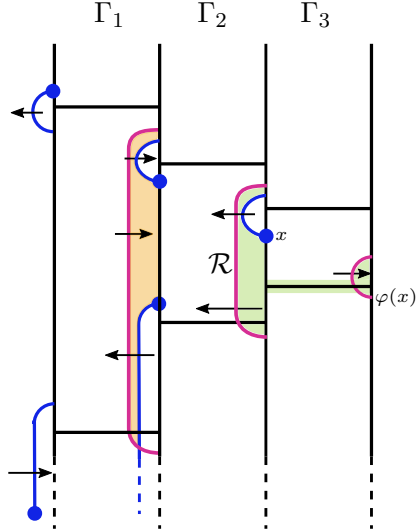
\begin{figure}[h!]\center
\labellist
\pinlabel {$\Gamma_1$} at 31 140
\pinlabel {$\Gamma_2$} at 58 140
\pinlabel {$\Gamma_3$} at 85 140
\pinlabel {$\mathcal{R}$} at 60 75
\tiny
\pinlabel {$x$} at 76 82
\pinlabel {$\varphi(x)$} at 106 66
\endlabellist
\includegraphics[scale=1.45]{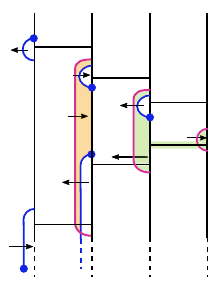}
\vspace{1em}
\caption{A positive 4--braid with $\Cmin = \Cthree$, and $d(2; 1,2) \geq 2$. We have applied our template to $\Gamma_1 \cup \Gamma_2$, and have chosen a single product disk from $\Gamma_3$. The resulting branched surface $B$ is sink disk free. $\mathcal{R}$ denotes the green sector whose boundary includes $\varphi(\alpha_{2,1})$ and $\alpha_{3,1}$.
}
\label{fig:Positive4Braid_CminC3_new}
\end{figure}

\begin{itemize}%}[\leftmargin=$\bullet$]
\item \Seifert by design, $\alpha_{1,2}$ points out of $\mathcal{S}_1$, so this sector is not a sink. By \Cref{lemma:SeifertDisksNotSinks}, $\mathcal{S}_2$ is not a sink. Since $d(2;1,2) \geq 2$ and $\alpha_{3,1}$ is the unique co-oriented plumbing arc in $S_3$, $\varphi(\alpha_{2,1})$ encloses only $\alpha_{3,1}$. In particular, $\partial \mathcal{S}_3$ contains $\varphi(\alpha_{2,2})$; this arc is a right pointer, so $\mathcal{S}_3$ is not a sink. Finally, since $\mathcal{S}_4 \sim \mathbbm{b}_{3,1}$, we have $\partial \mathcal{S}_4$ contains $\alpha_{3,1}$. This arc is a left pointer exiting $\mathcal{S}_4$, so this sector is not a sink. We deduce that none of the Seifert disk sectors are sinks. 
\item \Polygon if a polygon sector $\mathcal{P}$ exists, it must lie in either $S_2$ or $S_3$, as these are the only disk sectors containing both plumbing and image arcs. By \Cref{lemma:PolygonSectorsNotSinks}, the polygon sectors in $S_2$ are not sinks. By construction, there is a unique potential polygon sector (in \Cref{fig:Positive4Braid_CminC3_new}, it is the green sector labeled $\mathcal{R}$) in $S_3$, and it contains $\varphi(\alpha_{2,1})$ in its boundary. As $\alpha_{2,1}$ is a right pointer, $\varphi(\alpha_{2,1})$ is a left pointer exiting $\mathcal{R}$, so this sector is not a sink. We deduce that there are no polygon sector sink disks in $B$. 
\item \Horizontal by \Cref{lemma:TemplatesHorizontalSectorsNotSinks}, there is at most one possible horizontal sector $\mathcal{H}$ spanning $\Gamma_1 \cup \Gamma_2$ that is a sink, and it is the sector containing $\mathbbm{b}_{2,1}$. By design, the right endpoint of $\mathbbm{b}_{2,1}$ is not enclosed by any arcs, thus $\mathcal{H} \sim \mathcal{S}_3$; we already argued the latter is not a sink. We deduce that no horizontal sector spanning $\Gamma_1 \cup \Gamma_2$ is a sink. 

We argue that no horizontal sector $\mathcal{H}$ spanning $\Gamma_3$ is a sink. If $\mathcal{H}$ contains $\mathbbm{b}_{3,1}$, then $\mathcal{H} \sim \mathcal{S}_4$, which is not a sink. Next, notice that $\mathbbm{b}_{3, 2}, \ldots, \mathbbm{b}_{3, d(2; 1,2)}$ are all part of the same branch sector, as they are all enclosed, on the left, by $\varphi(\alpha_{2,1})$. This means that $\mathbbm{b}_{3, 2} \sim \ldots \sim \mathbbm{b}_{3, d(2; 1,2)} \sim \mathcal{R}$; in our polygon sector analysis, we already proved this sector is not a sink. Finally, suppose $\mathcal{H} \sim \mathbbm{b}_{3,s}$, for $d(2; 1,2) + 1 \leq s \leq c_3$. The bands $\mathbbm{b}_{3,s}$ are not enclosed, on the right, by any arcs. Therefore, $\mathcal{H} \sim \mathcal{S}_4$, which we already argued is not a sink. We deduce that no horizontal sectors in $B$ are sinks.
\end{itemize}

Therefore, we can construct a sink disk free branched surface in \textbf{Case (A)}. We now determine the number of linked pairs in $\tau_B$. By  \Cref{lemma:TemplateLinking}, the arcs $\alpha_{2,1}$ and $\alpha_{2,2}$ are the unique pair of linked arcs in $S_1 \cup S_2$. By consulting \Cref{fig:Positive4Braid_CminC3_new}, one sees that $\alpha_{3,1}$ does not link with any arcs in $S_2$: while traversing $\lambda$ on $\tau_B$, there are no sectors coming between $x$ and $\varphi(x)$. We therefore have a unique pair of linked arcs in $\tau_B$. This concludes the construction in \textbf{Case (A)}.

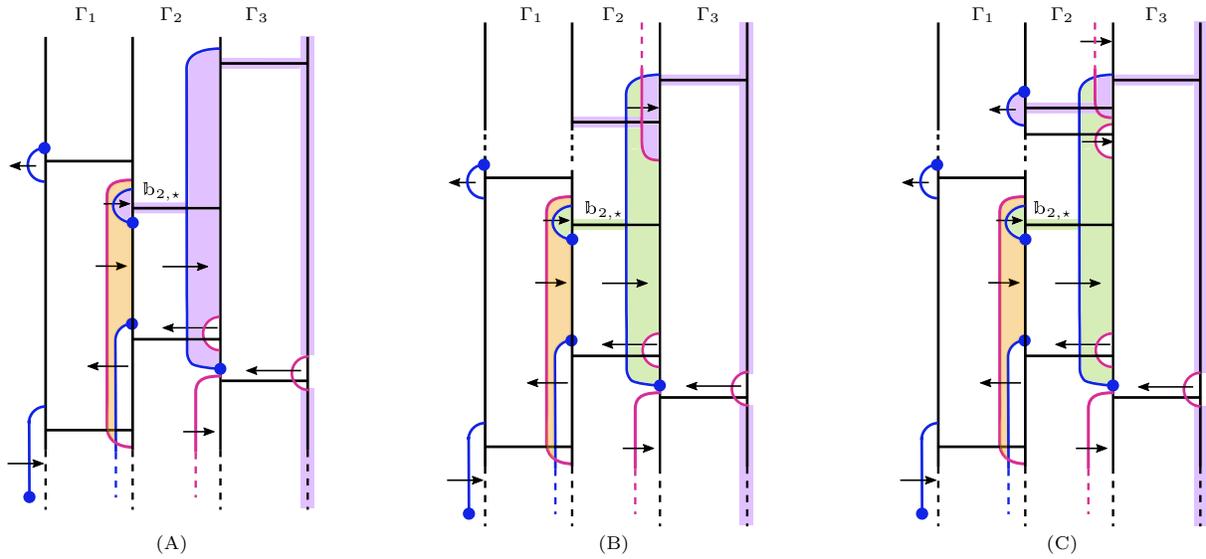
\begin{figure}[h!]\center
\labellist \tiny
\pinlabel {$\Gamma_1$} at 31 168
\pinlabel {$\Gamma_2$} at 58 168
\pinlabel {$\Gamma_3$} at 85 168
\pinlabel {$\mathbbm{b}_{2,\star}$} at 55 112
\pinlabel {(A)} at 58 0
%\pinlabel {$\mathcal{R}$} at 60 75
\pinlabel {$\Gamma_1$} at 172 168
\pinlabel {$\Gamma_2$} at 198 168
\pinlabel {$\Gamma_3$} at 228 168
\pinlabel {$\mathbbm{b}_{2,\star}$} at 195 106
\pinlabel {(B)} at 198 0
\pinlabel {$\Gamma_1$} at 315 168
\pinlabel {$\Gamma_2$} at 340 168
\pinlabel {$\Gamma_3$} at 370 168
\pinlabel {$\mathbbm{b}_{2,\star}$} at 337 106
\pinlabel {(C)} at 340 0
\endlabellist
\includegraphics[scale=1.2]{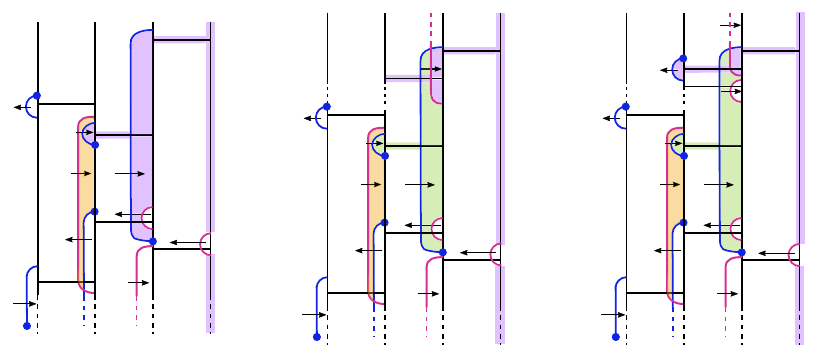}
\vspace{1em}
\caption{A positive 4--braid with $\Cmin = \Cthree$, and $d(2; 1,2) =0$. We have applied our template to $\Gamma_1 \cup \Gamma_2$, then modified the presentation of $\beta$ so that it begins with a $\sigma_3$ letter. The three cases of whether $\rho_3$ contains zero, one, or at least two $\sigma_2$ letters are displayed in frame \textbf{(A), (B),} and \textbf{(C)}, respectively. Said differently, \textbf{(A)} is where $d_L(3; 1,2) = 2$, \textbf{(B)} is where $d_L(3; 1,2) = 3$, and \textbf{(C)} is where $d_L(3; 1,2) \geq 4$.}
\label{fig:Positive4Braid_CminC3_new_2}
\end{figure}

\tcbox[size=fbox, colback=gray!30]{\textbf{Case (B):} $d(2;1,2) =0$.} 
We have already applied our template to $\Gamma_1 \cup \Gamma_2$; to proceed, we modify the presentation of $\beta$, and only then choose a product disk from $\Gamma_3$. We follow the following procedure:

\begin{enumerate}
\item At present, $\beta$ is presented as $\sigma_1 \omega_1 \sigma_1 \omega_2$, where $\omega_1$ contains at least two $\sigma_2$ letters and no $\sigma_1$ letters. By assumption, we have $d(2;1,2)=0$, so by applying braid relations if necessary, we can write our braid as $\sigma_1 \sigma_2^2 \theta_1 \sigma_1 \theta_2$, where $\theta_1$ contains no $\sigma_1$ letters. We now identify the last $\sigma_3$ letter in the braid. This allows us to write the braid as $\sigma_1 \sigma_2^2 \rho_1 \sigma_1 \rho_2 \sigma_3 \rho_3$, where $\rho_3$ contains no $\sigma_3$ letters. 
\item Now pivot $\beta$ about the identified $\sigma_3$. Now, our braid reads as $\sigma_3 \rho_3 \sigma_1 \sigma_2^2 \rho_1 \sigma_1 \rho_2$. 
\end{enumerate}

\Cref{fig:Positive4Braid_CminC3_new_2} presents three examples of the result. %: on the left, an example where $\beta''$ has no $\sigma_2$ letters; in the middle, and example where $\beta''$ has exactly one $\sigma_2$ letter, and potentially some $\sigma_1$ letters; on the right, an example where $\beta''$ has two $\sigma_2$ letters, and potentially some $\sigma_1$ letters.
This modified presentation identifies a new ``first'' product disk in $\Gamma_3$. We emphasize: it is possible that the relative position of the unique plumbing arc in $S_2$ which is enclosed, on the left, by a right pointer has changed -- this arc used to be called $\alpha_{2,1}$, but we know refer to it as $\alpha_{2, \star}$, as in \Cref{fig:Positive4Braid_CminC3_new_2}. Moreover, note that in our new presentation, the band $\mathbbm{b}_{3,1}$ may be enclosed, on the left, by a right pointing image arc, as in \Cref{fig:Positive4Braid_CminC3_new_2} (B, C). 
We already built our branched surface in $\Gamma_1 \cup \Gamma_2$; we now add the product disk swept out by $\alpha_{3,1}$, and co--orient this arc to the right. We check that the resulting $B$ is sink disk free.

\begin{itemize}%[\leftmargin=$\bullet$]
\item \Seifert the arc $\alpha_{2,1}$ points out of $\mathcal{S}_1$, so the sector $\mathcal{S}_1$ is not a sink. By \Cref{lemma:SeifertDisksNotSinks}, $\mathcal{S}_2$ is not a sink. By design, $\alpha_{3,1}$ is contained in $\partial \mathcal{S}_3$; this arc is a right pointer, so $\mathcal{S}_3$ is not a sink. We now argue that $\mathcal{S}_4$ is not a sink. 

Since $\beta$ is generic, $c_3 \geq 3$. In particular, $\mathbbm{b}_{3,3}$ is not enclosed, on the right, by any arc, so $\mathcal{S}_4 \sim \mathbbm{b}_{3,3}$. We now understand how $\mathbbm{b}_{3,3}$ is enclosed on the left. 

	\begin{itemize}
	\item[$\boldsymbol{\circ}$] There is a unique left pointing image arc in $S_3$, and by design, it does not enclose any bands in $\Gamma_3$. In particular, $\mathbbm{b}_{3,3}$ cannot be enclosed by a left pointing image arc. 
	\item[$\boldsymbol{\circ}$] If $\mathbbm{b}_{3,3}$ is not enclosed on the left, then $\mathcal{S}_4 \sim \mathcal{S}_3$; we already argued the latter is not a sink. 
	\item[$\boldsymbol{\circ}$] Suppose $\mathbbm{b}_{3,3}$ is enclosed, on the left, by the right pointer $\varphi(\alpha_{2,t})$. As image arcs always enclose bands on the right, we have that $\mathcal{S}_4 \sim \mathbbm{b}_{3,3} \sim \mathbbm{b}_{2,t+1}$. Therefore, $\mathcal{S}_4$ is contained in a horizontal sector which spans $\Gamma_2$. By \Cref{lemma:TemplatesHorizontalSectorsNotSinks}, we know that there is a unique horizontal sector spanning $\Gamma_1 \cup \Gamma_2$ which could be a sink, and it is the sector containing the band $\mathbbm{b}_{2, \star}$ which is enclosed, on the left, by a right pointing plumbing arc. Since $\varphi(\alpha_{2,t})$ is a right pointer, then $\alpha_{2,t}$ is a left pointer; by the design of our template, we know that $\alpha_{2,t+1}$ is either also a left pointer, or the band $\mathbbm{b}_{2,t}$ is not enclosed by any plumbing arc. Regardless, we deduce that $t+1 \neq \star$, and so $\mathcal{S}_4$ is not a sink.  
	\end{itemize}
Therefore, none of the Seifert disk sectors are sink disks.
\item \Polygon if a polygon sector $\mathcal{P}$ exists, it lies in either $S_2$ or $S_3$. By \Cref{lemma:PolygonSectorsNotSinks}, the polygon sectors in $S_2$ are not sinks. 

Suppose $\partial \mathcal{P} \subset S_3$. Then, by design, $\partial \mathcal{P}$ contains $\alpha_{3,1}$ and $\varphi(\alpha_{2,\star})$; see \Cref{fig:Positive4Braid_CminC3_new_2}. In particular, $\mathcal{P} \sim \mathbbm{b}_{2,\star}$. Once again, we have three cases to consider; the cases are determined by whether $d_L(3;1,2) = 2$ or $d_L(3;1,2) \geq 3$. (These are exactly the cases displayed in  \Cref{fig:Positive4Braid_CminC3_new_2}, and correspond to whether $\rho_3$ contains zero, one, or at least two $\sigma_2$ letters.)
	\begin{itemize}
	\item[$\boldsymbol{\circ}$] If $d_L(3;1,2) = 2$, as in \Cref{fig:Positive4Braid_CminC3_new_2} (A), then not only does $\mathcal{P}$ contain $\mathbbm{b}_{2,1}$, but $\mathcal{P} \sim \mathbbm{b}_{3,1}$. Therefore, $\mathcal{P} \sim \mathcal{S}_4$. We already proved $\mathcal{S}_4$ is never a sink, so neither is $\mathcal{P}$. 
	\item[$\boldsymbol{\circ}$] If $d_L(3;1,2) \geq 3$, as in  \Cref{fig:Positive4Braid_CminC3_new_2} (B, C), then $\partial \mathcal{P}$ also contains some right pointing image arc. (In our current presentation of our braid, this arc is $\varphi(\alpha_{2,c_2})$.) Therefore, $\mathcal{P}$ is not a sink.
	\end{itemize}
\item \Horizontal We first observe: if $d_L(3; 1,2) = 2$ (as in \Cref{fig:Positive4Braid_CminC3_new_2} (A)), then $\star = 1$; however, if $d_L(3; 1,2) \geq 3$, then $\star = d_L(3;1,2) - 1$. 

Applying \Cref{lemma:TemplatesHorizontalSectorsNotSinks}, there is at most one potential horizontal sink disk sector spanning $\Gamma_1 \cup \Gamma_2$, and in the current presentation of the braid, it is the sector containing $\mathbbm{b}_{2,\star}$. Therefore, to finish the proof, we need to prove two things: (1) that the horizontal sector containing $\mathbbm{b}_{2,\star}$ is not a sink, and (2) that no horizontal sector spanning $\Gamma_3$ is a sink. 
\end{itemize}

\tcbox[size=fbox, colback=gray!10]{The horizontal sector $\mathcal{H}$ containing $\mathbbm{b}_{2,\star}$ is not a sink.} 

	\begin{itemize}
	\item[$\boldsymbol{\circ}$] If $d_L(3;1,2) = 2$, then $\mathbbm{b}_{2,\star} = \mathbbm{b}_{2,1}$. When we studied the polygon sectors, we showed that $\mathbbm{b}_{2,1} \sim \mathcal{S}_4$, which is not a sink.
	\item[$\boldsymbol{\circ}$] If $d_L(3;1,2) \geq 3$, then $\mathbbm{b}_{2, \star}$ is in the same sector as a polygon sector in $S_3$, which we already proved is not a sink.
	\end{itemize}

Therefore, no horizontal sector spanning $\Gamma_1 \cup \Gamma_2$ is a sink.

\tcbox[size=fbox, colback=gray!10]{No horizontal sector spanning $\Gamma_3$ is a sink.} 

Suppose $\mathcal{H}$ is a horizontal sector spanning $\Gamma_3$. Notice that there is a unique band in $\Gamma_3$ which is enclosed, on the right, by an image arc: it is $\mathbbm{b}_{3,2}$. Therefore, if $\mathcal{H} \sim \mathbbm{b}_{3,2}$, then since this band is enclosed, on the right, by a left pointer, $\mathcal{H}$ is not a sink; see \Cref{fig:Positive4Braid_CminC3_new_2}. For all other bands $\mathbbm{b}_{3,r}, r \neq 2$, if $\mathcal{H} \sim \mathbbm{b}_{3,r}$, $\mathcal{H} \sim \mathcal{S}_4$. We already proved that $\mathcal{S}_4$ is not a sink, so we deduce that no horizontal sector spanning $\Gamma_3$ is a sink. 

This concludes our sink disk analysis; we deduce that when $d(2; 1,2) = 0$, the branched surface $B$ is sink disk free. We now turn to our analysis of the train track. 

By \Cref{lemma:TemplateLinking}, this is a unique linked pair coming from arcs in $S_1 \cup S_2$. Additionally, we see that $\alpha_{2, \star+1}$ is linked with both $\mathbbm{b}_{2, \star}$ and $\mathbbm{b}_{3,1}$, but $\mathbbm{b}_{2, \star}$ and $\mathbbm{b}_{3,1}$ are not linked with each other. Therefore, we have a single linked triple of arcs. 
\end{proof}

\begin{prop}  \label{prop:Generic4BraidCMinIsCOne}
Suppose $\beta$ is a generic positive 4--braid such that $\Cmin = \Cone$. Then there exists a sink disk free branched surface $B$ for $\widehat{\beta}$ such that $\tau_B$ has exactly one pair of linked arcs. 
\end{prop}

\begin{proof}

First, we find a suitable presentation of $\beta$: by conjugating and applying the braid relations, we may assume $\beta \approx \sigma_2 \sigma_1 w_1 \sigma_3 w_2$, where $w_1$ is either empty, or contains only $\sigma_1$ or $\sigma_2$ letters. (We quickly justify this assumption: as $\beta$ is generic, $c_1 \geq 3$. In particular, after potentially applying the braid relation $\sigma_1 \sigma_3 = \sigma_3 \sigma_1$, we can write the braid such that there exists at least one $\sigma_2$ letter followed by at least one $\sigma_1$ letter; conjugate $\beta$ to begin with said $\sigma_2$ letter. Thus, $\beta \approx \sigma_2 \sigma_1 \beta''$, where $\beta''$ is arbitrary. Identifying the first $\sigma_3$ letter in $\beta''$ allows us to write the braid as $\beta \approx \sigma_2 \sigma_1 w_1 \sigma_3 w_2$, where $w_1$ could be the empty word.)

We use this presentation of $\beta$ to build $B$: we will apply our standard template to $\Gamma_2 \cup \Gamma_3$, and include the product disk swept out by $\alpha_{1,1}$. The co--orientation of $\alpha_{1,1}$ is dependent on $d(1; 1,2)$; recall that this value is either zero, or at least two, as $\beta$ is standardized. 

\textbf{Case (A):} We set $\alpha_{1,1}$ to be a right pointer if:
\begin{itemize}
\item If $d(1; 1,2) = 0$, as in \Cref{fig:Positive4Braid_CminC1_new} (A).
\item If $d(1; 1,2) \geq 2$ and $\varphi(\alpha_{1,1})$ encloses at least two arcs, all of which are coherently oriented, as in \Cref{fig:Positive4Braid_CminC1_new}, (B), (C). These frames show two different cases: in (B), $2 \leq d(1; 1,2) \leq c_2 - 2$, and in (C), $d(1; 1,2) = c_2 - 1$. 
\end{itemize}

\textbf{Case (B):} We set $\alpha_{1,1}$ to be a left pointer if:
\begin{itemize}
\item If $d(1; 1,2) \geq 2$ and $\varphi(\alpha_{1,1})$ encloses exactly one co--oriented arc, as in \Cref{fig:Positive4Braid_CminC1_new}, (D), (E).
\item If $d(1; 1,2) \geq 2$ and $\varphi(\alpha_{1,1})$ encloses at least two arcs, all of which are co--oriented, but such that the arcs are not coherently oriented; see \Cref{fig:Positive4Braid_CminC1_new}, (F).
\item If $\varphi(\alpha_{1,1})$ encloses only right pointing plumbing arcs (not pictured). 
\end{itemize}

\begin{figure}[h!]\center
\labellist \tiny
\pinlabel {(A)} at 42 175
\pinlabel {(B)} at 152 175
\pinlabel {(C)} at 262 175
\pinlabel {(D)} at 42 -5
\pinlabel {(E)} at 152 -5
\pinlabel {(F)} at 262 -5
\pinlabel {$\Gamma_1$} at 31 335
\pinlabel {$\Gamma_2$} at 58 335
\pinlabel {$\Gamma_1$} at 141 335
\pinlabel {$\Gamma_2$} at 168 335
\pinlabel {$\Gamma_1$} at 251 335
\pinlabel {$\Gamma_2$} at 278 335
\pinlabel {$\Gamma_1$} at 31 155
\pinlabel {$\Gamma_2$} at 58 155
\pinlabel {$\Gamma_1$} at 141 155
\pinlabel {$\Gamma_2$} at 168 155
\pinlabel {$\Gamma_1$} at 251 155
\pinlabel {$\Gamma_2$} at 278 155
\pinlabel {$\mathbbm{b}_{2,c_2}$} at 276 240
\pinlabel {$\mathbbm{b}_{2,c_2}$} at 56 76
\pinlabel {$\mathbbm{b}_{2,c_2}$} at 165 76
%
%\tiny
%\pinlabel {$x$} at 75 92
%\pinlabel {$\varphi(x)$} at 108 75
%
\endlabellist
\includegraphics[scale=1.25]{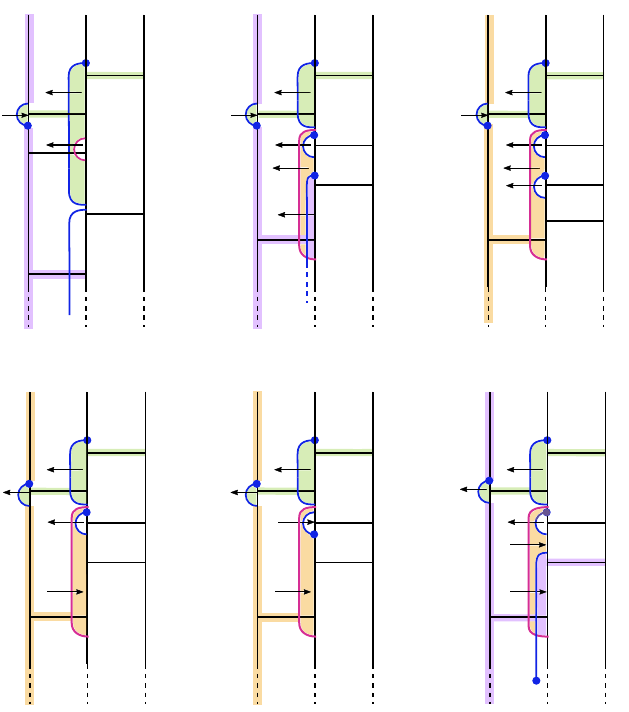}
\vspace{1em}
\caption{A positive 4--braid with $\Cmin = \Cone$. We only see $\Gamma_1 \cup \Gamma_2$ displayed in the figure. In (A), (B), and (C) $\alpha_{1,1}$ is a right pointer; in (D), (E), and (F), $\alpha_{1,1}$ is a left pointer. We see that $\alpha_{1,1}$ does not link with any arcs in $S_2$. In (C), (D), and (E), we also have that $d(1;1,2) = c_2-1$.}
\label{fig:Positive4Braid_CminC1_new}
\end{figure}
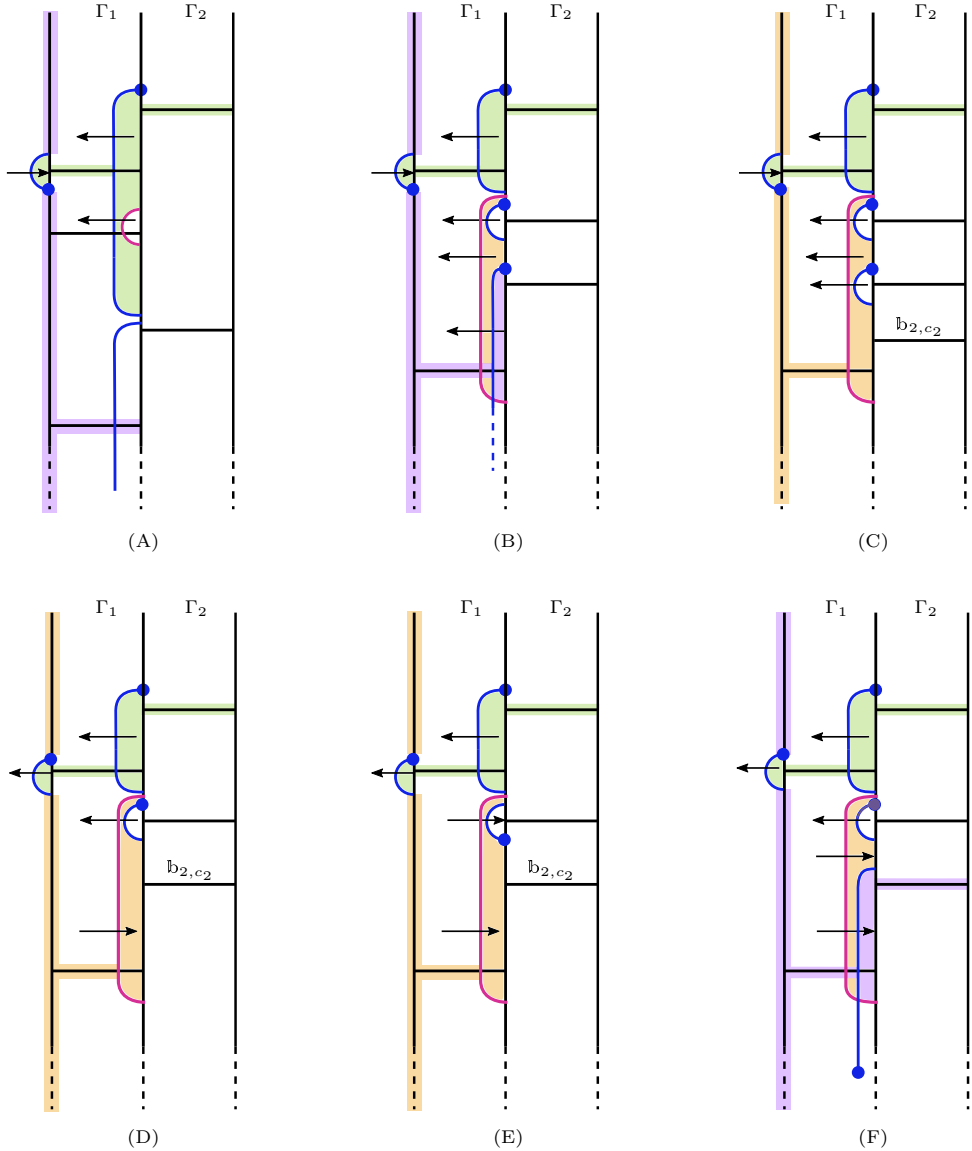

These instructions yield a branched surface $B$ with the characteristics described in \Cref{table:Generic4Braids}. By \Cref{lemma:TemplateLinking}, the application of our template to $\Gamma_2 \cup \Gamma_3$ contributes a single linked pair. Therefore, to show there is exactly one pair of linked arcs, it remains to show that $\alpha_{1,1}$ does not link with any other arcs. Since linking between arcs can only occur if the arcs are in adjacent columns, we need only verify that $\alpha_{1,1}$ does not link with any arc in $S_2$. By consulting \Cref{fig:Positive4Braid_CminC1_new}, we see that regardless of the co--orientation of $\alpha_{1,1}$, there is never a bolded endpoint in $\partial S_2$ that comes between the bolded endpoint of $\alpha_{1,1}$ and its isotoped image. Thus, $\alpha_{1,1}$ does link with any plumbing arc in $S_2$, and $B$ has a unique pair of linked arcs.

Next, we show that $B$ is sink disk free.

\begin{itemize}%[leftmargin=-2em]
\item \Seifert By \Cref{lemma:SeifertDisksNotSinks}, $\mathcal{S}_3$ is not a sink. By construction, $S_4$ contains only image arcs, and $\varphi(\alpha_{3,2}) \subset \partial S_4$. This arc is a right pointer, exiting $\mathcal{S}_4$, so $\mathcal{S}_4$ is not a sink. To show that $\mathcal{S}_1$ and $\mathcal{S}_2$ are not sinks, we analyze \textbf{Case (A)} and \textbf{Case (B)} separately.

\tcbox[size=fbox, colback=gray!30]{\textbf{Case (A):} $\alpha_{1,1}$ is a right pointer.} 

We refer to \Cref{fig:Positive4Braid_CminC1_new} (A), (B), (C)) throughout. In \textbf{Case (A)}, we have that $\alpha_{1,1}$ is a right pointer; this arc is in $\partial \mathcal{S}_1$, so $\mathcal{S}_1$ is not a sink. It remains to show that $\mathcal{S}_2$ is not a sink. 

If $S_2$ contains a right pointing plumbing arc $\alpha_{2,\star}$, then by design, this arc does not co--bound a polygon sector in $S_2$ with $\varphi(\alpha_{1,1})$. In particular, $\alpha_{2,\star} \subset \partial S_2$, and $\mathcal{S}_2$ is not a sink. Otherwise, $S_2$ contains only left pointing plumbing arcs, and we need to do some different analysis. 

\begin{itemize}
\item[$\boldsymbol{\circ}$] Suppose $d(1; 1,2) = 0$, as in \Cref{fig:Positive4Braid_CminC1_new} (A). Since $\beta$ is generic, $c_2 \geq 3$, and $\mathcal{S}_2 \sim \mathbbm{b}_{2,c_2}$. In particular, $\mathcal{S}_2$ is a horizontal sector spanning $\Gamma_2$. Applying \Cref{lemma:TemplatesHorizontalSectorsNotSinks}, we know that no horizontal sector spanning $\Gamma_2$ is a sink, so $\mathcal{S}_2$ is not a sink. 
\item[$\boldsymbol{\circ}$] Suppose $2 \leq d(1; 1,2) \leq c_2-2$, as in \Cref{fig:Positive4Braid_CminC1_new} (B). In this case, $\mathcal{S}_2 \sim \mathbbm{b}_{2,c_2}$, and we can apply the same argument as in the $d(1; 1,2)=0$ case to see that $\mathcal{S}_2$ is not a sink. 
\item[$\boldsymbol{\circ}$] Suppose $d(1; 1,2) = c_2-1$, as in \Cref{fig:Positive4Braid_CminC1_new} (C). Since $\beta$ is generic, $c_1 \geq 3$, $d_L(2; c_2, 1) \geq 1$ and $\mathcal{S}_2 \sim \mathbbm{b}_{1,3}$. But this band is not enclosed, on either side, by arcs, and so $\mathcal{S}_2 \sim \mathbbm{b}_{1,3} \sim \mathcal{S}_1$. But we already argued that $\mathcal{S}_1$ is not a sink, so neither is $\mathcal{S}_2$.
\end{itemize}

This concludes the Seifert disk analysis in \textbf{Case (A)}, so we move onto \textbf{Case (B)}. 

\tcbox[size=fbox, colback=gray!30]{\textbf{Case (B):} $\alpha_{1,1}$ is a left pointer.} 

We refer to \Cref{fig:Positive4Braid_CminC1_new} (D), (E), (F)) throughout. In this case, we have constructed $B$ so that $\partial \mathcal{S}_2$ contains the right pointing $\varphi(\alpha_{1,1})$. This arc points out of $\mathcal{S}_2$, and so $\mathcal{S}_2$ is not a sink. We now focus on proving that $\mathcal{S}_1$ is not a sink: in this case, $\mathcal{S}_1 \sim \mathbbm{b}_{1,2} \sim \mathbbm{b}_{2, d(1; 1,2)}$. In particular, $\mathcal{S}_1$ is a horizontal sector spanning $\Gamma_2$. By applying \Cref{lemma:TemplatesHorizontalSectorsNotSinks}, we know that no horizontal sector spanning $\Gamma_2$ is a sink, so $\mathcal{S}_1$ is not a sink.

We see that there are no Seifert disk sink disks in \textbf{Case (B)}. We turn to other sector types.

\item \Polygon as the only Seifert disks with both co--oriented plumbing and image arcs are $S_2$ and $S_3$, a polygon sector $\mathcal{P}$ must be contained in one of these two Seifert disks. By \Cref{lemma:PolygonSectorsNotSinks}, no polygon sector in $S_3$ is a sink. So, we need only check that if $\mathcal{P} \subset S_2$, it is not a sink. Consulting \Cref{fig:Positive4Braid_CminC1_new}, we see that a polygon sector can only arise when $d(1; 1,2) \leq c_2-2$, i.e. when $\alpha_{2,2} \ldots, \alpha_{2, d(1; 1,2)}$ are all co-oriented arcs. This occurs only in cases (B) and (F); indeed, our design ensures that neither is a sink.
\item \Horizontal By \Cref{lemma:TemplatesHorizontalSectorsNotSinks}, there is at most one potential sink disk spanning $\Gamma_2 \cup \Gamma_3$, and it is the horizontal sector $\mathcal{H} \sim \mathbbm{b}_{3,1}$. However, this band is also contained in the sector $\mathcal{S}_4$, which we already proved is not a sink. We deduce that no horizontal sector spanning $\Gamma_2 \cup \Gamma_3$ is a sink. 

Suppose $\mathcal{H}$ is a horizontal sector spanning $\Gamma_1$. By design, $\mathcal{S}_1$ contains all bands $\mathbbm{b}_{1, 2}, \ldots, \mathbbm{b}_{1,c_1}$. Since we already proved that $\mathcal{S}_1$ is never a sink, then if $\mathcal{H} \sim \mathbbm{b}_{1,k}$, where $2 \leq k \leq c_1$, $\mathcal{H}$ is not a sink. It remains to prove that a horizontal sector containing $\mathbbm{b}_{1,1}$ is not a sink: by design, the right endpoint of $\mathbbm{b}_{1,1}$ is enclosed by $\alpha_{2,1}$, which is left pointer exiting $\mathcal{H}$. We deduce that $\mathcal{H} \sim \mathbbm{b}_{1,1}$ is not a sink. 
\end{itemize}

Therefore, if $\beta$ is a generic positive 4--braid with $\Cmin = \Cone$, we have produced a sink disk free branched surface with a unique pair of linked arcs. 
\end{proof}
\begin{prop} \label{prop:Generic4BraidCMinIsCTwo}
Suppose $\beta$ is a generic positive 4--braid such that $\Cmin=\Ctwo$. Then there exists a sink disk free branched surface $B$ for $\widehat{\beta}$ such that $\tau_B$ has exactly one pair of linked arcs.
\end{prop}

\begin{proof}
First, we find a suitable presentation for $\beta$: conjugating as necessary, we assume $\beta \approx \sigma_1 \sigma_2 \beta'$. As $\beta$ is standardized, in this presentation of the braid word, $d(1; 1,2) \geq 2$; additionally, we know that $d(2; 1,2) = 0$ or $d(2; 1,2) \geq 2$. To build our branched surface $B$, we:

\begin{itemize}
\item co--orient $\alpha_{1,1}$ to the left, and all subsequent plumbing arcs to the right.  
\item co--orient $\alpha_{3,1}$ to the right, and all subsequent plumbing arcs to the left. 
\item if $d(2; 1,2) = 0$, co--orient $\alpha_{2,1}$ to the right, as in \Cref{fig:Positive4Braid_CminC2} (A); otherwise, $d(2; 1,2) \geq 2$, and we co--orient $\alpha_{2,1}$ to the left, as in \Cref{fig:Positive4Braid_CminC2} (B). In particular, we take exactly one product disk from $\Gamma_2$. 
\end{itemize}

Examples of the two types of resulting branched surfaces are demonstrated in \Cref{fig:Positive4Braid_CminC2}.

\begin{figure}[h!]\center
\labellist \tiny
%\pinlabel {$\mathcal{H}$} at 160 290
\pinlabel{(A)} at 60 5
\pinlabel{(B)} at 200 5
\pinlabel {$\mathcal{P}$} at 218 125
\pinlabel {$\mathcal{H}$} at 42 102
\pinlabel {$\mathcal{H}$} at 180 102
\endlabellist
\includegraphics[scale=1.2]{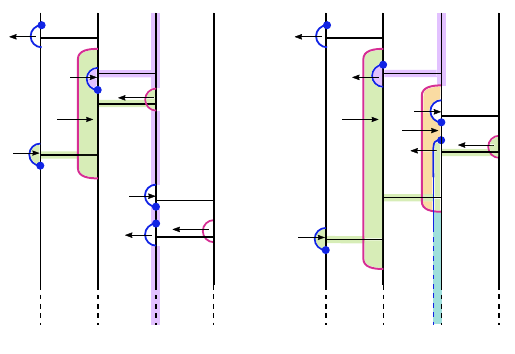}
\caption{Two examples of positive 4--braids with $\Cmin = \Ctwo$; $\alpha_{2,1}$ is the only co-oriented plumbing arc in $S_2$. In \textbf{(A)}, $d(2; 1,2) = 0$, and $\alpha_{2,1}$ is a right pointer. In \textbf{(B)}, $d(2; 1,2) \geq 2$, and $\alpha_{2,1}$ is a left pointer. Neither the polygon sector $\mathcal{P}$, nor the horizontal sectors $\mathcal{H}$, are sinks.}
\label{fig:Positive4Braid_CminC2}
\end{figure}
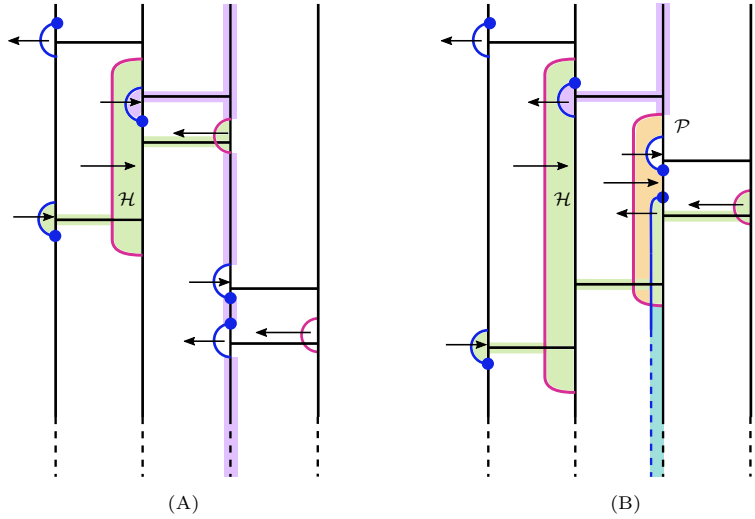

First, we check that $B$ has a unique pair of linked arcs: the proof of \Cref{lemma:TemplateLinking} ensures that there is no linking between any arcs $\alpha_{1,s}$ and $\alpha_{1,s+1}$ for any $s$, and that there is a single linked pair between $\alpha_{3,1}$ and $\alpha_{3,2}$. Consulting with \Cref{fig:Positive4Braid_CminC2}, we see that $\alpha_{2,1}$ does not link with any arc in either adjacent column. 

Next, we show $B$ is sink disk free:
\begin{itemize}
\item \Seifert since $S_1$ (resp. $S_4$) contains only co--oriented plumbing (resp. image) arcs with a mix of co--orientations, $\mathcal{S}_1$ and $\mathcal{S}_4$ are not sinks. Regardless of the value of $d(1; 1,2)$, $\varphi(\alpha_{1,1})$ is a right pointer in $\partial \mathcal{S}_2$; this arc points out of the sector, so $\mathcal{S}_2$ is not a sink. It remains to show $\mathcal{S}_3$ is not a sink disk.

When $\alpha_{2,1}$ is a right pointer, $\partial \mathcal{S}_3$ contains the left pointer $\alpha_{3,1}$, which points out of the sector; see \Cref{fig:Positive4Braid_CminC2} (A). When $\alpha_{2,1}$ is a left pointer, as in \Cref{fig:Positive4Braid_CminC2} (B), $\varphi(\alpha_{2,1})$ is a right pointer in $\partial \mathcal{S}_3$, so $\mathcal{S}_3$ is not a sink. 
\item \Polygon the only Seifert disks with both co--oriented plumbing and image arcs are $S_2$ and $S_3$; thus, these are the only disks with potential to contain polygon sectors. %Moreover, since there is exactly one co--oriented plumbing arc in $S_2$, there can be at most one polygon sector in either $S_2$ or $S_3$. 

If $\mathcal{P}$ is a polygon sector in $S_2$, then $\partial \mathcal{P}$ contains $\varphi(\alpha_{1,1})$ and $\alpha_{2,1}$. Since $\varphi(\alpha_{1,1})$ encloses $\mathbbm{b}_{2,2}, \ldots, \mathbbm{b}_{1, d(1; 1,2)}$, then $\mathcal{P} \sim \mathbbm{b}_{2,2} \sim \ldots \sim  \mathbbm{b}_{1, d(1; 1,2)}$. If $d(2; 1,2) = 0$, then $\mathbbm{b}_{2,2}$ is enclosed, on the right, by the left pointer $\varphi(\alpha_{2,1})$ which exits the sector, and it is not a sink; see \Cref{fig:Positive4Braid_CminC2} (A). Otherwise, $d(2; 1,2) \geq 2$, and $\mathbbm{b}_{2,2}$ is enclosed, on the right, by a right pointing plumbing arc, as in \Cref{fig:Positive4Braid_CminC2} (B); we see that it is not a sink. So, no polygon sector in $S_2$ is a sink. 

Suppose, instead, that $\mathcal{P}$ is a polygon sector in $S_3$. By design, this means $d(2; 1,2) \geq 2$, and $\partial \mathcal{P}$ contains $\varphi(\alpha_{2,1})$. Either (a) $\partial \mathcal{P}$ also contains $\alpha_{3,1}$, or (b) it doesn't. If the former occurs, then $\alpha_{3,1}$ leaves the sector, and it is not a sink (this is the orange sector in \Cref{fig:Positive4Braid_CminC2} (B)); if the latter occurs, then $\varphi(\alpha_{2,1})$ points out of the sector, and it is not a sink (see the teal sector in $S_3$ in \Cref{fig:Positive4Braid_CminC2} (B)). We deduce that no polygon sector in $B$ can be a sink. 
\item \Horizontal Suppose $\mathcal{H}$ is a horizontal sector spanning $\Gamma_1$. If $\mathcal{H} \sim \mathbbm{b}_{1,1}$, then $\mathcal{H} \sim \mathcal{S}_2$, which is not a sink. If $\mathcal{H} \sim \mathbbm{b}_{1,2}$, then $\mathcal{H}$ is in the same sector as the polygon sector $\mathcal{P} \subset S_2$; we already argued that this is not a sink. Suppose $\mathcal{H} \sim \mathbbm{b}_{1,r}$, for $3 \leq r \leq c_1$. These bands are all enclosed, on the right, by a left pointing image arc. Therefore, such an $\mathcal{H}$ is not a sink. 

Suppose $\mathcal{H}$ spans $\Gamma_2$. If $\mathcal{H} \sim \mathbbm{b}_{2,1}$, then $\mathcal{H} \sim \mathcal{S}_3$, so it is not a sink. While studying polygon sectors, we already argued that the sector containing $\mathbbm{b}_{2,2}$ is not a sink. Therefore, suppose $\mathcal{H} \sim \mathbbm{b}_{2, q}$, for $3 \leq q \leq c_2$. Such a band is either enclosed, on the right, by a left pointing plumbing arc $\alpha_{3,\star}$, or $\mathcal{H} \sim \mathbbm{b}_{2, q} \sim \mathcal{S}_3$. Since $\mathcal{S}_3$ is not a sink, neither is $\mathcal{H}$. 

Finally, suppose $\mathcal{H}$ spans $\Gamma_3$. If $\mathcal{H} \sim \mathbbm{b}_{3,1}$, then $\mathcal{H} \sim \mathcal{S}_4$, which is not a sink. Similarly, if $\mathcal{H} \sim \mathbbm{b}_{3, c_3}$, then $\mathcal{H} \sim \mathcal{S}_3$, which is not a sink. Otherwise, $\mathcal{H} \sim \mathbbm{b}_{3,t}$, for $2 \leq t \leq c_3-1$; such bands are enclosed, on the right, by the left pointing $\alpha_{3,t}$, so $\mathcal{H}$ is not a sink. 
\end{itemize}

Thus, $B$ is sink disk free branched surface, and $\tau_B$ contains a unique pair of linked arcs.
\end{proof}

We apply this work about generic positive 4-braids to produce taut foliations.

\begin{prop} \label{thm:MainThmGeneric4Braid}
Suppose $K \approx \widehat{\beta}$ is a non-trivial, generic prime positive braid knot on $n=4$ strands with $g(K) \geq 2$. Then for all $r < \frac{4}{3}g(K)$, $S^3_r(K)$ admits a taut foliation. 
\end{prop}

\begin{proof}
Let $K$ be any knot satisfying the stated hypotheses. By \Cref{defn:standard}, we can put $\beta$ into standard form; for simplicity, we also call the new braid $\beta$. After computing the value of $\Cmin$ for $\beta$, we build a branched surface with the characteristics described by \Cref{table:Generic4Braids}. By applying one of Propositions \ref{prop:Generic4BraidCMinIsCThree}, \ref{prop:Generic4BraidCMinIsCOne}, or \ref{prop:Generic4BraidCMinIsCTwo}, we obtain a sink disk free branched surface $B$ such that $\tau_B$ contains exactly one pair of linked arcs, or one linked triple of arcs. 

We now compute the slopes carried by $\tau_B$ in terms of $g(K)$: first, note that $\Cone + \Ctwo + \Cthree = \mathcal{C}$. Note: we cannot have that for all $j \in \{1, 2, 3\}, \ \mathcal{C}_{j} < \frac{\mathcal{C}}{3}$, nor can we have that for all $j$, $\mathcal{C}_{j} > \frac{\mathcal{C}}{3}$. Therefore, there exists some $j \in \{1,2,3\}$ with $\mathcal{C}_{j} \leq \frac{\mathcal{C}}{3}$, and for some $k \in \{1,2,3\}$, $\mathcal{C}_k \geq \frac{\mathcal{C}}{3}$. In particular, this implies that $\Cmin \leq \frac{\mathcal{C}}{3}$, and $\mathcal{C} - \Cmin \geq \frac{2}{3} \mathcal{C}$. 

\Cref{table:Generic4Braids} keeps track of the total number of product disks used to build $B$. In every case, we use $\mathcal{C}-\Cmin-1$ product disks. Therefore, there are $\mathcal{C}-\Cmin-1$ arcs contributing maximally to the train track. Since our construction guaranteed either a unique pair of linked arcs, or a unique linked triple, $\tau_B$ carries all slopes up to $\mathcal{C}-\Cmin-2$. That is, $\TauSup = \mathcal{C}-\Cmin-2$.

As $2g(K)-1 = \mathcal{C} - n = \mathcal{C}-4$, then $$g(K) = \frac{\mathcal{C}-3}{2}  \implies \frac{4}{3}g(K) = \frac{4}{3}\left( \frac{\mathcal{C}-3}{2}\right) = \frac{2\mathcal{C}}{3}- 2$$

We now compare $\TauSup$ to the $g(K)$:
\begin{align*}
\TauSup = \mathcal{C}-\Cmin-2 \geq  \frac{2\mathcal{C}}{3} - 2 = \frac{4}{3}g(K)
\end{align*}

Thus, $\tau_B$ carries all slopes $r < \frac{4}{3}g(K)$. Applying \Cref{prop:BisLaminar}, we see that $B$ is laminar; applying \Cref{thm:foliations}, we conclude that $S^3_r(K)$ admits a taut foliation for all $r < \frac{4}{3}g(K)$.
\end{proof}

%%%%%
\subsection{Sparse 4--braids} \label{section:Edge4Braids} We now prove an analogue of \Cref{thm:main} for sparse positive 4--braids. 

\begin{prop} \label{thm:MainThmSparse4Braid}
Suppose $K \approx \beta$, where $\beta$ is a prime, sparse, positive 4--braid. Then for all $r < 2g(K) - 2$, $S^3_r(K)$ has a taut foliation. 
\end{prop}

\begin{proof}
Assume $\widehat \beta$ is a knot satisfying the stated hypotheses. If $\beta$ is sparse, then there exists at least one column $\Gamma_i$ with $c_i =2$. Building a sink disk free branched surface carrying all slopes $r < 2g(K)-2$ requires some case analysis which is determined by how many columns have exactly two bands:

\begin{enumerate}
\item all three columns have exactly two bands, i.e. $c_1 = c_2 = c_3 = 2$.
\item exactly two columns have two bands, i.e. one of the following occurs:
	\begin{enumerate}
	\item  $c_1 = c_2 = 2$ and $c_3 \geq 3$, or
	\item $c_2 = c_3 = 2$ and $c_1 \geq 3$, or
	\item $c_1 = c_3 = 2$ and $c_2 \geq 3$.
	\end{enumerate}
\item exactly one of $c_1, c_2$, or $c_3$ has two bands, i.e one of the following holds:
	\begin{enumerate}
	\item  $c_1=2$ and $c_2, c_3 \geq 3$, or
	\item $c_2 =2$ and $c_1, c_3 \geq 3$, or
	\item $c_3 = 2$ and $c_1, c_2 \geq 3$.
	\end{enumerate}
\end{enumerate}

Our braid is realized on $4$ strands. By cyclically conjugating if necessary, we may assume that $\beta \approx \sigma_1 \sigma_2 \beta'$. Since $\beta$ is prime, $d(1; 1,2) \leq c_2-1$. Moreover, since $\beta$ is standardized, we have that $d(1; 1,2) \geq 2$. Combining these observations, we see that $2 \leq c_2-1$, hence $3 \leq c_2$. We deduce that cases (1), (2a), (2b), and (3b) do not occur. Therefore, we need only construct branched surfaces for cases (2c), (3a), and (3c). 

\tcbox[size=fbox, colback=gray!30]{\textbf{Case (2c):} $c_1 = c_3 = 2$, and $c_2 \geq 3$.} 

We begin with Case (2c), where $c_1 = c_3 = 2$. Notice that if $\beta$ is standardized, we have $d(1; 1,2) \geq 2$ and $d(1; 2,1) \geq 2$, thus $c_2 \geq 4$. To construct the branched surface in this case, we first cyclically conjugate the braid so that $\beta \approx \sigma_1 \sigma_2 \beta'$. Then, we apply our standard template to $\Gamma_1 \cup \Gamma_2$: there will be a single co--oriented plumbing arc in $S_1$, $\alpha_{1,1}$, and it will be co--oriented to the left. In this instance, we will not include any product disks from $\Gamma_3$. We claim that we have constructed a sink disk free branched surface. We refer to \Cref{fig:Positive4Braid_NotGeneric} throughout.

\begin{figure}[H]\center
\labellist
\pinlabel {$\Gamma_1$} at 35 165
\pinlabel {$\Gamma_2$} at 62 165
\pinlabel {$\Gamma_3$} at 90 165
\endlabellist
\includegraphics[scale=1.2]{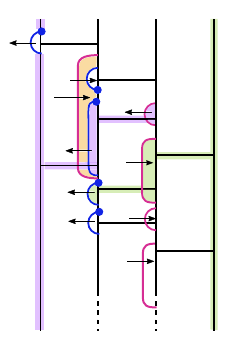}
\caption{A sample sparse positive 4-braid with $c_1 = c_3 = 2$. The branch sector $\mathcal{S}_4$ is not a sink.}
\label{fig:Positive4Braid_NotGeneric}
\end{figure}

\begin{itemize}
\item \Seifert 
First, note that $\mathcal{S}_1 \sim \mathbbm{b}_{1,2}$. Since $2 \leq d(1; 1,2) \leq c_2-1$, $\mathbbm{b}_{1,2}$ must be enclosed, on the right, by a left pointing plumbing arc; see \Cref{fig:Positive4Braid_NotGeneric}. We deduce that $\mathcal{S}_1$ is not a sink. 

Since $\mathbbm{b}_{1,1}$ is not enclosed on the right, $\mathcal{S}_2 \sim \mathbbm{b}_{1,1}$. As $\alpha_{1,1}$ is a left pointer, $\mathcal{S}_2$ is not a sink. 

The disk sector $\mathcal{S}_3$ contains only image arcs. Since $c_2 \geq 4$, then $S_3$ will contain both left pointing and right pointing image arcs. Thus, $\mathcal{S}_3$ is not a sink.

Finally, we analyze $\mathcal{S}_4$. Since we used no product disks from $\Gamma_3$, we know that $\mathcal{S}_4 \sim \mathbbm{b}_{3,1}$. By assumption, $\widehat{\beta}$ is not a connected sum, and $\beta$ is prime, then $d(2; 1, d(1; 1,2)) = 0$ and $d(2; d(1; 1,2)+1, c_2) = 0$. Therefore, we know that $\beta$ can be written as $\sigma_1 \sigma_2^{k_1} \sigma_1 \sigma_3 \sigma_2^{k_2} \sigma_3$, where $k_1, k_2 \geq 2$. In particular, $d_L(3; 1,2) \geq 1$, and $\mathbbm{b}_{3,1}$ is enclosed, on the left, by some arc $\alpha_{2,t}$, for $t \geq 2$. Therefore, $\mathcal{S}_4 \sim \mathbbm{b}_{3,1} \sim \mathbbm{b}_{2, t+1}$, as in \Cref{fig:Positive4Braid_NotGeneric}. We see that $\mathbbm{b}_{2,t+1}$ is enclosed on the left by the left pointing plumbing arc $\alpha_{2,t+1}$, which must be a left pointer. So, $\mathcal{S}_4$ is not a sink.

\item \Polygon The unique Seifert disk with both plumbing and image arcs is $S_2$. In fact, there is a single polygon sector in $S_2$ (see \Cref{fig:Positive4Braid_NotGeneric}), and it is not a sink because $\alpha_{2,1}$ points out of the sector.

\item \Horizontal Suppose $\mathcal{H}$ spans $\Gamma_1$. We assumed that $c_1=2$; as observed in our Seifert disk sector analysis, $\mathbbm{b}_{1,1} \sim \mathcal{S}_2$, and $\mathbbm{b}_{1,2} \sim \mathcal{S}_1$. But $\mathcal{S}_1$ and $\mathcal{S}_2$ are not sinks, so no horizontal sector spanning $\Gamma_1$ is a sink.

Suppose $\mathcal{H}$ spans $\Gamma_3$. Since there are no plumbing arcs in $S_3$, there are no image arcs in $S_4$. Therefore, $\mathcal{S}_4 \sim \mathbbm{b}_{3,1} \sim \mathbbm{b}_{3,2}$. We deduce that any horizontal sector spanning $\Gamma_3$ is in the same branch sector as $\mathcal{S}_4$, which we already proved is not a sink.

Finally, if $\mathcal{H}$ spans $\Gamma_2$, then $\mathcal{H}$ contains $\mathbbm{b}_{2,t}$, where $1 \leq t \leq c_2$. If $\mathcal{H}$ contains $\mathbbm{b}_{2,1}$ (resp. $\mathbbm{b}_{2,c_2}$), then $\mathcal{H} \sim \mathcal{S}_3$ (resp. $\mathcal{S}_2$), which we already proved is not a sink. Otherwise, $\mathcal{H}$ contains $\mathbbm{b}_{2,t}$, where $2 \leq t \geq c_2-1$. But $\mathbbm{b}_{2,t}$ is enclosed, on the left, by the left pointing $\alpha_{2,t}$, so $\mathcal{H}$ is not a sink. 

We deduce that no horizontal sectors in $B$ are sinks.
\end{itemize}

Thus, when we are in case (2c), we can build a sink disk free branched surface. The proof of \Cref{lemma:TemplateLinking} tells us that the unique linked pair of arcs in $\tau_B$ arises from $\alpha_{2,1}$ and $\alpha_{2,2}$. Therefore, $\TauSup = 1 + (\Ctwo - 1)-1 = \Ctwo - 1$. Since
$$\mathcal{C} = \Cone + \Ctwo + \Cthree = 2 + \Ctwo + 2 = \Ctwo + 4,$$ 
we have $2g(K) - 1 = \mathcal{C} - n = (\Ctwo + 4) - 4 = \Ctwo$. We deduce that $\TauSup = 2g(K) - 2$. 
%Whenever $g(K) \geq 2, \frac{4}{3}g(K) - 1 < 2g(K) - 2 = \TauSup$, so we need only consider the $g(K)=1$ case. The only positive braid knot of genus one is the right--handed trefoil, and Roberts \cite{Roberts:Part1} produced taut foliations in $S^3_r(K)$ for all $r < 1 = 2g(K)-1$ for this knot; so, we defer to her result in this case. 
Therefore, we have constructed a sink disk free branched surface $B$ whose train track $\tau_B$ carries all slopes $r < 2g(K)-2$.

To finish the sparse 4--braid case, we construct branched surfaces in Cases (3a) and (3c). %In both cases, we will choose no product disks from the column $\Gamma_i$ with $c_i =2$, and choose $c_j-1$ product disks from $\Gamma_j$ for $c_j \geq 3$. 

\tcbox[size=fbox, colback=gray!30]{\textbf{Case (3a):} $c_1= 2$, and $c_2, c_3 \geq 3$.} 

In this case, we choose no product disks from $\Gamma_1$, $c_2-1$ product disks from $\Gamma_2$, and $c_3-1$ product disks from $\Gamma_3$. To begin, we follow the setup of \Cref{prop:Generic4BraidCMinIsCOne} to find a suitable cyclic conjugation of our braid, and to then apply our template to $\Gamma_2 \cup \Gamma_3$. We now check that we have constructed a sink disk free branched surface. (We recommend the reader refer to \Cref{fig:Positive4Braid_CminC1_new} to see the presentation of $\beta$; erasing $\alpha_{1,1}$ and $\varphi(\alpha_{1,1})$ from the figures reveals how $B$ looks in $\Gamma_1 \cup \Gamma_2$.)

\begin{itemize}
\item \Seifert since $B$ contains no product disks from $\Gamma_1$, we have that $\mathcal{S}_1 \sim \mathbbm{b}_{1,1} \sim \mathbbm{b}_{1,2}$. By design, we have that $\mathbbm{b}_{1,1}$ is enclosed, on the right, by $\alpha_{2,1}$, which is a left pointer. Therefore, $\mathcal{S}_1$ is not a sink. 

If $S_2$ contains a right pointing plumbing arc, then $\partial \mathcal{S}_2$ contains an exiting arc, and so $\mathcal{S}_2$ is not a sink. On the other hand, suppose $S_2$ contains only left pointers. 
%Letting $d(1; 1,2) -1 := t$, we have that $d(2; 1,t) = 0$, and $d(2; t, t+1) \geq 1$; see \textcolor{red}{Figure XY}. 
Since we used no product disks from $\Gamma_1$ to build $B$, this means $\mathcal{S}_2 \sim \mathbbm{b}_{2,c_2}$; moreover, the design of our template tells us that $d(2; c_2-1, c_2) \geq 1$. We deduce that $\mathbbm{b}_{2,c_2}$ is enclosed, on the right, by some arc $\alpha_{3, s}$. In particular, $\mathcal{S}_2 \sim \mathbbm{b}_{2,c_2} \sim \mathbbm{b}_{3,s}$.

\begin{itemize}
\item[$\boldsymbol{\circ}$] If $s=1$, then $\mathbbm{b}_{2,c_2}$ is enclosed by $\alpha_{3,1} = \alpha_{3,s}$, then $\mathcal{S}_2 \sim \mathbbm{b}_{2,c_2} \sim \mathbbm{b}_{3,1} \sim \mathcal{S}_4$. But $\mathcal{S}_4$ contains both left and right pointing image arcs, so $\mathcal{S}_4$ is not a sink, and neither is $\mathcal{S}_2$. 
\item[$\boldsymbol{\circ}$] Otherwise, $s \geq 2$. By design, $\alpha_{3,s}$ is a left pointer, so $\mathcal{S}_2$ is not a sink.
\end{itemize}

We deduce that $\mathcal{S}_2$ is not a sink. By applying \Cref{lemma:SeifertDisksNotSinks}, $\mathcal{S}_3$ is not a sink. 

Finally, by design, $\mathcal{S}_4$ contains both left and right pointing image arcs, and so $\mathcal{S}_4$ is not a sink. Therefore, no Seifert disk sectors are sink disks.    

\item \Polygon The only Seifert disk with both plumbing and image arcs is $S_3$. Directly applying \Cref{lemma:PolygonSectorsNotSinks}, we see that no polygon sector in $S_3$ is a sink. 

\item \Horizontal Suppose $\mathcal{H}$ is a horizontal sector spanning $\Gamma_1$. Since $B$ contains no product disks from $\Gamma_1$, we have that $\mathcal{S}_1 \sim \mathbbm{b}_{1,1} \sim \mathbbm{b}_{1,2}$. In particular, if $\mathcal{H}$ spans $\Gamma_1$, then $\mathcal{H} \sim \mathcal{S}_1$, and we already proved the latter is not a sink. 

Suppose $\mathcal{H}$ is a horizontal sector spanning either $\Gamma_2$ or $\Gamma_3$. By \Cref{lemma:TemplatesHorizontalSectorsNotSinks}, the only potential horizontal sector spanning $\Gamma_2 \cup \Gamma_3$ is the sector containing $\mathbbm{b}_{3,1}$. By design, $\mathbbm{b}_{3,1} \sim \mathcal{S}_4$; we already proved the latter is not a sink, so we deduce that no horizontal sector of $B$ is a sink.
\end{itemize}

We deduce that we have produced a sink disk free branched surface for Case (3a). We will analyze the slopes carried by $\tau_B$ after executing the analogous construction for Case (3c).

\tcbox[size=fbox, colback=gray!30]{\textbf{Case (3c:)} $c_3= 2$, and $c_1, c_2 \geq 3$.} 

Similar to Case (3a), we choose $c_1-1$ product disks from $\Gamma_1$, $c_2-1$ product disks from $\Gamma_2$, and no product disks from $\Gamma_3$. As in \Cref{prop:Generic4BraidCMinIsCThree}, we cyclically conjugate $\beta$ (and potentially apply the braid relation $\sigma_1 \sigma_3 = \sigma_3\sigma_1$) to present it as $\sigma_1 \omega_1 \sigma_1 \omega_2$, where $\omega_1$ is non-empty, contains no $\sigma_1$ letters, and contains some $\sigma_2$ letters; since $\beta$ is standardized, we know that $d(1; 1,2) \geq 2$. We now build $B$ by directly applying our template to $\Gamma_1 \cup \Gamma_2$ (we suggest the reader refer to \Cref{fig:Positive4Braid_CminC3_new_2} to see the branched surface; it is obtained by erasing the $\alpha_{3,1}$ and its image from the figure). We claim the result has no sink disks:

\begin{itemize}
\item \Seifert By the design of our standard template, $\alpha_{1,2}$ is a right pointer, and $\alpha_{1,2} \subset  \partial \mathcal{S}_1$. Therefore, this arc points out of $\mathcal{S}_1$, and we know the sector is not a sink. Applying \Cref{lemma:SeifertDisksNotSinks}, we see that $\mathcal{S}_2$ is not a sink. Since we have chosen no product disks from $\Gamma_3$, the only arcs in $S_3$ are image arcs; in particular, $S_3$ contains both left and right pointing image arcs, and so $\mathcal{S}_3$ is not a sink. 

Finally, it remains to show that $\mathcal{S}_4$ is not a sink. First, we observe that $\mathcal{S}_4 \sim \mathbbm{b}_{3,1} \sim \mathbbm{b}_{3,2}$. Next, we notice that $d(2; 1,2)=0$: since $c_3 =2$, we must either have that $d(2; 1,2) = 0$, $d(2; 1,2) = 1$, or $d(2; 1,2) = 2$. However, if $d(2; 1,2) = 1$, then $\beta$ is not standardized, and if $d(2; 1,2) = 2$, then $\beta$ is a connected sum. We deduce that $d(2; 1,2) = 0$. This tells us some important information about $\mathbbm{b}_{3,1}$: if $t$ is the smallest value such that $d(2; t, t+1) \geq 1$, then $t \geq 2$. Since $\beta$ is not a connected sum, $t \leq c_2-1$. By the design of our standard template, the band $\mathbbm{b}_{3,1}$ must be enclosed, on the left, by the right pointing image arc $\varphi(\alpha_{2,t})$. Therefore, $\mathcal{S}_4 \sim \mathbbm{b}_{3,1} \sim \mathbbm{b}_{2, t+1}$. We deduce that $\mathcal{S}_4$ is part of a horizontal sector which spans $\Gamma_2$. Either $\mathcal{S}_4 \cap S_2$ contains $\alpha_{2, t+1}$ (which is a left pointer), or $\mathcal{S}_4 \sim \mathcal{S}_2$. Regardless of which occurs, we deduce that $\mathcal{S}_4$ is not a sink. 

\item \Polygon The unique Seifert disk with potential to contain polygon sectors is $S_2$. Applying \Cref{lemma:PolygonSectorsNotSinks}, we see that no polygon sector in $S_2$ is a sink.

\item \Horizontal By \Cref{lemma:TemplatesHorizontalSectorsNotSinks}, the only potential sink disk horizontal sector $\mathcal{H}$ spanning $\Gamma_1 \cup \Gamma_2$ is the sector containing $\mathbbm{b}_{2,1}$. However, we see that $\mathbbm{b}_{3,1} \sim \mathcal{S}_3$, which we already argued is not a sink. Finally, suppose $\mathcal{H}$ is a horizontal sector spanning $\Gamma_3$. Since $B$ contains no product disks from $\Gamma_3$, $\mathcal{H} \sim \mathcal{S}_4$, which we already proved is not a sink. 

\end{itemize}

We deduce that we can build a sink disk free branched surface for all braids satisfying the hypotheses of Case (3c). We now determine the supremal slopes carried by the train tracks for the branched surfaces build in both Case (3a) and (3c). This is straightforward: in both cases, $\Cmin = 2$, and we used $\mathcal{C}-\Cmin-2$ product disks, while having exactly one linked pair. Thus, $\TauSup = \mathcal{C} - \Cmin - 3 = \mathcal{C}- 5$. But $2g(K)-1 = \mathcal{C}-4$, so $\TauSup = 2g(K)-2$. %As we proved in the case when exactly two columns have 2 bands, the inequality $\frac{4}{3}g(K)-1 < 2g(K) - 2$ always holds.

Thus, if $\beta$ is sparse on 4-strands, we can produce a sink disk free branched surface $B$ such that $\tau_B$ carries all slopes $r <2g(K)-2$. Applying \Cref{prop:BisLaminar} and \Cref{thm:foliations}, we deduce that $S^3_r(K)$ admits a taut foliation for all $r < 2g(K)-2$.
\end{proof}

\subsection{\Cref{thm:main} for 4-braids.}

We can finally prove \Cref{thm:main} for positive 4-braids.

\begin{thm} \label{thm:Main4Braids}
If $K$ is a positive 4-braid with $g(K) \geq 2$, then $S^3_r(K)$ admits a taut foliation whenever $r < g(K)+1$.
\end{thm}

\begin{proof}
Suppose $K$ is a positive 4-braid, $g(K) \geq 2$, and $K \approx \widehat{\beta}$, where $\beta$ is a standardized 4-braid. If $\beta$ is generic, then by \Cref{thm:MainThmGeneric4Braid}, we know that for all $r < \frac{4}{3}g(K)$, $S^3_r(K)$ admits a taut foliation. In particular, for any such $K$ with $g(K) \geq 3$, we can immediately deduce that $S^3_r(K)$ admits a taut foliations for all $r < g(K)+1$. Our construction ensures that $\TauSup$ is an integer, so when $g(K)=2$, we see that $\displaystyle \TauSup = \bigg \lceil \frac{4}{3} \cdot 2 \bigg \rceil$, and we produce taut foliations in all surgeries $r < 3 = g(K)+1$. We deduce that the theorem holds for all generic positive braids on 4-strands. 

Otherwise, $\beta$ is a sparse positive 4-braid, and \Cref{thm:MainThmSparse4Braid} constructs taut foliations in $S^3_r(K)$ for all $r<2g(K)-2$. Therefore, for all $K$ with $g(K) \geq 3$, we can immediately deduce that $S^3_r(K)$ has a taut foliations whenever $r < g(K)+1$, and the same analysis as in the previous paragraph holds in the $g(K)=2$ case. 
\end{proof}

In addition to proving \Cref{thm:main} for positive 4--braids, this section demonstrates that branched surfaces that simultaneously minimize linking and avoid sink disks is a delicate process. Moreover, we hope it illustrates the combinatorial challenges presented by positive braids.

%%%%%%%%%%%%%%%%%%%%%%%
\section{Some Interesting Examples} \label{section:WeirdExamples}

In this section, we prove \Cref{thm:examples}:

\textbf{\Cref{thm:examples}.}
\textit{There are infinitely many hyperbolic 4-braid L-space knots $\{K_m\}$ such that $S^3_r(K_m)$ admits a taut foliation whenever $r < 2g(K_m)-2$.}

\begin{proof}
A \textit{twisted torus knot} $K(p, q; s, m)$ is obtained by taking the closure of a positive braid on $p$ strands, where the braid is of the form $\beta = (\sigma_{p-1} \sigma_{p-2} \ldots \sigma_2 \sigma_1)^q \cdot (\sigma_{p-1} \sigma_{p-2} \ldots \sigma_{p-s+1})^{ms}$ (that is, into the standard braid word for a positive torus knot, one adds in $m$ positive full twists into a subset of $s$ adjacent strands). Vafaee proved that twisted torus knots are sometimes L-space knots:

\begin{thm} \cite{Vafaee:TwistedTorusKnots} \label{Vafaee}
For $p \geq 2$, $k \geq 1$, $m \geq 1$, and $0 < s < p$, the twisted torus knot $K(p, kp \pm 1; s,m)$ is an L-space knot if and only if either $s=p-1$, or $s \in \{2, p-2\}$ and $m =1$. \hfill $\Box$
\end{thm} 

To prove \Cref{thm:examples}, we consider the twisted torus knots with $p=4, k=2,$ and $s=3$; using Vafaee's notation, these are the knots $K_m :=K(4, 7; 3, m)$. We have $s=p-1$, so applying \Cref{Vafaee}, we see that $\{K_m \ | \ m \geq 1\}$ are L-space knots.

Next, we construct taut foliations in $S^3_r(K_m)$ for all $r < 2g(K_m)-2$. Let $\beta_m := (\sigma_1 \sigma_2 \sigma_3)^7 (\sigma_3 \sigma_2)^{3m}$. To begin, we show that we can find a standardized synonym for $\beta_m$. Finding such a synonym for $\beta_m$ requires concrete applications of the braid relations, which we execute below. Throughout, we write $i$ for $\sigma_i$, and we underline the letters of the braid word to which we apply a braid relation. 
\begin{align}
\beta_m &= (1 2 3)^7 (3 2)^{3m} \nonumber \\
&= (1 \ 2 \ \underline{3 \ 1} \ 2 \ 3 \ 1 \ 2 \ \underline{3 \ 1} \ 2 \ 3 \ 1 \ 2 \ \underline{3 \ 1} \ 2 \ 3 \ 1 \ 2 \ 3) (3 \ 2 \ 3 \ 2 \ 3 \ 2)^m \nonumber \\ 
&= (\underline{1 \ 2 \ 1} \ 3 \ 2 \ 3 \ \underline{1 \ 2 \ 1} \ 3 \ 2 \ 3 \ \underline{1 \ 2 \ 1} \ 3 \ 2 \ 3 \ 1 \ 2 \ 3) (3 \ 2 \ 3 \ 2 \ 3 \ 2)^m \nonumber \\
&= (2 \ 1 \ \underline{2 \ 3 \ 2} \ 3 \ 2 \ 1 \ \underline{2 \ 3 \ 2} \ 3 \ 2 \ 1 \ \underline{2 \ 3 \ 2} \ 3 \ 1 \ 2 \ 3) (3 \ 2 \ 3 \ 2 \ 3 \ 2)^m \nonumber \\
&= (2 \ 1 \ 3 \ 2 \ 3 \ 3 \ 2 \ 1 \ 3 \ 2 \ 3 \ 3 \ 2 \ \underline{1 \ 3} \ 2 \ \underline{3 \ 3 \ 1} \ 2 \ 3) (3 \ 2 \ 3 \ 2 \ 3 \ 2)^m \nonumber \\
&= (2 \ 1 \ 3 \ 2 \ 3 \ 3 \ 2 \ 1 \ 3 \ 2 \ 3 \ 3 \ 2 \ 3 \ \underline{1 \ 2 \ 1} \ 3 \ 3 \ 2 \ 3) (3 \ 2 \ 3 \ 2 \ 3 \ 2)^m \nonumber \\
&= (2 \ 1 \ 3 \ 2 \ 3 \ 3 \ 2 \ 1 \ 3 \ 2 \ 3 \ 3 \ \underline{2 \ 3 \ 2} \ 1 \ 2 \ 3 \ 3 \ 2 \ 3) (3 \ \underline{2 \ 3 \ 2} \ 3 \ 2)^m \nonumber \\
&= (2 \ 1 \ 3 \ 2 \ 3 \ 3 \ 2 \ 1 \ 3 \ 2 \ 3 \ 3 \ 3 \ 2 \ 3 \ 1 \ 2 \ 3 \ 3 \ 2 \ 3) (3 \ 3 \ 2 \ 3 \ 3 \ 2)^m \label{example:FinalBraid}
\end{align}

We take inventory of the crossings in each column for $\beta_m: c_1 = 3, c_2 = 7+2m$, and $c_3 = 11 + 4m$. In particular, for all $m \geq 1$, $\Cmin = \Cone$. We apply the construction of \Cref{prop:Generic4BraidCMinIsCOne} to produce a sink disk free branched surface for $K_m$. \Cref{fig:CuriousExamples} showcases a brick diagram for the standardized form in \Cref{example:FinalBraid}, and most of the associated branched surface produced by applying the construction of \Cref{prop:Generic4BraidCMinIsCOne}. (In the $m=1$ case, the full branched surface is visible in \Cref{fig:Example_Main}.)

By our usual elementary calculations, we have
\begin{align*}
2g(K_m) -1 &= \mathcal{C}-n = (3 + 7 + 2m + 11 + 4m)-4 = (21+6m)-4 = 17+6m, \text{ and } \\
\TauSup &= 1 + (7+2m - 1) + (11+4m - 1) - 1 = 16+6m = 2g(K_m)-2.
\end{align*}

Applying \Cref{prop:BisLaminar}, we see that $B$ is laminar; applying \Cref{thm:foliations}, we conclude that we can produce taut foliations in $S^3_r(K_m)$ for all $r < 2g(K)-2$. 

In \Cref{appendix}, we use SnapPy \cite{SnapPy} and Thurston's hyperbolic Dehn surgery \cite{Thurston:Vol4} \cite[Theorem A]{NeumannZagier} theorem to prove that infinitely many of the $\{K_m \ | \ m \geq 1\}$ are hyperbolic. 

These knots have braid index 4: for all $m \geq 1$, $\beta_m$ is a positive 4-braid with a full twist on 4 strands. By Morton---Franks-Williams theorem tells us that braid index of $K_m$ is 4 \cite{Morton:KnotPolys, FranksWilliams}.
\end{proof}

\begin{figure}[H]\center
\labellist
\pinlabel {$(\sigma_3 \sigma_3 \sigma_2)^{2m}$} at 84 35
\pinlabel {$(\sigma_3 \sigma_3 \sigma_2)^{2m}$} at 250 35
\endlabellist
\includegraphics[scale=1.1]{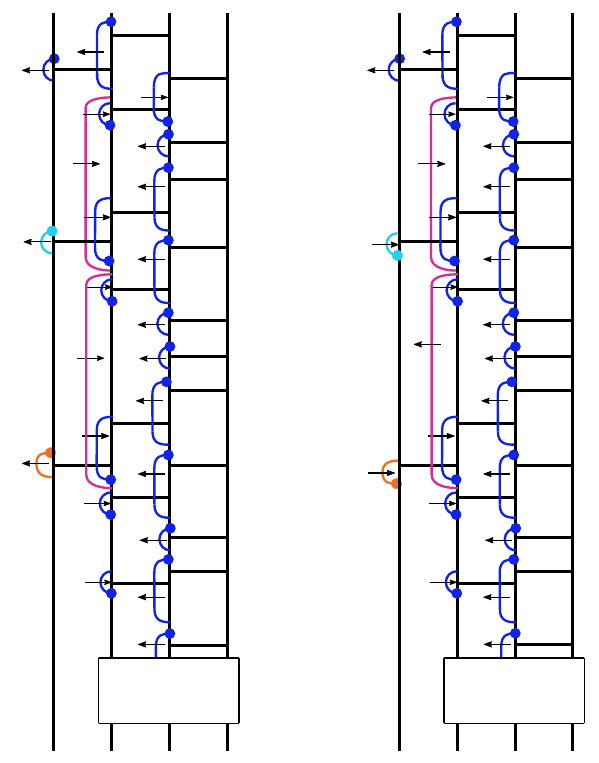}
\caption{Building the branched surface for $K_m$, where $m \geq 1$. We use $\alpha_{1,1}$ to build $B$, but no other arcs in $S_1$. For simplicity, we have suppressed all image arcs except $\varphi(\alpha_{1,1})$ and $\varphi(\alpha_{1,2})$, and hidden the part of the branched surface coming from the full twists on a subset of strands. We see that if we regardless of the co-orientation imposed on $\alpha_{1,2}$ (identified in light blue) or on $\alpha_{1,3}$ (identified in orange), the arc forces a linked pair to occur.}%=: if $\alpha_{1,2}$ is a left pointer (as in the left frame), then $\alpha_{1,2}$ links with $\alpha_{2,3}$; however, if $\alpha_{1,2}$ is a right pointer (as in the right frame), then $\alpha_{1,2}$ links with $\alpha_{2,5}$. }
\label{fig:CuriousExamples}
\end{figure}

Inspecting \Cref{fig:CuriousExamples}, one may try to build a ``better'' branched surface by adding a product disk swept out by $\alpha_{1,2}$ or $\alpha_{1,3}$, and hoping that the result carries more slopes than the existing branched surface. In \Cref{table:Examples_Linking}, we see that regardless of how we could co-orient these arcs, we will produce a linked pair with some other arc that we already used to build $B$. Therefore, there is no value added by including either disk into $B$, and it appears that our construction cannot produce taut foliations for slopes $r \in [2g(K_m)-2, 2g(K_m)-1)$. This leads us to ask the question posed in the introduction: \textit{for $r \in [2g(K_m)-2, 2g(K_m)-1)$, if $S^3_r(K_m)$ has a taut foliation, must a transversal for the foliation arise as a dual of the core of the Dehn surgery solid torus?}

% \begin{tabular}{c||c|c}
%  & If $\alpha_{1,2}$ is smoothed to the \ldots & If $\alpha_{1,3}$ is smoothed to the \ldots \\
%  \hline
 
%  \ldots left, then it links with & $\alpha_{2,3}$ &  \\
%  \hline
%  \ldots right, then it links with & & 
% \end{tabular}

\begin{table}[H]
\begin{tabular}{c||c|c}
 & left, then it links with: & right, then it links with: \\
 \hline
 If $\alpha_{1,2}$ is smoothed to the \ldots & $\alpha_{2,3}$ & $\alpha_{2,5}$ \\
 \hline
If $\alpha_{1,3}$ is smoothed to the \ldots & $\alpha_{2,5}$ & $\alpha_{1,1}$
\end{tabular}
\caption{Smoothing additional arcs in $S_1$ always produces a linked pair.}
\label{table:Examples_Linking}
\end{table}

%%%%PROOF OF MAIN THEOREM FOR 5 OR MORE STRANDS
%%%%%%%
\section{Proof of \Cref{thm:main} for positive braids on five or more strands} \label{section:proof}

In this section, we prove \Cref{thm:main} for all positive braid knots on at least five strands. We briefly sketch the strategy here, and then establish some preliminaries. 

To build our desired branched surface, we begin by standardizing our braid $\beta$ on $n \geq 5$ strands, and then determining whether the crossings are concentrated in the ``even'' or ``odd'' columns of the braid. \Cref{section:BuildInFront} provides instructions for how to build $B$ in $\Gamma_1 \cup \Gamma_2 \cup \Gamma_3$. \Cref{section:BuildInMiddle} details how to iteratively conjugate our braid to build $B$ in the intermediate columns of the braid. Eventually, we will reach $\Gamma_{n-3}$; at this point, we follow the instructions of \Cref{section:BuildInLast} to build $B$ in $\Gamma_{n-3} \cup \Gamma_{n-2} \cup \Gamma_{n-1}$. \Cref{section:Finale} details how to stitch all of the pieces of the construction together to construct a branched surface $B$ which is sink disk free (and therefore laminar). This sections also computes $\TauSup$ for a branched surface constructed this way, and shows that for any such $B$, $\tau_B$ carries all slopes $r < g(K)+1$.

Before we dive into the proof, we establish some preliminaries. 

\begin{defn}
A column $\Gamma_i$ is \textbf{homogeneous} $($resp. \textbf{heterogeneous}$)$ if, when we build our branched surface in $\Gamma_i$, the Seifert disk $S_i$ contains only left pointing \underline{plumbing} arcs $($resp. both left and right pointing \underline{plumbing} arcs$)$.
\end{defn}

We emphasize: the homogeneity or heterogeneity of $S_i$ (or $\Gamma_i$) is very much determined by and dependent upon the presentation of the braid when we apply our template to $\Gamma_i$. Much of our construction requires distinct analysis for the heterogeneous and homogeneous cases. 

\begin{lemma}
If $\Gamma_i$ is heterogeneous, then $c_i \geq 3$, but if $\Gamma_i$ is homogeneous, then $c_i \geq 2$. \hfill $\Box$
\end{lemma}

\begin{lemma} \label{lemma:GeneityAndSinks}
Suppose we apply our template to $\Gamma_i$ or to $\Gamma_i \cup \Gamma_{i+1}$. Further assume that $S_i$ contains only plumbing arcs. Then if $\Gamma_i$ is heterogeneous, $\mathcal{S}_i$ is not a sink. 
\end{lemma}

\begin{proof}
This is straightforward: if $S_i$ contains only plumbing arcs, then $\partial \mathcal{S}_i$ contains $\alpha_{i,1}, \ldots, \alpha_{i, c_1-1}$. Since $\Gamma_i$ is heterogeneous, then there must be some right pointing arc in $S_i$, and this arc points out of $\mathcal{S}_i$, so this sector is not a sink.
\end{proof}

\begin{defn} \label{defn:Cmax}
Suppose $\beta$ is a standardized positive $n$-braid. We define $\Codd$ and $\Ceven$:
\begin{align*}
\Codd := \sum_{i \equiv 1 \mod 2}^{n-1} c_i \qquad \qquad \qquad \qquad  \Ceven := \sum_{i \equiv 0 \mod 2}^{n-1} c_i
\end{align*}
\end{defn}

\begin{defn} \label{defn:Cmax}
We define $\Cmax$ as follows:
\begin{align*}
\Cmax := 
    \begin{cases}
    	\Codd & \Codd \geq \Ceven \\ %\Cone \neq \Ctwo \neq \Cthree 
	\Ceven & \text{otherwise} 
    \end{cases}
\end{align*}
\end{defn}

We can now detail our construction of our branched surface in various parts of the braid.

%%%%%%%%%%%%%%%%%%%%%%%%%%%%%%%%%%%%%%%%
%%%%%%%%%%%%%%%%%%%%%%%%%%%%%%%%%%%%%%%%
%%%%%%%%%% BUILDING B in EXTREMAL COLUMNS
%%%%%%%%%%%%%%%%%%%%%%%%%%%%%%%%%%%%%%%%
%%%%%%%%%%%%%%%%%%%%%%%%%%%%%%%%%%%%%%%%
\subsection{Building the branched surface in the leftmost columns} \label{section:BuildInFront}

We begin by describing the construction of the branched surface in the first three columns of the braid. There will be two constructions in this section: one for when $\Cmax = \Codd$, and the other for when $\Cmax = \Ceven$.

%%%%%%%%%Building $B$ in the front

\begin{lemma} \label{lemma:Generic_Cmax=Codd_Gammas12}
Let $\beta$ be a standardized braid on $n \geq 5$ such that $\Cmax = \Codd$. We can build a branched surface $B$ in $\Gamma_1 \cup \Gamma_2 \cup \Gamma_3$ such that: 
\begin{enumerate}
\item $B$ contains $c_1 - 1$ product disks from $\Gamma_1$,
\item $B$ contains 1 product disk from $\Gamma_2$,
\item $B$ contains $c_3 - 1$ product disks from $\Gamma_3$ , 
\item there is at most one potential branch sector spanning $\Gamma_1 \cup \Gamma_2$ which could be a sink, and %and it is $\mathcal{S}_3$,
\item no arcs $\alpha, \alpha'$, where $\alpha, \alpha' \subset S_1 \cup S_2 \cup S_3$, form a linked pair in $\tau_B$.
\end{enumerate}
In particular, we build $B$ in $\Gamma_1$ and $\Gamma_3$ by applying our template to these columns, for some cyclic conjugations of $\beta$.
\end{lemma}

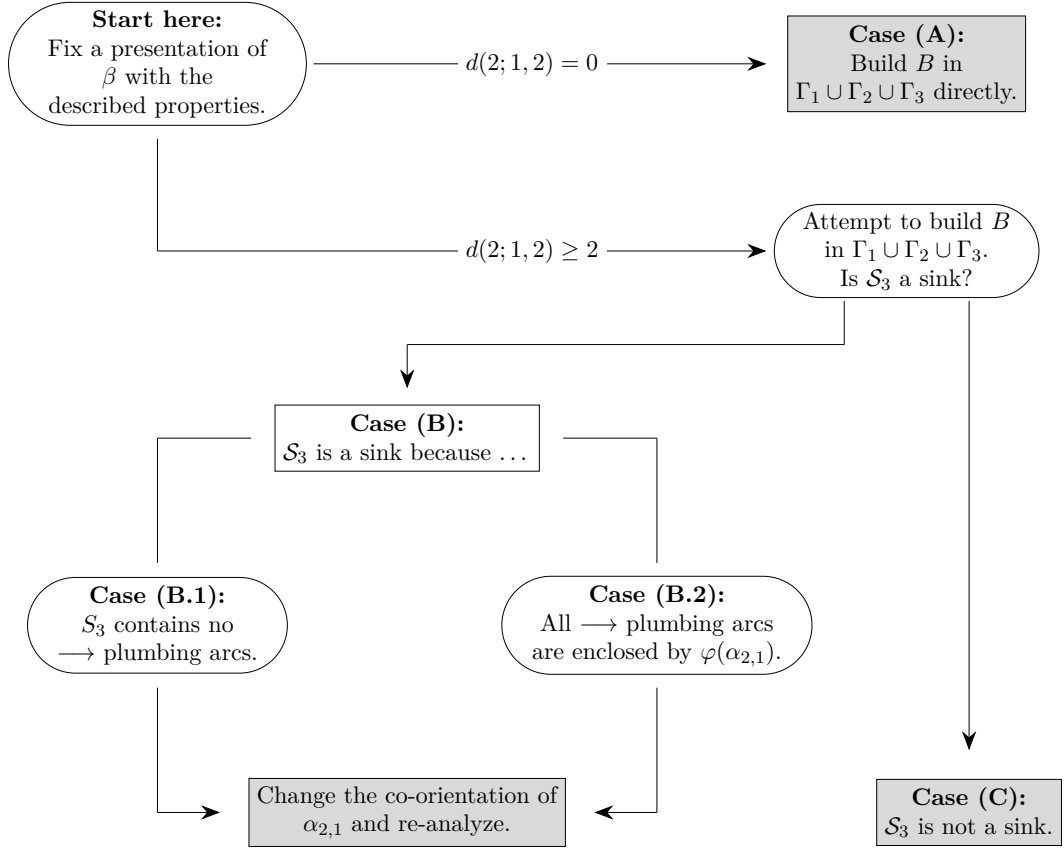
\begin{figure}[h!] \center
\scalebox{0.83}{
\begin{tikzpicture}
    \node[draw,
    	rounded rectangle,
        align=center, 
        minimum width=1cm,
        minimum height=1.2cm] at (0,0) {\textbf{Start here:} \\ Fix a  presentation of \\$\beta $ with the \\ described properties.};
 
    \node[draw,
        fill=gray!30,
    	align=center, 
        minimum width=2cm, 
        minimum height=1cm] at (12,0) {\textbf{Case (A):} \\ Build $B $ in \\ $ \Gamma_1 \cup \Gamma_2 \cup \Gamma_3$ directly.};
        
    \node[draw,
    	rounded rectangle,
        align=center, 
        minimum width=1cm,
        minimum height=1cm] at (12,-3) {Attempt to build $B$ \\ in $\Gamma_1 \cup \Gamma_2 \cup \Gamma_3$. \\ Is $\mathcal{S}_3$ a sink?};
        
    \node[draw,
        fill=gray!30,
    	align=center, 
        minimum width=2cm, 
        minimum height=1cm] at (13,-12) {\textbf{Case (C):} \\ $\mathcal{S}_3$ is not a sink.};
        
    \node[draw,
    	align=center, 
        minimum width=2cm, 
        minimum height=1cm] at (4,-6) {\textbf{Case (B):} \\ $\mathcal{S}_3$ is a sink because \ldots};        
        
    \node[draw,
        rounded rectangle,
    	align=center, 
        minimum width=2cm, 
        minimum height=1cm] at (0,-9) {\textbf{Case (B.1):} \\ $S_3$ contains no \\ $\longrightarrow$ plumbing arcs.};    

    \node[draw,
        rounded rectangle,
    	align=center, 
        minimum width=2cm, 
        minimum height=1cm] at (8,-9) {\textbf{Case (B.2):} \\ All $\longrightarrow$ plumbing arcs \\ are enclosed by $\varphi(\alpha_{2,1})$.};
        
    \node[draw,
    	fill=gray!30,
    	align=center, 
        minimum width=2cm, 
        minimum height=1cm] at (4,-12) {Change the co-orientation of \\ $\alpha_{2,1}$ and re-analyze.};       
        
      	%%%Draw in arrows
	\draw[-{Stealth[length=3mm]}] (2.5,0) -- (9.75,0);
	\node[fill=white] at (6,0) {$d(2; 1,2)=0$};
	
	\draw[] (0,-1.2) -- (0,-3);
	\draw[-{Stealth[length=3mm]}] (0,-3) -- (9.75,-3);
	\node[fill=white] at (6,-3) {$d(2; 1,2)\geq 2$};
	
	\draw[-{Stealth[length=3mm]}] (13,-3.8) -- (13,-11);
	\draw[] (11,-3.8) -- (11,-4.5);
	\draw[] (4,-4.5) -- (11,-4.5);
	\draw[-{Stealth[length=3mm]}] (4,-4.5) -- (4,-5.2);
	
	% for B.1
	\draw[] (1.5,-6) -- (0,-6);
	\draw[] (0,-6) -- (0,-8);
	\draw[] (0,-10) -- (0,-12);
	\draw[-{Stealth[length=3mm]}] (0,-12) -- (1,-12);
	
	% for B.2
	\draw[] (6.5,-6) -- (8,-6);
	\draw[] (8,-6) -- (8,-8);
	\draw[] (8,-10) -- (8,-12);
	\draw[-{Stealth[length=3mm]}] (8,-12) -- (7,-12);
\end{tikzpicture}
}
\caption{Outline of the proof of \Cref{lemma:Generic_Cmax=Codd_Gammas12}. Terminal nodes are shaded gray.}
\label{fig:ProofOutline_FirstTwoColumns_GenericOdd}
\end{figure}

\begin{proof}
Since $\beta$ is standardized, for all $1 \leq i \leq c_i$, $d(1; i, i+1) = 0$ or $d(1; i, i+1) \geq 2$. This means that we may cyclically conjugate $\beta$ to present it as follows:
 $\beta \approx \sigma_1 \sigma_2 \gamma \sigma_2 \gamma'$, where $\gamma$ is either trivial, or it contains no $\sigma_1$ letters. We fix this presentation of $\beta$ for the remainder of the proof. 
 
We want to build $B$ in $\Gamma_1 \cup \Gamma_2 \cup \Gamma_3$. To start, we build $B$ in $\Gamma_1 \cup \Gamma_2$ as follows:
\begin{itemize}
\item we make $\alpha_{1,1}$ a left pointer, and $\alpha_{1,i}$ right pointers for all $i \in \{2, \ldots, c_1-1\}$,
\item we make $\alpha_{2,1}$ a right pointer.
\end{itemize}

The construction of $B$ in $\Gamma_2 \cup \Gamma_3$ is determined by whether $d(2;1,2)\geq 2$ or $d(2; 1,2) =0$, i.e. whether $\gamma$ does or does not contain any $\sigma_3$ letters, respectively. Proving these two cases sometimes requires further analysis; we organize this information in the flowchart in \Cref{fig:ProofOutline_FirstTwoColumns_GenericOdd}.

%%%%%%%%%%%%%

\tcbox[size=fbox, colback=gray!30]{\textbf{Case (A):} $d(2;1,2) =0$. } 

If $d(2; 1,2) = 0$, we fall into Case (A) of \Cref{fig:ProofOutline_FirstTwoColumns_GenericOdd}, and we build $B$ in $\Gamma_2 \cup \Gamma_3$ by pivoting $\beta$ about $\alpha_{3,1}$, and applying our template to $\Gamma_3$. In our figures, we show $B$ in $\Gamma_1 \cup \Gamma_2 \cup \Gamma_3$ using the original presentation of $\beta$. \Cref{fig:Generic_CmaxCodd_Gammas12} contains examples when $c_1 = 2$ and $c_1 \geq 3$.

\begin{figure}[h!] \center
\labellist
\tiny
\pinlabel {$\mathcal{P}$} at 44 100
\pinlabel {$\mathcal{P}$} at 222 100
\pinlabel {$\Gamma_1$} at 35 160
\pinlabel {$\Gamma_2$} at 64 160
\pinlabel {$\Gamma_3$} at 90 160
\pinlabel {$\Gamma_1$} at 215 160
\pinlabel {$\Gamma_2$} at 242 160
\pinlabel {$\Gamma_3$} at 272 160
\endlabellist
    %\begin{framed}
        \includegraphics[scale=1.4]{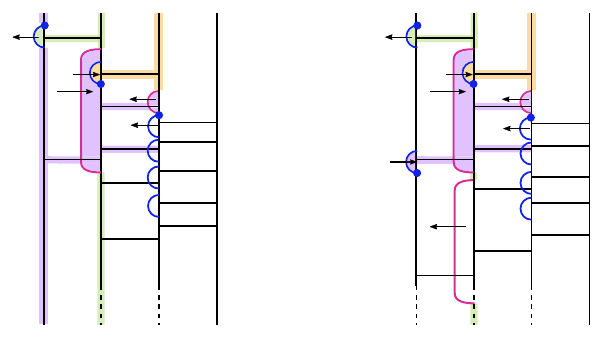}
        \caption{Constructing $B$ in $\Gamma_1 \cup \Gamma_2$ for Case \textbf{(A)}, where $d(2; 1,2)=0$.
        \textbf{Left:} The case where $c_1=2$. \textbf{Right:} An example where $c_1 \geq 3$. We know that $\alpha_{3,1}$ is a left pointer, but we do not know the co-orientations of the remaining plumbing arcs in $S_3$, so we suppress that information.}
        \label{fig:Generic_CmaxCodd_Gammas12}
    %\end{framed}
\end{figure}

By construction, we see that items (1), (2), and (3) of the desired conclusion of the lemma hold. We check item (4): that there is at most one potential sink disk spanning $\Gamma_1 \cup \Gamma_2$.

\begin{itemize}
\item \Seifert if $c_1 = 2$, then $\mathcal{S}_1 \sim \mathbbm{b}_{1,2}$, and $\partial \mathcal{S}_1$ contains both $\varphi(\alpha_{1,1})$ and $\alpha_{2,1}$, as in \Cref{fig:Generic_CmaxCodd_Gammas12} (left). In particular, $\alpha_{2,1}$ is a right pointer exiting the sector, so $\mathcal{S}_1$ is not a sink. Alternatively, suppose $c_1 \geq 3$. In this case, $\alpha_{1,2}$ is a right pointer, pointing out of $\mathcal{S}_1$, so the sector is not a sink. 

We see that $\mathcal{S}_2 \sim \mathbbm{b}_{1,1}$, and $\partial \mathcal{S}_2$ contains $\alpha_{1,1}$. This arc points out of $\mathcal{S}_2$, so $\mathcal{S}_2$ is not a sink. 

The branch sector $\mathcal{S}_3$, a priori, could be a sink. 

\item \Polygon by definition, a polygon sector must lie in a Seifert disk and also contain both plumbing and image arcs in its boundary. We claim there are no polygon sectors spanning $\Gamma_1 \cup \Gamma_2$: 
\begin{itemize}
\item[$\boldsymbol{\circ}$] $S_1$ contains only plumbing arcs, so it does not contain any polygon sectors. 
\item[$\boldsymbol{\circ}$] $S_2$ contains exactly one sector co-bounded by both plumbing and image arcs, but this sector also contains $\mathbbm{b}_{1,2}$. Therefore, it does not lie in a Seifert disk, so it is not a polygon sector. 
\item[$\boldsymbol{\circ}$] Recall that $d(2;1,2)=0$, and that we will apply our template to $\Gamma_3$ at $\alpha_{3,1}$. Therefore, the plumbing arcs in $S_3$ do not cobound a sector with the unique image arc in $S_3$, and there are no polygon sectors in $S_3$. See \Cref{fig:Generic_CmaxCodd_Gammas12}.
\end{itemize}

\item \Horizontal Suppose $\mathcal{H}$ is a horizontal sector spanning \underline{exactly} $\Gamma_1$. We analyze the cases $c_1=2$ and $c_1 \geq 3$ separately. 

\begin{itemize}
\item[$\boldsymbol{\circ}$] Suppose $c_1 = 2$. If $\mathcal{H}$ contains $\mathbbm{b}_{1,1}$, then $\mathcal{H} \sim \mathcal{S}_2$. Alternatively, if $\mathcal{H}$ contains $\mathbbm{b}_{1,2}$, then $\mathcal{H} \sim \mathcal{S}_1$. We already proved that no Seifert disk sectors are sinks, so no horizontal sector spanning $\Gamma_1$ is a sink. 
%%%
%%%
\item[$\boldsymbol{\circ}$] Suppose $c_1 \geq 3$. If $\mathcal{H} \sim \mathbbm{b}_{1,1}$, then $\mathcal{H} \sim \mathcal{S}_2$, which is not a sink. If $\mathcal{H} \sim \mathbbm{b}_{1,2}$, then $\mathcal{H} \sim \mathcal{P}$, where $\mathcal{P}$ denotes the sector with $\partial \mathcal{P}$ containing $\alpha_{1,2}, \alpha_{2,1},$ and $\varphi(\alpha_{1,1})$; see \Cref{fig:Generic_CmaxCodd_Gammas12} (right). We already argued that $\mathcal{P}$ is not a sink. Therefore, we need only check what happens if $\mathcal{H} \sim \mathbbm{b}_{1, i}$, for $3 \leq i \leq c_1$. By construction, for $3 \leq i \leq c_1$, $\mathbbm{b}_{1,i}$ is enclosed, on the $\varphi(\alpha_{1, i-1})$, which is a left pointer. In particular, it points out of $\mathcal{H}$; see \Cref{fig:Generic_CmaxCodd_Gammas12} (right). Therefore, no horizontal sector spanning $\Gamma_1$ is a sink. 
\end{itemize}

Next, suppose $\mathcal{H}$ spans $\Gamma_2$. There are four possibilities, which we analyze in turn.
\begin{itemize}
\item[$\boldsymbol{\circ}$] \textit{$\mathcal{H}$ is enclosed, on the left, by $\alpha_{2,1}$}. In this case, $\mathcal{H} \sim \mathcal{S}_3$; we do not know whether or not $\mathcal{S}_3$ is a sink or not. 
\item[$\boldsymbol{\circ}$] \textit{$\mathcal{H}$ is enclosed, on the left, by $\varphi(\alpha_{1,1})$}. In this case, $\partial \mathcal{H}$ also contains $\alpha_{2,1}$, and this arc points out of $\mathcal{H}$; see \Cref{fig:Generic_CmaxCodd_Gammas12}. Therefore, this sector is not a sink.
\item[$\boldsymbol{\circ}$] \textit{$\mathcal{H}$ is enclosed, on the left, by $\varphi(\alpha_{1,j})$, $j \geq 2$.} The arcs $\varphi(\alpha_{1,j})$, $j \geq 2$ are left pointers, pointing out of $\mathcal{H}$; see \Cref{fig:Generic_CmaxCodd_Gammas12}. 
\item[$\boldsymbol{\circ}$] \textit{$\mathcal{H}$ is not enclosed, on the left, by any arc.} In this case, $\mathcal{H} \sim \mathcal{S}_2$, which we already argued is not a sink.
\end{itemize}

We deduce that there is exactly one potential sink disk branch sector contained in $\Gamma_1 \cup \Gamma_2$: it is the sector $\mathcal{S}_3$. We record this in \Cref{table:Bookkeeping_Cmax=Codd_Gammas12}. This verifies (3) in the statement of \Cref{lemma:Generic_Cmax=Codd_Gammas12}.
\end{itemize}

Finally, we check item (5) from the statement of \Cref{lemma:Generic_Cmax=Codd_Gammas12}: that there are no linked pairs arising from arcs which lie in $S_1 \cup S_2 \cup S_3$. This is straightforward: \Cref{lemma:TemplateLinking} tells us that there is no pairwise linking between arcs in $S_1$ or in $S_3$, and one directly checks that $\alpha_{2,1}$ does not link with any arcs in $S_1$ or $S_3$. This concludes the construction of $B$ in the case where $d(2;1,2)=0$. 

\tcbox[size=fbox, colback=gray!30]{The case where $d(2;1,2) \geq 2$.}

Next, we detail the construction of $B$ in the case that $d(2;1,2) \geq 2$. As suggested by \Cref{fig:ProofOutline_FirstTwoColumns_GenericOdd}, this requires some finesse -- namely, we will start with our construction of $B$ in $\Gamma_1 \cup \Gamma_2$, and then attempt to build $B$ in $\Gamma_3$ in a naive way. When we do this, it is possible that we built something with a sink disk. In the event that we produced a sink disk, we modify the co-orientation of the arcs to eliminate the offending sector. Throughout this remainder of the construction, we let $\delta = d(2;1,2)$ and $\epsilon = d(1; 1,2)$. In particular, since $\beta$ is standardized, $\delta \geq 2$ and $\epsilon \geq 2$.  

We already specified the construction of $B$ in $\Gamma_1 \cup \Gamma_2$ at the start of this proof. We now explain how to build $B$ in $\Gamma_3$: we first pivot $\beta$ about $\alpha_{3,\delta}$, and then apply our template to $\Gamma_3$. For convenience of labelling bands, we draw $B$ and $\beta$ in the original/pre-pivoted presentation. 

With this construction in place, there are two possibilities at this stage: either
\begin{enumerate}
\item[(B)] $\mathcal{S}_3$ becomes a sink, or
\item[(C)] $\mathcal{S}_3$ is not a sink.
\end{enumerate}

We analyze each of these cases in turn. 

\tcbox[size=fbox, colback=gray!10]{\textbf{Case (B):} $\mathcal{S}_3$ became a sink.} 

Suppose Case (B) occurs, and $\mathcal{S}_3$ is a sink. This can happen for two reasons; either
\begin{enumerate}
\item[$(B.1)$] $S_3$ contains no right pointing $\alpha_{3,i}$ arcs, as in \Cref{fig:Generic_CmaxCodd_Gammas12_A}, or
\item[$(B.2)$] $S_3$ contains some right pointing $\alpha_{3,i}$ arcs, but they co-bound a branch sector with the unique image arc in $S_3$; see \Cref{fig:Generic_CmaxCodd_Gammas12_A.2}.
\end{enumerate}

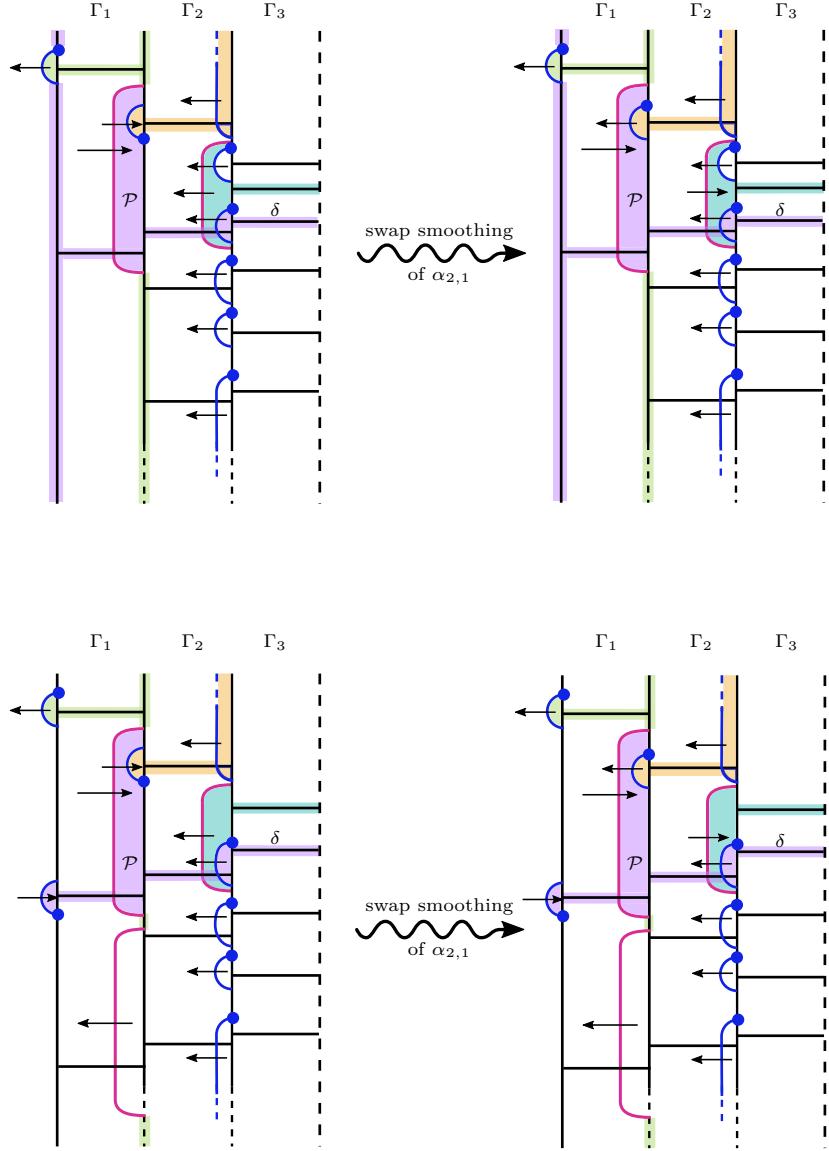
\begin{figure}[h!] \center
\labellist
\tiny
%bottom row
\pinlabel {$\mathcal{P}$} at 44 100
\pinlabel {$\mathcal{P}$} at 204 100
\pinlabel {$\Gamma_1$} at 35 170
\pinlabel {$\Gamma_2$} at 64 170
\pinlabel {$\Gamma_3$} at 90 170
\pinlabel {$\Gamma_1$} at 195 170
\pinlabel {$\Gamma_2$} at 225 170
\pinlabel {$\Gamma_3$} at 252 170
%% top row
\pinlabel {$\mathcal{P}$} at 44 310
\pinlabel {$\mathcal{P}$} at 204 310
\pinlabel {$\Gamma_1$} at 35 370
\pinlabel {$\Gamma_2$} at 64 370
\pinlabel {$\Gamma_3$} at 90 370
\pinlabel {$\Gamma_1$} at 195 370
\pinlabel {$\Gamma_2$} at 225 370
\pinlabel {$\Gamma_3$} at 252 370
\pinlabel {swap smoothing} at 142 86
\pinlabel {of $\alpha_{2,1}$} at 142 72
\pinlabel {swap smoothing} at 142 300
\pinlabel {of $\alpha_{2,1}$} at 142 286
\pinlabel {$\delta$} at 90 307
\pinlabel {$\delta$} at 250 307
\pinlabel {$\delta$} at 90 108
\pinlabel {$\delta$} at 250 107
\endlabellist
   % \begin{framed}
        \includegraphics[scale=1.2]{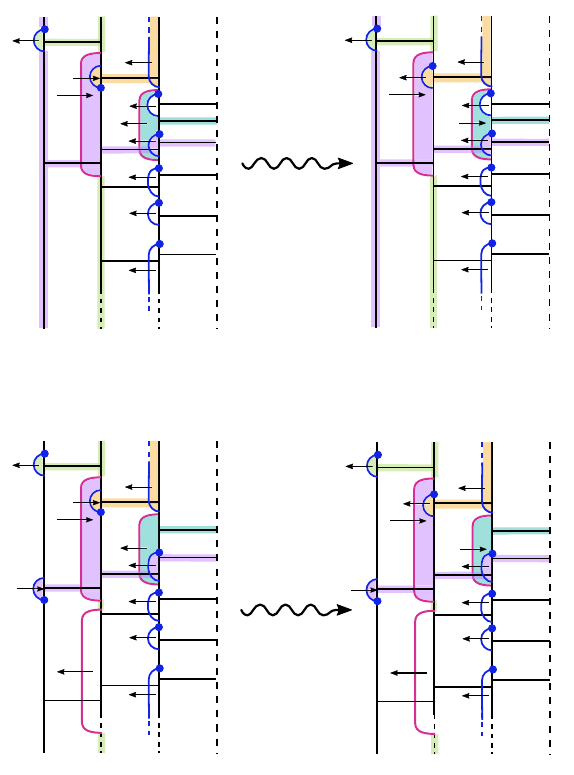}
        \caption{Understanding the $(B.1)$ case. \textbf{Top Row:} an example where $c_1=2$ and $\delta = d(2; 1,2) \geq 3$. We see that when $\alpha_{2,1}$ is a right pointer (as in the top row, left), $\mathcal{S}_3$ becomes a sink. So, we change the co-orientation of $\alpha_{2,1}$ and make it a left pointer (as in top row, right). Now, $\mathcal{S}_3$ is not a sink, because $\varphi(\alpha_{2,1})$ is a right pointer exiting $\mathcal{S}_3$. \textbf{Bottom Row:} an example where $c_1\geq 3$, and $\delta = d(2; 1,2) =2$. As in the top row, before doing any modifications, $\mathcal{S}_3$ is a sink, so we change the co-orientation of $\alpha_{2,1}$. Since we have not included any crossings from $\Gamma_4$, we dotted the vertical line corresponding to the $4^{\text{th}}$ strand of the braid. After the modification, the branch sector $\mathcal{P}$ is not a sink.}
        \label{fig:Generic_CmaxCodd_Gammas12_A}
   % \end{framed}
\end{figure}

(Note: if $(B.1)$ occurs, then by the definition of our template, we must have that in the pre-pivot presentation of $\beta$, $d(3; \delta-2, \delta)=c_4$; in particular, we have that $d(3; \delta, 1) + d(3; 1, \delta-2) = 0$.)

\begin{figure}[h!] \center
\labellist
\tiny
%bottom row
\pinlabel {$\mathcal{P}$} at 44 100
\pinlabel {$\mathcal{P}$} at 204 100
\pinlabel {$\Gamma_1$} at 35 170
\pinlabel {$\Gamma_2$} at 64 170
\pinlabel {$\Gamma_3$} at 90 170
\pinlabel {$\Gamma_1$} at 195 170
\pinlabel {$\Gamma_2$} at 225 170
\pinlabel {$\Gamma_3$} at 252 170
%% top row
\pinlabel {$\mathcal{P}$} at 44 310
\pinlabel {$\mathcal{P}$} at 204 310
\pinlabel {$\Gamma_1$} at 35 370
\pinlabel {$\Gamma_2$} at 64 370
\pinlabel {$\Gamma_3$} at 90 370
\pinlabel {$\Gamma_1$} at 195 370
\pinlabel {$\Gamma_2$} at 225 370
\pinlabel {$\Gamma_3$} at 252 370
\pinlabel {swap smoothing} at 142 86
\pinlabel {of $\alpha_{2,1}$} at 142 72
\pinlabel {swap smoothing} at 142 300
\pinlabel {of $\alpha_{2,1}$} at 142 286
\pinlabel {$\delta$} at 90 307
\pinlabel {$\delta$} at 250 307
\pinlabel {$\delta$} at 90 100
\pinlabel {$\delta$} at 250 100
\endlabellist
    %\begin{framed}
        \includegraphics[scale=1.15]{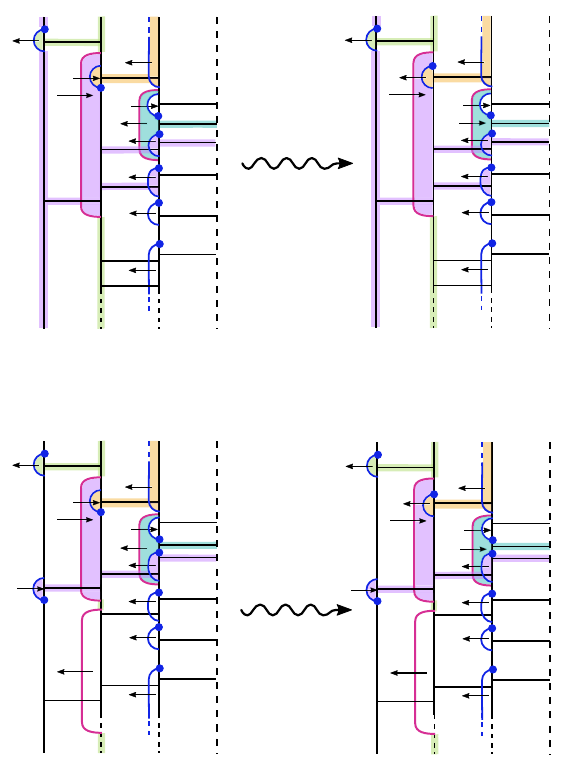}
        \caption{Understanding the $(B.2)$ case. \textbf{Top Row:} an example where $c_1=2$, $\epsilon = d(1; 1,2) = 3$, and $\delta = d(2; 1,2) \geq 3$. The arc $\varphi(\alpha_{2,1})$ encloses the only right pointing plumbing arcs in $S_3$. We see that when $\alpha_{2,1}$ is a right pointer (as in the top row, left), $\mathcal{S}_3$ becomes a sink. Changing the co-orientation of $\alpha_{2,1}$ yields (top row, right). Now, $\mathcal{S}_3$ is not a sink. \textbf{Bottom Row:} an example where $c_1\geq 3$, and $\delta = d(2; 1,2) =2$. As in the top row, before doing any modifications, $\mathcal{S}_3$ is a sink, so we change the co-orientation of $\alpha_{2,1}$.}
        \label{fig:Generic_CmaxCodd_Gammas12_A.2}
    %\end{framed}
\end{figure}

Regardless of whether $(B.1)$ or $(B.2)$ holds, we can quickly eliminate the existing sink disk sector $\mathcal{S}_3$: we first pivot $\beta$ with respect to the unique $\sigma_1$ which is enclosed by a left pointing plumbing arc, thereby ``resetting'' our original presentation of $\beta$. With this presentation, the unique co-oriented arc in $S_2$ is $\alpha_{2,1}$, which is currently a right pointer. We now make $\alpha_{2,1}$ a left pointer. The results are seen in \Cref{fig:Generic_CmaxCodd_Gammas12_A} (for Case (B.1))and \Cref{fig:Generic_CmaxCodd_Gammas12_A.2} (for Case (B.2)).

It is clear that items (1), (2), and (3) of the statement of \Cref{lemma:Generic_Cmax=Codd_Gammas12} hold. We now analyze the branch sectors, and verify that there is at most one sink disk sector whose span lies in $\Gamma_1 \cup \Gamma_2$. 

\begin{itemize}
\item \Seifert We analyze the cases $c_1=2$ and $c_1 \geq 3$ separately. 

\begin{itemize}
\item[$\boldsymbol{\circ}$] If $c_1=2$, then $\mathcal{S}_1 \sim \mathbbm{b}_{1,2} \sim \mathbbm{b}_{2, 2} \sim \mathbbm{b}_{3, \delta}$; see \Cref{fig:Generic_CmaxCodd_Gammas12_A} (top row) and \Cref{fig:Generic_CmaxCodd_Gammas12_A.2} (bottom row). Since $\mathbbm{b}_{3, \delta}$ is enclosed, on the left, by a left pointing plumbing arc, and this arc points out of $\mathcal{S}_1$, we deduce this sector is not a sink. 
\item[$\boldsymbol{\circ}$] Suppose $c_1 \geq 3$. Then $\alpha_{1,2}$ points out of $\mathcal{S}_1$, and the sector is not a sink. 
\end{itemize}

The sector $\mathcal{S}_2$ contains $\varphi(\alpha_{1,1})$ in its boundary. This image arc is a right pointer, pointing out of $\mathcal{S}_2$, so this sector is not a sink. See \Cref{fig:Generic_CmaxCodd_Gammas12_A} and \Cref{fig:Generic_CmaxCodd_Gammas12_A.2}. The arc $\varphi(\alpha_{2,1})$ is part of $\partial \mathcal{S}_3$, and it points out of this disk sector. Thus, $\mathcal{S}_3$ is not a sink.

%%%polygon sector analysis
\item \Polygon By definition, a polygon sector must lie in a Seifert disk and also contain both plumbing and image arcs in its boundary. In this case, there are no polygon sectors spanning $\Gamma_1 \cup \Gamma_2$: 
\begin{itemize}
\item[$\boldsymbol{\circ}$] $S_1$ contains only plumbing arcs, so it does not contain any polygon sectors. 
\item[$\boldsymbol{\circ}$] $S_2$ contains a sector co-bounded by both plumbing and image arcs, but this sector also contains $\mathbbm{b}_{1,2}$. Therefore, it does not lie in a Seifert disk, so it is not a polygon sector. 
\item[$\boldsymbol{\circ}$] Since $d(2;1,2)\geq 2$, and we applied our template to $\Gamma_3$ at $\alpha_{3,\delta}$, the band $\mathbbm{b}_{3,\delta-1}$ is not enclosed, on the left, by a plumbing arc; see \Cref{fig:Generic_CmaxCodd_Gammas12_A} and \Cref{fig:Generic_CmaxCodd_Gammas12_A.2}. In particular, the unique sector in $S_3$ which is co-bounded by both plumbing and image arcs must contain $\mathbbm{b}_{3,\delta-1}$, so it does not lie in a Seifert disk. Thus, there are no polygon sectors in $S_3$.
\end{itemize}

%%%horizontal sector analysis
\item \Horizontal Suppose $\mathcal{H}$ is a horizontal sector spanning $\Gamma_1$. We analyze the cases $c_1=2$ and $c_1 \geq 3$ separately. 

\begin{itemize}
\item[$\boldsymbol{\circ}$] When $c_1 = 2$, we can apply the same argument as in Case (A).% If $\mathcal{H} \sim \mathbbm{b}_{1,1}$, then $\mathcal{H} \sim \mathcal{S}_2$, which we already proved is not a sink. Otherwise, $\mathcal{H} \sim \mathbbm{b}_{1,2}$; as previously argued, $\mathcal{H} \sim \mathcal{S}_1$. We already proved this sector is not a sink, so we deduce that no horizontal sector spanning exactly $\Gamma_1$ is a sink. 
\item[$\boldsymbol{\circ}$] Suppose $c_1 \geq 3$. 
If $\mathcal{H} \sim \mathbbm{b}_{1,1}$, then $\mathcal{H} \sim \mathcal{S}_2$, which we already proved is not a sink. 
Suppose instead that $\mathcal{H} \sim \mathbbm{b}_{1,2}$. In this case, $\mathbbm{H} \sim \mathbbm{b}_{1,2} \sim \mathbbm{b}_{2,2} \sim \mathbbm{b}_{3, \delta}$. Since $\alpha_{3, \delta}$ is a left pointer, this arc points out of the sector, and $\mathcal{H}$ is not a sink.
%
%The band $\mathbbm{b}_{1,1} \sim \mathcal{S}_2$, so it is not a sink. As we observed above, $\partial \mathbbm{b}_{1,2}$ contains $\varphi(\alpha_{1,1})$; this arc is a left pointer exiting the sector, so it is not a sink. 
The remaining bands $\mathbbm{b}_{1,3}, \ldots, \mathbbm{b}_{1,c_1}$ are enclosed, on the right, by the left pointers $\varphi(\alpha_{1,j-1})$; see \Cref{fig:Generic_CmaxCodd_Gammas12_A} and \Cref{fig:Generic_CmaxCodd_Gammas12_A.2}. So, no horizontal sector spanning $\Gamma_1$ is a sink. 
\end{itemize}

Next, suppose $\mathcal{H}$ is a horizontal sector whose span includes $\Gamma_2$. There are four possibilities, each of which we analyze in turn. We refer the reader to \Cref{fig:Generic_CmaxCodd_Gammas12_A} and \Cref{fig:Generic_CmaxCodd_Gammas12_A.2} throughout.

\begin{itemize}
\item[$\boldsymbol{\circ}$] \textit{$\mathcal{H}$ is enclosed, on the left, by $\alpha_{2,1}$}. But $\alpha_{2,1}$ is a left pointer exiting $\mathcal{H}$, so it is not a sink. 
\item[$\boldsymbol{\circ}$] \textit{$\mathcal{H}$ is enclosed, on the left, by $\varphi(\alpha_{1,1})$}. %This is the unique place in which the arguments for the $(B.1)$ and $(B.2)$ cases differ:
        %\begin{itemize}
        %		\item[$\bullet$] The $(B.1)$ case: a priori, the sector $\mathcal{H}$ could be a sink disk! We will subsequently construct $B$ in $\Gamma_4$ to ensure it is not a sink. See \Cref{fig:Generic_CmaxCodd_Gammas12_A}.
        %
       %\item[$\bullet$] The $(B.2)$ case: 
       By construction, if $\mathcal{H}$ is enclosed on the left, by $\varphi(\alpha_{1,1})$, then $\mathcal{H} \sim \mathbbm{b}_{2,2}$. This band is enclosed, on the right, by $\alpha_{3, \delta}$, which is a left pointer. In particular, it points out of $\mathcal{H}$, so $\mathcal{H}$ is not a sink. %See \Cref{fig:Generic_CmaxCodd_Gammas12_A.1} and \Cref{fig:Generic_CmaxCodd_Gammas12_A.2}.
        %\end{itemize}
%
\item[$\boldsymbol{\circ}$] \textit{$\mathcal{H}$ is enclosed, on the left, by $\varphi(\alpha_{1,j})$, $j \geq 2$.} The arcs $\varphi(\alpha_{1,j})$, $j \geq 2$ are left pointers exiting $\mathcal{H}$, so such sectors are not sinks. %See \Cref{fig:Generic_CmaxCodd_Gammas12_A} and \Cref{fig:Generic_CmaxCodd_Gammas12_A.2}.
\item[$\boldsymbol{\circ}$] \textit{$\mathcal{H}$ is not enclosed, on the left, by any arc.} In this case, $\mathcal{H} \sim \mathcal{S}_2$, which is not a sink.
\end{itemize}
\end{itemize}

We deduce that there are no sink disks spanning $\Gamma_1 \cup \Gamma_2$ in Case (B).%: it occurs in the $(B.1)$ case, and the sector is the horizontal sector containing $\mathbbm{b}_{3,\delta-1}$.  

We now study the linking behavior between arcs in $S_1, S_2$, and $S_3$ for Case (B). 
%This concludes the analysis of item (4) from the statement of \Cref{lemma:Generic_Cmax=Codd_Gammas12} for the $(B.2)$ case, and we can move onto item (4): 
%
%that there are no linked pairs arising from arcs which lie in $S_1 \cup S_2$. 
\Cref{lemma:TemplateLinking} confirms that there is no pairwise linking between arcs which are both in $S_1$ or both in $S_3$. There can be no pairwise linking between arcs in $S_1$ and $S_3$. Finally, one checks that we co-oriented $\alpha_{2,1}$ so that it does not link with arcs in either $S_1$ or $S_3$. Therefore, we have the desired linking behavior. This concludes the proof of \Cref{lemma:Generic_Cmax=Codd_Gammas12} for Case $(B)$. 

\tcbox[size=fbox, colback=gray!10]{\textbf{Case (C):} $\mathcal{S}_3$ did not become a sink when we applied our template to $\Gamma_3$.} 

We now turn to Case (C), and refer the reader to \Cref{fig:Generic_CmaxCodd_Gammas12_B} throughout. Once again, we have presented $\beta$ in its original/pre-pivoted form for convinience while labelling bands. If, when we applied our template to $\Gamma_3$, we determined that $\mathcal{S}_3$ is not a sink, we must have that for some minimal $j$ with $\delta + 1 \leq j \leq c_3$, there is an arc $\alpha_{3,j}$ which is a right pointer; note that, by assumption, $\alpha_{3,j} \subset \partial \mathcal{S}_3$. Once again, items (1), (2), (3) from the conclusions of of \Cref{lemma:Generic_Cmax=Codd_Gammas12} hold, so we need to check that there is at most potential branch sector spanning $\Gamma_1 \cup \Gamma_2$ which could be a sink.

\begin{figure}[h!] \center
\labellist
\tiny
%bottom row
\pinlabel {$\mathcal{P}$} at 44 115
\pinlabel {$\mathcal{P}$} at 198 115
\pinlabel {$\Gamma_1$} at 35 175
\pinlabel {$\Gamma_2$} at 64 175
\pinlabel {$\Gamma_3$} at 90 175
\pinlabel {$\Gamma_1$} at 190 175
\pinlabel {$\Gamma_2$} at 215 175
\pinlabel {$\Gamma_3$} at 245 175
%% top row
\pinlabel {$\mathcal{P}$} at 44 310
\pinlabel {$\mathcal{P}$} at 198 310
\pinlabel {$\Gamma_1$} at 35 370
\pinlabel {$\Gamma_2$} at 64 370
\pinlabel {$\Gamma_3$} at 90 370
\pinlabel {$\Gamma_1$} at 190 370
\pinlabel {$\Gamma_2$} at 218 370
\pinlabel {$\Gamma_3$} at 245 370
\pinlabel {$j$} at 86 272
\pinlabel {$j$} at 238 237
\pinlabel {$j$} at 85 75
\pinlabel {$j$} at 238 58
\pinlabel {$\delta$} at 86 300
\pinlabel {$\delta$} at 238 300
\pinlabel {$\delta$} at 85 108
\pinlabel {$\delta$} at 238 108
\endlabellist
    %\begin{framed}
        %\includegraphics[scale=0.5]{Generic_CmaxCodd_Gammas1,2_B_ByHand}
        \includegraphics[scale=1.2]{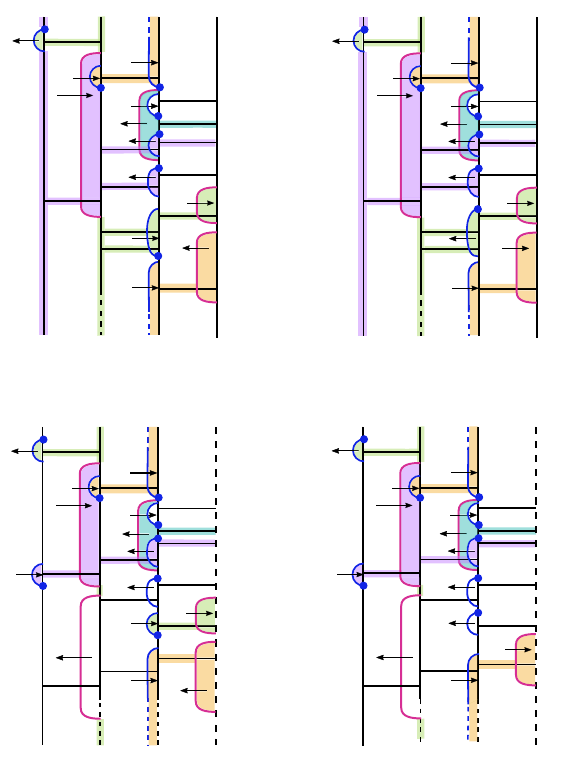}
        \caption{Understanding Case (C). Let $\mathcal{H}$ denote the horizontal sector where $\mathcal{H} \sim \mathbbm{b}_{2,1} \sim \mathbbm{b}_{3,c_3}$ (this sector is orange in all frames). \textbf{Top Row:} In the top row, we have that $c_1=2$. On the left, we have an example where $j < c_3$, and so the horizontal sector $\mathcal{H}$ is not a sink. On the right, we have $j = c_3$, and so a priori, it is possible that the horizontal sector $\mathcal{H}$ is a sink. \textbf{Bottom Row:} The analogous frames when $c_1 \geq 3$.}
        \label{fig:Generic_CmaxCodd_Gammas12_B}
    %\end{framed}
\end{figure}
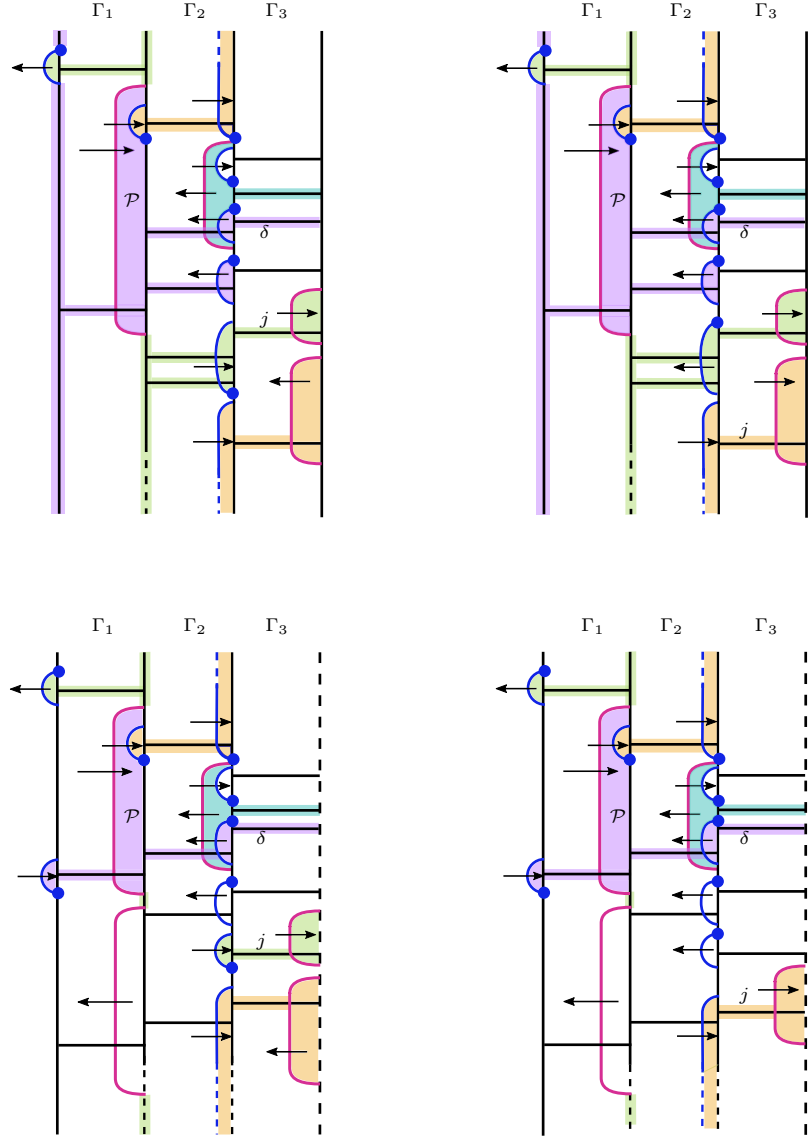

\begin{itemize}

\item \Seifert To see why $\mathcal{S}_1$ is not a sink, we must argue the cases where $c_1=2$ and $c_1 \geq 3$ separately. 

\begin{itemize}
\item[$\boldsymbol{\circ}$]  If $c_1=2$, then $\mathcal{S}_1$ contains $\mathbbm{b}_{1,2}$, and $\partial \mathcal{S}_1$ contains $\alpha_{2,1}$ in the boundary. Since this arc is a right pointer, pointing out of the sector, $\mathcal{S}_1$ is not a sink. 
\item[$\boldsymbol{\circ}$]  If $c_1 \geq 3$, then $\alpha_{1,2}$ points out of $\mathcal{S}_1$, so this sector is not a sink. 
\end{itemize}

As $\varphi(\alpha_{1,1}) \subset \partial \mathcal{S}_2$, and this arc is a right pointer, $\mathcal{S}_2$ is not a sink. By design, $\partial \mathcal{S}_3$ contains at least one outward pointing plumbing arc, namely $\alpha_{3,j}$.
\item \Polygon The same argument as in Case (B) verifies that there are no polygon sectors spanning $\Gamma_1 \cup \Gamma_2$. So, there are no polygon sectors which are sink disks. 
\item \Horizontal The same argument as Case (B) confirms that if $\mathcal{H}$ is a horizontal sector spanning exactly $\Gamma_1$, then $\mathcal{H}$ is not a sink. It remains to check that if $\mathcal{H}$ is a horizontal sector whose span includes $\Gamma_2$, then $\mathcal{H}$ is not a sink. We analyze the four possibilities in turn.

\begin{enumerate}
\item \textit{$\mathcal{H}$ is enclosed, on the left, by $\alpha_{2,1}$}. In this case, $\mathcal{H}$ meets $S_3$ in $\alpha_{3, c_3}$; see \Cref{fig:Generic_CmaxCodd_Gammas12_B}. In particular, this means $\mathcal{H} \sim \mathbbm{b}_{3,c_3}$. This band is enclosed, on the right, by $\varphi(\alpha_{3,c_3-1})$, which can either be a right or left pointer; see \Cref{fig:Generic_CmaxCodd_Gammas12_B}. If $j < c_3$, then $\varphi(\alpha_{3,c_3-1})$ is a left pointer (as in \Cref{fig:Generic_CmaxCodd_Gammas12_B} (left)), and $\mathcal{H}$ is not a sink. Otherwise, $j=c_3$ (as in \Cref{fig:Generic_CmaxCodd_Gammas12_B} (right)) and $\mathcal{H}$ is a right pointer -- in this case,  $\mathcal{H}$ could be a sink disk. (This is our one potential sink disk spanning $\Gamma_1 \cup \Gamma_2$ in this case.)
\item \textit{$\mathcal{H}$ is enclosed, on the left, by $\varphi(\alpha_{1,1})$}. In this case, $\partial \mathcal{H}$ also contains $\alpha_{2,1}$, which points out of the sector. Thus, $\mathcal{H}$ is not a sink. 
\item \textit{$\mathcal{H}$ is enclosed, on the left, by $\varphi(\alpha_{1,k})$, $k \geq 2$.} The arcs $\varphi(\alpha_{1,k})$, $k \geq 2$ are left pointers exiting $\mathcal{H}$, so such sectors are not sinks. See \Cref{fig:Generic_CmaxCodd_Gammas12_B}.
\item \textit{$\mathcal{H}$ is not enclosed, on the left, by any arc.} In this case, $\mathcal{H} \sim \mathcal{S}_2$, which is not a sink.
\end{enumerate}
\end{itemize}

Therefore, at present, there is at most one sector spanning $\Gamma_1 \cup \Gamma_2$ which could potentially be a sink -- it is the horizontal sector $\mathcal{H}$ which contains $\mathbbm{b}_{2,1}$. This concludes the proof of (4) from the statement of \Cref{lemma:Generic_Cmax=Codd_Gammas12} for the $(C)$ case. The same argument as in Case (B) confirms that we have no pairwise linking between any arcs in $S_1 \cup S_2 \cup S_3$.

This concludes the proof of \Cref{lemma:Generic_Cmax=Codd_Gammas12}. 
\end{proof}

In \Cref{table:Bookkeeping_Cmax=Codd_Gammas12}, we keep track of which sectors in $\Gamma_1 \cup \Gamma_2$ are the potential sink disks.

\begin{table}[H]
\begin{tabular}{ | c | c | c |}
\hline
\textsc{Case} &  \textsc{Status of $\Gamma_3$} &  \textsc{Potential Sink Disk contains} \\ \hline \hline
%\cellcolor{grey!20}{\textsc{Case}} &  \cellcolor{grey!20}{\textsc{Status of $\Gamma_3$}} &  \cellcolor{grey!20}{\textsc{Potential Sink Disk contains}} \\ \hline \hline
\hyperref[fig:Generic_CmaxCodd_Gammas12]{Case (A)} & unknown whether homogeneous or heterogeneous & $\mathcal{S}_3 \sim \mathbbm{b}_{2,1}$ \\  
 \hyperref[fig:Generic_CmaxCodd_Gammas12_A]{Case (B.1)} & homogeneous &  no sinks spanning $\Gamma_1 \cup \Gamma_2$ \\
 \hyperref[fig:Generic_CmaxCodd_Gammas12_A.2]{Case (B.2)} & heterogeneous &  no sinks spanning $\Gamma_1 \cup \Gamma_2$ \\
\hyperref[fig:Generic_CmaxCodd_Gammas12_B]{Case (C)} & heterogeneous & $\mathbbm{b}_{2,1} \sim \mathbbm{b}_{3, c_3}$ \\ \hline
\end{tabular} 
        \caption{Bookkeeping for \Cref{lemma:Generic_Cmax=Codd_Gammas12}, which builds $B$ in $\Gamma_1 \cup \Gamma_2 \cup \Gamma_3$ when $\Cmax=\Codd$. The third column details the unique potential sink disk spanning $\Gamma_1 \cup \Gamma_2$.}
        \label{table:Bookkeeping_Cmax=Codd_Gammas12}
\end{table}

Next, we describe how to build $B$ in $\Gamma_1 \cup \Gamma_2$ when $\Cmax = \Ceven$.

\begin{lemma} \label{lemma:Sparse_Cmax=Ceven_Gammas12}
Let $\beta$ be a standardized braid on $n \geq 5$ such that $\Cmax = \Ceven$. We can build a branched surface $B$ in $\Gamma_1 \cup \Gamma_2$ such that: 
\begin{enumerate}
\item $B$ contains $1$ product disk from $\Gamma_1$,
\item $B$ contains $c_2-1$ product disks from $\Gamma_2$,
\item $B$ contains no product disks from $\Gamma_3$, 
\item there is at most one sink disk sector spanning $\Gamma_1 \cup \Gamma_2$, and %and it is either $\mathcal{S}_3$ or a horizontal sector, and
\item no arcs $\alpha, \alpha'$, where $\alpha, \alpha' \subset S_1 \cup S_2$, form a linked pair in $\tau_B$.
\end{enumerate}
\end{lemma}

\begin{proof}
To begin, we will adapt the presentation of our braid. This will depend on the distribution of crossings in the braid. %At the outset, there will be four cases to consider, but we will simplify our analysis so that we need only apply two constructions. We first analyze the four possible presentations of the braid, and then we proceed with the construction afterwards.  

\tcbox[size=fbox, colback=gray!30]{\textbf{Case (1)}: $c_1 \geq 3$.} 
Since $\beta$ is standardized, we know that for all $1 \leq j \leq c_1$, $d(1; j, j+1) = 0$ or $d(1; j, j+1) \geq 2$. We cyclically conjugate $\beta$ so that it begins with a $\sigma_1$ letter, and so that $d(1; 1,2) \geq 2$. More specifically, we can cyclically conjugate our braid to write it as follows:
\begin{equation} \label{braid_special_form}
\beta \approx \sigma_1 w_0 \ \sigma_1^{k_1} w_1 \ \sigma_1^{k_2} w_2 \ \ldots \ \sigma_1^{k_j} w_{j} \ \sigma_1^{k_{j+1}}
\end{equation}
where the $w_i$ sub-words are non-empty for all $0 \leq i \leq j$, the $k_1, \ldots, k_j$ are all positive, and for $0 \leq i \leq j$, the $w_i$ contain no $\sigma_1$ letters. (We allow the possibility that $k_{j+1}$ = 0). 

Either two things are possible:

\begin{enumerate}
\item[(1A)] There exists some $i$, where $0 \leq i \leq j$, where $w_i = \theta_0 \ \sigma_2^{\ell_1} \ \theta_1 \ \sigma_2^{\ell_2} \ \theta_2 \ \ldots \ \sigma_2^{\ell_t} \ \theta_t$ such that 
	\begin{itemize}
	\item[$\bullet$] $\theta_0, \ldots, \theta_{t}$ contain no $\sigma_1$ or $\sigma_2$ letters,
	\item[$\bullet$] $\theta_0$ or $\theta_t$ could be empty but $\theta_1, \ldots, \theta_{t-1}$ are non-empty,
	\item[$\bullet$] $\ell_1, \ldots, \ell_t$ are all positive,
	\item[$\bullet$] when we restrict our distance function to the subword $w_i$, we have $d(2; 1, \sum_{r=1}^t \ell_r) \geq 2$.
	\end{itemize}

\item[(1B)] No such $w_i$ exists. 
\end{enumerate}

To start, we focus our attention to Case (1A). Let $w_\star$ denote the first subword in $\beta$ for which the conditions specified in (1A) occur. We now more concisely write the braid in \Cref{braid_special_form} as:
$$\beta \approx \beta_1 \sigma_1 w_\star \sigma_1 \beta_2$$
After cyclically conjugating our braid, we write it as 
\begin{equation} \label{braid_special_form_1}
\beta \approx \sigma_1 w_\star \sigma_1 \beta_2 \beta_1
\end{equation}
By design, this is a presentation of $\beta$ such that $d(1; 1,2) \geq 2$, and $d(2; 1, d(1; 1,2)) \geq 2$. Going forwards, we will assume that if we are in Case (1A), then $\beta$ is presented as in \Cref{braid_special_form_1}. 

Suppose, instead, that when we analyze the presentation of the braid as in \Cref{braid_special_form}, we fall into Case (1B). If this occurs, then for the braid in \Cref{braid_special_form}, for all $i$ with $1 \leq i \leq c_2$, we have that $d(2; i, i+1) = d_L(2; i, i+1) = 0$, \textbf{or} $d(2; i, i+1)$ and $d_L(2; i, i+1)$ are simultaneously positive. There must exist some $i$ such that the latter occurs (otherwise we have a split link); so, we rewrite our braid to spotlight those first two $\sigma_2$ letters where this occurs, and cyclically conjugate as follows to produce a new presentation of the braid:
%\begin{align*}
%\beta &\approx z_1 \ \sigma_2 \ (\theta_1 \sigma_1^a \theta_2 \sigma_3^b \theta_3) \  \sigma_2 \ z_2
%\end{align*}
%(Note that $\theta_1$ and $\theta_3$ could be the empty word, and $\theta_2$ contains no $\sigma_1$ or $\sigma_2$ letters. Additionally, we picked out these $\sigma_2$ letters specifically so that $d(2; i, i+1)$ and $d(2; i, i+1)$ are both positive, so we know that $\theta_1$ and $\theta_3$ do not contain any $\sigma_2$ letters, and that $a, b \geq 1$.) Given these conditions, we can further rewrite our braid as:
\begin{align}
%\beta &\approx z_1 \ \sigma_2 \ (\widetilde{\theta_1} \sigma_1 \widetilde{\theta_2} \sigma_3 \widetilde{\theta_3}) \  \sigma_2 \ z_2
\beta &\approx z_1 \ \sigma_2 \ (\theta_1 \sigma_1 \theta_2 \sigma_3 \theta_3) \  \sigma_2 \ z_2 \nonumber \\
%\end{align*}
\intertext{By assumption, we know that $\theta_1, \theta_2,$ and $\theta_3$ contain no $\sigma_2$ letters. We may assume that $\theta_2$ and $\theta_3$ contain no $\sigma_1$ letters (if they did, use braid relations to commute them into $\theta_1$). We cyclically conjugate $\beta$, and perform some braid relations, to get:}
%\begin{align}
 &\approx (\theta_1 \underline{\sigma_1 \theta_2} \sigma_3 \theta_3) \  \sigma_2 \ z_2 \ z_1 \ \sigma_2  \nonumber \\
&\approx (\underline{\theta_1  \theta_2} \sigma_1 \sigma_3 \theta_3) \  \sigma_2 \ z_2 \ z_1 \ \sigma_2  \nonumber \\
&\approx (\sigma_1 \sigma_3 \theta_3) \  \sigma_2 \ z_2 \ z_1 \ \sigma_2 \ \theta_1 \  \theta_2 \label{braid_special_form_2}
\end{align}

Going forwards, we assume that if $\beta$ falls under Case (1B), then $\beta$ is presented as in \Cref{braid_special_form_2}.

%In this case, we adapt our presentation of $\beta$ as follows: since $\beta$ is standardized, we know that for all $1 \leq j \leq c_1$, $d(1; j, j+1) = 0$ or $d(1; j, j+1) \geq 2$. We cyclically conjugate $\beta$ so that $\beta \approx \sigma_1 w_1 \sigma_1 w_2$, where $w_1$ contains no $\sigma_1$ letters, and so that $d(1; 1,2) \geq 2$. We will assume this presentation when we construct our branched surface in $\Gamma_1 \cup \Gamma_2$. 

\tcbox[size=fbox, colback=gray!30]{\textbf{Case (2)}: $c_1 = 2$.}
Since $\beta$ is standardized and $\widehat{\beta}$ is prime, $d(1; 1,2)$ and $d(1; 2,1)$ are both at least two. Cyclically conjugate $\beta$ so that $\beta = \sigma_1 w_1 \sigma_1 w_2$. Without loss of generality, we may assume that $w_1$ contains at least one $\sigma_3$ letter (if not, then $w_2$ contains all $c_3$-many $\sigma_3$ letters in $\beta$, so conjugate the braid to read as $\sigma_1 w_2 \sigma_1 w_1$, and then swap the labelings of $w_1$ and $w_2$). We assume that $\beta$ is presented in this form for now. We now study the distribution of the $\sigma_3$ letters in $\beta$. Let $\delta := d(1; 1,2)$. Since $\beta$ is standardized, we know that $\delta \geq 2$. We either have:

\begin{enumerate}
\item[(2A)] $d(2; 1, \delta) \geq 1$. 
\item[(2B)] $d(2; 1, \delta) = 0$, as in \Cref{fig:NonGeneric_CmaxCeven_Gammas12_CaseA1A2}. This can happen for two reasons:
	\begin{enumerate}
	\item[(2B.1)] $\beta = \sigma_1 \sigma_3^k w_1' \sigma_1 w_2$, where $k \geq 1$, or 
	\item[(2B.2)] $\beta = \sigma_1 w_1'  \sigma_3^k \sigma_1 w_2$, where $k \geq 1$, and $w_1'$ contains no $\sigma_3$ letters.
	\end{enumerate}
\end{enumerate}

\begin{figure}[h!] \center
\labellist
\pinlabel {$\delta$} at 47 104
\pinlabel {$\delta$} at 200 125
\pinlabel {$\Gamma_1$} at 23 165
\pinlabel {$\Gamma_2$} at 48 165
\pinlabel {$\Gamma_3$} at 75 165
\pinlabel {$\Gamma_1$} at 175 165
\pinlabel {$\Gamma_2$} at 204 165
\pinlabel {$\Gamma_3$} at 230 165
\endlabellist
    %\begin{framed}
        \includegraphics[scale=1.25]{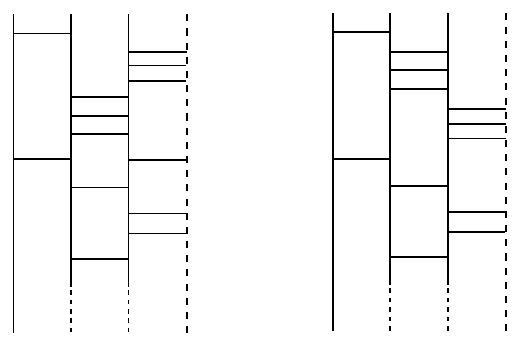}
        \caption{\textbf{Left:} An example of Case (B.1). \textbf{Right:} An example of Case (B.2).}
        \label{fig:NonGeneric_CmaxCeven_Gammas12_CaseA1A2}
    %\end{framed}
\end{figure}
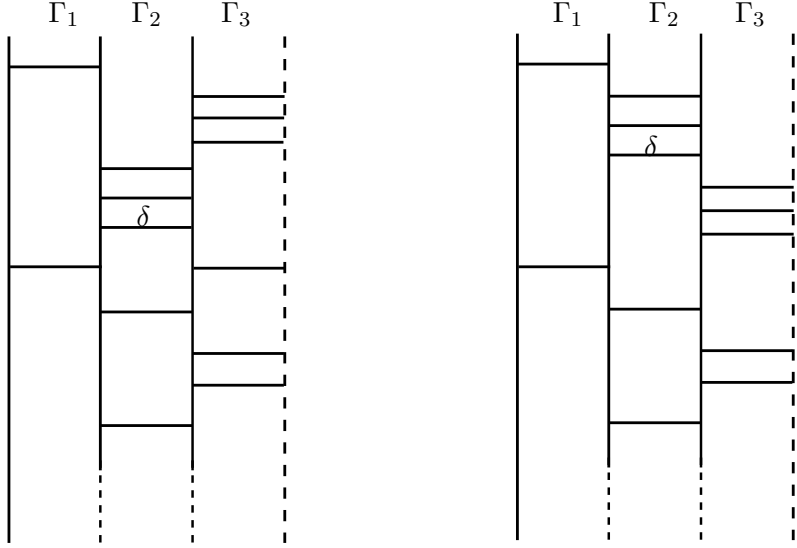

\textbf{\underline{Claim.}} After applying some braid moves, a braid of the form specified in Case (2B.2) can be put into the form specified in Case (2B.1). 

\textit{Proof of Claim.} This is straightforward. We begin with $\beta$ with the form specified in Case (2B.2).
\begin{align*}
\beta &= \sigma_1 w_1'  \underline{\sigma_3^k \sigma_1} w_2 \\
&= \underline{\sigma_1 w_1'} \sigma_1 \sigma_3^k w_2 \\
&\approx \sigma_1 \sigma_3^k w_2 \sigma_1 w_1'
\end{align*}
But this is exactly the form specified in Case (2B.1). \hfill \fbox{End of proof of Claim.}

Therefore, if $\beta$ appears as in Case (2B.2), we can and will modify the presentation and write it as in Case (2B.1). In fact, we will further modify this presentation in Case (2B.1) as follows:
\begin{align}
\beta & \approx \underline{\sigma_1 \sigma_3^k} w_1' \sigma_1 w_2 \nonumber \\
& \approx \underline{\sigma_3^{k-1}} \sigma_1 \sigma_3 w_1' \sigma_1 w_2 \nonumber \\
& \approx \sigma_1 \sigma_3 w_1' \sigma_1 w_2 \sigma_3^{k-1} \nonumber\\
& \approx \sigma_1 \sigma_3 w_1' \sigma_1 w_1'' \label{braid_special_form_3}
\end{align}

In particular, this means the $c_1=2$ case has only two sub-cases to consider:

\begin{enumerate}
\item[(2A)] $\beta \approx \sigma_1 w_1 \sigma_1 w_2$ and $d(2; 1, \delta) \geq 1$, or
\item[(2B)] $d(2; 1, \delta) = 0$, and $\beta \approx \sigma_1 \sigma_3 w_2 \sigma_1 w_1''$.
\end{enumerate}

To summarize, we have four cases total; letting $\delta = d(1; 1,2)$, we have 

\begin{itemize}
\item Case (1A): $c_1=3$, and $d(2; 1, \delta) \geq 2$, and $\beta \approx \sigma_1 w_\star \sigma_1 \beta_2 \beta_1$
\item Case (2A): $c_1=2$, $d(2; 1, \delta) \geq 2$, and $\beta \approx \sigma_1 w_1 \sigma_1 w_2$
\item Case (1B): $c_1=3$, but $d(2; 1, \delta) = 0$, and $\beta \approx \sigma_1 \sigma_3 \beta''$ 
\item Case (2B): $c_1=2$, $d(2; 1, \delta)=0$, and $\beta \approx \sigma_1 \sigma_3 w_2 \sigma_1 w_1''$
\end{itemize}

We can apply a single construction to both Case (1A) and Case (2A), and a different construction to both Case (1B) and Case (2B). We now build $B$ in $\Gamma_1 \cup \Gamma_2$. 

\tcbox[size=fbox, colback=gray!30]{Construction of $B$ in $\Gamma_1 \cup \Gamma_2$ in \textbf{Case (1A)} and \textbf{Case (2A)}.}

Note that not only is $d_L(2; 1, \delta) = 0$, but additionally, $\beta$ is standardized; therefore, there exists some minimal $j$, where $1 \leq j \leq \delta-1$ such that $d(2; j, j+1) \geq 2$. In particular, $d(2; 1, j) = 0$, and $d(2; j, j+1) \geq 2$, as seen in \Cref{fig:NotGeneric_CmaxCeven_Gammas12_CaseB} (we will refer to this figure throughout the proof). 

We now build $B$ in $\Gamma_1 \cup \Gamma_2$:
\begin{itemize}
\item smooth $\alpha_{1,1}$ to the left, 
\item smooth $\alpha_{2,1}, \ldots, \alpha_{2,j}$ to the left, 
\item smooth $\alpha_{2,j+1}, \ldots, \alpha_{2, c_2-1}$ to the right, and 
\item we take no product disks from $\Gamma_3$.
\end{itemize}

Items (1), (2), and (3) from the statement of the lemma are immediately satisfied; we need to check that there is a unique potential sink disk spanning $\Gamma_1 \cup \Gamma_2$. Note that by design, the arc $\varphi(\alpha_{2, j})$ encloses the bands $\mathbbm{b}_{3,1}, \ldots, \mathbbm{b}_{3, d(2; j, j+1)}$ on the left.

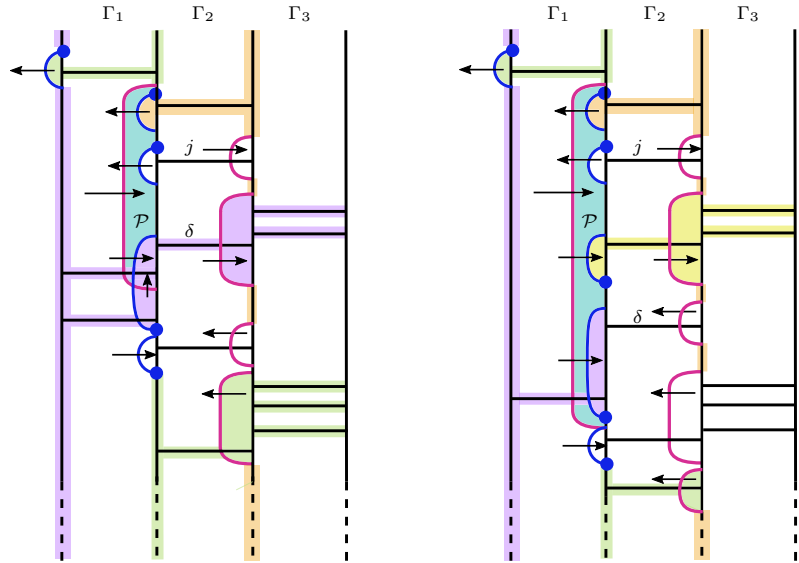
\begin{figure}[h!] \center
\labellist
\tiny
%bottom row
\pinlabel {$\mathcal{P}$} at 44 105
\pinlabel {$\mathcal{P}$} at 174 105
\pinlabel {$\Gamma_1$} at 36 165
\pinlabel {$\Gamma_2$} at 62 165
\pinlabel {$\Gamma_3$} at 90 165
\pinlabel {$\Gamma_1$} at 165 165
\pinlabel {$\Gamma_2$} at 193 165
\pinlabel {$\Gamma_3$} at 220 165
\pinlabel {$\delta$} at 58 102
\pinlabel {$\delta$} at 188 77
\pinlabel {$j$} at 58 126
\pinlabel {$j$} at 188 126
\endlabellist
    %\begin{framed}
        \includegraphics[scale=1.3]{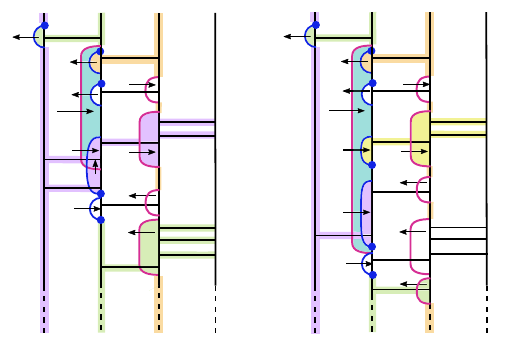}
        \caption{Examples where we building $B$ in $\Gamma_1 \cup \Gamma_2$ in Cases (1A) and (2A). We see that the sector $\mathcal{H} \sim \mathbbm{b}_{2, j+1} \sim \mathbbm{b}_{3,1}$ is the unique potential sink disk in $\Gamma_1 \cup \Gamma_2$. On the \textbf{left}, an example where $j+1 = \delta$, and on the \textbf{right}, an example where $j+1 < \delta$.}
        \label{fig:NotGeneric_CmaxCeven_Gammas12_CaseB}
   % \end{framed}
\end{figure}

\begin{itemize}
\item \Seifert We see that $\mathcal{S}_2 \sim \mathbbm{b}_{1,1}$, and so $\partial \mathcal{S}_2$ contains $\alpha_{1,1}$. This arc points out of the sector, so $\mathcal{S}_2$ is not a sink. 

By design, we see that $S_2$ contains both left and right pointing plumbing arcs. Since we are choosing no product disks from $\Gamma_3$, we immediately deduce that $\mathcal{S}_3$ contains both left and right pointing image arcs, so $\mathcal{S}_3$ is not a sink.

Finally, we know that $\mathcal{S}_1 \sim \mathbbm{b}_{1,2} \sim \mathbbm{b}_{2, \delta}$, as in \Cref{fig:NotGeneric_CmaxCeven_Gammas12_CaseB}. 
\begin{itemize}
\item[$\bullet$] If $j+1 < \delta$, then $\mathbbm{b}_{2, \delta}$ is enclosed, on the right, by a left pointing image arc, as in \Cref{fig:NotGeneric_CmaxCeven_Gammas12_CaseB} (right). In particular, we see that $\mathcal{S}_1$ is not a sink. 
\item[$\bullet$] If $j+1 = \delta$, then $\mathbbm{b}_{2, \delta}$ is enclosed, on the right, by a right pointer, as in \Cref{fig:NotGeneric_CmaxCeven_Gammas12_CaseB} (left). Thus, $\mathcal{S}_1$ is actually a horizontal sector $\mathcal{H}$ for which, at present, we do not know the sink disk status.
\end{itemize}

% Otherwise, $\mathcal{S}_1 \sim \mathbbm{b}_{1,2} \sim \mathbbm{b}_{2, \delta}$, and the latter is enclosed, on the right, by the right pointing image arc $\varphi(\alpha_{2, \delta-1})$, as in \Cref{fig:NotGeneric_CmaxCeven_Gammas12_CaseB} (left). 

% However, by design, we know that the arcs $\varphi(\alpha_{2,1}), \ldots, \varphi(\alpha_{2, j})$ are all right pointers, and the arcs $\varphi(\alpha_{2, j+1}), \ldots, \varphi(\alpha_{2, \delta-1})$ are all left pointers. Therefore, if $\varphi(\alpha_{2, \delta-1})$ is a right pointer, then $j=\delta-1$. In particular, $\mathcal{S}_1 \sim \mathbbm{b}_{2, j+1}$, as well as the bands $\mathbbm{b}_{3, 1}, \mathbbm{b}_{3, 2}, \ldots, \mathbbm{b}_{3, d(2; j, j+1)}$, as in \Cref{fig:NotGeneric_CmaxCeven_Gammas12_CaseB} (left).  
%
%
\item \Polygon Of $S_1, S_2,$ and $S_3$, only $S_2$ contains both plumbing and image arcs, so this is the only Seifert disk spanning $\Gamma_1 \cup \Gamma_2$ with the potential to contain a polygon sector. Moreover, any such polygon sector must contain $\varphi(\alpha_{1,1})$ in its boundary. The number of polygon sectors in $S_2$ is at least one and at most two. There is always a polygon sector $\mathcal{P}$ whose boundary contains the arcs $\varphi(\alpha_{1,1})$, $\alpha_{2,1}, \ldots, \alpha_{2,\delta}$. By design, $\alpha_{2, \delta}$ points out of $\mathcal{P}$, so this sector is not a sink; see \Cref{fig:NotGeneric_CmaxCeven_Gammas12_CaseB}. 

If $d_L(2; \delta, \delta+1) \geq 2$, then there is another sector $\mathcal{P}'$ in $S_2$ which contains $\varphi(\alpha_{1,1})$ in its boundary. Recall that we only took one product disk from $\Gamma_1$. Therefore, $\mathcal{P}'$ also contains $\mathbbm{b}_{1,2}, \mathbbm{b}_{1, 3}, \ldots, \mathbbm{b}_{1, d_L(2; \delta, \delta+1) + 1}$. In particular, $\mathcal{P}' \sim \mathcal{S}_1$, and we already proved the latter is not a sink. We deduce that there are no polygon sector sink disks spanning $\Gamma_1 \cup \Gamma_2$. 
\item \Horizontal Suppose $\mathcal{H}$ is a horizontal sector spanning $\Gamma_1$. If $\mathcal{H} \sim \mathbbm{b}_{1,1}$, then $\mathcal{S} \sim \mathcal{S}_2$, which we already argued is not a sink. Otherwise, $\mathcal{H} \sim \mathbbm{b}_{1, k}$, where $2 \leq k \leq c_1$; this means $\mathcal{H} \sim \mathcal{S}_1 \sim \mathbbm{b}_{2, \delta}$. As described above, the sink disk status of this sector is determined by whether $\delta = j+1$ or $\delta \geq j+2$. In particular, if $\delta = j+1$, $\mathcal{H} \sim \mathbbm{b}_{3,1}$ (and we do not yet know the sink disk status of the sector), but if $\delta \geq j+2$, then $\mathcal{H}$ is not a sink. 
%The band $\mathbbm{b}_{1,1}$ is in the same branch sector as $\mathcal{S}_2$, which we already argued is not a sink. The band $\mathbbm{b}_{1,2}$ is in the same sector as $\mathcal{S}_1$ and $\mathbbm{b}_{2, \delta}$; the status of this sector is determined (as described above) by whether or not $j < \delta-1$ or $j = \delta-1$.  In particular, $\mathcal{H}$ is not a sink in the former, but appears to be a potential sink in the latter.

Suppose $\mathcal{H}$ is a sector spanning $\Gamma_2$. Since $\alpha_{2, 1}, \ldots, \alpha_{2, j}$ are all left pointers, if $\mathcal{H} \sim \mathbbm{b}_{2, t}$, for $1 \leq t \leq j$, then $\partial \mathcal{H}$ contains an outward pointing arc in the boundary, and hence is not a sink. If $\mathcal{H}$ contains $\mathbbm{b}_{2, j+1}$, we cannot yet determine the sink disk status of this sector. Otherwise, $\mathcal{H}$ contains $\mathbbm{b}_{2, j+2}, \ldots, \mathbbm{b}_{2, c_2}$; as these bands are enclosed, on the right, by outward pointing image arcs, such sectors cannot be sinks. See \Cref{fig:NotGeneric_CmaxCeven_Gammas12_CaseB}.
Therefore, there is exactly one horizontal sector which appears to be a sink: it is the sector $\mathcal{H} \sim \mathbbm{b}_{2, j+1} \sim \mathbbm{b}_{3,1}$. 
\end{itemize}

We deduce, as desired, that there is a unique sector for which the sink disk status is unknown. The proof of \Cref{lemma:TemplateLinking} ensures that there is no linking between arcs which both lie in $S_2$, and one can inspect  \Cref{fig:NotGeneric_CmaxCeven_Gammas12_CaseB} to see that $\alpha_{1,1}$ does not link with any arc in $S_2$. Therefore, there is no linking between arcs in $\Gamma_1 \cup \Gamma_2$. 

\tcbox[size=fbox, colback=gray!30]{Construction of $B$ in $\Gamma_1 \cup \Gamma_2$ for Case (1B) and Case (2B).}
To build $B$ in $\Gamma_1 \cup \Gamma_2$, we: smooth $\alpha_{1,1}$ to the left, smooth $\alpha_{2,1}, \ldots, \alpha_{2,c_2-1}$ to the right, and we take no product disks from $\Gamma_3$. We refer to %\Cref{fig:CmaxCeven_Gammas12_Case1} and
\Cref{fig:NotGeneric_CmaxCeven_Gammas12_CaseA} throughout.
%\begin{itemize}
%\item smooth $\alpha_{1,1}$ to the left
%\item smooth $\alpha_{2,1}, \ldots, \alpha_{2,c_2-1}$ to the right
%\item we take no product disks from $\Gamma_3$
%\end{itemize}

It is clear that items (1), (2), and (3) from the statement of the lemma are satisfied; so, we study the sink disks and linked arcs in $\Gamma_1 \cup \Gamma_2$. We claim that there is only one potential sink disk sector spanning $\Gamma_1 \cup \Gamma_2$; that sector contains $\mathbbm{b}_{2,1} \sim \mathcal{S}_3 \sim \mathbbm{b}_{3,1}$.

\begin{figure}[h!] \center
\labellist
\tiny
%bottom row
\pinlabel {$\mathcal{P}$} at 42 102
\pinlabel {$\mathcal{P}$} at 180 102
\pinlabel {$\mathcal{P}$} at 342 102
\pinlabel {$\Gamma_1$} at 33 165
\pinlabel {$\Gamma_2$} at 62 165
\pinlabel {$\Gamma_3$} at 90 165
\pinlabel {$\Gamma_1$} at 174 165
\pinlabel {$\Gamma_2$} at 200 165
\pinlabel {$\Gamma_3$} at 230 165
\pinlabel {$\Gamma_1$} at 335 165
\pinlabel {$\Gamma_2$} at 362 165
\pinlabel {$\Gamma_3$} at 390 165
\pinlabel {$\delta$} at 62 102
\pinlabel {$\delta$} at 200 102
\pinlabel {$\delta$} at 360 102
\endlabellist
   % \begin{framed}
        \includegraphics[scale=1.1]{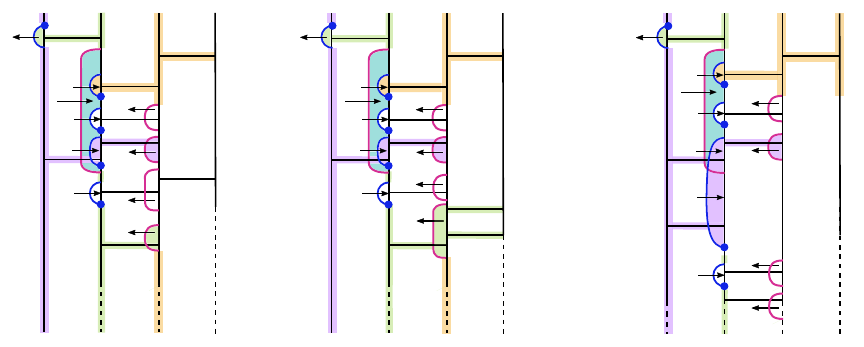}
        \caption{Some examples where we build $B$ in $\Gamma_1 \cup \Gamma_2$ in Cases (1B) and (2B). We see that $\mathcal{S}_2 \sim \mathbbm{b}_{1,1}$ and $\mathcal{S}_1 \sim \mathbbm{b}_{1,2}$, and that neither are sink disks. Additionally, we see that $\mathcal{S}_3$ appears to be a potential sink.}
        \label{fig:NotGeneric_CmaxCeven_Gammas12_CaseA}
   %\end{framed}
\end{figure}
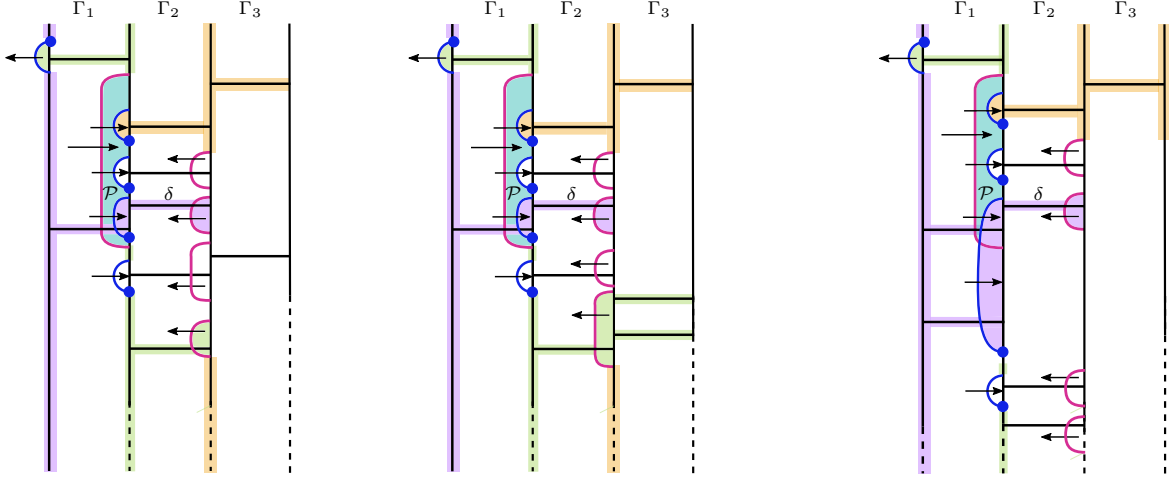

\begin{itemize}
%\item \Cref{lemma:ProductDisksAren'tSinkDisks} shows the (isotoped) product disks are never sinks. 
%
%
\item \Seifert By design, $\mathcal{S}_2 \sim \mathbbm{b}_{1,1}$. This band is enclosed, on the left, by the left pointer $\alpha_{1,1}$, so $\mathcal{S}_2$ is not a sink. 

We see that $\mathcal{S}_1 \sim \mathbbm{b}_{1, t}$, where $2 \leq t \leq c_1$. In particular, $\mathbbm{b}_{1,2}$ is enclosed, on the right, by $\alpha_{2,\delta}$, so $\mathcal{S}_1 \sim \mathbbm{b}_{2,\delta}$. As we have taken no product disks from $\Gamma_3$, we know that $\partial \mathcal{S}_1$ contains $\varphi(\alpha_{2, \delta-1})$, which is a left pointer. Therefore, $\mathcal{S}_1$ is not a sink.

As $S_3$ contains only left pointing image arcs, and $\mathcal{S}_3 \sim \mathbbm{b}_{2,1}$, we see that $\partial \mathcal{S}_3$ contains only inward pointing arcs. Therefore, at present, we cannot determine the sink disk status of $\mathcal{S}_3$.

\item \Polygon Of $S_1, S_2,$ and $S_3$, only $S_2$ contains both plumbing and image arcs, so this is the only Seifert disk spanning $\Gamma_1 \cup \Gamma_2$ with the potential to contain a polygon sector. Indeed, there is a polygon sector $\mathcal{P}$ whose boundary contains the arcs $\varphi(\alpha_{1,1})$, $\alpha_{2,1}, \ldots, \alpha_{2,\delta}$. By design, all the plumbing arcs $\alpha_{2,1}, \ldots, \alpha_{2, \delta}$ point out of $\mathcal{P}$, so this sector is not a sink. 
\item \Horizontal Suppose $\mathcal{H}$ is a horizontal sector spanning $\Gamma_1$. As noted in our Seifert disk sector analysis, every band in $\Gamma_1$ is in the same sector as either $\mathcal{S}_1$ or $\mathcal{S}_2$. Since neither are sinks, no horizontal sector spanning $\Gamma_1$ is a sink. 

Suppose $\mathcal{H}$ is a sector spanning $\Gamma_2$. We see that $\mathbbm{b}_{2,1} \sim \mathcal{S}_3$; we already described why we cannot know the sink disk status of this sector just yet. However, suppose $\mathcal{H}$ is a horizontal sector spanning $\Gamma_2$ and $\mathcal{H} \sim \mathbbm{b}_{2,j}$, for some $2 \leq j \leq c_2$. Such a band $\mathbbm{b}_{2,j}$ is enclosed, on the right, by $\varphi(\alpha_{2,j-1})$, and this arc points out of $\mathcal{H}$, so such sectors are not sinks. 
\end{itemize}

This confirms item (4) in the statement of the lemma. Finally, we check that there are no linked pairs between arcs in $S_1$ and $S_2$. The proof of \Cref{lemma:TemplateLinking} confirms that there can be no pairwise linking between arcs in $S_2$, and by design, $\alpha_{1,1}$ does not link with any arc in $S_2$. This completes the proof of the lemma for Case (1B) and Case (2B). 
\end{proof}

We summarize the distinguished bands and potential sink disks from the constructions of \Cref{lemma:Sparse_Cmax=Ceven_Gammas12}.

\begin{table}[H]
\begin{tabular}{ | c | c | c |}
\hline
\textsc{Case} &  \textsc{Distinguished Band} & \textsc{Potential Sink Disk Sector contains} \\ \hline \hline
%\cellcolor{grey!20}{\textsc{Case}} &  \cellcolor{grey!20}{\textsc{Distinguished Band}} &  \cellcolor{grey!20}{\textsc{Potential Sink Disk Sector contains}} \\ \hline \hline
\hyperref[fig:NotGeneric_CmaxCeven_Gammas12_CaseB]{Cases (1A), (2A)} & $\mathbbm{b}_{2, j+1}$ & $\mathbbm{b}_{2,j+1} \sim \mathbbm{b}_{3, 1}$ \\
%Case 2A &  $\mathbbm{b}_{2,1}$ & $\mathbbm{b}_{2,1} \sim \mathcal{S}_3$ \\
 %$c_1=2$, subcase A.2 & $\mathbbm{b}_{3, \delta-1} &  no sinks spanning $\Gamma_1 \cup \Gamma_2$ \\
\hyperref[fig:NotGeneric_CmaxCeven_Gammas12_CaseA]{Cases (1B), (2B)} & $\mathbbm{b}_{2, 1}$ & $\mathbbm{b}_{2,1} \sim \mathcal{S}_3 \sim \mathbbm{b}_{3,1}$ \\ \hline
\hline
\end{tabular} 
        \caption{Bookkeeping for \Cref{lemma:Sparse_Cmax=Ceven_Gammas12}, which builds $B$ in $\Gamma_1 \cup \Gamma_2$ when $\Cmax=\Ceven$.}
        \label{table:Bookkeeping_Cmax=Ceven_Gammas12}
\end{table}

\begin{rmk} \label{NoClumpingOfArcs}
\textup{For the remainder of this paper, assume that if we build $B$ in $\Gamma_i$, for $i \geq 4$, then we either take all the product disks from $\Gamma_i$, or none of them. Said differently, assume that if $S_i$ contains both plumbing and image arcs, then either:
    \begin{enumerate}
    \item $\Cmax = \Codd$ and $i=2, 3$, or 
    \item $\Cmax = \Ceven$ and $i = 2$.
    \end{enumerate}
}
\end{rmk}

\Cref{NoClumpingOfArcs} will be vital as we construct $B$ to non-trivially spans all the columns of the braid.

%%%%%%%%%%%%%%%%%%%%%%%%%%%%%%%%%%%%%%%%
%%%%%%%%%%%%%%%%%%%%%%%%%%%%%%%%%%%%%%%%
%%%%%%%%%% BUILDING B in INTERIOR COLUMNS
%%%%%%%%%%%%%%%%%%%%%%%%%%%%%%%%%%%%%%%%
%%%%%%%%%%%%%%%%%%%%%%%%%%%%%%%%%%%%%%%%
\subsection{Building the branched surface in interior columns}  \label{section:BuildInMiddle}

To build our laminar branched surface so that it spans many columns of the braid, we will need to fluidly move throughout the braid to decide where to apply our templates. This requires the following definitions and lemmas.

\begin{defn}\label{defn:calibration}
Let $\beta$ be a braid word, and let $\beta'$ be some cyclic conjugation thereof. 
Let $F$ and $F'$ be the Bennequin surfaces directly constructed from the braid words $\beta$ and $\beta'$, respectively. $($As remarked in \Cref{section:monodromy}, $F$ and $F'$ are isotopic.$)$ 
Suppose $\sigma_{s,t}$ $($resp. $\mathbbm{b}_{s, t})$ is a distinguished letter $($resp. band$)$ in $\beta$ $($resp. $F)$ which, after cyclically conjugating to $\beta'$, becomes $\sigma_{s,t'}$ $($resp. $\mathbbm{b}_{s,t'})$ in $F'$. 
We say $\beta'$ is a \textbf{calibration of $\beta$ with respect to $\alpha_{s,t}$ $($or $\mathbbm{b}_{s,t})$} if in $\beta'$, the plumbing arc $\alpha_{s+1,1}$ encloses $\mathbbm{b}_{s,t'}$ on the right. 
\end{defn}

For a simple example of a calibration, see \Cref{fig:Calibration}.

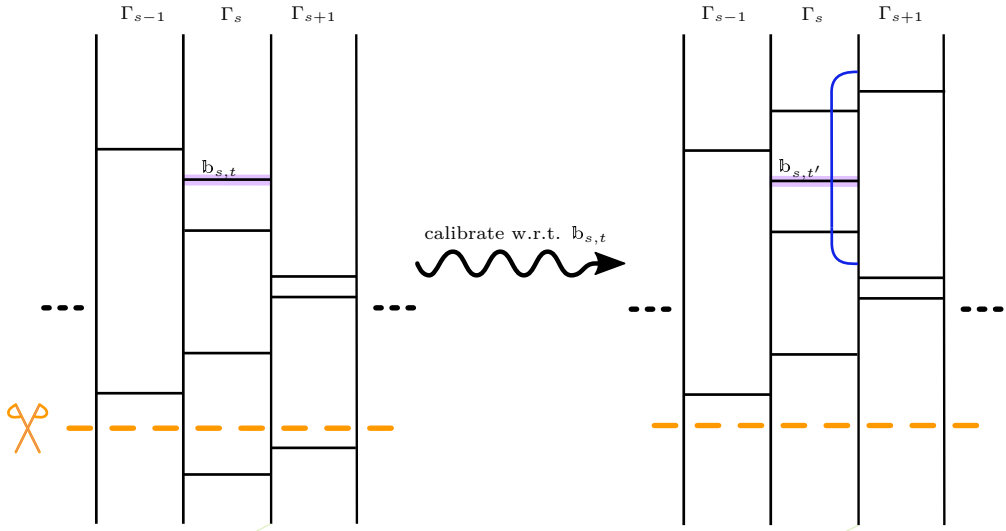
\begin{figure}[h!] \center
\labellist
\tiny
\pinlabel {$\Gamma_{s-1}$} at 52 175
\pinlabel {$\Gamma_{s}$} at 80 175
\pinlabel {$\Gamma_{s+1}$} at 106 175
\pinlabel {$\Gamma_{s-1}$} at 236 175
\pinlabel {$\Gamma_s$} at 264 175
\pinlabel {$\Gamma_{s+1}$} at 292 175
\pinlabel {$\mathbbm{b}_{s,t}$} at 76 126
\pinlabel {$\mathbbm{b}_{s,t'}$} at 260 126
\pinlabel {calibrate w.r.t. $\mathbbm{b}_{s,t}$} at 170 105
\endlabellist
    %\begin{framed}
        \includegraphics[scale=1.2]{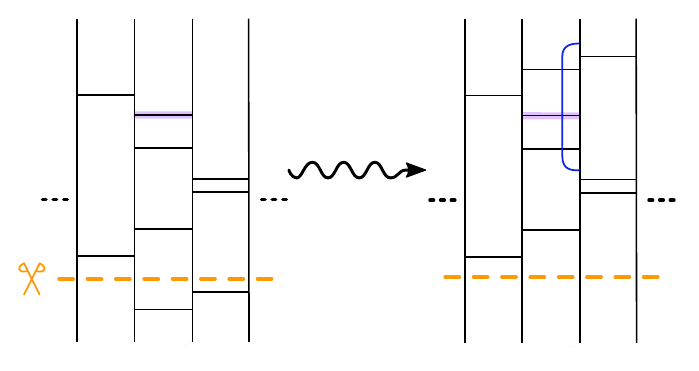}
        \caption{\textbf{Left:} We see $\beta$ and a distinguished band $\mathbbm{b}_{s,t}$ (here, the band is $\mathbbm{b}_{s,1}$).  The dotted orange line shows the axis along which we need to cyclically conjugate the braid. \textbf{Right:} The result after we (canonically) calibrate $\beta$ with respect to $\mathbbm{b}_{s,1}$, producing $\beta'$.  Note that in $\beta'$, the distinguished band has become $\mathbbm{b}_{s,2}$. }
        \label{fig:Calibration}
    %\end{framed}
\end{figure}

%\begin{defn} 
%Let $\beta$ be a braid word, and let $\beta'$ be some cyclic conjugation of $\beta$. Let $F$ and $F'$ be the Bennequin surfaces directly constructed from $\beta$ and $\beta'$. (Note that, as remarked in \Cref{section:monodromy}, $F$ and $F'$ are isotopic.) Suppose $\mathbbm{b}_{s, t}$ is a distinguished band in $F$ which, after the cyclic conjugation, becomes $\mathbbm{b}_{s,t'}$ in $F'$. We say that the band $\mathbbm{b}_{s+1, 1}$ is a \textbf{conduit} for $\mathbbm{b}_{s, t'}$ if, once $\beta$ is cyclically conjugated into $\beta'$, the plumbing arc $\alpha_{s+1, 1}$ encloses the band $\mathbbm{b}_{s, k'}$ on the right.  
%\end{defn}

\begin{lemma} \label{lemma:CanCalibrate}
Fix a braid word $\beta$. For any distinguished band $\mathbbm{b}_{s,t}$ in $\beta$, where $1\leq s \leq n-2$, there exists a calibration $\beta'$ of $\beta$ with respect to $\mathbbm{b}_{s,t}$. 
\end{lemma}

\begin{proof}
Let $\beta$ be a braid word and let $\sigma_{s,t}$ be a distinguished letter therein. Recall that $\mathbbm{b}_{s,t}$ is the band corresponding to $\alpha_{s,t}$ in $\beta$. In particular, we can write our braid as $\beta = \gamma_1 \sigma_s \gamma_2$, where the distinguished $\sigma_s$ is the distinguished letter $\sigma_{s,t}$. We pivot $\beta$ with respect to $\sigma_s$, producing the braid word $\widetilde{\beta}$, where $\widetilde{\beta} = \sigma_s \gamma_2 \gamma_1$. We next identify the last $\sigma_{s+1}$ letter in this presentation of the braid; that is, we write $\widetilde{\beta} = \sigma_s \ \widetilde{\gamma_2} \ \sigma_{s+1} \ \widetilde{\gamma_1}$, where there are no $\sigma_{s+1}$ letters in $\widetilde{\gamma_1}$. Pivoting $\beta$ with respect to $\sigma_{s+1}$, we produce $\beta' = \sigma_{s+1} \ \widetilde{\gamma_1} \ \sigma_s \ \widetilde{\gamma_2}$. (Note that $\widetilde{\gamma_1}$ could contain other $\sigma_s$ letters, and that is allowed.)

We claim that this is a calibration of $\beta$ with respect to $\mathbbm{b}_{s,t}$: by design, the presentation $\beta'$ of the braid is such that the first occurrence of a $\sigma_{s+1}$ letter appears before the distinguished $\sigma_s$ letter. Since $\beta$ is prime, we know that $c_{s+1} \geq 2$. In particular, a second occurrence of a $\sigma_{s+1}$ letter exists in $\beta'$, it happens after our distinguished $\sigma_s$ letter, and it somewhere in $\widetilde{\gamma_2}$. The plumbing arc $\alpha_{s+1,1}$ lies in $S_{s+1}$, has one endpoint above $\mathbbm{b}_{s+1,1}$, and the other above $\mathbbm{b}_{s+1,2}$. Moreover, $\alpha_{s+1,1}$ will enclose, on the left, all the bands $\mathbbm{b}_{s,1}, \ldots, \mathbbm{b}_{s, d_L(s+1; 1,2)}$ (as in \Cref{fig:Calibration} (right)). After calibration, our distinguished band is $\mathbbm{b}_{s,t'}$, where $1 \leq t' \leq d_L(s+1; 1,2)$. We conclude that $\alpha_{s+1,1}$ encloses our distinguished band on the right.
\end{proof}

We remark: there are many cyclic conjugations of $\beta$ which are \textit{also} calibrations of $\beta$ with respect to $\mathbbm{b}_{s,t}$. %They are produced by starting with the presentation $\beta'$ from \Cref{lemma:CanCalibrate}, and then cyclically conjugating letters from $\widetilde{\gamma_2}$ to precede the distinguished $\sigma_{s+1}$ letter until one hits another $\sigma_{s+1}$ letter. 
We do not want to consider these other calibrations, so we define the following.

\begin{defn}
The \textbf{canonical calibration of $\beta$ with respect to $\mathbbm{b}_{s,t}$} is the calibration of $\beta$ produced via the algorithm described in the proof of \Cref{lemma:CanCalibrate}. 
\end{defn}

\begin{lemma} \label{lemma:Continue}
%Suppose we already built $B$ in $\Gamma_i$, where $i \geq 3$, by applying our template to $\Gamma_i$. There is a calibration with respect to a distinguished band in $\Gamma_{i+1}$, so that when we apply our template to $\Gamma_{i+2}$, the only potential sink disk spanning $\Gamma_i \cup \Gamma_{i+1}$ is $\mathcal{S}_{i+2}$. 
Suppose we built a branched surface containing product disks from some non-trivial subset of the columns in $\Gamma_1 \cup \ldots \cup \Gamma_i$, where $4 \leq i \leq n-4$. Additionally, assume that we applied our template to $\Gamma_i$ to build $B$ in this column -- in particular, assume that our branched surface meets $S_i$ in only plumbing arcs and no image arcs.
After performing some pivots and canonical calibrations, we can build our branched surface in $\Gamma_{i+1} \cup \Gamma_{i+2}$ such that: $(1)$ there are no product disks chosen from $\Gamma_{i+1}$, $(2)$ we applied our template to $\Gamma_{i+2}$ $($so there are $c_{i+2}-1$ co-oriented product disks chosen from $\Gamma_{i+2})$, $(3)$ the unique potential sink disk spanning $\Gamma_i \cup \Gamma_{i+1}$ is $\mathcal{S}_{i+2}$, and $(4)$ there is no pairwise linking between arcs in $\Gamma_i \cup \Gamma_{i+1} \cup \Gamma_{i+2}$. 
\end{lemma}

\begin{proof}
Suppose we built $B$ in $\Gamma_1 \cup \ldots \cup \Gamma_i$ such that we already applied our template to $\Gamma_i$. (We refer the reader to \Cref{defn:Templates} to recall the design of our templates.) Therefore, there is a unique band $\mathbbm{b}_{i, k}$ such that $\mathbbm{b}_{i, k-1}$ is not enclosed, on the left, by a plumbing band. We pivot the braid about $\mathbbm{b}_{i, k}$, so that $\beta \approx \sigma_i \omega$; in this new presentation, $\sigma_{i,1}$ is enclosed (on the left) by a left pointing plumbing arc, and $\sigma_{i, c_i}$ is not enclosed (on the left) by any plumbing arcs. We refer to this cyclic presentation of the braid as $P_1$, and \Cref{fig:CalibrateAndContinue} (left) shows some examples of $\Gamma_i \cup \Gamma_{i+1}$ in $P_1$.

There are two cases to consider:

\begin{enumerate}
\item[$\boldsymbol{\circ}$] \textbf{Case (1)}: The heterogenous case. Here, $d(i; 1, c_i - 1) > 0$, as in \Cref{fig:CalibrateAndContinue} (top row), or 
\item[$\boldsymbol{\circ}$] \textbf{Case (2)}: The homogeneous case. Here, $d(i; 1, c_i-1) = 0$, as in \Cref{fig:CalibrateAndContinue} (bottom row).
\end{enumerate}

We begin by building our branched surface in $\Gamma_{i+2}$, and then go on to show that, in both the heterogeneous and homogeneous cases, we created as most one sink disk. 

We quickly establish some notation: let $\mathbbm{b}_{i,p}$ denote the first $\sigma_i$ band in $\Gamma_i$ which is followed by $\sigma_{i+1}$ bands; that is, $d(i; 1, p) = 0$, and $d(i; p, p+1) \geq 1$. Note that when $\Gamma_{i}$ is homogeneous, $p = c_{i-1}$. To build $B$, we canonically calibrate $\beta$ with respect to $\mathbbm{b}_{i+1,1}$, thereby presenting our braid as $\sigma_{i+2} \omega_1 \sigma_{i+1} \omega_2$, where the $\sigma_{i+1}$ letter written explicitly used to be $\sigma_{i+1,1}$ (in our old presentation of the braid as $P_1$). With this presentation of $\beta$, we apply our template to $\Gamma_{i+2}$. The result is seen in \Cref{fig:CalibrateAndContinue}, where we show an incomplete picture of what the branched surface looks like in $\Gamma_i \cup \Gamma_{i+1} \cup \Gamma_{i+2}$ with respect to the $P_1$ presentation of the braid. %Our labelings of bands is also with respect to the $P_1$ presentation of the braid seen in \Cref{fig:Continue}.

\begin{figure}[h!] \center
\labellist
\tiny
%% top row
\pinlabel {$\mathcal{H}$} at 41 320
\pinlabel {$\mathcal{H}$} at 244 320
\pinlabel {$\Gamma_i$} at 32 410
\pinlabel {$\Gamma_{i+1}$} at 60 410
\pinlabel {$\Gamma_{i+2}$} at 88 410
\pinlabel {$\Gamma_i$} at 234 410
\pinlabel {$\Gamma_{i+1}$} at 262 410
\pinlabel {$\Gamma_{i+2}$} at 290 410
%bottom row
\pinlabel {$\mathcal{H}$} at 41 52
\pinlabel {$\mathcal{H}$} at 244 52
\pinlabel {$\Gamma_i$} at 32 186
\pinlabel {$\Gamma_{i+1}$} at 60 186
\pinlabel {$\Gamma_{i+2}$} at 88 186
\pinlabel {$\Gamma_i$} at 234 186
\pinlabel {$\Gamma_{i+1}$} at 262 186
\pinlabel {$\Gamma_{i+2}$} at 290 186
% arrow labels
\pinlabel {calibrate w.r.t. $\mathbbm{b}_{i+1,1}$} at 158 322
\pinlabel {calibrate w.r.t. $\mathbbm{b}_{i+1,1}$} at 158 100
\pinlabel {$\mathbbm{b}_{i+1,1}$} at 60 352
\pinlabel {$\mathbbm{b}_{i+1,2}$} at 258 352
\pinlabel {$\mathbbm{b}_{i+1,1}$} at 60 90
\pinlabel {$\mathbbm{b}_{i+1,2}$} at 258 90
\pinlabel {$p$} at 28 370
%\pinlabel {$p$} at 230 370
\pinlabel {$p$} at 25 98
%\pinlabel {$p$} at 228 98
%
\endlabellist
        \includegraphics[scale=1.2]{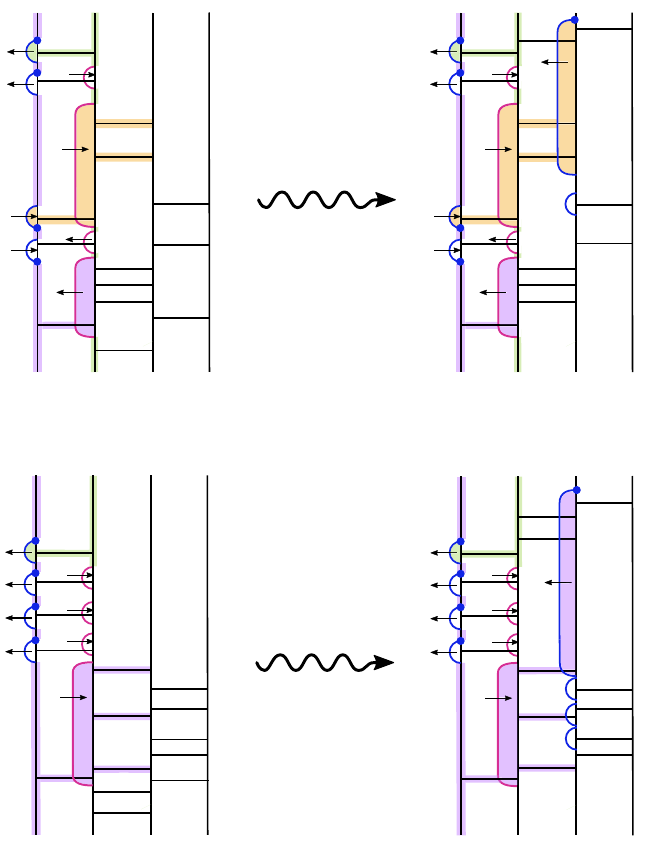}
        \caption{\textbf{Top Row:} We see an example of the heterogeneous case, how to canonically calibrate $\beta$ (on the left) with respect to $\mathbbm{b}_{i+1,1}$ to product $\beta'$ (on the right). The sector $\mathcal{P}$ is not a sink, because $\alpha_{i+2,1}$ points out of the region. \textbf{Bottom Row:} we see an analogous example when $\Gamma_i$ is homogeneous. Note: in both rows, we do not know whether $\alpha_{i+2,t}$, where $t \neq 1$, is a left or right pointer, so we have suppressed that information.}
        %        \caption{Applying our template to $\Gamma_i \cup \Gamma_{i+1}$ yields two cases to analyze. This figure shows the presentation $P_1$ of the braid, and what we see in $\Gamma_i \cup \Gamma_{i+1}$. \textbf{Case (1)}, i.e. the \textit{heterogeneous} case, is on the left, and \textbf{Case (2)}, i.e. the \textit{homogeneous} case, is on the right. It is straightforward to see that each has a horizontal sector that appears like a sink when restricting our attention to $\Gamma_i \cup \Gamma_{i+1}$. In Case (1), it is the sector containing $\mathbbm{b}_{i, p+1}$, and in Case (2), it is the sector containing $\mathbbm{b}_{i, c_i}$. \textcolor{red}{Add in what you see in $\Gamma_{i+2}$ when MAKING THE FIGURES.} Make sure it is also clear that in case 2, $p=c_{i-1}$.}
        \label{fig:CalibrateAndContinue}
\end{figure}
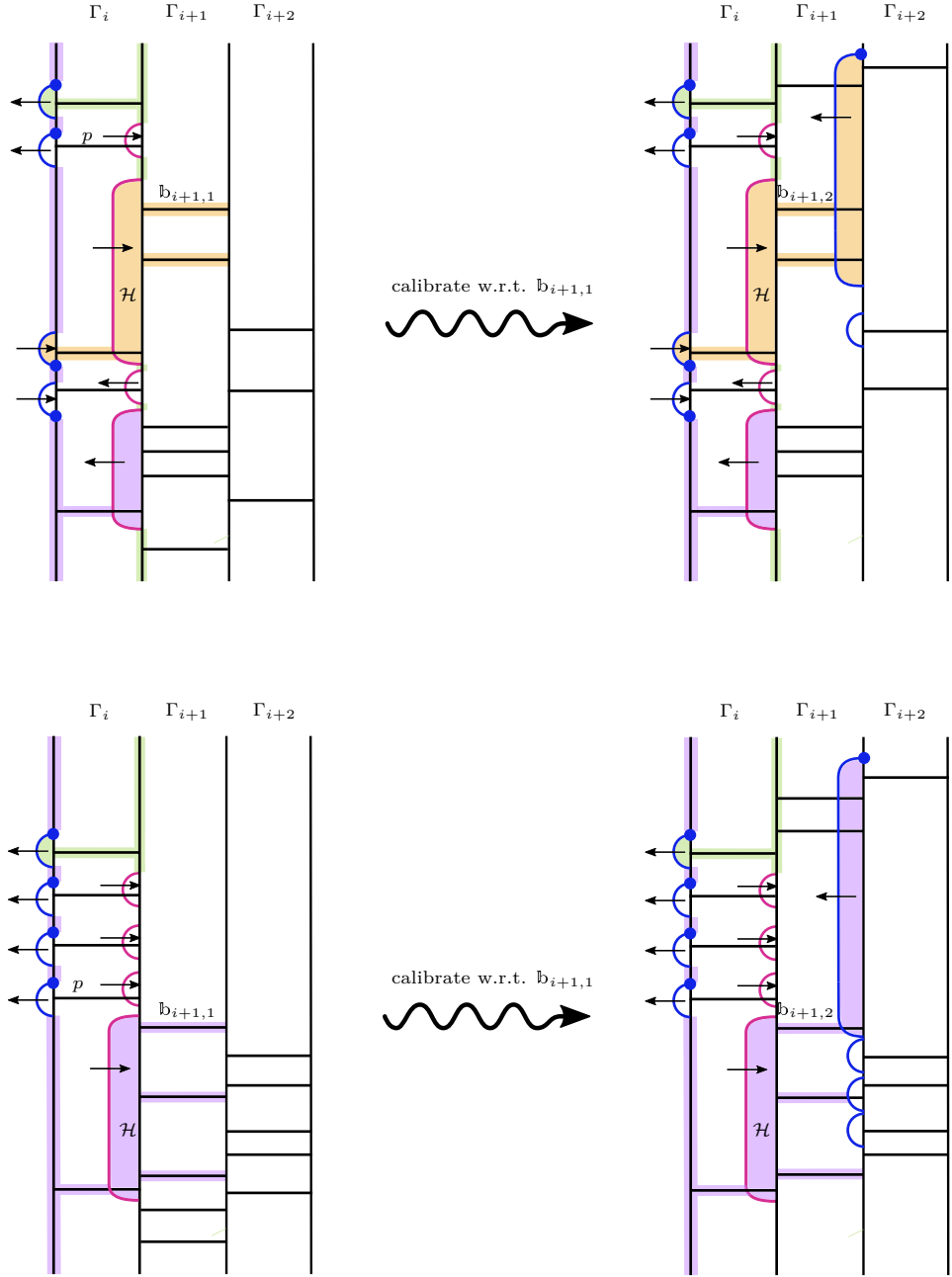

%We now do our sink disk analysis, beginning with the heterogeneous case. 
%Suppose we are in Case (1), and $d(i; 1, c_i - 1) > 0$. We quickly set some notation: let $\mathbbm{b}_{i,p}$ denote the first $\sigma_i$ band which is followed by $\sigma_{i+1}$ bands; that is, $d(i; 1, p) = 0$, and $d(i; p, p+1) \geq 1$. We canonically calibrate $\beta$ with respect to $\mathbbm{b}_{i+1,1}$, thereby presenting our braid as $\sigma_{i+2} \omega_1 \sigma_i \omega_2$, where the $\sigma_i$ letter written explicitly used to be $\sigma_{i+1,1}$. With this presentation of $\beta$, we apply our template to $\Gamma_{i+2}$. In \Cref{fig:Continue}, we show an incomplete picture of what the branched surface looks like in $\Gamma_i \cup \Gamma_{i+1} \cup \Gamma_{i+2}$; for convenience, we do so with respect to the $P_1$ presentation of the braid. 

\tcbox[size=fbox, colback=gray!30]{\textbf{Case (1):} Sink disk analysis when $\Gamma_i$ is heterogeneous.}

We now verify that there are no sink disk sectors spanning $\Gamma_i \cup \Gamma_{i+1}$ in the heterogeneous case, and refer to \Cref{fig:CalibrateAndContinue} (top row) throughout. Note that unless stated otherwise, our labelling of bands is with respect to the $P_1$ presentation of $\beta$. 

\begin{itemize}

\item \Seifert 
	\begin{itemize}
	\item[$\boldsymbol{\circ}$] By assumption, $S_i$ contains only plumbing arcs, and $\Gamma_i$ is heterogeneous. Therefore, $\partial \mathcal{S}_i$ contains $\alpha_{i, p+1}$, which points out of the sector; thus, $\mathcal{S}_i$ is not a sink. 
	\item[$\boldsymbol{\circ}$] Since we take no disks from $\Gamma_{i+1}$, $\varphi(\alpha_{i, 1})$ is contained in $\partial \mathcal{S}_{i+1}$, pointing out of the sector. Thus, $\mathcal{S}_{i+1}$ is not a sink. 
	\item[$\boldsymbol{\circ}$] Though we are assuming that we have applied our template to $\Gamma_{i+2}$, at present, we cannot argue that $\mathcal{S}_{i+2}$ is not a sink: we do not know whether $\Gamma_{i+2}$ is heterogeneous or homogeneous. Therefore, this is a potential sink disk sector. 
	\end{itemize}
\item \Polygon By definition, a polygon sector must lie in a Seifert disk and also contain both plumbing and image arcs in its boundary. By assumption, $S_{i}$ contains only plumbing arcs, so $S_i$ cannot contain any polygon sectors. We have taken no product disks from $\Gamma_{i+1}$, and have applied our template to $\Gamma_{i+2}$. Therefore, $S_{i+1}$ contains only image arcs, and $\Gamma_{i+2}$ contains only plumbing arcs. Therefore, none of $S_i, S_{i+1},$ and $S_{i+2}$ contain both plumbing and image arcs, and so there are no polygon sectors in $\Gamma_i \cup \Gamma_{i+1}$. 
\item \Horizontal Suppose $\mathcal{H}$ is a horizontal sector spanning $\Gamma_i$. There are three cases to consider:

	\begin{enumerate}
	\item \textit{$\mathcal{H}$ contains $\mathbbm{b}_{i, j}$, where $1 \leq j \leq p$.} Here, the bands $\mathbbm{b}_{i,j}$ are enclosed, on the left, by an outward pointing plumbing arc; so these sectors are not sinks; see \Cref{fig:CalibrateAndContinue} (top row). 
	\item \textit{$\mathcal{H}$ contains $\mathbbm{b}_{i,p+1}$.} This band is enclosed by inward pointing arcs on both the right and left, as in \Cref{fig:CalibrateAndContinue} (top row, left). When we canonically calibrated our braid with respect to $\mathbbm{b}_{i+1, 1}$, and then applied our template to $\Gamma_{i+2}$, we ensured that the sector $\mathcal{H}$ meets $S_{i+2}$ in a left pointing plumbing arc (in the new presentation of $\beta$, this arc is $\alpha_{i+2,1})$. 
We deduce that $\mathcal{H}$ is not a sink.
	\item \textit{$\mathcal{H}$ contains $\mathbbm{b}_{i, j}$, where $p+2 \leq j \leq c_i$.} By design, the band $\mathbbm{b}_{i,j}$ is enclosed, on the right, by an outward pointing image arc. Therefore, $\mathcal{H}$ is not a sink.
	\end{enumerate}

We now analyze the horizontal sectors spanning $\Gamma_{i+1}$. There are three cases to consider.
\begin{itemize}
\item[$\bullet$] The bands $\mathbbm{b}_{i+1, 1}, \ldots, \mathbbm{b}_{i+1, d(i; p, p+1)}$ are all part of the same branch sector $\mathcal{H}$; this sector also contains $\mathbbm{b}_{i, p+1}$. We already argued (in our analysis of the horizontal sectors spanning $\Gamma_i$) that the sector containing this band is not a sink. 
\item[$\bullet$] The bands $\mathbbm{b}_{i+1, j}$, where $j \geq d(i; p, p+1) +1$ are either enclosed (on the left) by a left pointing image arc,  or they are not enclosed (on the left) by any arc. When the former happens, the arc points out of the horizontal sector, so it is not a sink; when the latter happens, the horizontal sector also contains $\mathcal{S}_{i+1}$, which we already proved is not a sink.
% \item \textit{$\mathcal{H}$ is enclosed, on the left, by $\alpha_{i,1}$}. In this case, this is the sector $\mathcal{H}^*$ discussed above. At present, we cannot argue that it is not a sink.
% %
% \item \textit{$\mathcal{H}$ is enclosed, on the left, by $\varphi(\alpha_{i,j})$, where $2 \leq j \leq c_i-1$}. These sectors are enclosed, on the left, by an outward pointing image arc, so they are not sinks. 
% %
% \item \textit{$\mathcal{H}$ is not enclosed, on the left, by any arc.} In this case, $\mathcal{H}$ is in the same branch sector as $\mathcal{S}_2$, which we already argued is not a sink.
%
\end{itemize}
%We deduce that no horizontal sectors spanning $\mathcal{S}_{i+1}$ is a sink. 
%
\end{itemize}

We deduce that in \textbf{Case (1)}, $\mathcal{S}_{i+2}$ is the only potential sink disk spanning $\Gamma_i \cup \Gamma_{i+1}$.

\tcbox[size=fbox, colback=gray!30]{\textbf{Case (2):} Sink disk analysis when $\Gamma_i$ is homogeneous.}

We now do our sink disk analysis in Case (2), where $d(i; 1, c_i-1)=0$; a typical example appears in \Cref{fig:CalibrateAndContinue} (bottom row). Since we are assuming that $\beta$ is not a connected sum, we must have that $d(i; c_i - 1, c_i) := \delta_1 >0$ and $d(i; c_i, 1) =: \delta_2 >0$. Notice that in this case, the Seifert disk $S_i$ does not contain any right pointing plumbing arcs, which is why this analysis differs from Case (1). 

\begin{itemize}
\item \Seifert In this case, since $\mathcal{S}_i$ contains only plumbing arcs, $S_i \sim \mathbbm{b}_{i, c_i} \sim \mathbbm{b}_{i+1,1}$. By design, we canonically calibrated our braid with respect to $\mathbbm{b}_{i+1,1}$, and then applied our template to $\Gamma_{i+2}$. So, the band known as $\mathbbm{b}_{i+1,1}$ in the $P_1$ presentation of $\beta$ is subsequently enclosed, on the right, by a left pointing plumbing arc; see \Cref{fig:CalibrateAndContinue} (bottom row, right). Therefore, $\mathcal{S}_i$ is not a sink. $S_{i+1}$ contains only outward pointing image arcs, so it is not a sink. We cannot yet conclude that $S_{i+2}$ is not a sink, as we do not yet know whether $S_{i+2}$ is heterogeneous or homogeneous.
\item \Polygon By the same argument from \textbf{Case (1)}, there are no polygon sectors spanning $\Gamma_{i} \cup \Gamma_{i+1} \cup \Gamma_{i+2}$. 
\item \Horizontal No horizontal sector whose span includes $\mathbbm{b}_{i,j}$, where $1 \leq j \leq c_i-1$, is a sink: as the band $\mathbbm{b}_{i,j}$ is enclosed, on the left, by the outward pointing plumbing arc $\alpha_{i,j}$. See \Cref{fig:CalibrateAndContinue} (bottom row, right). The sector containing $\mathbbm{b}_{i, c_i}$ is the branch sector $\mathcal{S}_{i}$, which we already argued is not a sink.

If a horizontal sector $\mathcal{H}$ spans $\Gamma_{i+1}$, it contains some band $\mathbbm{b}_{i+1, k}$. There are two cases to consider, based on whether $1 \leq k \leq \delta_1$ or $\delta_1 + 1 \leq k \leq c_{i+1}$. 

Suppose $\mathcal{H}$ contains $\mathbbm{b}_{i+1, k}$ with $1 \leq k \leq \delta_1$. These bands are all enclosed, on the left, by $\varphi(\alpha_{i, c_i-1})$, and so they are all in the same sector as $\mathcal{S}_i$; we already proved this sector is not a sink. Otherwise, $\mathcal{H}$ contains $\mathbbm{b}_{i+1, k}$ where $\delta_1 + 1 \leq k \leq c_{i+1}$, then $\mathcal{H}$ is in the same horizontal sector as $\mathcal{S}_{i+1}$, which we already argued is not a sink. See \Cref{fig:CalibrateAndContinue} (bottom row).
\end{itemize}

We deduce that in Case (2), $\mathcal{S}_{i+2}$ is the only potential sink disk spanning $\Gamma_i \cup \Gamma_{i+1}$. \Cref{lemma:TemplateLinking} confirms that there is no pairwise linking between arcs which are both in $S_1$ or both in $S_3$, and there can be no pairwise linking between arcs in $S_1$ and $S_3$.
\end{proof}

We summarize our findings in \Cref{table:Bookkeeping_InteriorRedux}.

\begin{table}[H]
\begin{tabular}{ | c | c | c |}
\hline
%\cellcolor{grey!20}{\textsc{Case}} &  \cellcolor{grey!20}{\textsc{Distinguished Band}} &  \cellcolor{grey!20}{\textsc{Potential Sink Disk Sector contains}} \\ \hline \hline
\textsc{Case} &  \textsc{Distinguished Band} & \textsc{Potential Sink Disk Sector contains} \\ \hline \hline
%$\Gamma_i$ is heterogeneous & $\mathbbm{b}_{i+1, 1}$ & $\mathbbm{b}_{i+1,1}$ \\  
$\Gamma_i$ is heterogeneous & $\mathbbm{b}_{i+1, 1}$ & $\mathcal{S}_{i+2}$ \\  
%$\Gamma_i$ is homogeneous & $\mathbbm{b}_{i+1, 1} &  $\mathcal{S}_i$ \\ 
$\Gamma_i$ is homogeneous & $\mathbbm{b}_{i+1, 1}$ & $\mathcal{S}_{i+2}$ \\  \hline
\end{tabular} 
        \caption{Bookkeeping for \Cref{lemma:Continue}, where we build $B$ in the ``interior'' columns of the braid.}
        \label{table:Bookkeeping_InteriorRedux}
\end{table}

In \Cref{lemma:Continue}, we assumed that we had already built our branched surface in $\Gamma_1 \cup \ldots \Gamma_i$, where $i \geq 4$. We now prove an analogous statement to \Cref{lemma:Continue} for the case where $i=3$ and $\Cmax = \Codd$. 

\begin{lemma} \label{lemma:Continue_Special}
Suppose $\Cmax = \Codd$, and $n \geq 7$. Further suppose that we followed \Cref{lemma:Generic_Cmax=Codd_Gammas12} to construct a branched surface for $\widehat{\beta}$ in $\Gamma_1 \cup \Gamma_2 \cup \Gamma_3$ $($so, in particular, we added $c_1-1$ $($resp. $c_3-1$$)$ co-oriented product disks from $\Gamma_1$ $($resp. $\Gamma_3$$)$ into $B$, and exactly one co-oriented product disk from $\Gamma_2$ into $B$$)$. Suppose we have not yet added any product disks from $\Gamma_4 \cup \ldots \cup \Gamma_{n-1}$ into $B$. Then, after performing some pivots and canonical calibrations, we can build our branched surface in $\Gamma_4 \cup \Gamma_5$ so that: $(1)$ $B$ contains no co-oriented product disks from $\Gamma_4$, $(2)$ we applied our template to $\Gamma_5$ $($so there are $c_5-1$ co-oriented product disks from $\Gamma_5$ in our branched surface$)$, $(3)$ there is at most one potential sink disk spanning $\Gamma_1 \cup \ldots \cup \Gamma_4$, and it is $\mathcal{S}_5$, and $(4)$ there is no pairwise linking between arcs in $\Gamma_3 \cup \Gamma_{4} \cup \Gamma_{5}$. 
\end{lemma}

\begin{proof}
In \Cref{lemma:Generic_Cmax=Codd_Gammas12}, there were four possible ways that we constructed $B$ in $\Gamma_3$. To continue building $B$ in $\Gamma_3 \cup \Gamma_4 \cup \Gamma_5$, we will analyze the four cases in turn. To begin, we use the labeling of bands and presentation from \Cref{lemma:Generic_Cmax=Codd_Gammas12}.

\tcbox[size=fbox, colback=gray!30]{Construction of $B$ in $\Gamma_1 \cup \Gamma_2 \cup \Gamma_3$ is inherited from Case (A) of \Cref{lemma:Generic_Cmax=Codd_Gammas12}.}

As summarized in \Cref{table:Bookkeeping_Cmax=Codd_Gammas12}, there is one potential sink disk spanning $\Gamma_1 \cup \Gamma_2$, and it is $\mathcal{S}_3$. Additionally, in $S_3$, $\partial \mathcal{S}_3$ contains $c_3-1$ arcs and one image arc; see \Cref{fig:Generic_CmaxCodd_Gammas12}. In this case, we can apply the same procedure as in the proof of \Cref{lemma:Continue} to pivot $\beta$ and then build $B$ in $\Gamma_4 \cup \Gamma_5$; for an example of the result, see \Cref{fig:Continue_Special_A}. We claim that the sink disk analysis from \Cref{lemma:Continue} carries through without issue: this follows from the fact that the unique image arc on $S_3$ does not enclose any bands on the left (i.e. the arc $\varphi(\alpha_{2,1})$ does not enclose any other plumbing arcs on $S_3$), so the branch sector $\mathcal{S}_3$ is homeomorphic to the sector seen in \Cref{fig:CalibrateAndContinue} (Upper Right, Lower Right). We deduce that the unique potential sink disk spanning $\Gamma_1 \cup \ldots \cup \Gamma_4$ is $\mathcal{S}_5$. 

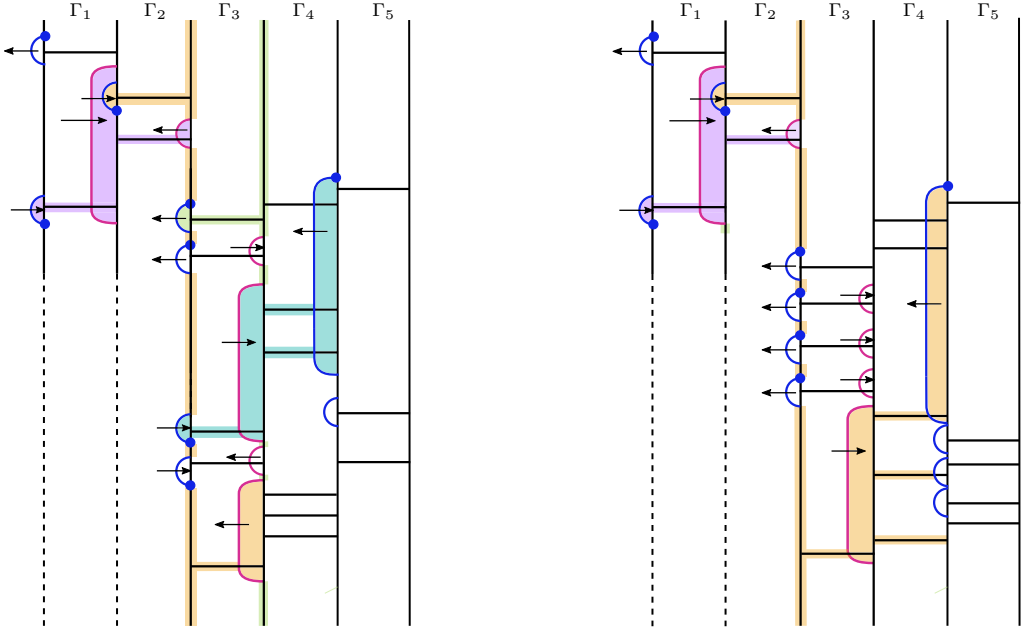
\begin{figure}[h!] \center
\labellist
\tiny
\pinlabel {$\Gamma_1$} at 32 240
\pinlabel {$\Gamma_{2}$} at 60 240
\pinlabel {$\Gamma_{3}$} at 88 240
\pinlabel {$\Gamma_{4}$} at 116 240
\pinlabel {$\Gamma_{5}$} at 146 240
\pinlabel {$\Gamma_1$} at 262 240
\pinlabel {$\Gamma_{2}$} at 290 240
\pinlabel {$\Gamma_{3}$} at 318 240
\pinlabel {$\Gamma_{4}$} at 346 240
\pinlabel {$\Gamma_{5}$} at 374 240
\endlabellist
        \includegraphics[scale=1]{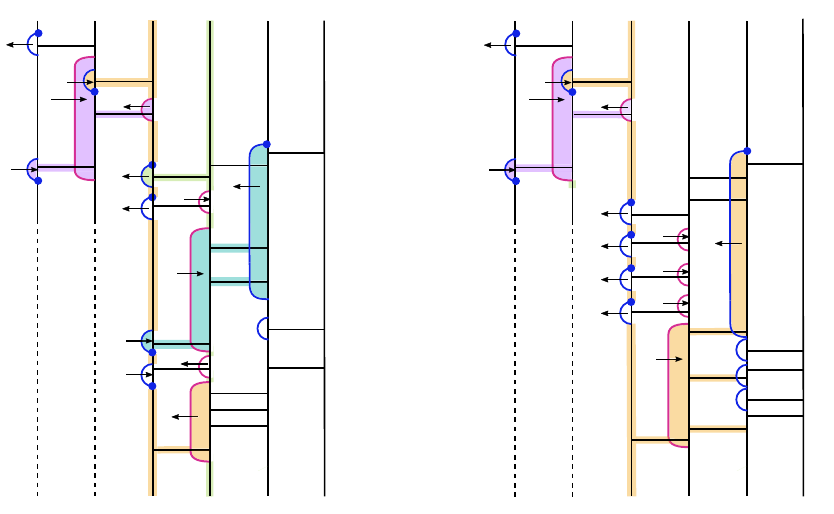}
        \caption{Assume we build $B$ in $\Gamma_1 \cup \Gamma_2 \cup \Gamma_3$ according to Case (A) of \Cref{lemma:Generic_Cmax=Codd_Gammas12}. Since $\partial \mathcal{S}_3 \cap S_3$ is $\varphi(\alpha_{2,1}), \alpha_{3,1}, \ldots, \alpha_{3,c_3-1}$, we can directly pivot our braid and then apply the construction -- and proof -- from \Cref{lemma:Continue} to build $B$ in $\Gamma_3 \cup \Gamma_4 \cup \Gamma_5$. On the \textbf{left}, $\Gamma_3$ is heterogeneous, and on the \textbf{right}, $\Gamma_3$ is homogeneous. We deduce that there is only one potential sink disk in $\Gamma_1 \cup \Gamma_2 \cup \Gamma_3 \cup \Gamma_4$, and it is $\mathcal{S}_5$. As we do not want to focus on the distribution of bands in $\Gamma_1 \cup \Gamma_2$, we have dotted the vertical lines representing the first and second strands of the braid.}
        \label{fig:Continue_Special_A}
\end{figure}

\tcbox[size=fbox, colback=gray!30]{Construction of $B$ in $\Gamma_1 \cup \Gamma_2 \cup \Gamma_3$ is inherited from Case (B.1) of \Cref{lemma:Generic_Cmax=Codd_Gammas12}.}

We need to do some new analysis for this and subsequent cases. We use the labelings from \Cref{lemma:Generic_Cmax=Codd_Gammas12} throughout, and we refer to \Cref{fig:Continue_Special_B.1} throughout. As summarized in \Cref{table:Bookkeeping_Cmax=Codd_Gammas12}, we know that the unique potential sink disk spanning $\Gamma_1 \cup \Gamma_2$ is the sector $\mathcal{H}$ containing $\mathbbm{b}_{3, \delta-1}$; we zoom in on this sector in \Cref{fig:Continue_Special_B.1}. 

We recall: in Case (B.1) of \Cref{lemma:Generic_Cmax=Codd_Gammas12}, we let $\delta = d(2; 1, 2)$; additionally, we knew that $d(3; \delta-2, \delta) = c_4$, and that $d(3; \delta-2, \delta-1)$ and $d(3; \delta-1, \delta)$ were both positive. Let $\epsilon = d(3; \delta-2, \delta-1)$. To build $B$ in $\Gamma_5$, we canonically calibrate $\beta$ with respect to $\mathbbm{b}_{4, \epsilon}$, and then apply our template to $\Gamma_5$; see \Cref{fig:Continue_Special_B.1}. We claim there is a unique sink disk spanning $\Gamma_1 \cup \Gamma_2 \cup \Gamma_3 \cup \Gamma_4$, and it is $\mathcal{S}_5$.

\begin{itemize}
\item \Seifert In \Cref{lemma:Generic_Cmax=Codd_Gammas12}, we proved that $\mathcal{S}_1, \mathcal{S}_2$, and $\mathcal{S}_3$ are not sinks. By assumption (from being in Case (B.1)), we know that $\Gamma_3$ is homogeneous; it follows that all the image arcs in $S_4$ are right pointers, so $\mathcal{S}_4$ is not a sink. Since we do not know whether $\Gamma_5$ is homogeneous or heterogeneous, we cannot yet argue that $\mathcal{S}_5$ is not a sink. 
\item \Polygon The only possible Seifert disks with polygon sectors are $S_2$ and $S_3$, since these are the only Seifert disks spanning $\Gamma_1 \cup \ldots \cup \Gamma_4$ that contain both plumbing and image arcs. Our proof of \Cref{lemma:Generic_Cmax=Codd_Gammas12} tells us that the unique polygon sector in $S_2$ is not a sink. By design, the polygon sector $\mathcal{P}$ such that $\mathcal{P} \cap S_3 \neq \varnothing$ is not a sink: we constructed $B$ so that $\mathcal{P} \sim \mathbbm{b}_{3, \delta-1} \sim \mathbbm{b}_{4, \epsilon}$, and we know that $\mathbbm{b}_{4, \epsilon}$ is enclosed, on the right, by $\alpha_{5,1}$, which is a left pointer (see \Cref{fig:Continue_Special_B.1}). We deduce that $\mathcal{P}$ is not a sink, and that there are no sink disk polygon sectors spanning $\Gamma_1 \cup \ldots \cup \Gamma_4$. 
\item \Horizontal Our proof of \Cref{lemma:Generic_Cmax=Codd_Gammas12} ensures that the unique horizontal sector spanning $\Gamma_1 \cup \Gamma_2$ that could be a sink is the sector containing $\mathbbm{b}_{3,\delta-1}$. As argued in our polygon sector analysis above, the branch sector containing $\mathbbm{b}_{3, \delta-1}$ is not a sink. Moreover, since $\Gamma_3$ is homogeneous, then for all $i \neq \delta-1$, the horizontal sector containing $\mathbbm{b}_{3,i}$ is not a sink because $\alpha_{3, i}$ points out of the sector. Therefore, at this point, we know that no horizontal sector spanning $\Gamma_1 \cup \Gamma_2 \cup \Gamma_3$ is a sink. 

It remains to show that no sector spanning $\Gamma_4$ is a sink. This is straightforward: recall that we have labelled our bands in $\Gamma_4$ such that $d(3; \delta-2, \delta-1)= \epsilon \geq 1$ and $d(3; \delta-2, \delta) = c_4$. We know that the bands $\mathbbm{b}_{4, 1}, \ldots, \mathbbm{b}_{4, \epsilon}$ are all enclosed, on the left by $\varphi(\alpha_{3, \delta-2})$, and in particular, this means $\mathbbm{b}_{4, 1} \sim \ldots \sim \mathbbm{b}_{4, \epsilon}$; see \Cref{fig:Continue_Special_B.1}.
We already proved that the sector containing $\mathbbm{b}_{4, \epsilon}$ is not a sink, so it follows that no horizontal sector containing $\mathbbm{b}_{4, j}$, where $j \in \{1, \ldots, \epsilon\}$ is a sink. Finally, notice that the remaining bands in $\Gamma_4$ are all concentrated between $\mathbbm{b}_{3, \delta-1}$ and $\mathbbm{b}_{3, \delta}$ (this is because $d(3; \delta-1, \delta) = c_4 - \epsilon$). Since we did not include the product disk swept out by $\alpha_{3, \delta-1}$ into our branched surface, the bands $\mathbbm{b}_{4, \epsilon+1}, \ldots, \mathbbm{b}_{c_4}$ are not enclosed, on the left, by any arcs. In particular, these bands are all in the same branch sector as $\mathcal{S}_4$, which we already proved is not a sink. We deduce that there are no sink disk horizontal sectors spanning $\Gamma_1 \cup \Gamma_2 \cup \Gamma_3 \cup \Gamma_4$. 
\end{itemize}

We deduce that in Case (B.1), we can build a branched surface which contains product disks from $\Gamma_1 \cup \Gamma_2 \cup \Gamma_3 \cup \Gamma_5$ such that the unique potential sink disk spanning $\Gamma_1 \cup \Gamma_2 \cup \Gamma_3 \cup \Gamma_4$ is $\mathcal{S}_5$.

%\begin{center}
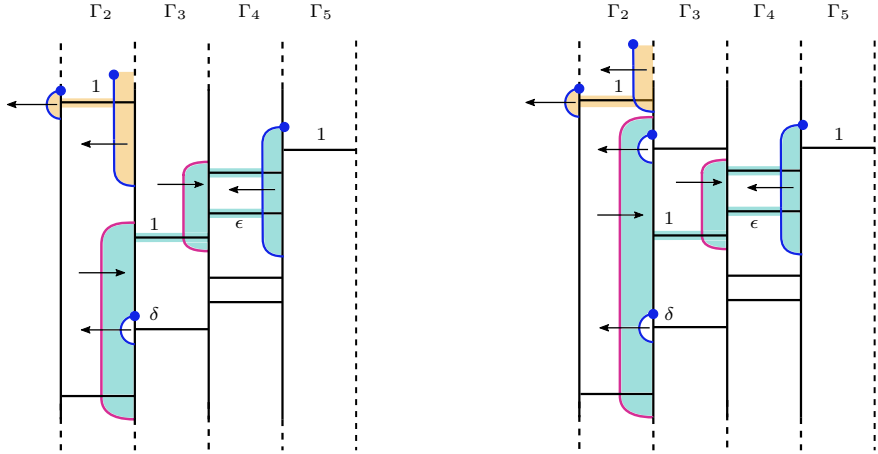
\begin{figure}[h!] \center
\labellist \tiny
%
%\pinlabel {$\Gamma_{n-3}$} at 218 174
\pinlabel {$\Gamma_{2}$} at 40 174
\pinlabel {$\Gamma_{3}$} at 68 174
\pinlabel {$\Gamma_{4}$} at 96 174
\pinlabel {$\Gamma_{5}$} at 123 174
\pinlabel {$1$} at 38 146
\pinlabel {$1$} at 60 94
\pinlabel {$1$} at 123 128
\pinlabel {$\delta$} at 60 60
\pinlabel {$\epsilon$} at 92 94
\pinlabel {$\Gamma_{2}$} at 235 174
\pinlabel {$\Gamma_{3}$} at 262 174
\pinlabel {$\Gamma_{4}$} at 290 174
\pinlabel {$\Gamma_{5}$} at 318 174
\pinlabel {$1$} at 235 146
\pinlabel {$1$} at 254 96
\pinlabel {$1$} at 318 128
\pinlabel {$\delta$} at 254 60
\pinlabel {$\epsilon$} at 286 94
\endlabellist
        \includegraphics[scale=1]{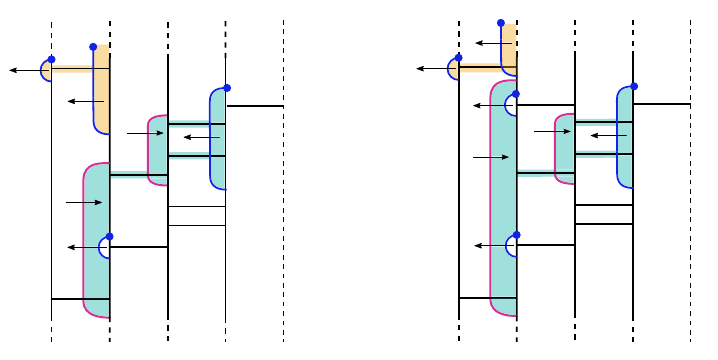}
        \caption{Assume we build $B$ in $\Gamma_1 \cup \Gamma_2 \cup \Gamma_3$ according to Case (B.1) of \Cref{lemma:Generic_Cmax=Codd_Gammas12}. Zooming in on the bands enclosed, on the left, by $\varphi(\alpha_{2,1})$ from \Cref{fig:Generic_CmaxCodd_Gammas12_A} (top right, bottom right); to be consistent, we have used the same cyclic conjugated presentation of the braid from that figure. On the \textbf{left}, the case where $d(2; 1,2) = 2$, and on the \textbf{right}, we have $d(2; 1, 2) \geq 3$. For emphasis, we have partially drawn $\alpha_{3, \alpha_{c_3}}$ on $S_3$. We see that the shaded sector, which contains $\mathbbm{b}_{3, \delta-1}$, is not a sink, and that $\mathcal{S}_4 \sim \mathbbm{b}_{4, \epsilon+1} \sim \ldots \sim \mathbbm{b}_{4, c_4}$.}
        \label{fig:Continue_Special_B.1}
\end{figure}

\tcbox[size=fbox, colback=gray!30]{Construction of $B$ in $\Gamma_1 \cup \Gamma_2 \cup \Gamma_3$ is inherited from Case (B.2) of \Cref{lemma:Generic_Cmax=Codd_Gammas12}.}

As summarized in \Cref{table:Bookkeeping_Cmax=Codd_Gammas12}, we can build $B$ in $\Gamma_1 \cup \Gamma_2 \cup \Gamma_3$ such that there are no sink disks spanning $\Gamma_1 \cup \Gamma_2$. We recall: in this case, we know that $\Gamma_3$ is heterogeneous, but that all the right pointing plumbing arcs co-bound a subsurface in $\Gamma_3$; see \Cref{fig:Generic_CmaxCodd_Gammas12_A.2}. Assume the same presentation of the braid as in \Cref{lemma:Generic_Cmax=Codd_Gammas12}, and now let $\mathbbm{b}_{3, p}$ denote the first right pointing plumbing arc in $\Gamma_3$; in particular, this means $\mathbbm{b}_{3, p-1}$ is a left pointer and $1 \leq p \leq \delta-2$. In \Cref{fig:Continue_Special_B.2}, we see examples where $p = 1$ and where $p = 2$. By the definition of our template, we know that $d(3; \delta, c_3) = 0$, $d(3; c_3, p-1) = 0$, and $d(3; p-1, p) := \epsilon \geq 1$. To build $B$ in $\Gamma_4 \cup \Gamma_5$, we canonically calibrate $\beta$ with respect to $\mathbbm{b}_{4, \epsilon}$, and then apply our template to $\Gamma_5$; see \Cref{fig:Continue_Special_B.2}. We claim that the resulting branched surface contains at most one sink disk in $\Gamma_1 \cup \ldots \cup \Gamma_4$, and it is $\mathcal{S}_5$.

%\begin{center}
\begin{figure}[h!] \center
\labellist \tiny
%
%\pinlabel {$\Gamma_{n-3}$} at 218 174
\pinlabel {$\Gamma_{2}$} at 40 174
\pinlabel {$\Gamma_{3}$} at 68 174
\pinlabel {$\Gamma_{4}$} at 96 174
\pinlabel {$\Gamma_{5}$} at 123 174
\pinlabel {$1$} at 38 144
\pinlabel {$p$} at 64 118
\pinlabel {$1$} at 123 150
\pinlabel {$\delta$} at 60 60
\pinlabel {$\epsilon$} at 92 122
\pinlabel {$\Gamma_{2}$} at 235 174
\pinlabel {$\Gamma_{3}$} at 262 174
\pinlabel {$\Gamma_{4}$} at 290 174
\pinlabel {$\Gamma_{5}$} at 318 174
\pinlabel {$1$} at 232 144
\pinlabel {$p$} at 254 94
\pinlabel {$1$} at 318 128
\pinlabel {$\delta$} at 254 60
\pinlabel {$\epsilon$} at 286 94
\endlabellist
        \includegraphics[scale=1]{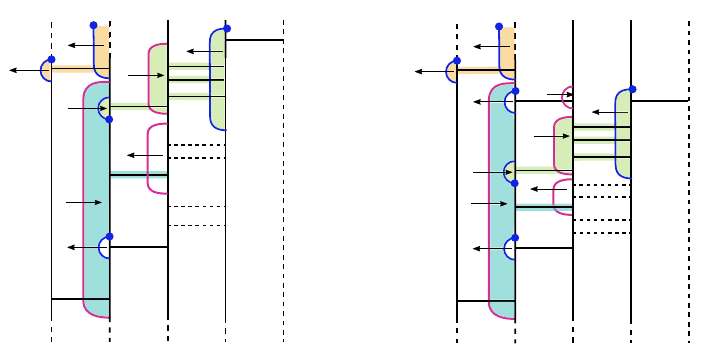}
        \caption{Assume we build $B$ in $\Gamma_1 \cup \Gamma_2 \cup \Gamma_3$ according to Case (B.2) of \Cref{lemma:Generic_Cmax=Codd_Gammas12}. Zooming in on the bands enclosed, on the left, by $\varphi(\alpha_{2,1})$ from \Cref{fig:Generic_CmaxCodd_Gammas12_A.2}; to be consistent, we have used the same cyclic conjugated presentation of the braid from that figure. On the \textbf{left}, the case where $p=1$, and on \textbf{right}, we have $p=2$. For emphasis, we have partially drawn $\alpha_{3, c_3}$ on $S_3$. The green shaded sector, which contains $\mathbbm{b}_{3, p}$, is not a sink. We have drawn bands in $\Gamma_4$ as dashed to indicate that we do not know their precise distribution within the column.
}
        \label{fig:Continue_Special_B.2}
\end{figure}
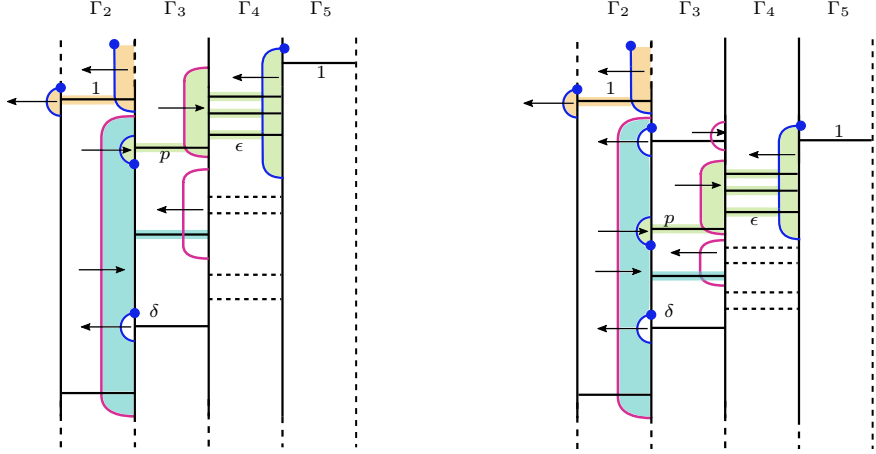

Since we already proved that there are no sink disks spanning $\Gamma_1 \cup \Gamma_2$, we need only prove that there is at most one sink disk spanning $\Gamma_3 \cup \Gamma_4$. 

\begin{itemize}
\item \Seifert The sector $\mathcal{S}_3$ is not a sink because $\varphi(\alpha_{2,1})$ points out of the sector. Since $\Gamma_3$ is heterogeneous and we have used no product disks from $\Gamma_4$ to build $B$, $\partial \mathcal{S}_4$ must contain an outward pointing arc. Since we do not know whether $\Gamma_5$ is heterogeneous or homogeneous, we cannot make any claims about $\mathcal{S}_5$; this is the unique potential sink disk spanning $\Gamma_3 \cup \Gamma_4$. 
\item \Polygon By \Cref{lemma:Generic_Cmax=Codd_Gammas12}, we know that no polygon sector in $S_3$ is a sink. Since $S_4$ (resp. $S_5$) contains only image (resp. plumbing) arcs, there are no polygon sectors in $S_4$ or $S_5$. We deduce that there are no sink disk polygon sectors in $\Gamma_3 \cup \Gamma_4$. 
\item \Horizontal Suppose $\mathcal{H}$ is a horizontal sector spanning $\Gamma_3$. We prove that $\mathcal{H}$ is not a sink:
	\begin{itemize}
	\item[$\circ$] If $\mathcal{H} \sim \mathbbm{b}_{3,i}$, where $i \in \{1, \ldots, p-1\} \cup \{\delta, \ldots, c_3\}$, then $\mathcal{H}$ is enclosed, on the left, by the left pointing plumbing arc $\alpha_{3,i}$; we deduce that $\mathcal{H}$ is not a sink. 
	\item[$\circ$] If $\mathcal{H} \sim \mathbbm{b}_{3,i}$ where $i \in \{p+1, \ldots \delta-1\}$, then $\mathcal{H}$ is enclosed, on the right, by the left pointing image arc $\varphi(\alpha_{3,i-1})$, so $\mathcal{H}$ is not a sink.
	\item[$\circ$] If $\mathcal{H} \sim \mathbbm{b}_{3, p}$, then $\mathcal{H}$ is in the same branch sector as $\mathbbm{b}_{4, \epsilon}$. By construction, this sector meets $S_5$ in $\alpha_{5,1}$, which is a right pointer. We deduce that $\mathcal{H}$ is not a sink.
	\end{itemize}
	
It remains to show that if a horizontal sector $\mathcal{H}$ spans $\Gamma_4$, then $\mathcal{H}$ is not a sink. Recall that by the definition of Case (B.2), we know that $d(3; \delta, c_3) + d(3; c_3, 1) + d(3; 1, p-1) = 0$. In particular, this means that all of our bands in $\Gamma_4$ are concentrated in between $\mathbbm{b}_{3, p-1}$ and $\mathbbm{b}_{3, \delta}$. The bands $\mathbbm{b}_{4,1}, \ldots, \mathbbm{b}_{4, \epsilon}$ are all part of the same branch sector; by design, $\mathbbm{b}_{4, \epsilon}$ is enclosed, on the right, by an outward pointing arc. So, if $\mathcal{H} \sim \mathbbm{b}_{4, i}$, where $i \in \{1, \ldots, \epsilon\}$, then $\mathcal{H}$ is not a sink. If $\mathcal{H} \sim \mathbbm{b}_{4, i}$, where $i \in \{\epsilon +1, \ldots, c_4\}$, then $\mathcal{H}$ meets $S_4$ in an outward pointing image arc, so it is not a sink. Otherwise, $\mathcal{H} \sim \mathbbm{b}_{4,i} \sim \mathcal{S}_4$, and we already proved that $\mathcal{S}_4$ is not a sink. We conclude that there are no horizontal sink disk sectors in $\Gamma_3 \cup \Gamma_4$. 
\end{itemize}

Therefore, the unique potential sink disk spanning $\Gamma_1 \cup \Gamma_2 \cup \Gamma_3 \cup \Gamma_4$ is the branch sector $\mathcal{S}_5$.

\tcbox[size=fbox, colback=gray!30]{Construction of $B$ in $\Gamma_1 \cup \Gamma_2 \cup \Gamma_3$ is inherited from Case (C) of \Cref{lemma:Generic_Cmax=Codd_Gammas12}.}

As summarized in \Cref{table:Bookkeeping_Cmax=Codd_Gammas12}, there is one potential sink disk spanning $\Gamma_1 \cup \Gamma_2$, and it is the horizontal sector $\mathcal{H} \sim \alpha_{2,1} \sim \alpha_{3, c_3}$. Therefore, to prove that we can construct $B$ in $\Gamma_1 \cup \ldots \cup \Gamma_5$ such that there are no sink disks in $\Gamma_1 \cup \Gamma_2 \cup \Gamma_3 \cup \Gamma_4$, it need to first build $B$ in $\Gamma_4 \cup \Gamma_5$ and then show that there are no sink disks spanning $\Gamma_3 \cup \Gamma_4$. We will refer to \Cref{fig:Generic_CmaxCodd_Gammas12_B} throughout.

To build $B$ in $\Gamma_4 \cup \Gamma_4$, we first define $\epsilon := d(3; j-1, j)$, where (as in \Cref{fig:Generic_CmaxCodd_Gammas12_B}) $\alpha_{3, j}$ is the first right pointer in $S_4$. By the definition of our template, we know that $\epsilon \geq 1$. To build $B$ in $\Gamma_5$, we canonically calibrate $\beta$ with respect to $\mathbbm{b}_{4, \epsilon}$, and then apply our template to $\Gamma_5$. We check there are no sink disks spanning $\Gamma_3 \cup \Gamma_4$. 

\begin{itemize}
\item \Seifert By being in Case (C) of \Cref{lemma:Generic_Cmax=Codd_Gammas12}, we know that $\mathcal{S}_3$ is not a sink. Since $\Gamma_3$ is heterogeneous and we have included no product disks from $\Gamma_4$ into $B$, $\partial \mathcal{S}_4$ contains both left and right pointing image arcs, so $\mathcal{S}_4$ is not a sink. The sink disk status of $\mathcal{S}_5$ is unknown, as we do not know whether it is heterogeneous or homogeneous. 
\item \Polygon The only polygon sector in $\Gamma_3 \cup \Gamma_4$ lies in $S_3$; our proof of \Cref{lemma:Generic_Cmax=Codd_Gammas12} ensures that this sector is not a sink. 
\item \Horizontal First, suppose $\mathcal{H}$ spans $\Gamma_3$. Adapting our usual argument, we see that the only sector that could potentially be a sink is the sector $\mathcal{H} \sim \mathbbm{b}_{3, j}$, because this sector meets $S_3$ in a right pointer and $S_4$ in a left pointer. However, by design, this sector contains $\mathbbm{b}_{4,\epsilon}$, and this band is enclosed, on the right, by a left pointing plumbing arc. We deduce that $\mathcal{H}$ is not a sink. Instead suppose $\mathcal{H}$ spans $\Gamma_4$. In this case, as in \Cref{fig:Continue_Special_B.2}, every band in $\Gamma_4$ is either (1) enclosed, on the left, by a left pointing image arc, or (2) in the same branch sector as $\mathcal{S}_4$, or (3) it is in the same branch sector as $\mathbbm{b}_{4, \epsilon}$. In each of these three cases, we deduce that the horizontal sector is not a sink, hence there are no horizontal sink disk sectors in $\Gamma_3 \cup \Gamma_4$. 
\end{itemize}

We deduce that, in Case (C) of \Cref{lemma:Generic_Cmax=Codd_Gammas12}, we can build $B$ in the first five columns of the braid such that the unique potential sink disk spanning $\Gamma_1 \cup \ldots \cup \Gamma_4$ is $\mathcal{S}_5$. \Cref{lemma:TemplateLinking} confirms that there is no pairwise linking between arcs which are both in $S_3$ or both in $S_5$, and there can be no pairwise linking between arcs in $S_1$ and $S_3$.
\end{proof}

%%%%%%%%%%%%%%%%%%%%%%%%%
%%%%%%%%%%%%%%%%%%%%%%%%%
%%% BUILD B IN THE LAST FEW COLUMNS
%%%%%%%%%%%%%%%%%%%%%%%%%
%%%%%%%%%%%%%%%%%%%%%%%%%

\subsection{Building $B$ in the rightmost columns.} \label{section:BuildInLast}

In this section, we describe how to build our branched surface $B$ in the last few columns of the braid. 

\begin{lemma} \label{lemma:Generic_EndOfBraid_LastColumnIsSparse}
Let $\beta$ be a positive braid on $n$ strands, where $n \geq 5$. Suppose we constructed a branched surface in $\Gamma_1 \cup \ldots \cup \Gamma_{n-2}$, such that we applied our template to $\Gamma_{n-2}$. Further assume that $\partial \mathcal{S}_{n-2}$ contains $c_{n-2}-1$ plumbing arcs and at most one image arc, and that the Seifert disk $S_{n-2}$ does not contain any polygon sectors. If we build $B$ in $\Gamma_{n-2} \cup \Gamma_{n-1}$ by taking no product disks from $\Gamma_{n-1}$, then there are no sink disks spanning $\Gamma_{n-2} \cup \Gamma_{n-1}$. 
\end{lemma}

\begin{proof}
Just as in \Cref{lemma:Continue}, there will be two cases to consider, based on whether $\Gamma_{n-2}$ is heterogeneous or homogeneous. 

\begin{enumerate}
\item[$\boldsymbol{\circ}$] \textbf{Case (1)}: The heterogenous case. Here, $d(n-2; 1, c_i - 1) > 0$, as in \Cref{fig:LastTwoColumns}.
\item[$\boldsymbol{\circ}$] \textbf{Case (2)}: The homogeneous case. Here, $d(n-2; 1, c_i-1) = 0$, as in \Cref{fig:LastTwoColumns_Homogeneous}.
\end{enumerate}

Throughout, we assume three things: (1) $\beta$ is cyclically conjugated so that it begins with a $\sigma_{n-2}$ letter, (2) $\alpha_{n-2,1}$ is a left pointer, and (3) $\mathbbm{b}_{n-2,c_{n-2}}$ is not enclosed, on the left, by any arc. We quickly justify why we can do this:

\begin{itemize}
\item If $S_{n-2}$ has no image arcs, then this is immediate. 
\item Suppose $S_{n-2}$ contains exactly one image arc, $c_{n-2}$ plumbing arcs, and no polygon sectors. Since $S_{n-2}$ contains one image arc, \Cref{NoClumpingOfArcs} tells us that either (a) $\Cmax = \Codd$ and $\Gamma_3 = \Gamma_{n-2}$, or (b) $\Cmax = \Ceven$ and $\Gamma_{2} = \Gamma_{n-2}$. However, since we are also assuming that $n \geq 5$, we deduce that 
%The constructions outlined in Lemmas \ref{lemma:Generic_Cmax=Codd_Gammas12}, \ref{lemma:Sparse_Cmax=Ceven_Gammas12}, \ref{lemma:Continue}, guarantee that if we take a single product disk from a column $\Gamma_i$, $i \not \equiv m \mod 2$, and include it into $B$, then either (a) $\Cmax = \Ceven$ and $i=1$, or (b) $\Cmax = \Codd$ and $i =2$. Therefore, if $S_{n-2}$ contains a single image arc (and, by assumption, $c_{n-2}-1$ plumbing arcs), and $n-2 \equiv m \mod 2$, then 
$n=5$, $\Cmax = \Codd$, and the construction of $B$ in $\Gamma_{n-2} = \Gamma_{3}$ is inherited from Case (A) of \Cref{lemma:Generic_Cmax=Codd_Gammas12}. In particular, this means that $\mathbbm{b}_{n-2, c_{n-2}}$ is not enclosed, on the left, by any arc, and so we can cyclically conjugate $\beta$ into the form described above.
\end{itemize}

We now move onto the constructions of $B$ in $\Gamma_{n-2} \cup \Gamma_{n-1}$.

\tcbox[size=fbox, colback=gray!30]{Sink disk analysis for \textbf{Case (1):} $\Gamma_{n-2}$ is heterogeneous.}
Suppose $\beta$ is cyclically conjugated as described, and $\Gamma_{n-2}$ is heterogeneous; see \Cref{fig:LastTwoColumns}. Additionally, let $\mathbbm{b}_{n-2, p}$ denote the first band which is followed by $\mathbbm{b}_{n-1}$ bands. (That is, $2 \leq p \leq c_{n-2}-2$,  $d(n-2; 1, p) = 0$, and $d(n-2; p, 1) = c_{n-1}$.) We now perform our sink disk analysis.

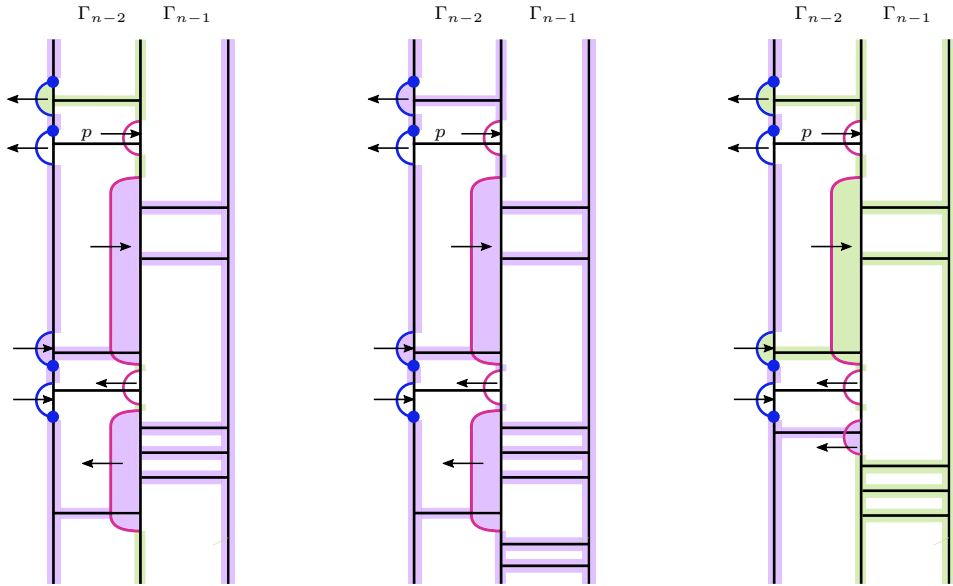
\begin{figure}[h!] \center
\labellist
\tiny
\pinlabel {$\Gamma_{n-2}$} at 35 190
\pinlabel {$\Gamma_{n-1}$} at 62 190
\pinlabel {$\Gamma_{n-2}$} at 148 190
\pinlabel {$\Gamma_{n-1}$} at 178 190
\pinlabel {$\Gamma_{n-2}$} at 262 190
\pinlabel {$\Gamma_{n-1}$} at 290 190
%%
%\pinlabel {$\delta$} at 58 102
%\pinlabel {$\delta$} at 188 77
%
\pinlabel {$p$} at 30 152
\pinlabel {$p$} at 142 152
\pinlabel {$p$} at 258 152
\endlabellist
        \includegraphics[scale=1.2]{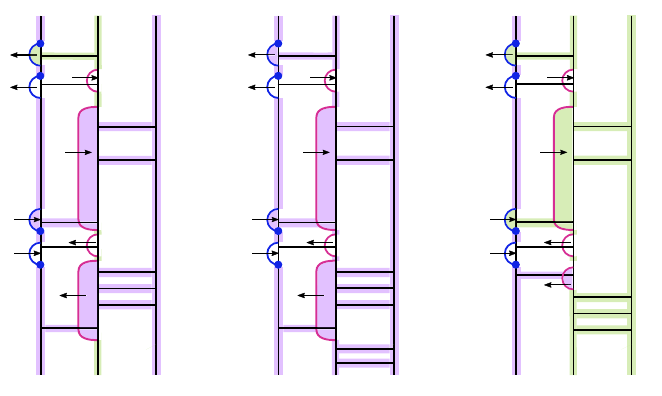}
        \caption{Building $B$ in $\Gamma_{n-2} \cup \Gamma_{n-1}$ when $\Gamma_{n-2}$ is heterogeneous and $n-2 \equiv m \mod 2$. We see that $\mathcal{S}_n$ contains all the bands in $\Gamma_{n-1}$, as well as $\mathbbm{b}_{n-2,p+1}$ and $\mathbbm{b}_{n-1, 1}$.}
        \label{fig:LastTwoColumns}
\end{figure}

\begin{itemize}
\item \Seifert $\mathcal{S}_{n-2}$ is not a sink: the arc $\alpha_{n-2, p+1}$ is in $\partial \mathcal{S}_{n-2}$, and it points out of this Seifert disk. Since $\alpha_{n-2,1}$ is a left pointer, and we have taken no disks from $\Gamma_{i+1}$, then $\varphi(\alpha_{n-2,1})$ is in $\partial \mathcal{S}_{n-1}$. It points out of the sector, so $\mathcal{S}_{n-1}$ is not a sink.

Finally, we argue that $\mathcal{S}_n$ is not a sink. Since $B$ includes no product disks from $\Gamma_{n-1}$, there are no arcs on $S_n$, hence $\mathcal{S}_n \sim \mathbbm{b}_{n-1,1} \sim \ldots \sim \mathbbm{b}_{n-1, c_{n-1}}$, as seen in \Cref{fig:LastTwoColumns}. 
%as seen in \Cref{fig:LastTwoColumns}, $\mathcal{S}_n$ contains all bands in $\Gamma_{n-1}$; that is, $\mathcal{S}_n \sim \mathbbm{b}_{n-1,1} \sim \ldots \sim \mathbbm{b}_{n-1, c_{n-1}}$. 
%
Since $\Gamma_{n-2}$ is heterogeneous, we know that there are {both} left and right pointing image arcs in $S_{n-1}$. The distribution of the bands in $\Gamma_{n-1}$ determines whether or not some or none of the bands will be enclosed, on the left, by a left pointing image arc. In particular, either (a) there is some band $\mathbbm{b}_{n-1, j}$ which is enclosed, on the left, by a left pointing image arc (as in \Cref{fig:LastTwoColumns} (left and middle)), or (b) there is no such band (as in \Cref{fig:LastTwoColumns} (right)). If (a) occurs, then $\mathcal{S}_{n-1}$ contains an outward pointing arc in its boundary, hence it is not a sink; however, if (b) occurs, then $\mathcal{S}_{n-1} \sim \mathcal{S}_{n-1}$, and we already argued that the latter is not a sink. Therefore, none of the Seifert disk sectors are sinks. 
\item \Polygon By definition, a polygon sector must lie in a Seifert disk and also contain both plumbing and image arcs in its boundary. By assumption, $\mathcal{S}_{n-2}$ does not contain a polygon sector. As $S_{n-1}$ and $S_n$ do not contain both co-oriented plumbing and image arcs, there are no polygon sectors spanning $\Gamma_{n-2} \cup \Gamma_{n-1}$. 
\item \Horizontal Suppose $\mathcal{H}$ is a horizontal sector spanning $\Gamma_{n-2}$. There are three possibilities:
\begin{enumerate}
\item \textit{$\mathcal{H} \sim \mathbbm{b}_{n-2,j}$ where $1 \leq j \leq p$.} In this case, $\mathcal{H}$ meets $S_{n-2}$ in an outward pointing plumbing arc, so $\mathcal{H}$ is not a sink. 
\item \textit{$\mathcal{H} \sim \mathbbm{b}_{n-2,p+1}$.} In this case, by design, $\mathcal{H} \sim \mathbbm{b}_{n-1,1}$. We previously argued that the sector $\mathcal{S}_n$ contains all the bands in $\Gamma_{n-1}$, so in particular, $\mathcal{H} \sim \mathcal{S}_n$. We already showed that $\mathcal{S}_n$ is not a sink, so neither is $\mathcal{H}$.
\item \textit{$\mathcal{H} \sim \mathbbm{b}_{n-2,j}$ where $p+2 \leq j \leq c_{n-2}$.} In this case, $\mathcal{H}$ meets $S_{n-1}$ in an outward pointing image arc, so this sector is not a sink. 
\end{enumerate}

Therefore, no horizontal sector $\mathcal{H}$ spanning $\Gamma_{n-2}$ is a sink. Suppose $\mathcal{H}$ spans $\Gamma_{n-1}$. Since $\mathcal{S}_{n}$ contains all the bands in $\Gamma_{n-1}$, then such an $\mathcal{H}$ must be in the same sector as $\mathcal{S}_{n}$, which we already argued is not a sink. We deduce that there are no horizontal sectors spanning $\Gamma_{n-2} \cup \Gamma_{n-1}$. 
\end{itemize}

So, when $\Gamma_{n-2}$ is heterogenous, there are no sink disks spanning $\Gamma_{n-2} \cup \Gamma_{n-1}$. We turn to Case (2).

\tcbox[size=fbox, colback=gray!30]{Sink disk analysis for \textbf{Case (2):}  $\Gamma_{n-2}$ is homogeneous.}

Throughout, we refer to  \Cref{fig:LastTwoColumns_Homogeneous}. 

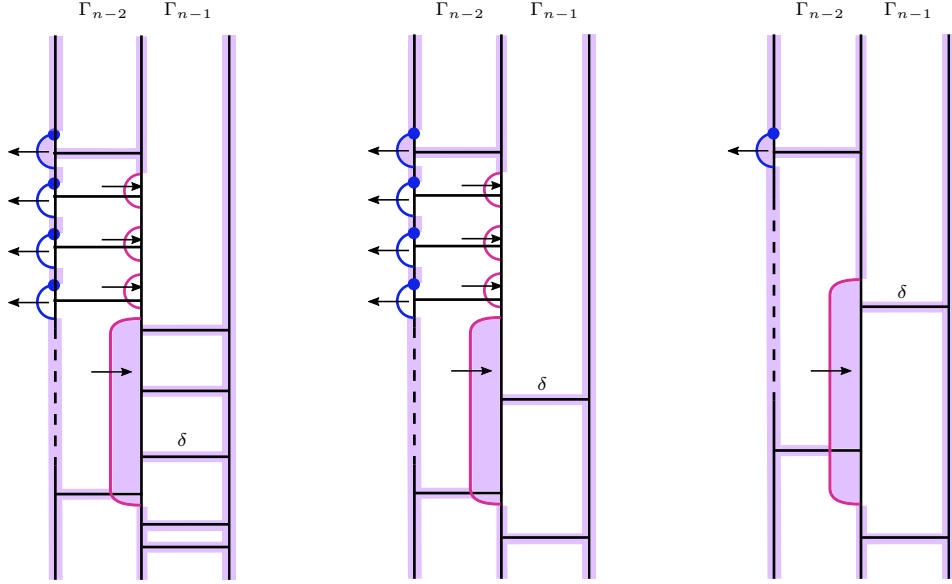
\begin{figure}[h!] \center
\labellist
\tiny
\pinlabel {$\Gamma_{n-2}$} at 35 190
\pinlabel {$\Gamma_{n-1}$} at 62 190
\pinlabel {$\Gamma_{n-2}$} at 148 190
\pinlabel {$\Gamma_{n-1}$} at 178 190
\pinlabel {$\Gamma_{n-2}$} at 262 190
\pinlabel {$\Gamma_{n-1}$} at 290 190
%%
%\pinlabel {$\delta$} at 58 102
%\pinlabel {$\delta$} at 188 77
%
\pinlabel {$\delta$} at 60 54
\pinlabel {$\delta$} at 174 72
\pinlabel {$\delta$} at 288 101
\endlabellist
        \includegraphics[scale=1.2]{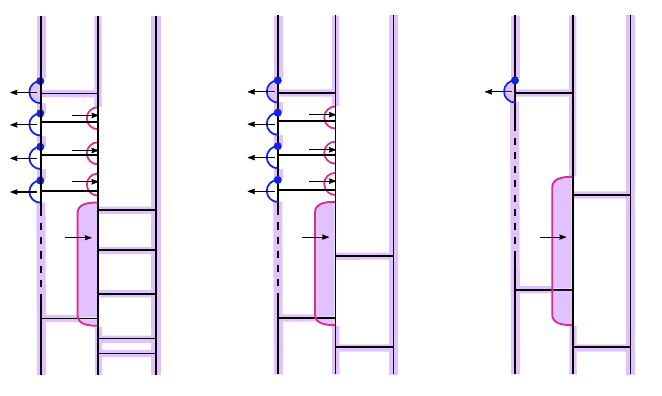}
        \caption{Building $B$ in $\Gamma_{n-2} \cup \Gamma_{n-1}$ when $\Gamma_{n-2}$ is homogeneous and $n-2 \equiv m \mod 2$. We see that $\mathcal{S}_{n-2} \sim \mathcal{S}_{n-1} \sim \mathcal{S}_n$, and that this sector contains all the bands in $\Gamma_{n-1}$. We note that we have included the case where $c_{n-2} \geq 3$ and $c_{n-1}=2$, and also the case where $c_{n-2} \geq 3$ and $c_{n-1}=2$.}
        \label{fig:LastTwoColumns_Homogeneous}
\end{figure}

\begin{itemize}
\item \Seifert We will prove that the Seifert disk sectors $\mathcal{S}_{n-2}$, $\mathcal{S}_{n-1}$, and $\mathcal{S}_n$ are all part of the same branch sector. Let $\mathcal{S}$ denote the branch sector containing $\mathcal{S}_{n-2}$. 

We already argued that we could cyclically conjugate $\beta$ so that $\mathbbm{b}_{n-2,c_{n-2}}$ is not enclosed, on the left, by any arc. This automatically tells us that regardless of whether $\partial \mathcal{S}_{n-2}$ contains an image arc or not, we must have that $\mathcal{S} \sim \mathbbm{b}_{n-2, c_{n-2}}$. %If $S_{n-2}$ contains no image arcs, then by the design of our template, we must have that the band $\mathbbm{b}_{n-2, c_{n-2}}$ is not enclosed by any plumbing arcs, and hence the claim holds. Otherwise, suppose $\partial S_{n-2}$ contains a single image arc. The constructions outlined in Lemmas \ref{lemma:Generic_Cmax=Codd_Gammas12}, \ref{lemma:Sparse_Cmax=Ceven_Gammas12}, \ref{lemma:Continue}, guarantee that if we take a single product disk from a column $\Gamma_i$, $i \not \equiv m \mod 2$, and include it into $B$, then either (a) $\Cmax = \Ceven$ and $i=1$, or (b) $\Cmax = \Codd$ and $i =2$. Therefore, if $S_{n-2}$ contains a single image arc (and, by assumption, $c_{n-2}-1$ plumbing arcs), then $n=5$, $\Cmax = \Codd$, and the construction of $B$ in $\Gamma_{n-2} = \Gamma_{3}$ is inherited from Case (A) of \Cref{lemma:Generic_Cmax=Codd_Gammas12}. In particular, this means that $\mathbbm{b}_{n-2, c_{n-2}}$ is not enclosed, on the left, by any arc, and so $\mathcal{S} \sim  \mathbbm{b}_{n-2, c_{n-2}}$, as desired. 
%%%%
%We know that $\mathcal{S} \sim \mathbbm{b}_{n-2, c_{n-2}}$, because by the design of our template, this band is not enclosed by any plumbing arcs (and these are the only types of arcs in $S_{n-2}$). 
But this band is enclosed, on the right, by $\varphi(\alpha_{n-2, c_{n-2}-1})$, as seen in \Cref{fig:LastTwoColumns_Homogeneous}. Let $\delta:= d(n-2; c_{n-2}-1, c_{n-2})$. The arc $\varphi(\alpha_{n-2, c_{n-2}-1})$ encloses the bands $\mathbbm{b}_{n-1, 1}, \ldots, \mathbbm{b}_{n-1, \delta}$ on the left, and so $\mathcal{S} \sim \mathcal{S}_{n-2} \sim \mathbbm{b}_{n-1,1} \sim \ldots \sim \mathbbm{b}_{n-1, \delta}$. However, these bands are not enclosed on the right by \textit{any} arcs, so $\mathcal{S} \sim \mathcal{S}_n$.

In fact, \textit{none} of the bands in $\Gamma_{n-1}$ are enclosed, on the right, by arcs. Therefore, we have that $\mathcal{S} \sim \mathbbm{b}_{n-1, \delta+1} \sim \ldots \sim \mathbbm{b}_{n-1, c_{n-1}}$. But this is the homogenous case, so we know that these bands are not enclosed, on the left, by any arcs, as in \Cref{fig:LastTwoColumns_Homogeneous}. Therefore, not only does $\mathcal{S}$ contain \textit{all} the bands in $\Gamma_{n-1}$, it also contains $\mathcal{S}_{n-1}$. Therefore, $\mathcal{S} \sim \mathcal{S}_{n-2} \sim \mathcal{S}_{n-1} \sim \mathcal{S}_n$, proving the claim. 

We now argue that $\mathcal{S}$ is not a sink. By construction, $\mathcal{S} \sim \mathbbm{b}_{n-2,1}$, as in \Cref{fig:LastTwoColumns_Homogeneous}. In particular, $\alpha_{n-2,1}$ is contained in $\partial \mathcal{S}$; this arc is smoothed to the left, and it points out of $\mathcal{S}$, so this sector is not a sink. Therefore, none of the Seifert disk sectors spanning $\Gamma_{n-2} \cup \Gamma_{n-1}$ are sinks. 
\item \Polygon The same argument as in the heterogeneous case allows us to deduce that there are no polygon sectors spanning $\Gamma_{n-2} \cup \Gamma_{n-1}$. 
\item \Horizontal Suppose $\mathcal{H}$ is a horizontal sector spanning $\Gamma_{n-2}$. We already argued that $\mathbbm{b}_{n-2,1}$ and $\mathbbm{b}_{n-2,c_{n-2}}$ are part of $\mathcal{S}$, so it remains to study the other bands in $\Gamma_{n-2}$. For all $2 \leq j \leq c_{n-2}-1$, the band $\mathbbm{b}_{n-2,j}$ is enclosed, on the left, by the outward pointing plumbing arc $\alpha_{n-2,j}$. Thus, no sector containing these bands is a sink. 

We already argued that the all the bands in $\Gamma_{n-1}$ are part of the sector $\mathcal{S}$, which is not a sink. Therefore, no horizontal sector spanning $\Gamma_{n-1}$ is a sink. We deduce that no horizontal sectors spanning $\Gamma_{n-2} \cup \Gamma_{n-1}$ are sink disks. 
\end{itemize}

Therefore, there are no sink disks spanning $\Gamma_{n-2} \cup \Gamma_{n-1}$ in the homogeneous case.
\end{proof}

\begin{lemma} \label{lemma:Generic_EndOfBraid_LastColumnIsFull}
Let $\beta$ be a positive braid in $\mathcal{B}_n$, where either $n \geq 6$, or $n=5$ with $\Cmax = \Ceven$ and $\Gamma_2$ heterogeneous. %Suppose $n \not \equiv m \mod 2$.
Suppose we build $B$ in $\Gamma_1 \cup \ldots \cup \Gamma_{n-3}$, such that $S_{n-3}$ contains $c_{n-3}-1$ plumbing arcs and at most one image arc. Then we can build $B$ in $\Gamma_{n-2} \cup \Gamma_{n-1}$ such that there are no sink disks spanning $\Gamma_{n-3} \cup \Gamma_{n-2} \cup \Gamma_{n-1}$, and such that there is no pairwise linking between arcs in $\Gamma_{n-3} \cup \Gamma_{n-2} \cup \Gamma_{n-1}$. 
\end{lemma}

\begin{rmk}
\textup{We will prove the analogue of \Cref{lemma:Generic_EndOfBraid_LastColumnIsFull} when $n=5$ and $\Gamma_2$ is not heterogeneous in \Cref{thm:PositiveNBraids_Even}.}% \Cref{section:Finale}, where we prove \Cref{thm:main}.}
\end{rmk}

\textit{Proof of \Cref{lemma:Generic_EndOfBraid_LastColumnIsFull}.} %\begin{proof}
First, we establish the presentation of $\beta$ that we will use for the proof of this lemma. %Notice that, by assumption, we already applied our template to $\Gamma_{n-3}$. So, in particular, $S_{n-3}$ contains plumbing arcs. 
We conjugate $\beta$ so that:
        \begin{itemize}
            \item $\beta$ begins with $\sigma_{n-3}$, and 
            %\item $\alpha_{n-3,1}$ is a left pointer, and 
            \item $\mathbbm{b}_{n-3, c_{n-3}}$ is not enclosed, on the right, by a plumbing arc.
        \end{itemize}

In particular, this means that $\beta = \sigma_{n-3} \ w_1 \ \sigma_{n-2} \ w_2 \ \sigma_{n-1} \ w_3$, where the identified $\sigma_{n-2}$ is the first letter of its kind in $\beta$. Without loss of generality, we may assume that $w_1$ contains no $\sigma_{n-1}$ letters (if it did, we may use the braid relations and then cyclically conjugate $\beta$ to move them into $w_3$). This is the presentation of $\beta$ that we use going forward: $\beta = \sigma_{n-3} w_1 \sigma_{n-2} w_2$, where $w_1$ contains no $\sigma_{n-2}$ or $\sigma_{n-1}$ letters. See \Cref{fig:LastThreeColumns_Heterogeneous} for an example. 

To prove the lemma, we will need to consider two cases:

        \begin{enumerate}
        \item[$\boldsymbol{\circ}$] \textbf{Case (1)}: $\Gamma_{n-3}$ is heterogenous, i.e. when $d(n-3; 1, c_{n-3} - 1) > 0$.
        \item[$\boldsymbol{\circ}$] \textbf{Case (2)}: $\Gamma_{n-3}$ is homogeneous, i.e. when $d(n-3; 1, c_{n-3}-1) = 0$.
        \end{enumerate}

%%%%%%
\tcbox[size=fbox, colback=gray!30]{Sink disk analysis for \textbf{Case (1)}: $\Gamma_{n-3}$ is heterogeneous.} 

We begin with the heterogeneous case, and establish some preliminaries. 

\textbf{Observation \#1:} Suppose we build $B$ in $\Gamma_{n-3}$ by applying our template (as in \Cref{defn:Templates}) to $\Gamma_{n-3}$. Since we build $B$ in $\Gamma_{n-3}$ in a standard way, we can label bands in this column accordingly: we let $\alpha_{n-3,p}$ denote the last left pointing plumbing arc in $S_{n-3}$, and let $\delta = d(n-3; p,p+1)$. (Note that since $\Gamma_{n-3}$ is heterogeneous, $p \leq c_{n-3}-2$, and that $\delta \geq 1$; see \Cref{fig:LastThreeColumns_Heterogeneous} for some examples.) Therefore, the arcs $\alpha_{n-3, 1}, \ldots, \alpha_{n-3, p}$ are all left pointers, and $\alpha_{n-3, p+1}, \ldots, \alpha_{n-3, c_{n-3}-1}$ are all right pointers.

This setup tells us information about how the bands in $\Gamma_{n-3}$ are enclosed on the left and right:
    \begin{itemize}
        \item[$\boldsymbol{\circ}$] The bands $\mathbbm{b}_{n-3,2}, \ldots, \mathbbm{b}_{n-3, p+1}$ are all enclosed, on the right, by right pointing image arcs.
	\item[$\boldsymbol{\circ}$] The bands $\mathbbm{b}_{n-3, p+2}, \ldots, \mathbbm{b}_{n-3, c_{n-3}}$ are all enclosed, on the right, by left pointing image arcs. 
	\item[$\boldsymbol{\circ}$] The arc $\varphi(\alpha_{n-3, p})$ encloses the bands $\mathbbm{b}_{n-2,1}, \ldots, \mathbbm{b}_{n-2,\delta}$ on the left. 
	\item[$\boldsymbol{\circ}$] The bands $\mathbbm{b}_{n-2,\delta+1}, \ldots, \mathbbm{b}_{n-2,c_{n-2}}$ are either enclosed, on the left, by left pointing image arcs, or not enclosed by image arcs at all.
    \end{itemize}

See \Cref{fig:LastThreeColumns_Heterogeneous} for some examples.

\textbf{Observation \#2:} Regardless of the design of $B$ in $\Gamma_{n-3}$, we know something about the distribution of bands in $\Gamma_{n-1}$ with respect to $\Gamma_{n-2}$. Namely, there can be no $i$, where $1 \leq i \leq c_{n-2}$, such that $d(c_{n-2}; i, i+1) = c_{n-1}$ (if there were, that would mean that $\beta$ is a connected sum, which contradictions our assumption that $\widehat{\beta}$ is a prime knot). Therefore, we know that there exists some minimal $j$ such that $d_L(n-1; j, j+1) \geq 1$. Said differently, we let $\mathbbm{b}_{n-1, j}$ denote the first band in $\Gamma_{n-1}$ which is followed by a $\mathbbm{b}_{n-2}$ band, and so that $d_L(n-1; 1, j) = 0$. There are two sub-cases to consider: 
\begin{itemize}
	\item \textbf{Case (1a):} $\Gamma_{n-3}$ is heterogeneous and $j=1$.
	\item \textbf{Case (1b):} $\Gamma_{n-3}$ is heterogeneous and $j \geq 2$.
\end{itemize}

With these preliminary observations in place, we now construct $B$ in $\Gamma_{n-2} \cup \Gamma_{n-1}$ in the case where $\Gamma_{n-3}$ is heterogeneous and show that there are no sink disks spanning $\Gamma_{n-3} \cup \Gamma_{n-2} \cup \Gamma_{n-1}$.

 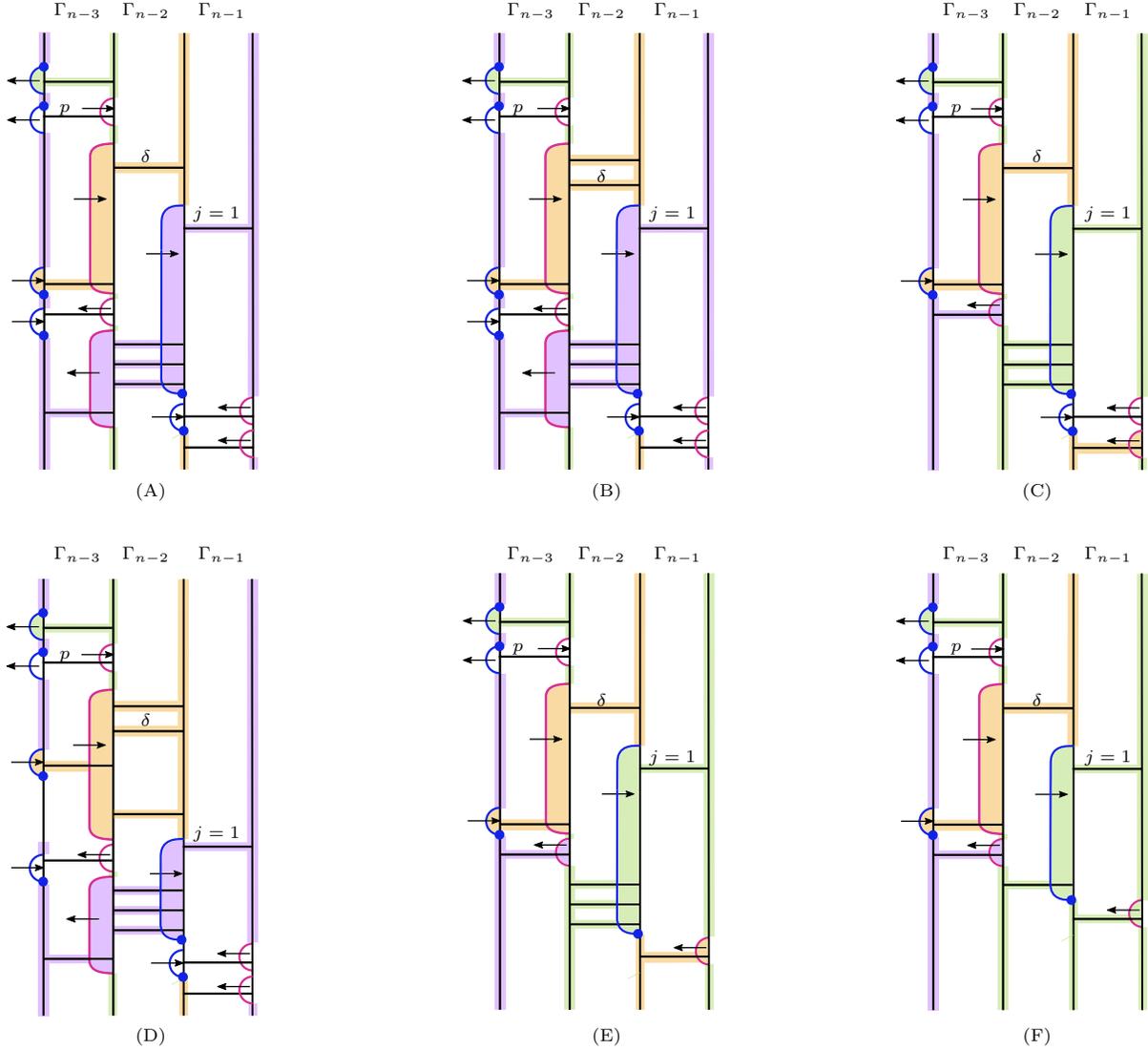
\begin{figure}[h!] \center
\labellist \tiny
\pinlabel {(A)} at 72 225
\pinlabel {(B)} at 252 225
\pinlabel {(C)} at 422 225
\pinlabel {(D)} at 72 10
\pinlabel {(E)} at 252 10
\pinlabel {(F)} at 422 10
% top row
\pinlabel {$\Gamma_{n-3}$} at 43 415
\pinlabel {$\Gamma_{n-2}$} at 70 415
\pinlabel {$\Gamma_{n-1}$} at 100 415
\pinlabel {$\Gamma_{n-3}$} at 222 415
\pinlabel {$\Gamma_{n-2}$} at 250 415
\pinlabel {$\Gamma_{n-1}$} at 280 415
\pinlabel {$\Gamma_{n-3}$} at 394 415
\pinlabel {$\Gamma_{n-2}$} at 422 415
\pinlabel {$\Gamma_{n-1}$} at 450 415
\pinlabel {$j=1$} at 98 335
\pinlabel {$j=1$} at 278 335
\pinlabel {$j=1$} at 450 335
\pinlabel {$p$} at 38 376
\pinlabel {$p$} at 217 376
\pinlabel {$p$} at 390 376
\pinlabel {$\delta$} at 70 358
\pinlabel {$\delta$} at 250 350
\pinlabel {$\delta$} at 422 358
% bottom row
\pinlabel {$\Gamma_{n-3}$} at 43 200
\pinlabel {$\Gamma_{n-2}$} at 70 200
\pinlabel {$\Gamma_{n-1}$} at 100 200
\pinlabel {$\Gamma_{n-3}$} at 222 200
\pinlabel {$\Gamma_{n-2}$} at 250 200
\pinlabel {$\Gamma_{n-1}$} at 280 200
\pinlabel {$\Gamma_{n-3}$} at 394 200
\pinlabel {$\Gamma_{n-2}$} at 422 200
\pinlabel {$\Gamma_{n-1}$} at 450 200
\pinlabel {$j=1$} at 98 90
\pinlabel {$j=1$} at 278 120
\pinlabel {$j=1$} at 450 120
\pinlabel {$p$} at 38 160
\pinlabel {$p$} at 217 163
\pinlabel {$p$} at 390 163
\pinlabel {$\delta$} at 70 135
\pinlabel {$\delta$} at 250 143
\pinlabel {$\delta$} at 422 143
\endlabellist
        \includegraphics[scale=1]{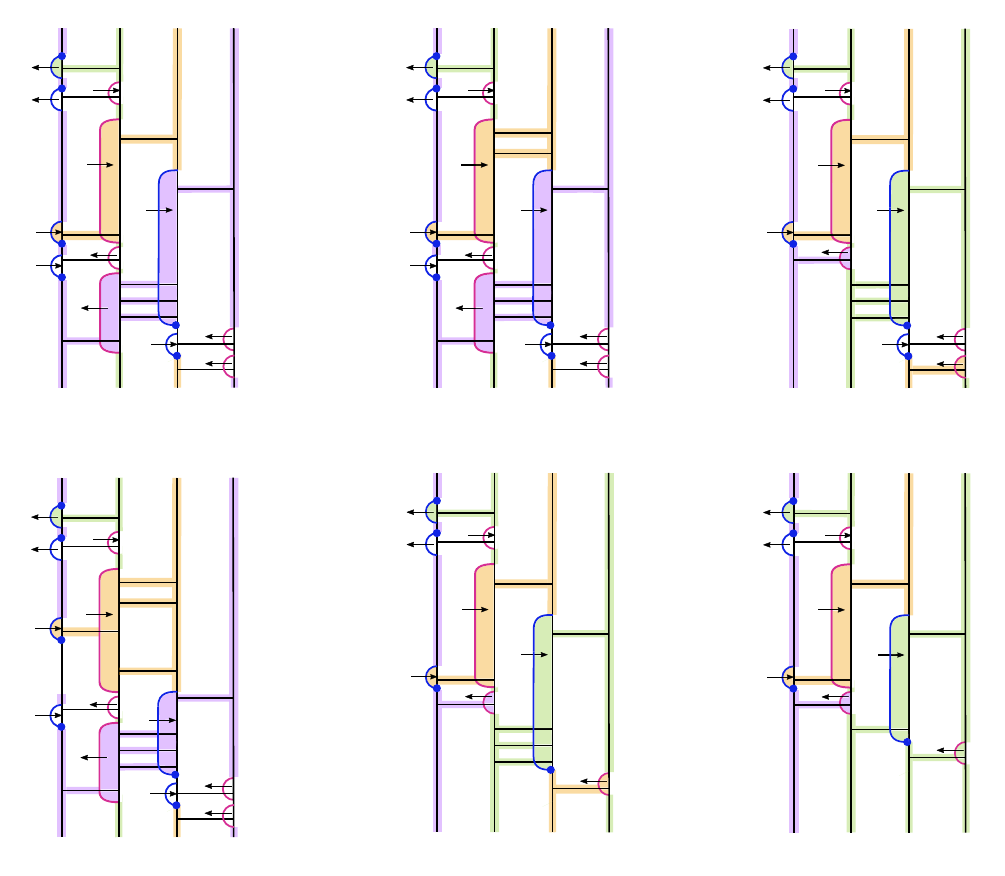}
        \caption{Building $B$ in $\Gamma_{n-3} \cup \Gamma_{n-2} \cup \Gamma_{n-1}$ in Case (1a), where $\Gamma_{n-3}$ is heterogeneous and $j=1$. We apply our template to $\Gamma_{n-3}$, and then finish building the branched surface by coherently orienting all the plumbing arcs in $\Gamma_{n-1}$ to the right. Frames (A) -- (F) show some different possible configurations of $\Gamma_{n-1}$ with respect to $\Gamma_{n-2}$. Let $\delta = d(n-3; p, p+1)$, where $\alpha_{n-3,p}$ is the last left pointing plumbing arc in $S_{n-3}$. (For simplicity, in this figure, $p=2$ in all the frames.) In (A), $\delta=1$. In (B), $\delta=2$. In (C), $\mathcal{S}_n \sim \mathcal{S}_{n-2}$. In (D), the first band following $\mathbbm{b}_{n-1,1}$ is $\mathbbm{b}_{n-2, t}$, where $t > \delta+1$. In (E), $c_{n-1} = 2$ and $c_{n-2} \geq 3$. In (F), $c_{n-1} = c_{n-2} = 2$.} 
         %We adapted the presentation of $\beta$ such that $\mathbbm{b}_{n-2,1}$ is in the same branch sector as $\mathcal{S}_{n-1}$, and so that $\mathbbm{b}_{n-1,1}$ is in the same branch sector as $\mathcal{S}_n$.}
        \label{fig:LastThreeColumns_Heterogeneous}
\end{figure}

\tcbox[size=fbox, colback=gray!10]{\textbf{Case (1a):} $\Gamma_{n-3}$ is heterogeneous and $j = 1$.}

To build $B$ in $\Gamma_{n-2} \ \cup \ \Gamma_{n-1}$, we take no product disks from $\Gamma_{n-2}$, and we co-orient the arcs $\alpha_{n-1, 1}, \alpha_{n-1, 2}, \ldots, \alpha_{n-1, c_{n-1}-1}$ to the right; see \Cref{fig:LastThreeColumns_Heterogeneous}. Note that we chose our presentation of $\beta$ so that $\mathbbm{b}_{n-2,1}$ precedes $\mathbbm{b}_{n-1,1}$. This means, in particular, that  $\mathbbm{b}_{n-2,1}$ is not enclosed, on the right, by any plumbing arcs. We claim that there are no sink disks spanning $\Gamma_{n-3} \cup \Gamma_{n-2} \cup \Gamma_{n-1}$:

\begin{itemize}
\item \Seifert we show that $S_{n-3}, S_{n-2}, S_{n-1}$, and $S_n$ are not sinks.

\begin{itemize}[label={$\bullet$}]
	\item To begin, we show that $\mathcal{S}_{n-3}$ is never a sink. This analysis depends upon  the value of $n$.

        \begin{itemize}[label={$\boldsymbol{\circ}$}]
        %\item Suppose $n=5$. In this case, $n-3 = 2$, and if $n \not \equiv m \mod 2$, then $\Cmax = \Ceven$. Thus, the construction of $B$ in $\Gamma_{n-3} = \Gamma_2$ is inherited from \Cref{lemma:Sparse_Cmax=Ceven_Gammas12}. We want to show that there are no Seifert disk sink disk sectors spanning $\Gamma_2 \cup \Gamma_3 \cup \Gamma_4$. As summarized in \Cref{table:Bookkeeping_Cmax=Ceven_Gammas12}, there is at most one sink disk spanning $\Gamma_2$, and it is not $\mathcal{S}_2 = \mathcal{S}_{n-3}$. 
        %
        \item Suppose $n=5$. In this case, $n-3 = 2$. By the assumptions of this lemma, we know that $\Cmax = \Ceven$ and $\Gamma_{n-3} = \Gamma_2$ is heterogeneous. In particular, the construction of $B$ in $\Gamma_1 \cup \Gamma_2$ is inherited from Case (1A) or (2A) of \Cref{lemma:Sparse_Cmax=Ceven_Gammas12}; see \Cref{fig:NotGeneric_CmaxCeven_Gammas12_CaseB}. In \Cref{lemma:Sparse_Cmax=Ceven_Gammas12}, we proved that $\mathcal{S}_2 = \mathcal{S}_{n-3}$ is not a sink.
        \item Suppose $n = 6$. In this case, $n-3 = 3$, hence $\mathcal{S}_{n-3} = \mathcal{S}_3$ and $\Cmax = \Codd$. Thus, the construction of $B$ in $\Gamma_{n-3} = \Gamma_3$ is inherited from an application of \Cref{lemma:Generic_Cmax=Codd_Gammas12} to $\Gamma_1 \cup \Gamma_2 \cup \Gamma_3$. In particular, $B \cap S_3$ contains $c_3-1$ plumbing arcs and one image arc. 
        
        As summarized in \Cref{table:Bookkeeping_Cmax=Codd_Gammas12}, there were four sub-cases of \Cref{lemma:Generic_Cmax=Codd_Gammas12}. We check that $\mathcal{S}_3$ is not a sink when $\Gamma_{n-3}$ is heterogeneous.
    
            	\begin{itemize}[label={$\bullet$}]
            	\item \textbf{Case (A) of \Cref{lemma:Generic_Cmax=Codd_Gammas12}.} Referring to \Cref{fig:Generic_CmaxCodd_Gammas12}, we see that $\partial \mathcal{S}_3$ contains $\alpha_{2,1}$, $\varphi(\alpha_{2,1})$, and $\alpha_{3,1}, \ldots, \alpha_{3, c_3-1}$. In particular, we assumed that $\Gamma_3$ is heterogeneous, so $\mathcal{S}_3$ contains some outward pointing plumbing arc $\alpha_{3, k}$, so it is not a sink.
            	\item \textbf{Case (B.1) of \Cref{lemma:Generic_Cmax=Codd_Gammas12}.} We are currently assuming that $\Gamma_3$ is heterogeneous. Since Case (B.1) of \Cref{lemma:Generic_Cmax=Codd_Gammas12} assumes that $\Gamma_3$ is homogeneous, we need not consider this case. 
            	\item \textbf{Case (B.2) of \Cref{lemma:Generic_Cmax=Codd_Gammas12}.} As noted in \Cref{table:Bookkeeping_Cmax=Codd_Gammas12}, there are no sink disks spanning $\Gamma_1 \cup \Gamma_2$, so in particular, $\mathcal{S}_3$ is not a sink. 
            	\item \textbf{Case (C) of \Cref{lemma:Generic_Cmax=Codd_Gammas12}.} As noted in \Cref{table:Bookkeeping_Cmax=Codd_Gammas12}, the unique potential sink disk spanning $\Gamma_1 \cup \Gamma_2$ is a horizontal sector $\mathcal{H} \not \sim \mathcal{S}_3$, so as desired, $\mathcal{S}_3$ is not a sink. 
            	\end{itemize}
        \item Suppose $n \geq 7$. \Cref{NoClumpingOfArcs} tells us that $\mathcal{S}_{n-3}$ contains only plumbing arcs. Since $\alpha_{n-3, p+1}$ is a right pointer, it points out of the Seifert disk, and so $\mathcal{S}_{n-3}$ is not a sink.
        \end{itemize}

We deduce that $\mathcal{S}_{n-3}$ is not a sink. Next, we show that $S_{n-2}, S_{n-1}$, $S_n$ are not sinks. 

%\begin{itemize}[label={$\bullet$}]
\item By assumption, $\Gamma_{n-3}$ is heterogeneous, so there exist both left and right pointing plumbing arcs in $S_{n-3}$. Therefore, there are both right and left pointing image arcs in $S_{n-2}$. As we include zero product disks from $\Gamma_{n-2}$ into $B$, $S_{n-2}$ contains only image arcs. Therefore, $\partial S_{n-2}$ contains an outward pointing arc, and the sector is not a sink. See \Cref{fig:LastThreeColumns_Heterogeneous}. 
\item By design, $S_{n-1}$ contains only right pointing plumbing arcs, so $\mathcal{S}_{n-1}$ is not a sink. 
\item Finally, we analyze $\mathcal{S}_n$. By design, $\mathbbm{b}_{n-1, 1}$ is not enclosed, on the right, by any arc; therefore, $\mathcal{S}_n \sim \mathbbm{b}_{n-1, 1}$. In particular, $\partial \mathcal{S}_n$ contains $\alpha_{n-1,1}$. Since $j=1$, then by assumption, $\mathbbm{b}_{n-1,1}$ is followed by some band in $\Gamma_{n-2}$; we call this band $\mathbbm{b}_{n-2, k}$. In particular, notice that $k \geq \delta +1$, and $\alpha_{n-1,1}$ encloses $\mathbbm{b}_{n-2,k}$ on the right; see \Cref{fig:LastThreeColumns_Heterogeneous}. There are now two cases to analyze separately: either $\mathbbm{b}_{n-2,k}$ is enclosed, on the left, by an image arc, or $\mathbbm{b}_{n-2,k}$   is not bounded, on the left, by any image arcs. %We refer to \Cref{fig:LastThreeColumns_Heterogeneous} throughout. 
	\begin{itemize}[label={$\boldsymbol{\circ}$}]
	\item[$\circ$] Suppose $\mathbbm{b}_{n-2,k}$ is enclosed, on the left, by an image arc (as in frames (A), (B), and (D) of \Cref{fig:LastThreeColumns_Heterogeneous}). We see that it must be a left pointer: the only bands in $\Gamma_{n-2}$ which are enclosed, on the left, by right pointers are the bands $\mathbbm{b}_{n-2,1}, \ldots, \mathbbm{b}_{n-2, \delta}$; the remaining bands are all enclosed, on the left, by left pointers. Since $k \geq \delta+1$, $\mathbbm{b}_{n-2, k}$ is enclosed, on the left, by a left pointer. We deduce that $\mathcal{S}_{n}$ is not a sink, because $\partial \mathcal{S}_n$ contains the outward pointing arc. 
	\item[$\circ$] Suppose $\mathbbm{b}_{n-2, k}$ is not enclosed, on the left, by any arc (as in frames (C), (E), and (F) of \Cref{fig:LastThreeColumns_Heterogeneous}). In this case, $\mathbbm{b}_{n-2,k}$ is contained in $\mathcal{S}_{n-2}$, which we already proved is not a sink. Therefore, we can deduce that $\mathcal{S}_n$ is not a sink, as desired.
\end{itemize}
\end{itemize}

We deduce that $\mathcal{S}_n$ is not a sink, so none of the Seifert disks sectors spanning $\Gamma_{n-3} \cup \Gamma_{n-2} \cup \Gamma_{n-1}$ are sink disks. 

\item \Polygon Polygon sectors exist when a Seifert disk contains both plumbing and image arcs. Since $S_{n-2}, S_{n-1},$ and $S_n$ contain only one type of arc, none of them can contain polygon sectors. It remains to see whether or not $S_{n-3}$ contains both plumbing and image arcs. This is dependent on the value of $n$.

\begin{itemize}[label={$\bullet$}]
%
%\item Suppose $n=5$. The construction of $B$ in $\Gamma_{n-3} = \Gamma_2$ is inherited from \Cref{lemma:Sparse_Cmax=Ceven_Gammas12}, where we proved that there are no polygon sink disk sectors spanning $\Gamma_1 \cup \Gamma_2$, as desired. 
%
\item Suppose $n=5$. By the assumptions of this lemma, we know that $\Cmax = \Ceven$ and $\Gamma_{n-3} = \Gamma_2$ is heterogeneous. In particular, the construction of $B$ in $\Gamma_1 \cup \Gamma_2$ is inherited from Case (2B) of \Cref{lemma:Sparse_Cmax=Ceven_Gammas12}. By design, we know that there is a unique polygon sector in $S_2$, and we proved it is not a sink in \Cref{lemma:Sparse_Cmax=Ceven_Gammas12}.
\item Suppose $n=6$. In this case, the construction of $B$ in $\Gamma_{n-3} = \Gamma_3$ is inherited from \Cref{lemma:Generic_Cmax=Codd_Gammas12}. We check that there are no polygon sinks in this case, based on the four sub-cases of \Cref{lemma:Generic_Cmax=Codd_Gammas12}.
	\begin{itemize}[label={$\boldsymbol{\circ}$}]
	\item \textbf{Case (A) of \Cref{lemma:Generic_Cmax=Codd_Gammas12}.} Referring to \Cref{fig:Generic_CmaxCodd_Gammas12}, we see that $\partial \mathcal{S}_3$ contains $\alpha_{2,1}$, $\varphi(\alpha_{2,1})$, and $\alpha_{3,1}, \ldots, \alpha_{3, c_3-1}$. In particular, the only branch sector co-bounded by the unique image arc in $S_3$ and the plumbing arcs in $S_3$ is $\mathcal{S}_3$. Hence, $S_3$ does not contain any polygon sectors. 
	\item \textbf{Case (B.1) of \Cref{lemma:Generic_Cmax=Codd_Gammas12}.} We are assuming that $\Gamma_3$ is heterogeneous. However, Case (B.1) of \Cref{lemma:Generic_Cmax=Codd_Gammas12} assumes that $\Gamma_3$ is homogeneous, so we need not consider this case.
	\item \textbf{Case (B.2) of \Cref{lemma:Generic_Cmax=Codd_Gammas12}.} As noted in \Cref{table:Bookkeeping_Cmax=Codd_Gammas12}, there are no sink disks spanning $\Gamma_1 \cup \Gamma_2$, so in particular, there are no polygon sink disk sectors in $S_3$. 
	\item \textbf{Case (C) of \Cref{lemma:Generic_Cmax=Codd_Gammas12}.} As seen in \Cref{table:Bookkeeping_Cmax=Codd_Gammas12}, no polygon sector in $S_3$ is a sink.
	\end{itemize}
\item Suppose $n \geq 7$. As noted in \Cref{NoClumpingOfArcs}, the Seifert disk $S_{n-3}$ must contain only plumbing arcs, so it cannot contain any polygon sectors.
\end{itemize}
We deduce that there are no polygon sink disk sectors spanning $\Gamma_{n-3} \cup \Gamma_{n-2} \cup \Gamma_{n-1}$. 
\item \Horizontal Finally, we analyze the horizontal sectors spanning $\Gamma_{n-3} \cup \Gamma_{n-2} \cup \Gamma_{n-1}$. Throughout, we let $\mathcal{H}$ denote a horizontal sector, and we consult \Cref{fig:LastThreeColumns_Heterogeneous}.

	\begin{itemize}
	\item[$\circ$] Suppose $\mathcal{H}$ spans $\Gamma_{n-3}$.
        			\begin{itemize}[label={$\bullet$}]
        			\item If $\mathcal{H} \sim \mathbbm{b}_{n-3,i}$, for $1 \leq i \leq p$, then $\mathcal{H}$ is not a sink, because $\partial \mathcal {H}$ contains $\alpha_{n-3, i}$, and these arcs all point into $S_{n-3}$.
        			\item  If $\mathcal{H} \sim \mathbbm{b}_{n-3, p+1}$, then $\mathcal{H} \sim \mathbbm{b}_{n-2, 1}, \ldots, \mathbbm{b}_{n-2, \delta}$. However, we know that $\mathbbm{b}_{n-2,1}$ is not enclosed, on the right, by any arc. Therefore, $\mathcal{H} \sim S_{n-1}$, and we already proved this sector is not a sink.
        			\item  Finally, if $\mathcal{H} \sim \mathbbm{b}_{n-3, i}$, for $p+2 \leq i \leq c_{n-3}$, then $\mathcal{H}$ is enclosed, on the right, $\varphi(\alpha_{n-3, i-1})$, and this arc points out of $\mathcal{H}$ and into $S_{n-2}$.
        			\end{itemize}
                We deduce that any horizontal sector $\mathcal{H}$ spanning $\Gamma_{n-3}$ is not a sink. 
	\item[$\circ$] Suppose $\mathcal{H}$ spans $\Gamma_{n-2}$.
        			\begin{itemize}[label={$\bullet$}]
        			\item If $\mathcal{H} \sim \mathbbm{b}_{n-2, i}$, for $1\leq i \leq \delta$, then $\mathcal{H} \sim S_{n-1}$, which we already argued is not a sink. 
        			\item Otherwise, $\mathcal{H} \sim \mathbbm{b}_{n-2, i }$ where $ \delta+1 \leq i \leq c_{n-2}$. Either $\mathbbm{b}_{n-2,i}$ is not enclosed, on the left, by an image arc (so it is in the same sector as $\mathcal{S}_{n-2}$, which we already proved is not a sink), or it is enclosed, on the left, by an image arc $\varphi(\alpha_{n-3, j})$, where $ p+2 \leq j \leq c_{n-3}$; these arcs point out of $\mathcal{H}$ and into in $S_{n-2}$, so $\mathcal{H}$ is not a sink. 
        			\end{itemize}
                We deduce that any horizontal sector $\mathcal{H}$ spanning $\Gamma_{n-2}$ is not a sink. 
	\item[$\circ$] Suppose $\mathcal{H}$ spans $\Gamma_{n-1}$. 
        			\begin{itemize}[label={$\bullet$}]
        			\item Suppose $\mathcal{H} \sim \mathbbm{b}_{n-1,1}$. Since this band is not enclosed, on the right, by any arc, we know that $\mathbbm{b}_{n-1,1} \sim \mathcal{S}_n$. Since $\mathcal{S}_n$ is not a sink, neither is $\mathcal{H}$. 
        			\item Instead suppose that $\mathcal{H}$ contains $\mathbbm{b}_{n-1, s}$, where $2 \leq s \leq c_n$. The band $\mathbbm{b}_{n-1, s}$ is enclosed, on the right, by $\varphi(\alpha_{n-1,s-1})$; this arc points out of $\mathbbm{b}_{n-1,s}$ and into $S_n$, so $\mathcal{H}$ is not a sink.
        			\end{itemize}
            We deduce that no horizontal sector $\mathcal{H}$ spanning $\Gamma_{n-1}$ is a sink.
	\end{itemize}
	Therefore, no horizontal sector spanning $\Gamma_{n-3} \cup \Gamma_{n-2} \cup \Gamma_{n-1}$ is a sink. 
\end{itemize}

This concludes the proof the $j=1$ sub-case of the heterogenous case. We now move onto the $j \geq 2$ case, and refer to \Cref{fig:LastThreeColumns_Heterogeneous_Case2} throughout.

\tcbox[size=fbox, colback=gray!10]{\textbf{Case (1b):} $\Gamma_{n-3}$ is heterogeneous and $j \geq 2$} 

First, we describe the construction of the branched surface in $\Gamma_{n-1}$: we co-orient $\alpha_{n-1, 1}, \ldots, \alpha_{n-1, j-1}$ to the left, and co-orient $\alpha_{n-1, j}, \ldots, \alpha_{n-1, c_{n-1}-1}$ to the right; see \Cref{fig:LastThreeColumns_Heterogeneous_Case2} for an example. We note: since $\widehat{\beta}$ is not a connected sum, if $j \geq 2$, then $c_{n-1} \geq j+1$, hence $c_{n-1} \geq 3$. We claim that there are no sink disk spanning $\Gamma_{n-3} \cup \Gamma_{n-2} \cup \Gamma_{n-1}$:

\begin{itemize}
\item \Seifert applying the same arguments from the $j=1$ case, we see that $\mathcal{S}_{n-3}$ and $\mathcal{S}_{n-2}$ are not sinks. 

Recall that we defined $j$ so that $d_L(n-1; j, j+1) \geq 1$, and $\mathbbm{b}_{n-1,j}$ is the first band in $\Gamma_{n-1}$ for which this holds. In particular, since $j \geq 2$, we know that $c_{n-1} \geq 3$. In particular, since $\alpha_{n-1, 1}, \ldots, \alpha_{n-1, j-1}$ are all left pointers, and $\alpha_{n-1, j}, \ldots, \alpha_{n-1, c_{n-1}-1}$ are all right pointers, we see that $S_{n-1}$ and $S_n$ each contain both left and right pointing arcs. We deduce that neither $\mathcal{S}_{n-1}$ and $\mathcal{S}_n$ are sinks. 
\item \Polygon The same argument as in the $j=1$ case shows that there are no sink disk polygon sectors spanning $\Gamma_{n-3} \cup \Gamma_{n-2} \cup \Gamma_{n-1}$.  
\item \Horizontal Suppose $\mathcal{H}$ is a horizontal sector spanning $\Gamma_{n-3} \cup \Gamma_{n-2} \cup \Gamma_{n-1}$. 
	\begin{itemize}
	\item[$\circ$] If $\mathcal{H}$ spans $\Gamma_{n-1}$, we can apply the proof from the $j=1$ case to see that $\mathcal{H}$ is not a sink. 
	\item[$\circ$] If $\mathcal{H}$ spans $\Gamma_{n-2}$, the proof from the $j=1$ case applies, and we see that $\mathcal{H}$ is not a sink. 
	\item[$\circ$] Suppose $\mathcal{H}$ spans $\Gamma_{n-1}$.
	\begin{itemize}
		      \item[$\bullet$] If $\mathcal{H} \sim \mathbbm{b}_{n-1, i}$ where $1 \leq i \leq j-1$, then $\partial \mathcal{H}$ contains $\alpha_{n-1,i}$, which points out of $\mathcal{H}$ and into $S_{n-1}$. Therefore, $\mathcal{H}$ is not a sink. 
		      \item[$\bullet$] If $\mathcal{H} \sim \mathbbm{b}_{n-1, j}$, then $\mathcal{H}$ also contains some band $\mathbbm{b}_{n-2, k}$, where $k \geq 2$. (Note that our presentation of $\beta$ ensures that $k \neq 1$, and recall that $\delta = d(n-3; p, p+1)$.)
		
        		\begin{itemize}
        		          \item[$\circ$] If $2 \leq k \leq \delta$, then $\mathcal{H}$ meets $S_{n-2}$ in $\varphi(\alpha_{n-3,p})$. Therefore, we know that $\mathcal{H} \sim \mathbbm{b}_{n-2,1} \sim \mathbbm{b}_{n-2, 2} \sim \ldots \sim \mathbbm{b}_{n-2, \delta}$. But $\mathbbm{b}_{n-2,1} \sim \mathcal{S}_{n-1}$, which we already argued is not a sink. We deduce that $\mathcal{H} \sim \mathcal{S}_{n-1}$, so $\mathcal{H}$ is not a sink. 
        		          \item[$\circ$] If $k \geq \delta+1$, then $\mathcal{H}$ meets $S_{n-2}$ in a left pointing image arc, or it is not enclosed by any arc (and $\mathcal{H} \sim \mathcal{S}_{n-2}$); see \Cref{fig:LastThreeColumns_Heterogeneous_Case2}. In both cases, we can deduce that $\mathcal{H}$ is not a sink.
        		\end{itemize}
          
	           \item[$\bullet$]  Finally, suppose $\mathcal{H}$ contains $\mathbbm{b}_{n-1, i}$ for $j+1 \leq i \leq c_{n}$. The band $\mathbbm{b}_{n-1, i}$ is enclosed, on the right, by $\varphi(\alpha_{n-1, j-1})$; this arc points out of $\mathbbm{b}_{n-1, i}$, and so $\mathcal{H}$ is not a sink.
            
	\end{itemize}
	\end{itemize}
Therefore, we have no horizontal sink disk sectors spanning $\Gamma_{n-3} \cup \Gamma_{n-2} \cup \Gamma_{n-1}$. 
\end{itemize}

\begin{figure}[h!] \center
\labellist
\tiny
%% top row
\pinlabel {$\mathcal{H}$} at 94 331
\pinlabel {$\mathcal{H}$} at 264 278
\pinlabel {$\Gamma_{n-3}$} at 44 400
\pinlabel {$\Gamma_{n-2}$} at 72 400
\pinlabel {$\Gamma_{n-1}$} at 100 400
\pinlabel {$\Gamma_{n-3}$} at 215 400
\pinlabel {$\Gamma_{n-2}$} at 242 400
\pinlabel {$\Gamma_{n-1}$} at 270 400
\pinlabel {$j=2$} at 98 322
\pinlabel {$j=2$} at 268 268
\pinlabel {$p$} at 38 362
\pinlabel {$p$} at 210 362
\pinlabel {$\delta$} at 65 328
\pinlabel {$\delta$} at 235 336

%bottom row
\pinlabel {$\mathcal{H}$} at 94 70
\pinlabel {$\mathcal{H}$} at 264 70
\pinlabel {$\Gamma_{n-3}$} at 44 185
\pinlabel {$\Gamma_{n-2}$} at 72 185
\pinlabel {$\Gamma_{n-1}$} at 100 185
\pinlabel {$\Gamma_{n-3}$} at 215 185
\pinlabel {$\Gamma_{n-2}$} at 242 185
\pinlabel {$\Gamma_{n-1}$} at 270 185
\pinlabel {$j=2$} at 98 62
\pinlabel {$j=2$} at 268 62
\pinlabel {$p$} at 38 152
\pinlabel {$p$} at 210 152
\pinlabel {$\delta$} at 65 132
\pinlabel {$\delta$} at 235 132
% 
%
%\pinlabel {$\mathbbm{b}_{i+1,1}$} at 60 352
%\pinlabel {$\mathbbm{b}_{i+1,2}$} at 258 352
%\pinlabel {$\mathbbm{b}_{i+1,1}$} at 60 90
%\pinlabel {$\mathbbm{b}_{i+1,2}$} at 258 90
%
\pinlabel {(A)} at 72 210
\pinlabel {(B)} at 242 210
\pinlabel {(D)} at 72 5
\pinlabel {(E)} at 242 5
%
%
%\pinlabel {$j=1$} at 98 90
%\pinlabel {$j=1$} at 278 120
%\pinlabel {$j=1$} at 450 120
%%
\endlabellist
        \includegraphics[scale=1.2]{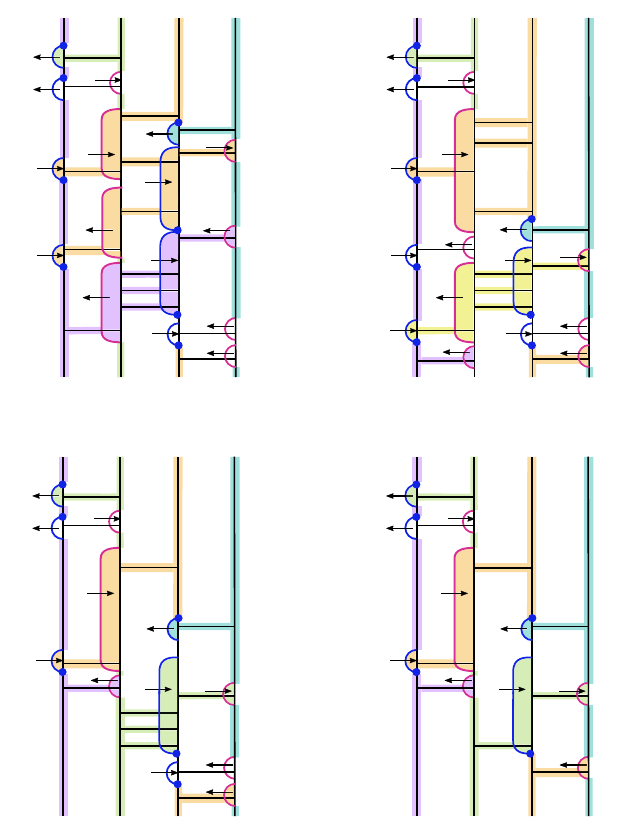}
        \caption{Building $B$ in $\Gamma_{n-3} \cup \Gamma_{n-2} \cup \Gamma_{n-1}$ in Case (1b), where $\Gamma_{n-3}$ is heterogeneous and $j\geq 2$ (for simplicity of the figures, we let $j=2$ throughout). We apply our template to $\Gamma_{n-3}$, and then finish building the branched surface co-orienting $\alpha_{n-1,1} , \ldots, \alpha_{n-1, j-1}$ to the left, and $\alpha_{n-1,j}, \ldots, \alpha_{n-1, c_{n-1}-1}$ to the right. Frames (A) -- (D) show some different possible configurations of $\Gamma_{n-1}$ with respect to $\Gamma_{n-2}$. Recall that $\delta := d(n-3; p, p+1)$, where $\alpha_{n-3,p}$ is the last left pointing plumbing arc in $S_{n-3}$. (For simplicity, in this figure, $p=2$ in all the frames.) Finally, let $\mathcal{H}$ denote the horizontal sector containing $\mathbbm{b}_{n-1,j}$. In (A), $\mathcal{H}$ meets $S_{n-2}$ in a right pointing image arc, and so $\mathcal{H} \sim \mathcal{S}_{n-1}$. In (B), $\mathcal{H}$ meets $S_{n-2}$ in a left pointing image arc. In (C), $\mathcal{H}$ is not enclosed by any arcs when it meets $S_{n-2}$, and so $\mathcal{H} \sim \mathcal{S}_{n-2}$. In (D), $c_{n-2} = 2$.} 
        \label{fig:LastThreeColumns_Heterogeneous_Case2}
\end{figure}
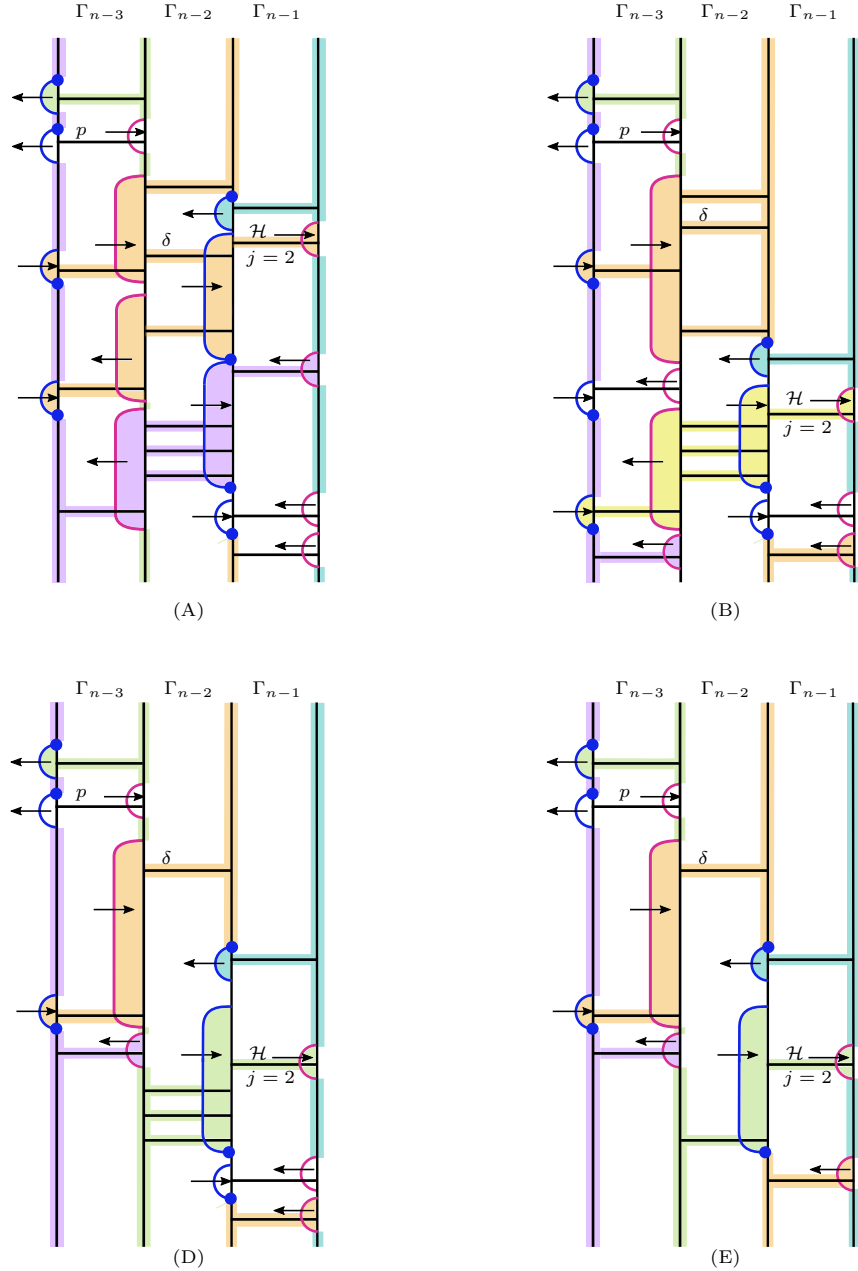

We see that there are no sink disks spanning $\Gamma_{n-3} \cup \Gamma_{n-2} \cup \Gamma_{n-1}$ in the $j\geq 2$ subcase of the heterogeneous case. This concludes the proof of the case where $\Gamma_{n-3}$ is heterogeneous, and we can proceed to study the case where $\Gamma_{n-3}$ is homogeneous.

\tcbox[size=fbox, colback=gray!30]{Sink disk analysis for \textbf{Case (2):} $\Gamma_{n-3}$ is homogeneous.}

We now turn to the case where $\Gamma_{n-3}$ is homogeneous, i.e. $d(n-3; 1, c_{n-3}-1) = 0$.  As in Cases (1a) and (1b), we let $\delta := d(n-3; c_{n-3}-1, c_{n-3})$; additionally, we let  $\mathbbm{b}_{n-1,j}$ denote the first band in $\Gamma_{n-1}$ such that $d_L(n-1; j, j+1) \geq 1$. 

We make a few quick observations. 

\textbf{Observation \#1:} since $\Gamma_{n-3}$ is homogeneous, the bands $\mathbbm{b}_{n-2, 1}, \ldots, \mathbbm{b}_{n-2, \delta}$ are all enclosed, on the left, by $\varphi(\alpha_{c_{n-3}, c_{n-3}-1})$, and that the bands $\mathbbm{b}_{n-2, \delta+1}, \ldots, \mathbbm{b}_{n-2, c_{n-2}}$ are not enclosed, on the left, by any arcs. (One can refer  to Figures \ref{fig:LastThreeColumns_Homogeneous} and \ref{fig:LastThreeColumns_Homogeneous_Case2}.) 

\textbf{Observation \#2:} since $\Gamma_{n-3}$ is homogeneous, $d(n-3; c_{n-3}-1, c_{n-3}) + d(n-3; c_{n-3}, 1) = c_{n-3}$, where each term on the left is strictly positive. In particular, all the bands in $\Gamma_{n-2}$ are concentrated between three consecutive bands in $\Gamma_{n-3}$: $\mathbbm{b}_{n-3, c_{n-3}-1}$, $\mathbbm{b}_{n-3, c_{n-3}}$, and $\mathbbm{b}_{n-3,1}$. Finally, this also means that $d_L(n-2; 1, \delta) = 0$ and $d_L(n-2; \delta+1, c_{n-2}) = 0$, as in Figures \ref{fig:LastThreeColumns_Homogeneous} and \ref{fig:LastThreeColumns_Homogeneous_Case2}.

With these preliminaries in place, we can proceed with building $B$ in $\Gamma_{n-2} \cup \Gamma_{n-1}$. As in the case where $\Gamma_{n-3}$ is heterogeneous (i.e. Cases (1a) and (1b)), we have two sub-cases to consider: 

\begin{itemize}
	\item \textbf{Case (2a):} $\Gamma_{n-3}$ is homogeneous and $j=1$.
	\item \textbf{Case (2b):} $\Gamma_{n-3}$ is homogeneous and $j \geq 2$.
\end{itemize}

\tcbox[size=fbox, colback=gray!10]{\textbf{Case (2a):} $\Gamma_{n-3}$ is homogeneous and $j=1$}

We begin by building $B$ in $\Gamma_{n-2} \cup \Gamma_{n-1}$: we take no product disks from $\Gamma_{n-2}$, and we take $c_{n-1}-1$ disks from $\Gamma_{n-1}$ by co-orienting $\alpha_{n-1,1}, \ldots, \alpha_{n-1, c_{n-1}-1}$ to the right. 

Since $j=1$, we must have that $d(n-2, 1, \delta) = 0$ (otherwise, there would exist some $s$ such that $1 \leq s \leq \delta$ such that $d(n-2; s, s+1) = 1$, and this contradicts that $\beta$ is standardized). Analogously, we see that there cannot exist some $s$ with $\delta+1 \leq s \leq c_{n-2}$ such that $d(n-2; s, s+1) = 1$. We deduce that $\mathbbm{b}_{n-1, 1}$ must occur at the same ``level'' as $\mathbbm{b}_{n-3, c_{n-3}}$, as seen in \Cref{fig:LastThreeColumns_Homogeneous}. More precisely, we have that $d_L(n-2; \delta, \delta+1) = d(n-2; \delta, \delta+1) = 1$. 

With this established, we can now prove that there are no sink disks spanning $\Gamma_{n-3} \cup \Gamma_{n-2} \cup \Gamma_{n-1}$; we refer to \Cref{fig:LastThreeColumns_Homogeneous} throughout.

\begin{itemize}
\item \Seifert We first show that $\mathcal{S}_{n-3}$ is not a sink. This requires analyzing the cases $n=6$ and $n \geq 7$ separately. 

	\begin{itemize}[label={$\bullet$}]
%	\item Suppose $n=5$. In this case, we have that $S_{n-3} = S_2$. Our construction in \Cref{lemma:Sparse_Cmax=Ceven_Gammas12} never co-orients all plumbing arcs in $S_2$ to the left; thus, we may assume that $n \geq 6$. 
	%
	\item Suppose $n = 6$. In this case, $n-3 = 3$, hence $\mathcal{S}_{n-3} = \mathcal{S}_3$ and $\Cmax = \Codd$. Thus, the construction of $B$ in $\Gamma_{n-3} = \Gamma_3$ is inherited from an application of \Cref{lemma:Generic_Cmax=Codd_Gammas12} to $\Gamma_1 \cup \Gamma_2 \cup \Gamma_3$. In particular, $B \cap S_3$ contains $c_3-1$ plumbing arcs and one image arc. Moreover, by consulting \Cref{table:Bookkeeping_Cmax=Codd_Gammas12}, we see that if $\Gamma_3$ is homogeneous, then we are either in Case (A) or Case (B.1) of \Cref{lemma:Generic_Cmax=Codd_Gammas12}. So, we must argue that for these sub-cases, $S_{n-3}$ is not a sink when we build $B$ in $\Gamma_{n-1} = \Gamma_{5}$ as above. 

        	\begin{itemize}[label={$\boldsymbol{\circ}$}]
        		      \item In Case (A) of \Cref{lemma:Generic_Cmax=Codd_Gammas12}, we know that $\partial \mathcal{S}_3$ contains $\varphi(\alpha_{2,1}), \alpha_{3, 1}, \ldots, \alpha_{3, c_3-1}$ (see \Cref{fig:Generic_CmaxCodd_Gammas12}), and that $\mathcal{S}_3 \sim \mathbbm{b}_{3, c_3}$. 
        		      \item In Case (B.1) of \Cref{lemma:Generic_Cmax=Codd_Gammas12}, we consult \Cref{table:Bookkeeping_Cmax=Codd_Gammas12} to see that $\mathcal{S}_3$ is not a sink.
          %know that $\mathcal{S}_3$ is not a sink because the unique image arc in $S_3$ points out of the sector; see \Cref{fig:Generic_CmaxCodd_Gammas12_A}. % the unique potential sink disk in $\Gamma_3$ is the sector containing $\mathbbm{b}_{3, c_3}$. 
        	\end{itemize}
	
        Therefore, to prove that $\mathcal{S}_{n-3}$ is not a sink when $n=6$, it suffices to show that the sector containing $\mathbbm{b}_{3, c_3}$ is not a sink. In fact, this observation also holds in the $n \geq 7$ case:
	
	\item If $n \geq 7$, then by \Cref{NoClumpingOfArcs}, $S_{n-3}$ contains only plumbing arcs, so $\mathcal{S}_{n-3} \sim \mathbbm{b}_{n-3, c_{n-3}}$.
	\end{itemize}

Therefore, we need to show that when $n \geq 6$, the sector containing $\mathbbm{b}_{n-3, c_{n-3}}$ is not a sink. Let $\mathcal{S}$ denote this sector. By design, the sector $\mathcal{S} \sim \mathbbm{b}_{n-2, 1}$. This band is not enclosed, on the right, by any bands, so $\mathcal{S}_{n-3} \sim \mathcal{S}_{n-1}$. As $S_{n-1}$ contains only right pointing plumbing arcs, the arc $\alpha_{{n-1},1}$ points out of $\mathcal{S}_{n-1}$. Thus, $\mathcal{S} \sim \mathcal{S}_{n-1} \sim \mathcal{S}_{n-3}$ is not a sink, as desired.

It remains to show that $\mathcal{S}_{n-2}$ and $\mathcal{S}_n$ are not sinks.  All the arcs in $S_{n-2}$ are right pointing image arcs, so $\mathcal{S}_{n-2}$ is not a sink. 

Finally, by design, we see that $\mathcal{S}_{n} \sim \mathbbm{b}_{n-1, 1} \sim \mathbbm{b}_{n-2, \delta+1}$; see \Cref{fig:LastThreeColumns_Homogeneous}. However, we already noticed that $\mathbbm{b}_{n-2, \delta+1}$ is not enclosed, on the left, by any arcs. We deduce that $\mathcal{S}_{n} \sim \mathcal{S}_{n-2}$, which we already argued is not a sink. 

We deduce that none of the Seifert disk sectors spanning $\Gamma_{n-3} \cup \Gamma_{n-2} \cup \Gamma_{n-1}$ are sinks.
\item \Polygon As above, we analyze the $n=6$ and $n \geq 7$ cases separately.

	\begin{itemize}[label={$\bullet$}]
	\item[$\circ$] If $n=6$, $S_{n-3}$ contains both plumbing and image arcs, so $S_{n-3} = S_3$ could contain a polygon sector. We are assuming that $\Gamma_3$ is homogeneous, putting us in Cases (A) and (B.1) of \Cref{lemma:Generic_Cmax=Codd_Gammas12}. However, as we saw in the proof of that lemma, there are no polygon sectors in $\Gamma_3$ in Case (A); see  \Cref{fig:Generic_CmaxCodd_Gammas12}. In Case (B.1), there is a sector which appears like it could be a polygon sector, but it contains a band in $\Gamma_3$ (see \Cref{fig:Generic_CmaxCodd_Gammas12_A}), so it does not lie in $S_{n-3}$. We deduce that there are no polygon sectors in $S_3$, regardless of which sub-case of \Cref{lemma:Generic_Cmax=Codd_Gammas12} occurs. 
	
	As we chose no product disks from $\Gamma_{n-2}$ and $c_{n-1}-1$ product disks from $\Gamma_{n-1}$, the remaining Seifert disks cannot contain both plumbing and image arcs, and there are no polygon sectors in $S_{n-2}, S_{n-1}$, and $S_n$. 
	%However, by our construction in \Cref{table:Bookkeeping_Cmax=Codd_Gammas12}, there is no sector contained entirely in $S_3$, so in fact, there are no real polygon sectors in $\Gamma_{n-3} \cup \Gamma_{n-2} \cup \Gamma_{n-1}$. 
	%
	%
	\item[$\circ$] If $n \geq 7$, then \Cref{NoClumpingOfArcs} ensures that $S_{n-2}, S_{n-1}$, and $S_n$ all only contain only plumbing or image arcs. Therefore, none of these Seifert disks can contain a polygon sector, and in particular, there are no polygon sink disk sectors spanning $\Gamma_{n-3} \cup \Gamma_{n-2} \cup \Gamma_{n-1}$. 
	\end{itemize}

\item \Horizontal Suppose $\mathcal{H}$ is a horizontal sector spanning $\Gamma_{n-3} \cup \Gamma_{n-2} \cup \Gamma_{n-1}$. (We recommend the reader refer to \Cref{fig:LastThreeColumns_Homogeneous} throughout.) Note: we need not analyze the $n=6$ and $n \geq 7$ cases separately.

	\begin{itemize}
	\item[$\circ$] Suppose $\mathcal{H}$ spans $\Gamma_{n-3}$. Suppose $\mathcal{H}$ contains $\mathbbm{b}_{n-3, i}$, for $1 \leq i \leq c_{n-3}-1$. This means $ \partial \mathcal{H}$ contains $\alpha_{n-3, i}$, and this arc points out of $\mathbbm{b}_{n-3, i}$ and into $S_{n-3}$. Therefore, $\mathcal{H}$ is not a sink. Alternatively, if $\mathcal{H} \sim \mathbbm{b}_{n-3, c_{n-3}}$, then $\mathcal{H} \sim \mathbbm{b}_{n-2, 1} \sim \ldots \sim \mathbbm{b}_{n-2, \delta} \sim \mathcal{S}_{n-1}$; we already argued that $\mathcal{S}_{n-1}$ is not a sink. 
	\item[$\circ$] Suppose $\mathcal{H}$ spans $\Gamma_{n-2}$. If $\mathcal{H}$ contains $\mathbbm{b}_{n-2, i}$, for $1 \leq i \leq \delta$, then $\mathcal{H} \sim \mathcal{S}_{n-1}$, which we already proved is not a sink. Alternatively, suppose $\mathcal{H} \sim \mathbbm{b}_{n-2, i}$, for $\delta +1 \leq i \leq c_{n-2}$. These arcs are not enclosed, on the left, by any arcs, so $\mathcal{H} \sim \mathcal{S}_{n-2}$, which we proved is not a sink. Therefore, no horizontal sector spanning $\Gamma_{n-2}$ is a sink.
	\item[$\circ$] Suppose $\mathcal{H}$ spans $\Gamma_{n-1}$. If $\mathcal{H} \sim \mathbbm{b}_{n-1, 1}$, then $\mathcal{H} \sim \mathcal{S}_{n}$, which is not a sink. Otherwise, $\mathcal{H} \sim \mathbbm{b}_{n-1, i}$, where $2 \leq i \leq c_{n-1}$. Notice that $\mathbbm{b}_{n-1, i}$ is enclosed, on the right, by $\varphi(\alpha_{n-1, i-1})$, and this arc points out of $\mathbbm{b}_{n-1, i}$. Therefore, $\mathcal{H}$ is not a sink. 
	\end{itemize}
	
Therefore, no horizontal sector whose span includes $\Gamma_{n-3} \cup \Gamma_{n-2} \cup \Gamma_{n-1}$ is a sink.
\end{itemize}
We deduce there are no sink disks spanning $\Gamma_{n-3} \cup \Gamma_{n-2} \cup \Gamma_{n-1}$ in Case (2a).

%\begin{center}
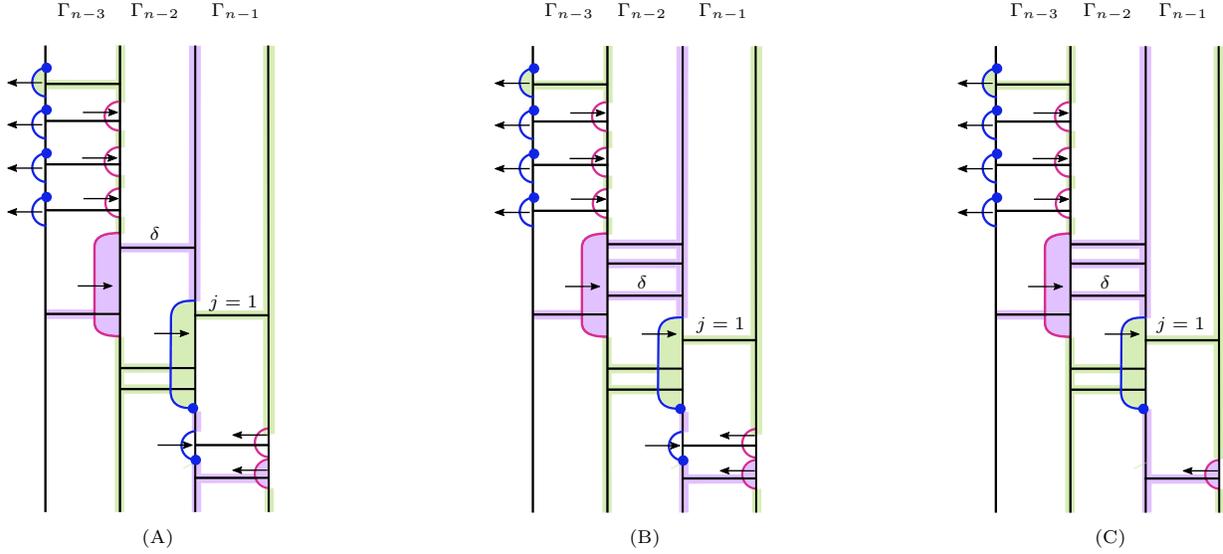
\begin{figure}[h!] \center
\labellist \tiny
\pinlabel {(A)} at 66 5
\pinlabel {(B)} at 246 5
\pinlabel {(C)} at 418 5
%% top row
%\pinlabel {$\Gamma_{n-3}$} at 43 415
%\pinlabel {$\Gamma_{n-2}$} at 70 415
%\pinlabel {$\Gamma_{n-1}$} at 100 415
%%
%\pinlabel {$\Gamma_{n-3}$} at 222 415
%\pinlabel {$\Gamma_{n-2}$} at 250 415
%\pinlabel {$\Gamma_{n-1}$} at 280 415
%%
%\pinlabel {$\Gamma_{n-3}$} at 394 415
%\pinlabel {$\Gamma_{n-2}$} at 422 415
%\pinlabel {$\Gamma_{n-1}$} at 450 415
%%
%\pinlabel {$j=1$} at 98 335
%\pinlabel {$j=1$} at 278 335
%\pinlabel {$j=1$} at 450 335
% bottom row
\pinlabel {$\Gamma_{n-3}$} at 38 200
\pinlabel {$\Gamma_{n-2}$} at 65 200
\pinlabel {$\Gamma_{n-1}$} at 95 200
\pinlabel {$\Gamma_{n-3}$} at 218 200
\pinlabel {$\Gamma_{n-2}$} at 245 200
\pinlabel {$\Gamma_{n-1}$} at 275 200
\pinlabel {$\Gamma_{n-3}$} at 390 200
\pinlabel {$\Gamma_{n-2}$} at 417 200
\pinlabel {$\Gamma_{n-1}$} at 445 200
\pinlabel {$j=1$} at 94 92
\pinlabel {$j=1$} at 274 84
\pinlabel {$j=1$} at 444 84
\pinlabel {$\delta$} at 65 118
\pinlabel {$\delta$} at 245 100
\pinlabel {$\delta$} at 416 100
\endlabellist
        \includegraphics[scale=1.02]{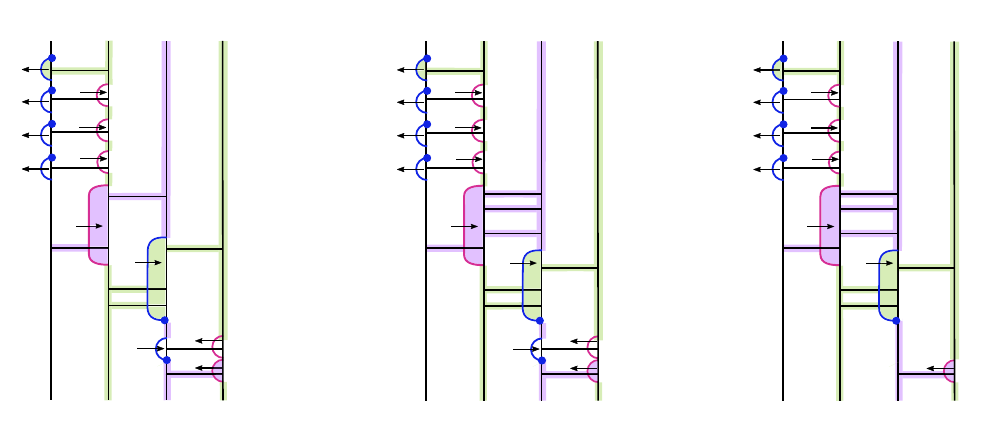}
        \caption{Constructing $B$ in $\Gamma_{n-3} \cup \Gamma_{n-2} \cup \Gamma_{n-1}$ when $\Gamma_{n-3}$ is homogeneous and $j=1$. All the arcs in $\Gamma_{n-1}$ are coherently oriented to the right. By design, $S_{n-3} \sim \mathcal{S}_{n-1}$ and $\mathcal{S}_{n-2} \sim \mathcal{S}_n$. In (A), $\delta := d(n-3; c_{n-3}-1, c_{n-3}) = 1$. In (B), $\delta \geq 3$. In (C), $c_{n-1} = 2$.}
        \label{fig:LastThreeColumns_Homogeneous}
\end{figure}
%\end{center}

%%%%%%%%%%%
\tcbox[size=fbox, colback=gray!10]{\textbf{Case (2b):} $\Gamma_{n-3}$ is homogeneous and $j \geq 2$} 

First, we describe how to build the branched surface in $\Gamma_{n-1}$: we choose no product disks from $\Gamma_{n-2}$, and $c_{n-1}-1$ product disks from $\Gamma_{n-1}$ by co-orienting $\alpha_{n-1,1}, \ldots, \alpha_{n-1, j-1}$ to the left and we co-orienting $\alpha_{n-1, j}, \ldots, \alpha_{n-1, c_{n-1}-1}$ to the right. 

We note: since $\widehat{\beta}$ is not a connected sum, if $j \geq 2$, then $c_{n-1} \geq j+1$, hence $c_{n-1} \geq 3$. We prove that there are no sink disks spanning $\Gamma_{n-3} \cup \Gamma_{n-2} \cup \Gamma_{n-1}$, and refer to \Cref{fig:LastThreeColumns_Homogeneous_Case2} throughout.

\begin{figure}[h!] \center
\labellist \tiny
\pinlabel {(A)} at 66 5
\pinlabel {(B)} at 235 5
\pinlabel {(C)} at 406 5
\pinlabel {$\Gamma_{n-3}$} at 38 200
\pinlabel {$\Gamma_{n-2}$} at 65 200
\pinlabel {$\Gamma_{n-1}$} at 95 200
\pinlabel {$\Gamma_{n-3}$} at 208 200
\pinlabel {$\Gamma_{n-2}$} at 235 200
\pinlabel {$\Gamma_{n-1}$} at 265 200
\pinlabel {$\Gamma_{n-3}$} at 380 200
\pinlabel {$\Gamma_{n-2}$} at 408 200
\pinlabel {$\Gamma_{n-1}$} at 435 200
\pinlabel {$j$} at 88 74
\pinlabel {$j$} at 258 46
\pinlabel {$j$} at 430 74
\pinlabel {$\delta$} at 65 118
\pinlabel {$\delta$} at 235 100
\pinlabel {$\delta$} at 406 118
\endlabellist
        \includegraphics[scale=1.05]{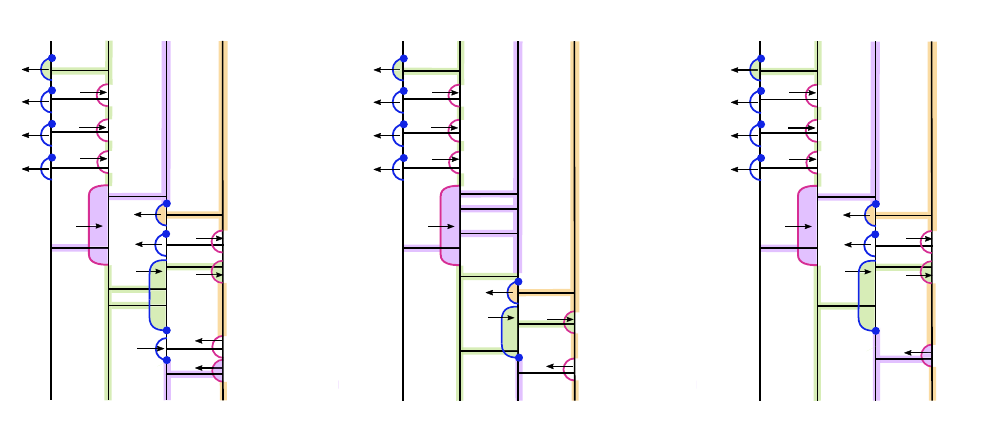}
        \caption{Constructing $B$ in $\Gamma_{n-3} \cup \Gamma_{n-2} \cup \Gamma_{n-1}$ when $\Gamma_{n-3}$ is homogeneous and $j \geq 2$. By design, $S_{n-3} \sim \mathcal{S}_{n-1}$. In (A), $\delta := d(n-3; c_{n-3}-1, c_{n-3}) = 1$. In (B), $\delta \geq 3$. In (C), $c_{n-2} = 2$.}
        \label{fig:LastThreeColumns_Homogeneous_Case2}
\end{figure}
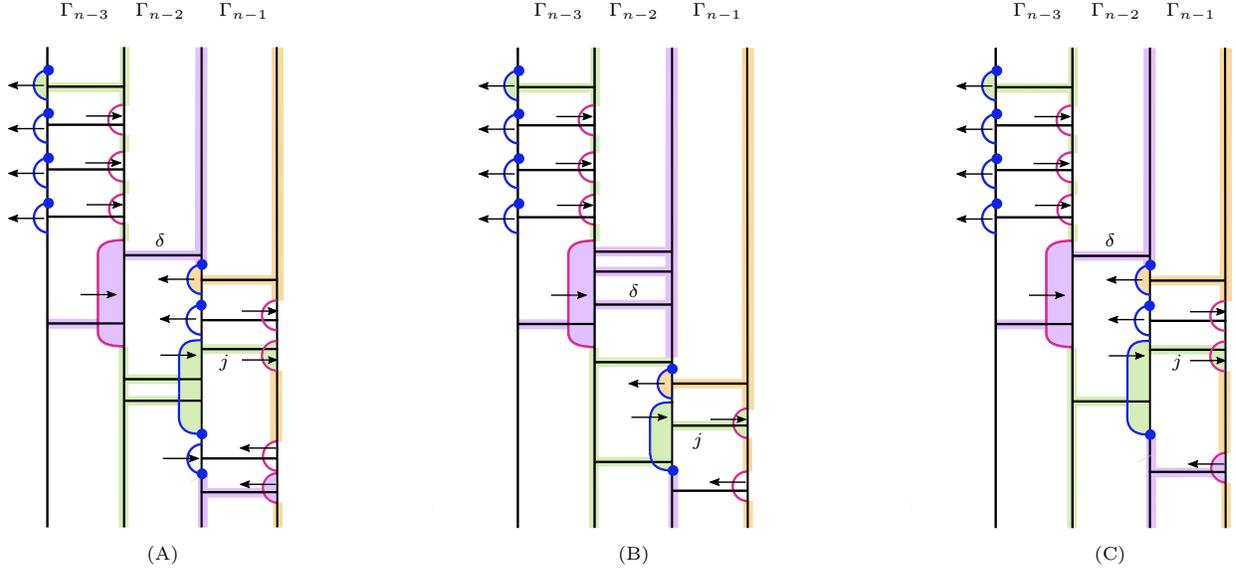

\begin{itemize}
\item \Seifert The same argument as in Case (2a) tells us that either $n = 6$ and $\mathcal{S}_{n-3} = \mathcal{S}_3$ is not a sink (because the construction of $B$ in $\Gamma_3$ is inherited from Case (B.1) of \Cref{lemma:Generic_Cmax=Codd_Gammas12}), or $n \geq 6$ and $\mathcal{S}_{n-3} \sim \mathbbm{b}_{n-3, c_{n-3}} \sim \mathcal{S}_{n-1}$.  When the latter occurs, we know that $\alpha_{n-1, j}$ is a right pointing plumbing arc, pointing out of $\mathcal{S}_{n-1}$, so $\mathcal{S}_{n-3} \sim \mathcal{S}_{n-1}$ is not a sink. Additionally, all the image arcs in $S_{n-2}$ point out of $\mathcal{S}_{n-2}$, and so $\mathcal{S}_{n-2}$ is not a sink. Finally, $\varphi(\alpha_{n-1,1})$ is in $\partial \mathcal{S}_n$, and points out of $\mathcal{S}_{n-1}$ and into $\mathbbm{b}_{n-1, 2}$, so $\mathcal{S}_n$ is not a sink. We deduce that none of the Seifert disk sectors spanning $\Gamma_{n-3} \cup \Gamma_{n-2} \cup \Gamma_{n-1}$ are sink disks. 
\item \Polygon The same argument from Case (2a) tells us that there are no sink disk polygon sectors spanning these columns. 
\item \Horizontal Suppose $\mathcal{H}$ is a horizontal sector spanning $\Gamma_{n-3} \cup \Gamma_{n-2} \cup \Gamma_{n-1}$. We argue that such an $\mathcal{H}$ cannot be a sink. 
	\begin{itemize}
	\item[$\circ$] If $\mathcal{H}$ spans $\Gamma_{n-3}$, we apply the same argument from Case (2a) to see that $\mathcal{H}$ is not a sink.
	\item[$\circ$] If $\mathcal{H}$ spans $\Gamma_{n-2}$, we apply the same argument from Case (2a) to see that $\mathcal{H}$ is not a sink. 
	\item[$\circ$] Suppose $\mathcal{H}$ spans $\Gamma_{n-1}$. If $\mathcal{H} \sim \mathbbm{b}_{n-1, i}$, where $1 \leq i \leq j-1$, then $\partial \mathcal{H}$ contains $\alpha_{n-1, i}$. This arc points out of $\mathcal{H}$ and into $S_{n-1}$, so $\mathcal{H}$ is not a sink. 
	
	If $\mathcal{H} \sim \mathbbm{b}_{n-1, j}$, then $\mathcal{H}$ also contains $\mathbbm{b}_{n-2, \delta+1}$, as in \Cref{fig:LastThreeColumns_Homogeneous_Case2}. This band is not enclosed, on the left, by any arcs, so $\mathcal{H} \sim \mathcal{S}_{n-2}$; we already argued that this sector is not a sink. 
	
	Finally, if $\mathcal{H} \sim \mathbbm{b}_{n-1, i}$ for $j+1 \leq i \leq c_{n-1}$, then $\partial \mathcal{H}$ contains $\varphi(\alpha_{n-1, i-1})$; this arc points out of the $\mathbbm{b}_{n-1, i}$ and into $S_n$, so $\mathcal{H}$ is not a sink.
	\end{itemize}
We deduce that no horizontal sector spanning $\Gamma_{n-3} \cup \Gamma_{n-2} \cup \Gamma_{n-1}$ is a sink. 
\end{itemize}

Therefore, there are no sink disk branch sectors spanning $\Gamma_{n-3} \cup \Gamma_{n-2} \cup \Gamma_{n-1}$.

Finally, we simultaneously verify that for all four of the cases of this lemma (i.e. Cases (1a), (1b), (2a), and (2b)), there is no pairwise linking between arcs in $\Gamma_{n-3} \cup \Gamma_{n-2} \cup \Gamma_{n-1}$. Since we chose no disks from $\Gamma_{n-2}$, we need only show that there is no pairwise linking between arcs in $S_{n-3}$ and $S_{n-1}$. As linking only occurs between $\alpha$ and $\alpha'$ if they are in the same column or adjacent columns, it suffices to show that there is no pairwise linking if $\alpha$ and $\alpha'$ lie in the same Seifert disk. This follows from the proof of \Cref{lemma:TemplateLinking}. \hfill $\Box$ \\
%\end{proof}

\subsection{Final\'e, a.k.a. sewing the constructions together.} \label{section:Finale} 

Finally, we combine the constructions of the previous sections to build taut foliations for positive $n$-braid knots, where $n \geq 5$.

\begin{thm} \label{thm:PositiveNBraids_Odd}
Suppose $K$ is a positive $n$-braid knot. Fix a standardized positive $n$-braid $\beta$ with $\widehat{\beta} \approx K$ and $n \geq 5$. Suppose, for this $\beta$, that $\Cmax = \Codd$. There is a laminar branched surface $B$ for $K$ such that $\tau_B$ fully carries all slopes $r < g(K)+1$. 
\end{thm}

Before we execute the construction of the branched surface, we note: at every stage of the construction (i.e. within the proofs of Lemmas \ref{lemma:Generic_Cmax=Codd_Gammas12}, \ref{lemma:Sparse_Cmax=Ceven_Gammas12}, \ref{lemma:Continue}, \ref{lemma:Generic_EndOfBraid_LastColumnIsSparse}, and \ref{lemma:Generic_EndOfBraid_LastColumnIsFull}), we always conjugated our brain into a specific presentation before building our branched surface in those columns. Throughout the forthcoming proof, we will need to intermittently conjugate $\beta$.

\begin{proof}
%Finally, suppose $K$ is a positive $n$-braid, where $n \geq 5$. Let $\beta$ denote the standardized positive $n$-braid with $\widehat{\beta} \approx K$. 
%There will be two cases to consider separately, corresponding to whether $\Cmax = \Codd$ or $\Cmax = \Ceven$. 
%First, suppose $\Cmax = \Codd$. 
To start, we follow the procedure outlined in \Cref{lemma:Generic_Cmax=Codd_Gammas12}, which builds $B$ in $\Gamma_1 \cup \Gamma_2 \cup \Gamma_3$. As such, the branched surface contains product disks from $\Gamma_1$, $\Gamma_2$, and $\Gamma_3$. Our construction of the branched surface in the remaining columns depends on $n$. We first build $B$ for all $n \geq 5$, and subsequently produce a lower bound on $\TauSup$.  

\tcbox[size=fbox, colback=gray!30]{Suppose $n=5$.} 
In this case, our braid has four columns, and we already constructed $B$ in $\Gamma_1\cup \Gamma_2 \cup \Gamma_3$. Therefore, it suffices to see how to build $B$ in $\Gamma_4$, and then check that no sink disks were produced. When $n=5$, we take no product disks from $\Gamma_4$. As there were four sub-cases of \Cref{lemma:Generic_Cmax=Codd_Gammas12}, we need to analyze each separately to determine why $B$ has no sink disks. Recall that we summarized the outcomes of \Cref{lemma:Generic_Cmax=Codd_Gammas12} in \Cref {table:Bookkeeping_Cmax=Codd_Gammas12}. 

\tcbox[size=fbox, colback=gray!10]{Suppose $n=5$, and Case (A) from \Cref{lemma:Generic_Cmax=Codd_Gammas12} holds.}  In this case, as we recorded in \Cref {table:Bookkeeping_Cmax=Codd_Gammas12}, we know that there is a unique potential sink disk spanning $\Gamma_1 \cup \Gamma_2$, and it is the sector $\mathcal{S} \sim \mathcal{S}_3 \sim \mathbbm{b}_{2,1}$ (we suggest the reader consult  \Cref{fig:Generic_CmaxCodd_Gammas12} to ``see'' $B$ in $\Gamma_1 \cup \Gamma_2 \cup \Gamma_3$). By design, $\partial S_{n-2} = \partial S_3$ contains exactly one image arc and $c_3-1$ plumbing arcs, and there are no polygon sectors in $S_3$. Therefore, if we take no product disks from $\Gamma_4$, we can apply \Cref{lemma:Generic_EndOfBraid_LastColumnIsSparse} and deduce that there are no sink disks spanning $\Gamma_3 \cup \Gamma_4$. This produces a sink disk free branched surface. 

\tcbox[size=fbox, colback=gray!10]{Suppose $n=5$, and Case (B.1) from \Cref{lemma:Generic_Cmax=Codd_Gammas12} holds.} 
As we recorded in \Cref {table:Bookkeeping_Cmax=Codd_Gammas12}, we know that there are no sink disks spanning $\Gamma_1 \cup \Gamma_2$; we suggest the reader refer to \Cref{fig:Generic_CmaxCodd_Gammas12_A} to ``see'' $B$ in $\Gamma_1 \cup \Gamma_2 \cup \Gamma_3$. 

In this case, we know that $d(3; \delta-2, \delta)=c_4$. Moreover, since $\widehat{\beta}$ is not a connected sum, then we cannot have all the bands in $\Gamma_4$ concentrated between a pair of consecutive bands in $\Gamma_3$. We deduce that $d(3; \delta-2, \delta-1) \geq 1$ and $d(3; \delta-1, \delta) \geq 2$. Let $d(3; \delta-2, \delta) = \epsilon_1$ and $d(3; \delta-1, \delta) = \epsilon_2$. We include no product disks from $\Gamma_4$ into $B$, and now argue that there are no sink disks spanning $\Gamma_3 \cup \Gamma_4$. This will allow us to deduce that we have constructed a sink disk free branched surface. 

\begin{itemize}
\item \Seifert As summarized in \Cref {table:Bookkeeping_Cmax=Codd_Gammas12}, $\mathcal{S}_3$ is not a sink. Since $S_4$ contains only right pointing image arcs, $\mathcal{S}_4$ is not a sink. 

We now argue that $\mathcal{S}_5 \sim \mathcal{S}_4$, hence $\mathcal{S}_5$ is not a sink. By design, we have chosen no product disks from $\Gamma_4$, so in particular, $\mathcal{S}_5 \sim \mathbbm{b}_{4,\epsilon_1 + \epsilon_2}$. Moreover, by design, $\mathbbm{b}_{4, \epsilon_1 + \epsilon_2}$ is not enclosed, on the left, by any arc; see \Cref{fig:LastThreeColumns_Local_1}. Therefore, $\mathcal{S}_4 \sim \mathbbm{b}_{4, \epsilon_1 + \epsilon_2}$, and we deduce that $\mathcal{S}_4 \sim \mathcal{S}_5$.

%\begin{center}
\begin{figure}[h!] \center
\labellist \tiny
%
%\pinlabel {$\Gamma_{n-3}$} at 218 174
\pinlabel {$\Gamma_{3}$} at 42 174
\pinlabel {$\Gamma_{4}$} at 70 174
\pinlabel {$\delta$} at 40 62
\pinlabel {$\delta-2$} at 40 130
\pinlabel {$\epsilon_1$} at 70 100
\pinlabel {$\epsilon_1 + \epsilon_2$} at 69 62
\endlabellist
        \includegraphics[scale=1.02]{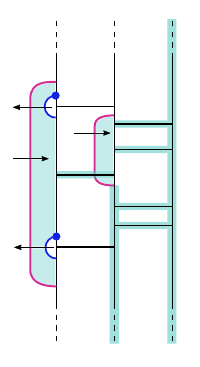}
        \caption{Zooming on on the bands $\mathbbm{b}_{3,\delta-2}, \mathbbm{b}_{3, \delta-1}$, and $\mathbbm{b}_{3, \delta}$ in \Cref{fig:Generic_CmaxCodd_Gammas12_A}. We see that $\mathcal{S}_{4} \sim \mathcal{S}_5$.}
        \label{fig:LastThreeColumns_Local_1}
\end{figure}

\item \Polygon There is a unique sector $\mathcal{P}$ in $\Gamma_3 \cup \Gamma_4$ which lies in a Seifert disk and contains both plumbing and image arcs in its boundary; see \Cref{fig:LastThreeColumns_Local_1}. By design, $\mathcal{P} \sim \mathbbm{b}_{3, \delta-1} \sim \mathbbm{b}_{4, \epsilon_1} \sim \mathcal{S}_5$, and we already proved that $\mathcal{S}_5$ is not a sink. 
\item \Horizontal Suppose $\mathcal{H}$ is a horizontal sector spanning $\Gamma_3$. If $\mathcal{H}$ spans $\Gamma_3$, it contains some band $\mathbbm{b}_{3, t}$. By design, all bands in $\Gamma_3$ are either enclosed, on the left, by a left pointing plumbing arc (so it is not a sink), or the band is part of the sector $\mathcal{P}$ described above (which we just proved is not a sink). Therefore, no horizontal sector spanning $\Gamma_3$ is a sink. Alternatively, if $\mathcal{H}$ spans $\Gamma_4$, then since there are no plumbing or image arcs in $S_5$, $\mathcal{H} \sim \mathcal{S}_5$, which is not a sink.
\end{itemize}

We deduce that we can construct a sink disk free branched surface in this case. 

%%%%
\tcbox[size=fbox, colback=gray!10]{Suppose $n=5$, and Case (B.2) from \Cref{lemma:Generic_Cmax=Codd_Gammas12} holds.}

As summarized in \Cref{table:Bookkeeping_Cmax=Codd_Gammas12}, in this case, we know that $\Gamma_3$ is heterogeneous, and that there are no sink disks spanning $\Gamma_1 \cup \Gamma_2$. We now argue that there are no sink disks spanning $\Gamma_3 \cup \Gamma_4$:

\begin{itemize}
\item \Seifert \Cref{table:Bookkeeping_Cmax=Codd_Gammas12} confirms that $\mathcal{S}_3$ is not a sink. $S_3$ contains both left and right pointing plumbing arcs, so $\mathcal{S}_4$ contains both left and right pointing image arcs; thus, it is not a sink. We now argue that $\mathcal{S}_5$ is not a sink. 

Throughout the forthcoming analysis, we use the same band labellings as in Case (B.2) and \Cref{fig:Generic_CmaxCodd_Gammas12_A.2}. By the assumptions of Case (B.2), we have $d(3; c_3, \delta) = c_4$. Let $\alpha_{3,p}$ be the first right pointing plumbing arc in $\Gamma_3$; by assumption, $1 \leq p \leq \delta$. Since $\widehat{\beta}$ is not a connected sum, then we cannot have that $d(3; p, p+1) = c_4$. Therefore, $d(3; p, p+1) : = \epsilon < c_4$ and $d(3; p+1, \delta) := \eta \geq 2$. Note that $\epsilon + \eta = c_4$. 

Since $d(3; p+1, \delta) \geq 2$, then there exists some band $\mathbbm{b}_{4, k}$, where $\epsilon+1 \leq k \leq \epsilon + \eta$ such that either (a) $\mathbbm{b}_{4, k}$ is enclosed, on the left, by a left pointing image arc, or (b) $\mathbbm{b}_{4, k}$ is not enclosed, on the left, by any image arcs; see 
%(a) there exists some band $\mathbbm{b}_{4, i}$, where $1 \leq i \leq \epsilon$, such that $\mathbbm{b}_{4,i}$ is enclosed, on the left, by a left pointing image arc, or (b) all the bands $\mathbbm{b}_{4,1}, \ldots, \mathbbm{b}_{4, \epsilon}$ are not enclosed, on the left, by any image arcs; 
see \Cref{fig:LastThreeColumns_Local_2}. 
However, since we included no product disks from $\Gamma_4$ while building $B$, we know that no band in $\Gamma_4$ is enclosed, on the right, by any arcs. Therefore, $\mathcal{S}_5 \sim \mathbbm{b}_{4, 1} \sim \ldots \sim \mathbbm{b}_{4, k} \sim \ldots \sim \mathbbm{b}_{4, c_4}$.
%that all the bands $\mathbbm{b}_{4, 1}, \ldots, \mathbbm{b}_{4,\epsilon}$ are not enclosed, on the right, by any arcs. Therefore, $\mathcal{S}_5 \sim \mathbbm{b}_{4,1} \sim \mathbbm{b}_{4, 2} \sim \ldots \sim \mathbbm{b}_{4, \epsilon}$. 
If (a) occurs, then $\partial \mathcal{S}_5$ contains an outward pointing image arc, but if (b) happens, then $\mathcal{S}_5 \sim \mathcal{S}_4$ (and we already proved the latter is not a sink). We deduce that $\mathcal{S}_5$ is never a sink. 

%\begin{center}
\begin{figure}[h!] \center
\labellist \tiny
%%%%%% LEFT
\pinlabel {$\Gamma_{3}$} at 42 170
\pinlabel {$\Gamma_{4}$} at 70 170
\pinlabel {$p=1$} at 44 124
\pinlabel {$\delta$} at 44 52
\pinlabel {$\epsilon$} at 70 100
%%%%%% MIDDLE
\pinlabel {$\Gamma_{3}$} at 195 170
\pinlabel {$\Gamma_{4}$} at 225 170
\pinlabel {$p=2$} at 192 90
\pinlabel {$\delta$} at 195 52
\pinlabel {$\epsilon$} at 223 100
%%%%%% RIGHT
\pinlabel {$\Gamma_{3}$} at 350 170
\pinlabel {$\Gamma_{4}$} at 378 170
\pinlabel {$p=2$} at 344 90
\pinlabel {$\delta$} at 350 52
\pinlabel {$\epsilon$} at 376 100
\endlabellist
        \includegraphics[scale=1.12]{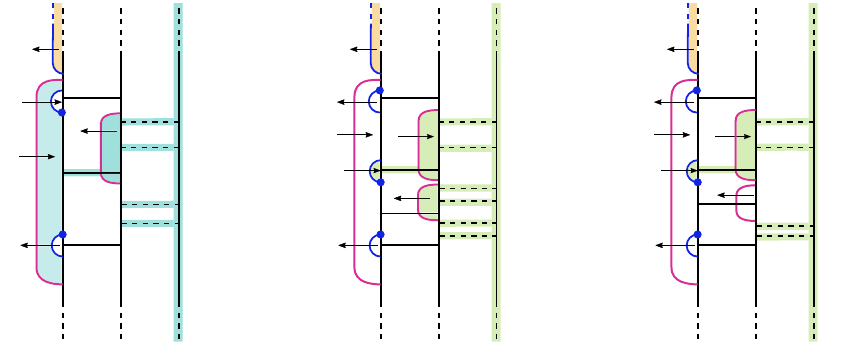}
        \caption{Zooming in on the bands enclosed, on the left, by $\varphi(\alpha_{2,1})$ from \Cref{fig:Generic_CmaxCodd_Gammas12_A.2} (to be consistent, we have used the same cyclic conjugated presentation of the braid from that figure). %There are bands (not pictured) between $\mathbbm{b}_{3,c_3}$ (not pictured) and $\mathbbm{b}_{3,1}$.
        We know that $d(3; 1, \delta) \geq 1$, but we do not know the distribution of the bands within $\Gamma_4$, so the visible bands in $\Gamma_4$ are dashed. No band in $\Gamma_4$ is enclosed, on the right, by any arc. If there is at least one enclosed, on the left, by an image arc, the image arc is a left pointer, and $\mathcal{S}_5$ is not a sink. Otherwise, $\mathcal{S}_5 \sim \mathcal{S}_4$, which we already proved is not a sink. }
        \label{fig:LastThreeColumns_Local_2}
\end{figure}
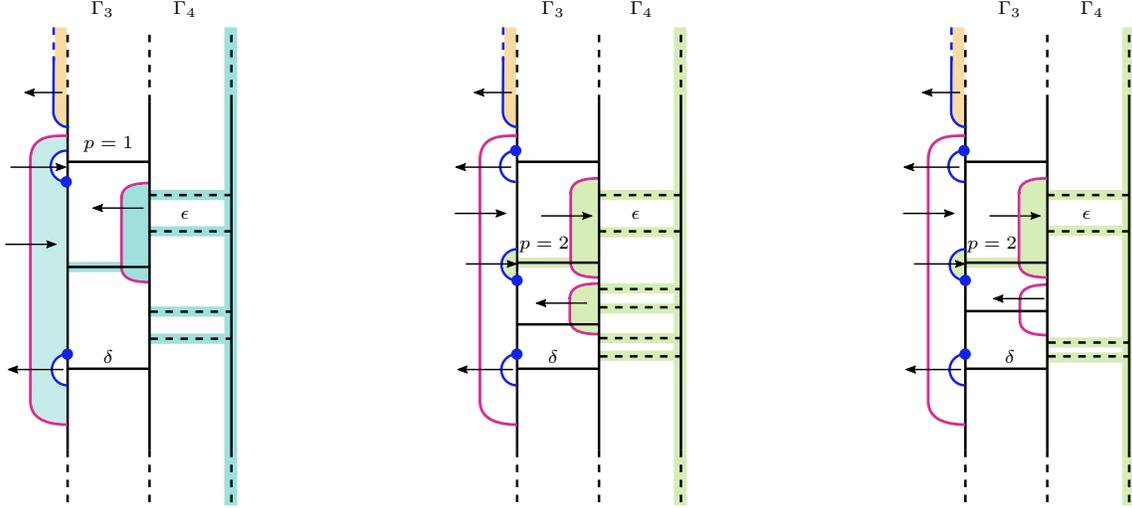

\item \Polygon There is a unique sector $\mathcal{P}$ in $S_3$ which contains both plumbing and image arcs in its boundary. Notice that $\partial \mathcal{P}$ contains $\alpha_{3, \delta-2}$, which is a right pointing plumbing arc. 
%; note that $\mathcal{P} \sim \mathbbm{b}_{3,\delta-1}$. By design, $\mathbbm{b}_{3, \delta-1}$ is enclosed, on the right, by a left pointing image arc. 
Therefore, $\mathcal{P}$ is not a sink. 

\item \Horizontal If $\mathcal{H}$ is a horizontal sector spanning $\Gamma_4$, then $\mathcal{H} \sim \mathbbm{b}_{4, i}$, for some $1 \leq i \leq c_4$. But these bands are not enclosed, on the right, by any arcs, so $\mathcal{H} \sim \mathcal{S}_5$, which we already proved is not a sink. 

Next, suppose that $\mathcal{H}$ is a horizontal sector spanning $\Gamma_3$, so that $\mathcal{H} \sim \mathbbm{b}_{3, j}$, for some $1 \leq j \leq c_3$. 

        \begin{itemize}
        \item[$\circ$] The bands $\mathbbm{b}_{3, \delta}, \mathbbm{b}_{3, \delta+1}, \ldots, \mathbbm{b}_{3, c_3}$, and $\mathbbm{b}_{3,1}, \ldots, \mathbbm{b}_{3, p-1}$ are all enclosed, on the left, by a left pointing image arc. Therefore, if $\mathcal{H}$ contains one of these bands, $\mathcal{H}$ is not a sink. 
        \item[$\circ$] Suppose $\mathcal{H} \sim \mathbbm{b}_{3, p}$. Since $d(3; p, p+1) \geq 2$, then $\mathcal{H} \sim \mathbbm{b}_{4,1} \sim \ldots \sim \mathbbm{b}_{4, \epsilon} \sim \mathcal{S}_5$; we already proved $\mathcal{S}_5$ is not a sink.
        \item[$\circ$] The bands $\mathbbm{b}_{3, p+1}, \ldots, \mathbbm{b}_{3, \delta-1}$ are all enclosed, on the right, by a left pointing image arc. Thus, if $\mathcal{H}$ contains one of these bands, it is not a sink.
        % \item[$\circ$] Suppose $\mathcal{H} \sim \mathbbm{b}_{3,1}$. Since $d(3; c_3, 1) \geq 1$, then $\mathcal{H} \sim \mathbbm{b}_{3,1}$ implies that $\mathcal{H}$ also contains all the bands in $\Gamma_4$ between $\mathbbm{b}_{3,c_3}$ and $\mathbbm{b}_{3, 1}$. But these bands are not enclosed, on the right, by any arcs, so $\mathcal{H} \sim \mathcal{S}_5$, which we already proved is not a sink. 
        % %
        % \item[$\circ$] Suppose $\mathcal{H} \sim \mathbbm{b}_{3, j}$, where $2 \leq j \leq \delta-1$. These bands are all enclosed, on the right, by a left pointing image arc which exists $\mathcal{H}$. We deduce that such $\mathcal{H}$ are not sinks. 
        % %
        % \item[$\circ$] Suppose $\mathcal{H} \sim \mathbbm{b}_{3, j}$, where $\delta \leq j \leq c_3$. These bands are all enclosed, on the left, by a left pointing plumbing arc which exists $\mathcal{H}$. We deduce that these horizontal sectors are not sinks.
        \end{itemize}
Therefore, $B$ contains no sink disk spanning $\Gamma_3 \cup \Gamma_4$, hence it is sink disk free.
\end{itemize}

%%%%
\tcbox[size=fbox, colback=gray!10]{Suppose $n=5$, and Case (C) from \Cref{lemma:Generic_Cmax=Codd_Gammas12} holds.}

As summarized in \Cref{table:Bookkeeping_Cmax=Codd_Gammas12}, in this case, after applying our template to $\Gamma_3$, the column is heterogeneous. Additionally, we know that there is a unique potential sink disk spanning $\Gamma_1 \cup \Gamma_2$, and it is the sector $\mathcal{S}$ containing $\mathbbm{b}_{2,1} \sim \mathbbm{b}_{3, c_3}$ (we refer the reader to  \Cref{fig:Generic_CmaxCodd_Gammas12_B} to recall our labelling conventions). We include no product disks from $\Gamma_4$ into $B$, and show that the result contains no sink disks spanning $\Gamma_3 \cup \Gamma_4$:

\begin{itemize}
\item \Seifert In \Cref{lemma:Generic_Cmax=Codd_Gammas12}, we proved that $\mathcal{S}_3$ is not a sink. By design, $S_4$ contains both left and right pointing image arcs; since $B$ includes no product disks from $\Gamma_4$, we deduce that $\partial \mathcal{S}_4$ must contain an outward pointing arc, so it is not a sink. 

We now prove that $\mathcal{S}_5$ is not a sink. Since we included no product disks from $\Gamma_4$, we know that $\mathcal{S}_5 \sim \mathbbm{b}_{4,1} \sim \ldots \sim \mathbbm{b}_{4, c_4}$, and it suffices to show that some band in $\Gamma_4$ is enclosed, on the left, by a left pointing image arc, or that $\mathcal{S}_5 \sim \mathcal{S}_4$.

We first notice: by the definition of our template (and the fact that we are in Case (C) of \Cref{lemma:Generic_Cmax=Codd_Gammas12}), we know that the following inequalities hold (once again, we refer the reader to \Cref{fig:Generic_CmaxCodd_Gammas12_B} to recall the labelings):
\begin{align*}
d(3; \delta, j-1) &= 0 \text{ (by the construction in \Cref{lemma:Generic_Cmax=Codd_Gammas12})}\\
d(3; j-1, j) &\geq 1 \text{ (by the definition of our template)} \\
d(3; j-1, j) &< c_4 \text{ (because $\beta$ is not a connected sum)} 
\end{align*}

These three inequalities imply that $d(3; j+1, \delta) \geq 1$. This means that either (a) there is some band $\mathbbm{b}_{4, k}$ which is enclosed, on the left, by a left pointing image arc, or (b) that there is some band $\mathbbm{b}_{4,k}$ which is not enclosed, on the left, by any image arc, so $\mathcal{S}_5 \sim \mathcal{S}_4$. We deduce that $\mathcal{S}_5$ is not a sink.

\item \Polygon Applying the same analysis from Case (B.2) tells us that there are no sink disk polygon sectors in $\Gamma_3 \cup \Gamma_4$; see \Cref{fig:LastThreeColumns_Local_2}. 

\item \Horizontal As described in \Cref{table:Bookkeeping_Cmax=Codd_Gammas12}, there is a unique potential sink disk spanning $\Gamma_1 \cup \Gamma_2$, and it is the horizontal sector containing both $\mathbbm{b}_{2,1}$ and $ \mathbbm{b}_{3, c_3}$. So, to prove that $B$ is completely sink disk free, we need to argue that there are no horizontal sink disks spanning $\Gamma_3 \cup \Gamma_4$. 

We quickly argue that no horizontal sector spanning $\Gamma_4$ is a sink: we already argued that $\mathcal{S}_5 \sim \mathbbm{b}_{4,1} \sim \ldots \sim \mathbbm{b}_{4, c_4}$. Since we already proved that $\mathcal{S}_5$ is not a sink, we immediately know that no horizontal sector spanning $\Gamma_4$ is a sink. 

Finally, we argue that no horizontal sector spanning $\Gamma_3$ is a sink. Let $\mathcal{H}$ denote such a sector. 
\begin{itemize}
\item[$\circ$] If $\mathcal{H} \sim \mathbbm{b}_{3, k}$, for $1 \leq k \leq \delta-1$, then $\partial \mathcal{H}$ contains $\varphi(\alpha_{3,k-1})$ which points out of the sector, so $\mathcal{H}$ is not a sink. 
\item[$\circ$] If $\mathcal{H} \sim \mathbbm{b}_{3, k}$, for $\delta \leq k \leq j-1$, then $\partial \mathcal{H}$ contains $\alpha_{3,k}$ which points out of the sector, so $\mathcal{H}$ is not a sink. 
\item[$\circ$] Suppose $\mathcal{H} \sim \mathbbm{b}_{3, j}$. Then by definition of our template, $\mathcal{H} \sim \mathbbm{b}_{4,s}$ for some $s$. But as discussed above, this means $\mathcal{H} \sim \mathcal{S}_5$, which we already proved is not a sink. 
\item[$\circ$] If $\mathcal{H} \sim \mathbbm{b}_{3, k}$, for $j+1 \leq k \leq c_3$, then $\partial \mathcal{H}$ contains $\varphi(\alpha_{3,k-1})$ which points out of the sector, so $\mathcal{H}$ is not a sink. 
\end{itemize}

\end{itemize}

For emphasis, we remark: suppose $\mathcal{H} \sim \mathbbm{b}_{2,1} \sim \mathbbm{b}_{3, c_3}$. If $j = c_3$, then $\mathcal{H} \sim \mathcal{S}_5$ (which is not a sink), and if $j < c_3$, then $\mathcal{H}$ meets $\mathcal{S}_4$ is a left pointing image arc. Regardless, it is not a sink. Therefore, we see that no horizontal sector spanning $\Gamma_3 \cup \Gamma_4$ is a sink, and $B$ is sink disk free. 

This concludes the analysis when $n=5$. %; we postpone the analysis of the slopes carried by $\tau_B$. 

\tcbox[size=fbox, colback=gray!30]{Suppose $n= 6$.} 

In \Cref{lemma:Generic_Cmax=Codd_Gammas12}, we constructed our branched surface so that it contained product disks from $\Gamma_1 \cup \Gamma_2 \cup \Gamma_3$. When $n=6$, our braid has five columns; to complete the construction, we must specify some presentation of our braid, and then build $B$ in the remaining two columns. We will choose no product disks from $\Gamma_4$ and $c_5-1$ product disks from $\Gamma_5$.

As recalled in \Cref{table:Bookkeeping_Cmax=Codd_Gammas12}, were four sub-cases in \Cref{lemma:Generic_Cmax=Codd_Gammas12}, and each of these four cases, there is at most one potential sink disk spanning $\Gamma_1 \cup \Gamma_2$: in Cases (A) and (C), the sector is $\mathcal{S}_3$ or $\mathbbm{b}_{3, c_3}$, respectively; in Cases (B.1) and (B.2), there were no sink disks spanning $\Gamma_1 \cup \Gamma_2$. In particular, the unique potential sink disk spanning $\Gamma_1 \cup \Gamma_2$ always spans $\Gamma_3$. \Cref{lemma:Generic_EndOfBraid_LastColumnIsFull} describes how to cyclically conjugate $\beta$ with respect to some distinguished band in $\Gamma_3$, and then build $B$ in $\Gamma_4 \cup \Gamma_5$. \Cref{lemma:Generic_EndOfBraid_LastColumnIsFull} allows us to conclude that there are no sink disks spanning $\Gamma_3 \cup \Gamma_4 \cup \Gamma_5$. We deduce that we can build a branched surface $B$ which contains co-oriented product disks from all columns except $\Gamma_4$, such that $B$ is sink disk free. %We postpone our computation of a lower bound for $\TauSup$.% for now.

\tcbox[size=fbox, colback=gray!30]{Suppose $n=7$.} 

By assumption, we already built $B$ in $\Gamma_1 \cup \Gamma_2 \cup \Gamma_3$ via \Cref{lemma:Generic_Cmax=Codd_Gammas12}; we now follow \Cref{lemma:Continue_Special} to construct $B$ in $\Gamma_4 \cup \Gamma_5$. The conclusions of \Cref{lemma:Continue_Special} tell us that the unique potential sink disk spanning $\Gamma_1 \cup \ldots \cup \Gamma_5$ is $\mathcal{S}_5$. Applying \Cref{lemma:Generic_EndOfBraid_LastColumnIsSparse}, we see that there are no sinks spanning $\Gamma_5 \cup \Gamma_6$. We deduce that we have constructed a sink disk free branched surface. %We later argue that $\tau_B$ carries all slopes $r < g(K)+1$. 

\tcbox[size=fbox, colback=gray!30]{Suppose $n\geq 8$.} 
We already built $B$ in $\Gamma_1 \cup \Gamma_2 \cup \Gamma_3$ via \Cref{lemma:Generic_Cmax=Codd_Gammas12}, and we now follow \Cref{lemma:Continue_Special} to construct $B$ in $\Gamma_4 \cup \Gamma_5$. We now iteratively build $B$ in the remaining columns. To simplify our exposition, we define the value $f$ as follows:
\begin{align*}
f :=
\begin{cases}
\displaystyle \frac{(n-1)-5}{2}-2 & \qquad n-1 \equiv 1 \mod 2 \\ \\
\displaystyle \frac{(n-1)-6}{2} & \qquad n-1 \equiv 0 \mod 2
\end{cases}
\end{align*}

To build $B$, we iteratively apply \Cref{lemma:Continue}, by assuming that we have build $B$ to contain co-oriented product disks in some non-trivial subset of columns $\Gamma_1 \cup \ldots \Gamma_{5+2t}$, for $t \in \{0, \ldots, f-1\}$. By construction, before each application of \Cref{lemma:Continue}, the only potential sink disk spanning $\Gamma_1 \cup \ldots \cup \Gamma_{5+2t}$ is $\mathcal{S}_{5+2t}$; after building $B$ in $\Gamma_1 \cup \ldots \cup \Gamma_{5+2t} \cup \Gamma_{5+2t+1} \cup \Gamma_{5+2t+2}$ as directed by \Cref{lemma:Continue}, the only potential sink disk spanning these columns is $\mathcal{S}_{5+2t+2}$. 

Suppose $n-1 \equiv 1 \mod 2$. When $t=f-1$, our last application of \Cref{lemma:Continue} assumes we have constructed $B$ to include product disks from $\Gamma_1 \cup \ldots \cup \Gamma_{n-5}$, and then concludes by producing a branched surface which includes product disks from $\Gamma_1 \cup \ldots \cup \Gamma_{n-3}$. By design, at this stage, the only potential sink disk spanning these columns is $\mathcal{S}_{n-3}$. We now apply  \Cref{lemma:Generic_EndOfBraid_LastColumnIsFull} to build $B$ in $\Gamma_{n-2} \cup \Gamma_{n-1}$.  \Cref{lemma:Generic_EndOfBraid_LastColumnIsFull} ensures that there are no sink disks spanning $\Gamma_{n-3} \cup \Gamma_{n-2} \cup \Gamma_{n-1}$. We deduce that our resulting branched surface, which contains co-oriented product disks from $\Gamma_1 \cup \Gamma_2 \cup \Gamma_3 \cup \Gamma_5 \cup \ldots \cup \Gamma_{n-1}$, is sink disk free. 

Instead suppose $n-1 \equiv 0 \mod 2$. When $t=f-1$, our last application of \Cref{lemma:Continue} assumes that we have already constructed our branched surface in $\Gamma_1 \cup \ldots \cup \Gamma_{n-4}$, and then outputs a branched surface containing product disks spanning $\Gamma_1 \cup \ldots \cup \Gamma_{n-2}$. The only potential sink disk spanning these columns is $\mathcal{S}_{n-2}$. We now apply \Cref{lemma:Generic_EndOfBraid_LastColumnIsSparse} to see that without including any product disks from $\Gamma_{n-1}$, we have constructed a sink disk free branched surface which contains co-oriented product disks from $\Gamma_1 \cup \Gamma_2 \cup \Gamma_3 \cup \Gamma_5 \cup \ldots \cup \Gamma_{n-2}$.

We deduce that when $\Cmax = \Codd$ and $n \geq 5$, we have produced a sink disk free branched surface $B$ which contains a copy of the fiber surface, all product disks from $\Gamma_i$, $i \equiv 1 \mod 2$, and a single product disk from $\Gamma_2$. Applying \Cref{prop:BisLaminar}, we deduce that $B$ is a laminar branched surface. It remains to determine what slopes are fully carried by $B$. Once again, our analysis will be determined by the parity of $n$.  Combining Lemmas \ref{lemma:Generic_Cmax=Codd_Gammas12}, \ref{lemma:Continue}, \ref{lemma:Continue_Special}, \ref{lemma:Generic_EndOfBraid_LastColumnIsSparse}, and \ref{lemma:Generic_EndOfBraid_LastColumnIsFull}, we see that there is no pairwise linking between any of the plumbing arcs used to build $B$. Thus, it suffices to put a lower bound on the number of product disks used to build $B$. 

\begin{itemize}
\item If $n-1 \equiv 1 \mod 2$, then $n \equiv 0 \mod 2$, and for $2g(K)-1=\mathcal{C}-n$ to hold, we must have that $\mathcal{C}$ is odd. Therefore, since we assumed that $\Cmax = \Codd$, we have that $\Codd \geq \frac{\mathcal{C}+1}{2}$. Therefore, 
\begin{align*}
\TauSup &= (\# \text{product disks used to construct $B$}) \\
&= \left( \ \sum_{i \text{ odd}} c_i \right) - (\# \Gamma_i, i \text{ odd}) +(\# \text{ product disks from } \Gamma_2) \\
&= \Codd - \frac{n}{2} + 1 \\
&\geq \frac{\mathcal{C}+1}{2} - \frac{n}{2}+1 = \frac{\mathcal{C}-n+1}{2}+1 = g(K)+1
\end{align*}
We deduce that $\tau_B$ carries all slopes $r < g(K)+1$, as desired. 
\item If $n-1 \equiv 0 \mod 2$, then $n \equiv 1 \mod 2$, and for $2g(K)-1=\mathcal{C}-n$ to hold, we have that $\mathcal{C}$ is even. Therefore, since $\Cmax = \Codd$, we have that $\Codd \geq \frac{\mathcal{C}}{2}$. Performing the analogous computation from above, we see that:
\begin{align*}
\TauSup \geq \left( \frac{\mathcal{C}}{2} - \frac{n}{2}\right)+1 = \frac{\mathcal{C}-n+1}{2}+\frac{1}{2} = g(K)+\frac{1}{2}
\end{align*}
Since $B$ was constructed from a single copy of the fiber surface and a collection of product disks, the supremal slope carried by $\tau$ is integral. Therefore, $\TauSup \geq g(K)+1$, and $B$ carries all slopes $r < g(K)+1$ as desired. 
\end{itemize}

We deduce that $\tau_B$ carries all slopes $r < g(K)+1$, as desired. 
\end{proof}

\begin{thm} \label{thm:PositiveNBraids_Even}
%Suppose $K$ is a positive $n$-braid knot, with $\widehat{\beta} \approx K$, for $\beta$ standardized and $n \geq 5$. Suppose $\Cmax = \Ceven$. There is a laminar branched surface $B$ for $K$ such that $\tau_B$ fully carries all slopes $r < g(K)+1$. 
Suppose $K$ is a positive $n$-braid knot. Fix a standardized positive $n$-braid $\beta$ with $\widehat{\beta} \approx K$ and $n \geq 5$. Suppose, for this $\beta$, that $\Cmax = \Ceven$. There is a laminar branched surface $B$ for $K$ such that $\tau_B$ fully carries all slopes $r < g(K)+1$. 
\end{thm}

\begin{proof}
To begin, we follow the procedure outlined in \Cref{lemma:Sparse_Cmax=Ceven_Gammas12}. This produces a branched surface in $\Gamma_1 \cup \Gamma_2$. In particular, we have selected a single product disk from $\Gamma_1$ and $c_2-1$ product disks from $\Gamma_2$. \Cref{table:Bookkeeping_Cmax=Ceven_Gammas12} summarizes the potential sink disks spanning $\Gamma_1 \cup \Gamma_2$. In particular, regardless of construction used, \Cref{table:Bookkeeping_Cmax=Ceven_Gammas12} identifies that the unique potential sink disk spanning $\Gamma_1 \cup \Gamma_2$ contains $\mathbbm{b}_{3,1}$. To build $B$ in the remaining columns, we study the $n=5$ and $n \geq 6$ cases separately. We subsequently produce a lower bound on $\TauSup$.

%\fcolorbox{}{gray!30}{Suppose $n=5$.} 
\tcbox[size=fbox, colback=gray!30]{Suppose $n=5$.} 

In this case, there are four columns of the braid, and we have already built $B$ in $\Gamma_1 \cup \Gamma_2$. As summarized in \Cref{table:Bookkeeping_Cmax=Ceven_Gammas12}, there were two primary cases for how we built $B$ in $\Gamma_1 \cup \Gamma_2$, and each requires a slightly different construction to build $B$ in $\Gamma_4$. We emphasize: our naming convention for the sub-cases, as well as the initial labellings of the bands, is the same as those specified in \Cref{lemma:Sparse_Cmax=Ceven_Gammas12} and \Cref{table:Bookkeeping_Cmax=Ceven_Gammas12}.

%\fcolorbox{}{gray!10}
\tcbox[size=fbox, colback=gray!10]{Suppose $n=5$, and construction in $\Gamma_1 \cup \Gamma_2$ is inherited from Case (1B) or (2B) of \Cref{lemma:Sparse_Cmax=Ceven_Gammas12}.} % Consider Case 1, where $c_1 \geq 3$. 

We canonically calibrate our braid with respect to $\mathbbm{b}_{2, 1}$, and then canonically calibrate $\beta$ with respect to $\mathbbm{b}_{3,1}$. The braid $\beta$ now begins with a $\sigma_4$ letter, and $d_L(4; 1,2) \geq 1$. Define $j$, where $2 \leq j \leq c_4$, to be the minimal integer so that $d_L(4; 2,j)=0$; see \Cref{fig:CmaxCeven_Nis5_CaseA}. We co-orient $\alpha_{4, 1}, \ldots, \alpha_{4,j-1}$ to the left, and co-orient $\alpha_{4,j}, \ldots, \alpha_{4, c_4-1}$ to the right. We claim the resulting branched surface is sink disk free. 

As noted in \Cref{table:Bookkeeping_Cmax=Ceven_Gammas12}, we already argued that there is a unique sink disk spanning $\Gamma_1 \cup \Gamma_2$. According to the presentation of the braid in \Cref{lemma:Sparse_Cmax=Ceven_Gammas12}, and as seen in \Cref{fig:NotGeneric_CmaxCeven_Gammas12_CaseA}, it is the sector $\mathcal{S}$ with $\mathcal{S} \sim \mathbbm{b}_{2,1} \sim \mathcal{S}_3 \sim \mathbbm{b}_{3,1}$. Since we have performed two canonical calibrations of our braid, the relative positions of $\mathbbm{b}_{2,1}$ and $\mathbbm{b}_{3,1}$ may have changed. So, in our current presentation of $\beta$, we refer to these bands as $\mathbbm{b}_{2, \theta}$ and $\mathbbm{b}_{3, \eta}$, as in \Cref{fig:CmaxCeven_Nis5_CaseA}. (In particular, $\alpha_{2, \theta}$ is a right pointer and $\mathbbm{b}_{2, \theta-1}$ is not enclosed, on the left, by any plumbing arcs.) By design, $\mathbbm{b}_{2, \theta}$ is not enclosed, on the right, by any arcs, so $\mathbbm{b}_{2, \theta} \sim \mathcal{S}_3$.  Therefore, to show that there are no sink disks in $B$, it suffices to show that there are no sink disks spanning $\Gamma_3 \cup \Gamma_4$. 

\begin{figure}[h!]
\labellist \tiny
%\pinlabel {(A)} at 66 5
%\pinlabel {(B)} at 235 5
%\pinlabel {(C)} at 406 5
%
\pinlabel {$\Gamma_{1}$} at 32 175
\pinlabel {$\Gamma_{2}$} at 60 175
\pinlabel {$\Gamma_{3}$} at 88 175
\pinlabel {$\Gamma_{4}$} at 115 175
\pinlabel {$\Gamma_{1}$} at 208 175
\pinlabel {$\Gamma_{2}$} at 235 175
\pinlabel {$\Gamma_{3}$} at 265 175
\pinlabel {$\Gamma_{4}$} at 290 175
\pinlabel {$\Gamma_{1}$} at 375 175
\pinlabel {$\Gamma_{2}$} at 403 175
\pinlabel {$\Gamma_{3}$} at 430 175
\pinlabel {$\Gamma_{4}$} at 455 175
\pinlabel {$j$} at 110 95
\pinlabel {$j$} at 285 95
\pinlabel {$j$} at 453 95
\pinlabel {$\theta$} at 58 132
\pinlabel {$\theta$} at 233 132
\pinlabel {$\theta$} at 401 132
\pinlabel {$\eta$} at 84 128
\pinlabel {$\eta$} at 260 128
\pinlabel {$\eta$} at 425 128
\endlabellist
    %\begin{framed}
        \includegraphics[scale=0.98]{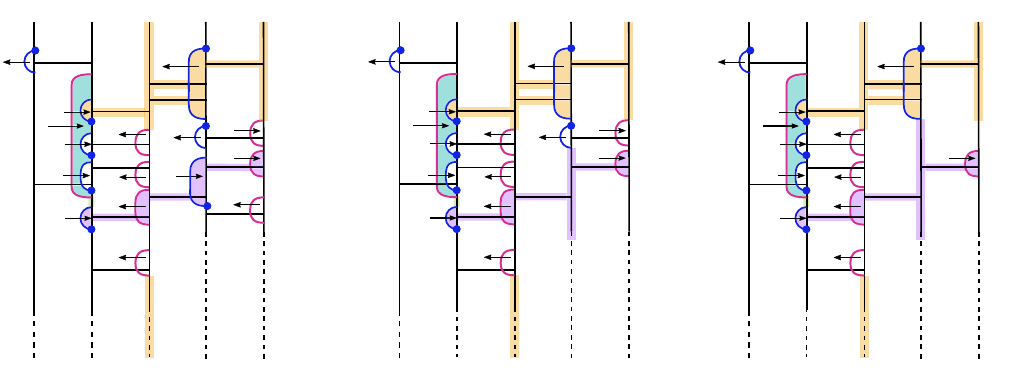}
        \caption{Building $B$ when $\Cmax = \Ceven$ and $n=5$. The construction in $\Gamma_1 \cup \Gamma_2$ is inherited from either Case (1) or Case (2A) \Cref{lemma:Sparse_Cmax=Ceven_Gammas12}. On the \textbf{left}, we see that $j < c_4$. In the \textbf{middle}, $j = c_4 \geq 3$. On the \textbf{right}, $j = c_4 = 2$. Regardless of the frame, we see that $\mathcal{S}_3 \sim \mathcal{S}_5$. Note: while it is possible that the band $\mathbbm{b}_{3, \eta-1}$ is enclosed, on the left, by some left pointing image arc, we suppress this possibility from our figures, as it does not impact our sink disk analysis.}
        \label{fig:CmaxCeven_Nis5_CaseA}
    %\end{framed}
\end{figure}
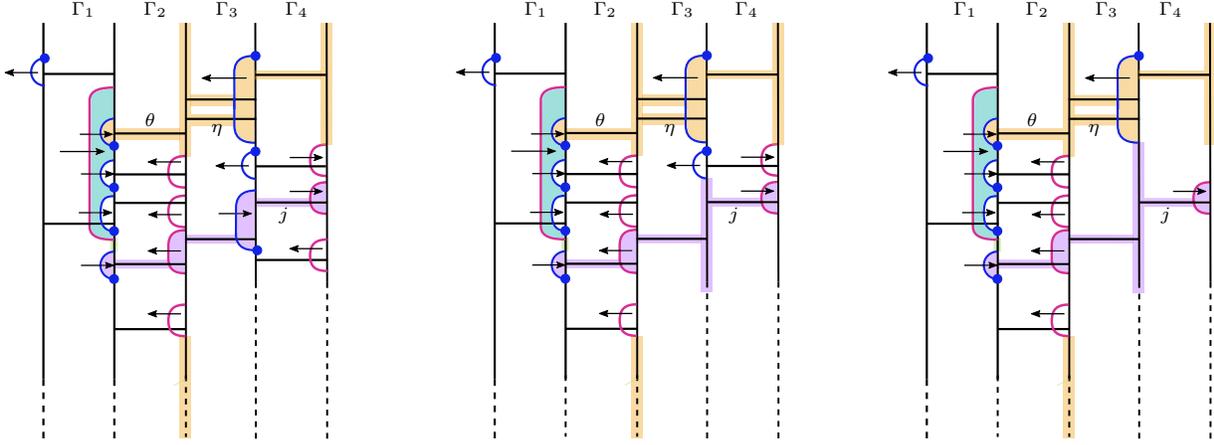

\begin{itemize}
\item \textbf{Seifert disk sector analysis:} By design, $\mathcal{S}_3 \sim \mathbbm{b}_{2, \theta} \sim \mathbbm{b}_{3, \eta}$. By construction, $\mathbbm{b}_{3, \eta}$ is enclosed, on the right, by the left pointing plumbing arc $\alpha_{4,1}$. In particular, this means $\mathbbm{b}_{3, \eta} \sim \mathbbm{b}_{4,1}$, where the band $\mathbbm{b}_{4,1}$ is not enclosed, on the right, by any arcs. Therefore, $\mathcal{S}_3 \sim \mathcal{S}_5$. Since $\partial \mathcal{S}_5$ contains $\varphi(\alpha(4,1))$, which points out of the sector, we deduce that $\mathcal{S}_3$ and $\mathcal{S}_5$ are not sinks. See \Cref{fig:CmaxCeven_Nis5_CaseA}. 

We analyze $\mathcal{S}_4$. If $j < c_4$, then $\alpha_{4,j}$ points out of $\mathcal{S}_4$, so it is not a sink; see \Cref{fig:CmaxCeven_Nis5_CaseA} (left). Otherwise, $j=c_4$, and $\mathcal{S}_4 \sim \mathbbm{b}_{3,x}$, where $x=d_L(4; 1,2) +1$, as in \Cref{fig:CmaxCeven_Nis5_CaseA} (middle, right). This band is either not enclosed on the left by any image arc (in which case $\mathcal{S}_4 \sim \mathcal{S}_3$, so it is not a sink), or it is enclosed on the left by some image arc. However, all the image arcs in $S_3$ are left pointers, so there is some arc pointing out of $\mathcal{S}_4$. We deduce that none of the disk sectors spanning $\Gamma_3 \cup \Gamma_4$ are sinks. 
\item \textbf{Polygon sector analysis:} None of the Seifert disks $S_3, S_4,$ or $S_5$ contain both plumbing and image arcs, so there are no polygon sectors contained in $\Gamma_3 \cup \Gamma_4$.
\item \textbf{Horizontal sector analysis:} Suppose $\mathcal{H}$ is a horizontal sector spanning $\Gamma_3$. So, by definition, $\mathcal{H} \sim \mathbbm{b}_{3,k}$, where $1 \leq k \leq c_3$. If $\mathbbm{b}_{3,k}$ is not enclosed, on the left, by any image arcs, then $\mathcal{H} \sim \mathcal{S}_3$, which we already argued is not a sink. Otherwise, $\mathbbm{b}_{3,k}$ is enclosed, on the left, by a left pointing image arc, so $\mathcal{H}$ is not a sink. 

If $\mathcal{H}$ is a horizontal sector spanning $\Gamma_4$, then $\mathcal{H} \sim \mathbbm{b}_{4,k}$, for $1 \leq k \leq c_4$. 

\begin{itemize}
\item[$\circ$] It $\mathcal{H} \sim \mathbbm{b}_{4,k}$, for $1 \leq k \leq j-1$, then $\alpha_{4,k}$ points out of the sector, and it is not a sink. 
\item[$\circ$] Next, suppose $\mathcal{H} \sim \mathbbm{b}_{4,j}$. 

	\begin{itemize}
	\item[$\bullet$] If $j=c_4$, then $\mathcal{H} \sim \mathbbm{b}_{4,c_4}$. But $\mathbbm{b}_{4,c_4} \sim \mathcal{S}_4$, which is not a sink, so neither is $\mathcal{H}$. 
	\item[$\bullet$] Suppose $j < c_4$. In this case, $\mathbbm{b}_{4,j}$ is enclosed, on the left, by the right pointing arc $\alpha_{4,j}$. But this means $\mathcal{H} \sim \mathbbm{b}_{3, \star}$, and in particular, $\mathcal{H}$ travels westwards to meet $S_3$. Thus, $\mathcal{H}$ is either in the same branch sector as $\mathcal{S}_3$ (which we already argued is not a sink), or it meets $S_3$ in an left pointing image arc (so it is not a sink). 
	\end{itemize}
\item[$\circ$] Finally, notice that the band $\mathbbm{b}_{4,k}$ with $j+1 \leq k \leq c_4$ is enclosed, on the right, by $\varphi(\alpha_{4,k-1})$, which points out of the band. Thus, if $\mathcal{H}$ contains $\mathbbm{b}_{4,k}$, $\mathcal{H}$ is not a sink.  	
\end{itemize}

Therefore, there are no horizontal sink disk sectors spanning $\Gamma_3 \cup \Gamma_4$. 
\end{itemize}

We deduce that there are no sink disks spanning $\Gamma_3 \cup \Gamma_4$, and so we have constructed a sink disk free branched surface. Next, we construct the branched surface in $\Gamma_3 \cup \Gamma_4$ when we are in Cases (1A) or (2A) of \Cref{lemma:Sparse_Cmax=Ceven_Gammas12}.

\tcbox[size=fbox, colback=gray!10]{Suppose $n=5$, and construction in $\Gamma_1 \cup \Gamma_2$ is inherited from Cases (1A) or (2A) from \Cref{lemma:Sparse_Cmax=Ceven_Gammas12}.} 

In this case, we built $B$ in $\Gamma_1 \cup \Gamma_2$ using Case (1A) or Case (2A) of \Cref{lemma:Sparse_Cmax=Ceven_Gammas12}. In particular, this means $\Gamma_2$ is heterogeneous. We then follow the instructions of  \Cref{lemma:Generic_EndOfBraid_LastColumnIsFull} to build $B$ in $\Gamma_3 \cup \Gamma_4$. \Cref{table:Bookkeeping_Cmax=Ceven_Gammas12} tells us that there was a unique potential sink disk spanning $\Gamma_1 \cup \Gamma_2$, and it was a horizontal sector which also spanned $\Gamma_3$. However, by applying the conclusions of \Cref{lemma:Generic_EndOfBraid_LastColumnIsFull}, we know that there are no sink disks spanning $\Gamma_2 \cup \Gamma_3 \cup \Gamma_4$. Therefore, we deduce that we have produced a branched surface containing co-oriented product disks from all columns except $\Gamma_3$, and this branched surface is sink disk free. 

%We canonically calibrate our braid with respect to $\mathbbm{b}_{2, j+1}$, and then canonically calibrate $\beta$ with respect to $\mathbbm{b}_{3,1}$. The braid $\beta$ now begins with a $\sigma_4$ letter.
%Here, we directly apply the construction of \Cref{lemma:Generic_EndOfBraid_LastColumnIsFull}} to build $B$ in $\Gamma_4$, and know that the resulting branched surface is sink disk free. \Cref{fig:CmaxCeven_Nis5_CaseB} shows $B$ in $\Gamma_1 \cup \Gamma_2 \cup \Gamma_3$. 

%\textcolor{blue}{We canonically calibrate our braid with respect to $\mathbbm{b}_{2, \star}$, which presents our braid as $\beta \approx \sigma_3 \omega_1 \sigma_2 \omega_2$ (where the explicit $\sigma_2$ in the braid word was formally known a $\mathbbm{b}_{2, \star}$). We now canonically calibrate $\beta$ with respect to $\mathbbm{b}_{3,1}$, so that $\beta$ begins with a $\sigma_4$ letter. }

%Therefore, in the $n=5$ case, we can build a sink disk free branched surface which contains a single product disk from $\Gamma_1$ and $c_i-1$ product disks from $\Gamma_i$ for $i=2,4$. We now proceed with building $B$ for all $n \geq 6$.

%\begin{center}
%\begin{figure}[h!]
%    \begin{framed}
%        \includegraphics[scale=0.5]{CmaxCeven_Nis5_CaseB}
%        \caption{\textcolor{red}{add caption}}
%        \label{fig:CmaxCeven_Nis5_CaseB}
%    \end{framed}
%\end{figure}
%\end{center}

\tcbox[size=fbox, colback=gray!30]{Suppose $n \geq 6$.}

We begin by explaining how to build $B$ in $\Gamma_1 \cup \Gamma_2 \cup \Gamma_3 \cup \Gamma_4$: we apply \Cref{lemma:Sparse_Cmax=Ceven_Gammas12} to construct $B$ in $\Gamma_1 \cup \Gamma_2$. As summarized in \Cref{table:Bookkeeping_Cmax=Ceven_Gammas12}, despite the fact that we used two different constructions across the four sub-cases, we were able to ensure that the unique potential sink disk spanning $\Gamma_1 \cup \Gamma_2$ always contained $\mathbbm{b}_{3,1}$. To continue building $B$ in $\Gamma_3 \cup \Gamma_4$, we canonically calibrate our braid with respect to $\mathbbm{b}_{3,1}$, and then apply our template to $\Gamma_4$.%, and then include no product disks from $\Gamma_5$. 

We claim that the only potential sink disk spanning $\Gamma_1 \cup \Gamma_2 \cup \Gamma_3$ is $\mathcal{S}_4$. Since \Cref{lemma:Sparse_Cmax=Ceven_Gammas12} (as summarized in \Cref{table:Bookkeeping_Cmax=Ceven_Gammas12}) ensures that the only potential sink disk spanning $\Gamma_1 \cup \Gamma_2$ is the branch sector containing $\mathbbm{b}_{3,1}$, it suffices to analyze $\Gamma_3$. %However, since $n=6$, and we directly applied our template to $\Gamma_4$, we can apply \Cref{lemma:Generic_EndOfBraid_LastColumnIsSparse} to see that there are no sink disks spanning $\Gamma_4 \cup \Gamma_5$. Therefore, it suffices to prove that there are no sink disks spanning $\Gamma_3$.

We split our proof depending on whether the construction of $B$ in $\Gamma_1 \cup \Gamma_2$ followed via Cases (1A) and (2A), or Cases (1B) and (2B) of \Cref{lemma:Sparse_Cmax=Ceven_Gammas12}, and refer to Figures \ref{fig:NotGeneric_CmaxCeven_Gammas12_CaseB} and \ref{fig:NotGeneric_CmaxCeven_Gammas12_CaseA} (and the labelings of bands in those figures) throughout.

\tcbox[size=fbox, colback=gray!10]{Suppose $n \geq 6$ and construction in $\Gamma_1 \cup \Gamma_2$ is inherited from Cases (1A) or (2A) from \Cref{lemma:Sparse_Cmax=Ceven_Gammas12}.} 

\begin{itemize}
\item \Seifert In \Cref{lemma:Sparse_Cmax=Ceven_Gammas12}, we already proved that $\mathcal{S}_3$ is not a sink. We do not know if $\Gamma_4$ is homogeneous or heterogeneous, so we cannot argue that $\mathcal{S}_4$ is not a sink. 
\item \Polygon Neither $S_3$ nor $S_4$ contain both plumbing and image arcs, so neither Seifert disk contains polygon sectors. 
\item \Horizontal Let $\epsilon := d(2; j, j+1)$. As we see in \Cref{fig:NotGeneric_CmaxCeven_Gammas12_CaseB}, the bands $\mathbbm{b}_{3,1}, \ldots, \mathbbm{b}_{3, \epsilon}$ are all enclosed, on the left, by $\varphi(\alpha_{2,j})$. Therefore, $\mathbbm{b}_{3,1} \sim \ldots \sim \mathbbm{b}_{3, \epsilon}$. Since we canonically calibrated $\beta$ with respect to $\mathbbm{b}_{3,1}$ and then applied our template to $\Gamma_4$, we know that $\mathbbm{b}_{3,1}$ is enclosed, on the right, by $\alpha_{4,1}$, which is a left pointer. We deduce that if $\mathcal{H} \sim \mathbbm{b}_{3, \ell}$, for $1 \leq \ell \leq \epsilon$, then $\mathcal{H}$ is not a sink. 

Alternatively, if $\mathcal{H} \sim \mathbbm{b}_{3, \ell}$ for $\epsilon +1 \leq \ell \leq c_3$, then $\mathcal{H}$ is either enclosed, on the left, by a left pointing image arc, or $\mathcal{H} \sim \mathcal{S}_3$. Regardless of which occurs, we know $\mathcal{H}$ is not a sink.
\end{itemize}

We deduce that no branch sector spanning $\Gamma_3$ is a sink, and we have produced a branched surface $B$ such that the unique potential sink disk spanning $\Gamma_1 \cup \Gamma_2 \cup \Gamma_3$ is $\mathcal{S}_4$. 

\tcbox[size=fbox, colback=gray!10]{Suppose $n= 6$ and construction in $\Gamma_1 \cup \Gamma_2$ is inherited from Cases (1B) or (2B) from \Cref{lemma:Sparse_Cmax=Ceven_Gammas12}.}

\begin{itemize}
\item \Seifert In this case $\mathcal{S}_3 \sim \mathbbm{b}_{3,1}$; by design, this band is enclosed, on the right, by $\alpha_{4,1}$. This arc is a left pointer, pointing out of $\mathcal{S}_3$, so we deduce that the sector is not a sink. We do not know whether $\mathcal{S}_4$ is a sink or not.
\item \Polygon Neither $S_3$ nor $S_4$ contain both plumbing and image arcs, so neither Seifert disk contains polygon sectors. 
\item \Horizontal Suppose $\mathcal{H}$ spans $\Gamma_3$. If $\mathcal{H} \sim \mathbbm{b}_{3,1}$, then we can apply the same argument as in the Seifert disk sector analysis to show see that $\mathcal{H}$ is not a sink. If $\mathcal{H} \sim \mathbbm{b}_{3, \ell}$, for $2 \leq \ell \leq c_3$, then, as we see in \Cref{fig:NotGeneric_CmaxCeven_Gammas12_CaseA}, $\mathcal{H}$ either meets $S_3$ in an exiting image arc, or $\mathcal{H} \sim \mathcal{S}_3$, which we already argued is not a sink.
\end{itemize}

We deduce that no branch sector spanning $\Gamma_3$ is a sink, and we have produced a branched surface $B$ such that the unique potential sink disk spanning $\Gamma_1 \cup \Gamma_2 \cup \Gamma_3$ is $\mathcal{S}_4$. 

\tcbox[size=fbox, colback=gray!30]{Returning to build $B$ in the remaining columns of the braid.}

We have now constructed $B$ in $\Gamma_1 \cup \Gamma_2 \cup \Gamma_3 \cup \Gamma_4$, and we know that the only potential sink disk spanning $\Gamma_1 \cup \Gamma_2 \cup \Gamma_3$ is $\mathcal{S}_4$. 

Suppose $n=6$. We apply \Cref{lemma:Generic_EndOfBraid_LastColumnIsSparse} to conclude that there are no sink disks spanning $\Gamma_4 \cup \Gamma_5$. We deduce that we have constructed a sink disk free branched surface containing co-oriented product disks from $\Gamma_1, \Gamma_2$, and $\Gamma_4$. 

Suppose $n \geq 7$. We iteratively build $B$ in the remaining columns. We define the value $f$ as follows:
\begin{align*}
f :=
\begin{cases}
\displaystyle \frac{(n-1)-5}{2}-1 & \qquad n-1 \equiv 1 \mod 2 \\ \\
\displaystyle \frac{(n-1)-4}{2}-2 & \qquad n-1 \equiv 0 \mod 2
\end{cases}
\end{align*}

To build $B$, we iteratively apply \Cref{lemma:Continue}, in sequence, to $\Gamma_{4 + 2t}$, for $t \in \{0, \ldots, f\}$. By construction, before each application of \Cref{lemma:Continue}, the only potential sink disk is $\mathcal{S}_{4+2t}$, but after building $B$ in $\Gamma_{4+2t} \cup \Gamma_{4+2t+1} \cup \Gamma_{4+2t+2}$ as directed by \Cref{lemma:Continue}, the only potential sink disk is $\mathcal{S}_{4+2t+2}$. 

Suppose $n-1 \equiv 1 \mod 2$. When $t=f$, our last application of \Cref{lemma:Continue} assumes our template has been applied to $\Gamma_{n-4}$, and concludes by building $B$ in $\Gamma_{n-2}$. At the end of this procedure, the only potential sink disk is $\mathcal{S}_{n-2}$. We now apply \Cref{lemma:Generic_EndOfBraid_LastColumnIsSparse} to see that there are no sinks spanning $\Gamma_{n-2} \cup \Gamma_{n-1}$. Therefore, we have constructed a sink disk free branched surface which contains product disks from $\Gamma_1 \cup \Gamma_2 \cup \Gamma_4 \cup \ldots \cup \Gamma_{n-2}$. 

Instead suppose $n-1 \equiv 0 \mod 2$. When $t=f$, our last application of \Cref{lemma:Continue} assumes our template has been applied to $\Gamma_{n-3}$, and the only potential sink disk is $\mathcal{S}_{n-3}$. We now apply \Cref{lemma:Generic_EndOfBraid_LastColumnIsFull} to build $B$ in $\Gamma_{n-1}$. By design,  \Cref{lemma:Generic_EndOfBraid_LastColumnIsFull} ensures that there are no sink disks spanning $\Gamma_{n-3} \cup \Gamma_{n-2} \cup \Gamma_{n-1}$. We deduce that we have produced a sink disk free branched surface containing co-oriented product disks from $\Gamma_1 \cup \Gamma_2 \cup \Gamma_4 \cup \ldots \cup \Gamma_{n-1}$.

Therefore, when $\Cmax = \Ceven$ and $n \geq 5$, we have constructed a sink disk free branched surface. Applying \Cref{prop:BisLaminar}, we deduce that $B$ is a laminar branched surface. It remains to determine what slopes are fully carried by $B$. Our analysis is determined by the parity of $n$. 

Since $\Cmax = \Ceven$, our branched surface $B$ contains a copy of the fiber surface, a single product disk from $\Gamma_1$, and all product disks from $\Gamma_i$, $i \equiv 0 \mod 2$. Combining Lemmas \ref{lemma:Sparse_Cmax=Ceven_Gammas12}, \ref{lemma:Continue}, \ref{lemma:Generic_EndOfBraid_LastColumnIsSparse}, and \ref{lemma:Generic_EndOfBraid_LastColumnIsFull}, we see that there is no pairwise linking between any of the plumbing arcs used to build $B$. Therefore, it suffices to put a lower bound on the number of product disks used to build $B$. 

\begin{itemize}
\item If $n-1 \equiv 1 \mod 2$, then $n \equiv 0 \mod 2$, and for $2g(K)-1=\mathcal{C}-n$ to hold, we must have that $\mathcal{C}$ is odd. Therefore, since $\Cmax \geq \frac{\mathcal{C}+1}{2}$. Performing our usual estimations, we see that
\begin{align*}
\TauSup &= (\# \text{product disks used to construct $B$}) \\
&= \left( \ \sum_{i \text{ even}} c_i \right) - (\# \Gamma_i, i \text{ even}) +(\# \text{ product disks from } \Gamma_1) \\ 
&= \Ceven - \frac{n-2}{2} + 1 \\
&\geq \left( \frac{\mathcal{C}+1}{2} - \frac{n-2}{2}\right)+1 = \frac{\mathcal{C}-n+3}{2}+1 = g(K)+2
\end{align*}
We deduce that $\tau_B$ carries all slopes $r < g(K)+1$, as desired. 
\item If $n-1 \equiv 0 \mod 2$, then $n \equiv 1 \mod 2$, and for $2g(K)-1=\mathcal{C}-n$ to hold, we have that $\mathcal{C}$ is even. Therefore, since $\Cmax \geq \frac{\mathcal{C}}{2}$. Therefore, 
\begin{align*}
\TauSup \geq \left( \frac{\mathcal{C}}{2} - \frac{n-1}{2}\right)+1 = \frac{\mathcal{C}-n+1}{2}+1 = g(K)+1
\end{align*}
Therefore, $\tau_B$ carries all slopes $r < g(K)+1$ as desired. 
\end{itemize}

We deduce that if $\Cmax = \Ceven$, we can construct a laminar branched surface such that $\tau_B$ fully carries all slopes $r < g(K)+1$. 
\end{proof}

We are now ready to prove \Cref{thm:main} in its entirety. 

\textbf{\Cref{thm:main}.} \textit{If $K$ is a positive braid knot with $g(K) \geq 2$, then $S^3_r(K)$ admits a taut foliation whenever $r < g(K)+1$.}

\begin{proof}
If $K$ is a positive braid on $n \leq 3$ strands, then we can apply the main theorem of \cite{Krishna:3Braids} to deduce that $S^3_r(K)$ admits a taut foliation whenever $r < 2g(K)-1$. Since we assumed that $g(K) \geq 2$, we know, in particular, that $S^3_r(K)$ admits a taut foliation whenever $r < g(K)+1$. 

If $K$ is a positive braid knot on $n = 4$ strands, we apply \Cref{thm:Main4Braids}. Otherwise, $K$ is a positive braid knot on $n \geq 5$ strands, so we combine \Cref{thm:PositiveNBraids_Odd} and \Cref{thm:PositiveNBraids_Even} to deduce that there is a sink disk free branched surface $B$ such that $\tau_B$ fully carries all slopes $r < g(K)+1$.  Applying \Cref{prop:BisLaminar} and \Cref{thm:foliations}, we deduce that $S^3_r(K)$ admits a taut foliation for all $r < g(K)+1$.
\end{proof}

\section{Splicing knot exteriors} \label{section:Splicing}

To conclude, we apply \Cref{thm:main} to construct taut foliations in splicings of braid positive knot exteriors. 

\textbf{\Cref{cor:splicing}.} \textit{Suppose $K_1$ and $K_2$ are positive braid knots of genus at least two. Let $M(K_1, K_2)$ denote the splicing of the knot exteriors of $K_1$ and $K_2$. Then $M(K_1, K_2)$ is a non-L-space if and only if it admits a taut foliation. Moreover, $\pi_1(M(K_1, K_2))$ is left-orderable.}

As noted in \Cref{subsection:results}, this is a special case of a theorem of Boyer-Gordon-Hu \cite[Theorem 2.1]{BoyerGordonHu:Toroidal}. In our proof, the restriction of the taut foliation on the essential torus is a foliation by circles; in the Boyer-Gordon-Hu proof, the foliation induced on the essential torus is a suspension. 

\begin{proof}
We claim that the only thing to prove is that if $K_1, K_2$ are positive braid knots of genus at least two, then $M(K_1, K_2)$ admits a taut foliation. This is because:

\begin{itemize}
\item If $M(K_1, K_2)$ admits a taut foliation, then it is a non-L-space \cite{OSz:HolDisks, Bowden:Approx, KazezRoberts}.
\item By Hedden-Levine \cite[Corollary 2]{HeddenLevine:Splicing}, $M(K_1, K_2)$ is not an L-space. 
\item If $Y$ is an integer homology sphere with a taut foliation, then $\pi_1(Y)$ is left-orderable \cite{BoyerRolfsonWiest}. 
\end{itemize}

Therefore, the rest of the proof is devoted to constructing taut foliations in the splicings.

First, we make a few observations about the construction used to prove \Cref{thm:main}. %The final statement we presented is the following: if $K$ is a positive braid knot, then for each $r < g(K)+1$, $S^3_r(K)$ admits a taut foliation, which we denote by $\widehat{\mathcal{F}_r}$. But something stronger is true: 
The construction produces, for any $K$ a positive braid knot, and any $r < g(K)+1$, a taut foliation $\mathcal{F}_r$ in the knot exterior $X_K$, such that $\mathcal{F}_r$ meets $\partial X_K$ in parallel simple closed curves of slope $r$. When we perform the $r$-framed Dehn filling of $X_K$, the taut foliation $\mathcal{F}_r$ gets capped off to produce the taut foliation $\widehat{\mathcal{F}_r}$. Moreover, the dual knot of the surgery (i.e. the core of the Dehn surgery solid torus) \textit{is} a transversal for the taut foliation $\widehat{\mathcal{F}_r}$. In particular, suppose $\rho$ is a simple closed curve on $\partial X_K$ of slope $r$. If $\gamma$ is a simple closed curve on $\partial X_K$ such that $\gamma$ and $\rho$ have intersection number $\pm 1$, then we can push $\gamma$ slightly into the interior of $X_K$ to produce a transversal for the taut foliation $\mathcal{F}_r$ of $X_K$. The simple closed curve $\gamma$ remains a transversal for $\widehat{\mathcal{F}_r}$, the taut foliation of $S^3_r(K)$, as this is produced from $\mathcal{F}_r$ by capping off. We use this essential observation to prove \Cref{cor:splicing}.  

Suppose $K_1$ and $K_2$ are positive braid knots, and let $X_1$ and $X_2$ denote their respective knot exteriors. Since $g(K_1) \geq 2$ and $g(K_2) \geq 2$, by the observation above, $X_1$ (resp. $X_2$) admits a taut foliation $\mathcal{F}_1$ (resp. $\mathcal{F}_2$) which meets $\partial X_1$ (resp. $\partial X_2$) in parallel simple closed curves of slope +1. 

Next, we study the splicing $M(K_1, K_2)$. As described in \Cref{subsection:results}, $M(K_1, K_2)$ is constructed by identifying $\partial X_1$ and $\partial X_2$ together via an orientation-reversing homeomorphism; we call this map $\varphi$. The map $\varphi$ identifies $\lambda_1$ and $\mu_2$, and also identifies $\lambda_2$ and $\mu_1$. It follows that $\varphi$ identifies the slope $+1 = \mu_1 + \lambda_1$ on $\partial X_1$ with the slope $+1 = \mu_2 + \lambda_2$ on $\partial X_2$. 

Therefore, we have endowed both $X_1$ and $X_2$ with the taut foliations $\mathcal{F}_1$ and $\mathcal{F}_2$, and after performing the splicing operation, the boundaries of the leaves on $\partial X_1$ and $\partial X_2$ exactly line up, thereby producing a foliation $\mathcal{F}$ of $M(K_1, K_2)$. 

We emphasize that this foliation $\mathcal{F}$ really is taut: suppose $\rho$ is a simple closed curve on $\partial X_1$ of slope $+1$. Constructing any curve $\gamma$ with intersection number one with $\rho$ (as described in the first paragraph), and push $\gamma$ into the interior of $X_1$ prior to splicing. This curve will be transverse to all the leaves of the foliation produced after splicing, so $\mathcal{F}$ is taut. Said differently, rather than capping off the leaves of $\mathcal{F}_1$ with disks (as you would when performing Dehn surgery), the leaves of $\mathcal{F}_1$ are getting capped off by the leaves of $\mathcal{F}_2$. 
\end{proof}

\appendix
\section{Appendix A: Infinitely many $K_m$ are hyperbolic} \label{appendix}

We prove the knot $K_1$ of \Cref{thm:examples} is hyperbolic. This requires using SnapPy \cite{SnapPy}: we input $\beta_1$ (which we defined in our proof of \Cref{thm:examples}) into Snappy, which computes the volume of the complement of $\widehat{\beta_1}$.

\begin{quotation}
\begin{verbatim}
In[1]: K=Link(braid_closure=[1,2,3,1,2,3,1,2,3,1,2,3,1,2,3,1,2,3,1,2,3,
                             3,2,3,2,3,2])

In[2]: X=K.exterior()

In[3]: X.volume()

Out[3]: 4.1885842865

In[4]: X.solution_type()

Out[4]: `all tetrahedra positively oriented'
\end{verbatim}
\end{quotation}

\medskip

Therefore, $K_1$ is a hyperbolic knot in $S^3$, and $vol(S^3-K_1) \approx 4.188$. We now show that infinitely many other of the $\{K_m \ | \ m \geq 2\}$ from \Cref{thm:examples} are hyperbolic as well. We begin by constructing a 5-braid $\gamma$ as follows: $$\gamma: = \beta_1 \sigma_4^{-1}\sigma_3^{-1}\sigma_2^{-1}\sigma_2^{-1}\sigma_3^{-1}\sigma_4^{-1}$$

The braid $\gamma$ is seen in \Cref{fig:4Braid_Hyperbolic} (left). It is straightforward to see that $\widehat{\gamma}$ is a 2-component link, where one of the components is $\widehat{\beta_1}$, and the other is an unknot which encircles the second, third, and fourth strands of $\beta_1$; see \Cref{fig:4Braid_Hyperbolic} (right). Below, we include the SnapPy verification that the link $\widehat{\gamma}$ is hyperbolic:

\begin{quotation}
\begin{verbatim}
In[1]: L=Link(braid_closure=[1,2,3,1,2,3,1,2,3,1,2,3,1,2,3,1,2,3,1,2,3,
                             3,2,3,2,3,2,-4,-3,-2,-2,-3,-4])

In[2]: X=L.exterior()

In[3]: X.volume()

Out[3]: 6.2903026840

In[4]: X.solution_type()

Out[4]: `all tetrahedra positively oriented'
\end{verbatim}
\end{quotation}

\medskip

Doing $\frac{1}{-(m-1)}$--framed Dehn surgery along the unknotted component of $\gamma$ yields $S^3$ again, and adds $(m-1)$ positive full twists into the second, third, and fourth strands of the braid (see \cite[Figure 5.27]{GompfAndStipsicz}). Therefore, we can produce all the knots $K_m$, where $m \geq 2$, via Dehn surgery along one of the components of $\widehat{\gamma}$. Applying Thurston's hyperbolic Dehn surgery theorem \cite{Thurston:Vol4}, \cite[Theorem A]{NeumannZagier}, we know that infinitely many of these knots are hyperbolic.

\begin{figure}[h!]
\labellist
\pinlabel {$\beta_1$} at 70 250
\pinlabel {$\beta_1$} at 375 250
\pinlabel {$\frac{1}{-(m-1)}$} at 480 150
\endlabellist
    %\begin{framed}
        \includegraphics[scale=0.75]{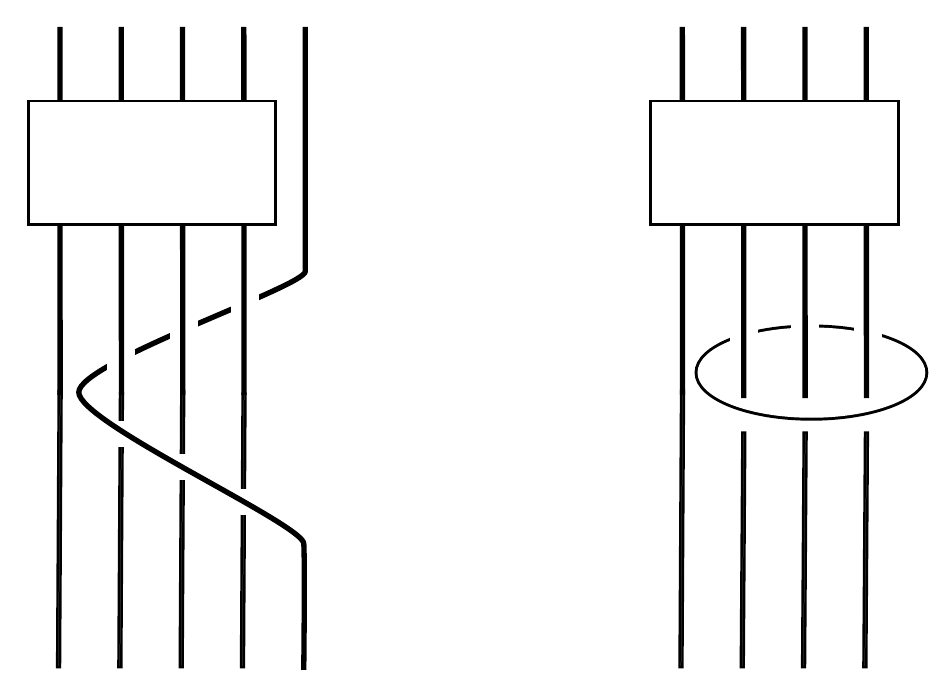}
        \caption{\textbf{Left:} The 5-braid $\gamma$. \textbf{Right:} The link $\widehat{\gamma}$, but we have suppressed drawing the closure of $\beta_1$ for simplicity.Performing $-1/m$ surgery along the unknotted component yields $\widehat{\beta_m}$.}
        \label{fig:4Braid_Hyperbolic}
    %\end{framed}
\end{figure}

\bibliographystyle{amsalpha2}
\bibliography{masterbiblio}

\providecommand{\bysame}{\leavevmode\hbox to3em{\hrulefill}\thinspace}
\providecommand{\MR}{\relax\ifhmode\unskip\space\fi MR }
\providecommand{\MRhref}[2]{%
  \href{http://www.ams.org/mathscinet-getitem?mr=#1}{#2}
}
\providecommand{\href}[2]{#2}
\begin{thebibliography}{HRRW15}

\bibitem[Baa13]{Baader:MaximalSignature}
Sebastian Baader, \emph{Positive braids of maximal signature}, Enseign. Math.
  \textbf{59} (2013), no.~3-4, 351--358. \MR{3189041}

\bibitem[BLL18]{BaaderLewarkLiechti}
Sebastian Baader, Lukas Lewark, and Livio Liechti, \emph{Checkerboard graph
  monodromies}, Enseign. Math. \textbf{64} (2018), no.~1-2, 65--88.
  \MR{3959848}

\bibitem[BS22]{BaldwinSivek:Splicing}
John~A. Baldwin and Steven Sivek, \emph{Instanton {$L$}-spaces and splicing},
  Ann. H. Lebesgue \textbf{5} (2022), 1213--1233. \MR{4526251}

\bibitem[Ber18]{Berge}
John Berge, \emph{Some knots with surgeries yielding lens spaces},
  https://arxiv.org/abs/1802.09722 (2018).

\bibitem[BB15]{BoileauBoyer:HomologySpheres}
Michel Boileau and Steven Boyer, \emph{Graph manifold {$\Bbb{Z}$}-homology
  3-spheres and taut foliations}, J. Topol. \textbf{8} (2015), no.~2, 571--585.
  \MR{3356771}

\bibitem[BBG19a]{BoileauBoyerGordon:SQP1}
Michel Boileau, Steven Boyer, and Cameron~McA. Gordon, \emph{Branched covers of
  quasi-positive links and {L}-spaces}, J. Topol. \textbf{12} (2019), no.~2,
  536--576. \MR{4072174}

\bibitem[BBG19b]{BoileauBoyerGordon:SQP2}
\bysame, \emph{On definite strongly quasipositive links and {L}-space branched
  covers}, Adv. Math. \textbf{357} (2019), 106828, 63. \MR{4016557}

\bibitem[Bow16]{Bowden:Approx}
Jonathan Bowden, \emph{Approximating {$C^0$}-foliations by contact structures},
  Geom. Funct. Anal. \textbf{26} (2016), no.~5, 1255--1296.

\bibitem[BC15]{BoyerClay2}
Steven Boyer and Adam Clay, \emph{Slope detection, foliations in graph
  manifolds, and {L}-spaces}, https://arxiv.org/abs/1510.02378 (2015).

\bibitem[BC17]{BoyerClay}
Steven Boyer and Adam Clay, \emph{Foliations, orders, representations,
  {L}-spaces and graph manifolds}, Adv. Math. \textbf{310} (2017), 159--234.

\bibitem[BGH21]{BoyerGordonHu:Toroidal}
Steven Boyer, Cameron~McA Gordon, and Ying Hu, \emph{Slope detection and
  toroidal 3-manifolds}, 2021.

\bibitem[BGW13]{BoyerGordonWatson}
Steven Boyer, Cameron~McA. Gordon, and Liam Watson, \emph{On {L}-spaces and
  left-orderable fundamental groups}, Math. Ann. \textbf{356} (2013), no.~4,
  1213--1245.

\bibitem[BRW05]{BoyerRolfsonWiest}
Steven Boyer, Dale Rolfsen, and Bert Wiest, \emph{Orderable 3-manifold groups},
  Ann. Inst. Fourier (Grenoble) \textbf{55} (2005), no.~1, 243--288.
  \MR{2141698}

\bibitem[BNR97]{BrittenhamNaimiRoberts}
Mark Brittenham, Ramin Naimi, and Rachel Roberts, \emph{Graph manifolds and
  taut foliations}, J. Differential Geom. \textbf{45} (1997), no.~3, 446--470.

\bibitem[CLW13]{ClayLidmanWatson}
Adam Clay, Tye Lidman, and Liam Watson, \emph{Graph manifolds,
  left-orderability and amalgamation}, Algebr. Geom. Topol. \textbf{13} (2013),
  no.~4, 2347--2368.

\bibitem[Cro89]{Cromwell:HomogeneousLinks}
P.~R. Cromwell, \emph{Homogeneous links}, J. London Math. Soc. (2) \textbf{39}
  (1989), no.~3, 535--552. \MR{1002465}

\bibitem[CDGW]{SnapPy}
Marc Culler, Nathan~M. Dunfield, Matthias Goerner, and Jeffrey~R. Weeks,
  \emph{Snap{P}y, a computer program for studying the geometry and topology of
  $3$-manifolds}, Available at \url{http://snappy.computop.org} (DD/MM/YYYY).

\bibitem[DL04]{DasbachLi}
Oliver~T. Dasbach and Tao Li, \emph{Property {P} for knots admitting certain
  {G}abai disks}, Topology Appl. \textbf{142} (2004), no.~1-3, 113--129.
  \MR{2071298}

\bibitem[DR]{DelmanRoberts:composite}
Charles Delman and Rachel Roberts, \emph{Persistently foliar composite knots},
  https://arxiv.org/abs/1905.04838.

\bibitem[Eft15]{Eftekhary}
Eaman Eftekhary, \emph{Floer homology and splicing knot complements}, Algebr.
  Geom. Topol. \textbf{15} (2015), no.~6, 3155--3213. \MR{3450759}

\bibitem[EHN81]{EisenbudHirschNeumann}
David Eisenbud, Ulrich Hirsch, and Walter Neumann, \emph{Transverse foliations
  of {S}eifert bundles and self-homeomorphism of the circle}, Comment. Math.
  Helv. \textbf{56} (1981), no.~4, 638--660.

\bibitem[Fel18]{Feller:4Braids}
Peter Feller, \emph{A sharp signature bound for positive four-braids}, Q. J.
  Math. \textbf{69} (2018), no.~1, 271--283. \MR{3771393}

\bibitem[FW87]{FranksWilliams}
John Franks and R.~F. Williams, \emph{Braids and the {J}ones polynomial},
  Trans. Amer. Math. Soc. \textbf{303} (1987), no.~1, 97--108. \MR{896009}

\bibitem[Gab85]{Gabai:MurasugiSumII}
David Gabai, \emph{The {M}urasugi sum is a natural geometric operation. {II}},
  Combinatorial methods in topology and algebraic geometry ({R}ochester,
  {N}.{Y}., 1982), Contemp. Math., vol.~44, Amer. Math. Soc., Providence, RI,
  1985, pp.~93--100.

\bibitem[Gab86]{Gabai:Fibered}
\bysame, \emph{Detecting fibred links in {$S^3$}}, Comment. Math. Helv.
  \textbf{61} (1986), no.~4, 519--555.

\bibitem[Gab87]{Gabai:FoliationsIII}
\bysame, \emph{Foliations and the topology of {$3$}-manifolds. {III}}, J.
  Differential Geom. \textbf{26} (1987), no.~3, 479--536.

\bibitem[GO89]{GabaiOertel}
David Gabai and Ulrich Oertel, \emph{Essential laminations in {$3$}-manifolds},
  Ann. of Math. (2) \textbf{130} (1989), no.~1, 41--73.

\bibitem[GS99]{GompfAndStipsicz}
Robert~E. Gompf and Andr\'{a}s~I. Stipsicz, \emph{{$4$}-manifolds and {K}irby
  calculus}, Graduate Studies in Mathematics, vol.~20, American Mathematical
  Society, Providence, RI, 1999. \MR{1707327}

\bibitem[Gor75]{Gordon:HomologySpheres}
C.~McA. Gordon, \emph{Knots, homology spheres, and contractible
  {$4$}-manifolds}, Topology \textbf{14} (1975), 151--172. \MR{402762}

\bibitem[GL89]{GordonLuecke:KnotComplements}
C.~McA. Gordon and J.~Luecke, \emph{Knots are determined by their complements},
  J. Amer. Math. Soc. \textbf{2} (1989), no.~2, 371--415. \MR{965210}

\bibitem[Gre13]{Greene:LensSpaceRealization}
Joshua~Evan Greene, \emph{The lens space realization problem}, Ann. of Math.
  (2) \textbf{177} (2013), no.~2, 449--511. \MR{3010805}

\bibitem[GLV18]{GLV:11Lspace}
Joshua~Evan Greene, Sam Lewallen, and Faramarz Vafaee, \emph{{$(1,1)$}
  {L}-space knots}, Compos. Math. \textbf{154} (2018), no.~5, 918--933.

\bibitem[GW10]{GrigsbyWehrli}
J.~Elisenda Grigsby and Stephan~M. Wehrli, \emph{On the colored {J}ones
  polynomial, sutured {F}loer homology, and knot {F}loer homology}, Adv. Math.
  \textbf{223} (2010), no.~6, 2114--2165. \MR{2601010}

\bibitem[HRRW15]{HRRW}
Jonathan Hanselman, Jacob Rasmussen, Sarah~Dean Rasmussen, and Liam Watson,
  \emph{Taut foliations on graph manifolds}, https://arxiv.org/abs/1508.05911
  (2015).

\bibitem[HRW16]{HanselmanRasmussenWatson}
Jonathan Hanselman, Jacob Rasmussen, and Liam Watson, \emph{Bordered floer
  homology for manifolds with torus boundary via immersed curves}, 2016.

\bibitem[Hat07]{Hatcher:3Manifolds}
Allen Hatcher, \emph{Notes on basic 3-manifold topology},
  https://pi.math.cornell.edu/~hatcher/3M/3M.pdf (2007).

\bibitem[Hed09a]{Hedden:KhovanovHomology2Cable}
Matthew Hedden, \emph{Khovanov homology of the 2-cable detects the unknot},
  Math. Res. Lett. \textbf{16} (2009), no.~6, 991--994. \MR{2576686}

\bibitem[Hed09b]{Hedden:CablingII}
\bysame, \emph{On knot {F}loer homology and cabling. {II}}, Int. Math. Res.
  Not. IMRN (2009), no.~12, 2248--2274. \MR{2511910}

\bibitem[Hed10]{Hedden:Positivity}
\bysame, \emph{Notions of positivity and the {O}zsv\'{a}th-{S}zab\'{o}
  concordance invariant}, J. Knot Theory Ramifications \textbf{19} (2010),
  no.~5, 617--629. \MR{2646650}

\bibitem[HK]{HeddenKrishna}
Matthew Hedden and Siddhi Krishna, \emph{Thurston-bennequin numbers and
  cabling}, In preparation.

\bibitem[HL16]{HeddenLevine:Splicing}
Matthew Hedden and Adam~Simon Levine, \emph{Splicing knot complements and
  bordered {F}loer homology}, J. Reine Angew. Math. \textbf{720} (2016),
  129--154. \MR{3565971}

\bibitem[HO08]{HeddenOrding}
Matthew Hedden and Philip Ording, \emph{The {O}zsv\'{a}th-{S}zab\'{o} and
  {R}asmussen concordance invariants are not equal}, Amer. J. Math.
  \textbf{130} (2008), no.~2, 441--453. \MR{2405163}

\bibitem[Ito22]{Ito:BraidPositiveHOMFLY}
Tetsuya Ito, \emph{A note on {HOMFLY} polynomial of positive braid links},
  Internat. J. Math. \textbf{33} (2022), no.~4, Paper No. 2250031, 18.
  \MR{4402791}

\bibitem[Juh15]{Juhasz:Survey}
Andr\'as Juh\'asz, \emph{A survey of {H}eegaard {F}loer homology}, New ideas in
  low dimensional topology, Ser. Knots Everything, vol.~56, World Sci. Publ.,
  Hackensack, NJ, 2015, pp.~237--296.

\bibitem[KLT21]{KarakurtLidmanTweedy}
\c{C}a\u{g}ri Karakurt, Tye Lidman, and Eamonn Tweedy, \emph{Heegaard {F}loer
  homology and splicing homology spheres}, Math. Res. Lett. \textbf{28} (2021),
  no.~1, 93--106. \MR{4247996}

\bibitem[KR17]{KazezRoberts}
William~H. Kazez and Rachel Roberts, \emph{{$C^0$} approximations of
  foliations}, Geom. Topol. \textbf{21} (2017), no.~6, 3601--3657.

\bibitem[Kri20]{Krishna:3Braids}
Siddhi Krishna, \emph{Taut foliations, positive 3-braids, and the {L}-space
  conjecture}, J. Topol. \textbf{13} (2020), no.~3, 1003--1033. \MR{4100124}

\bibitem[KMOS07]{KMOSz}
P.~Kronheimer, T.~Mrowka, P.~Ozsv\'ath, and Z.~Szab\'o, \emph{Monopoles and
  lens space surgeries}, Ann. of Math. (2) \textbf{165} (2007), no.~2,
  457--546.

\bibitem[Li02]{TaoLi:SinkDisk}
Tao Li, \emph{Laminar branched surfaces in 3-manifolds}, Geom. Topol.
  \textbf{6} (2002), 153--194.

\bibitem[Li03]{TaoLi:BoundarySinkDisk}
\bysame, \emph{Boundary train tracks of laminar branched surfaces}, Topology
  and geometry of manifolds ({A}thens, {GA}, 2001), Proc. Sympos. Pure Math.,
  vol.~71, Amer. Math. Soc., Providence, RI, 2003, pp.~269--285.

\bibitem[LS09]{LiscaStipsicz}
Paolo Lisca and Andr\'as~I. Stipsicz, \emph{On the existence of tight contact
  structures on {S}eifert fibered 3-manifolds}, Duke Math. J. \textbf{148}
  (2009), no.~2, 175--209.

\bibitem[Liv04]{Livingston:TauInvariant}
Charles Livingston, \emph{Computations of the {O}zsv\'{a}th-{S}zab\'{o} knot
  concordance invariant}, Geom. Topol. \textbf{8} (2004), 735--742.
  \MR{2057779}

\bibitem[Mis16]{Misev:Thesis}
Filip Misev, \emph{On the plumbing structure of fibre surfaces}, Ph.D. thesis,
  Universit{\"a}t Bern, 2016.

\bibitem[Mor86]{Morton:KnotPolys}
H.~R. Morton, \emph{Seifert circles and knot polynomials}, Math. Proc.
  Cambridge Philos. Soc. \textbf{99} (1986), no.~1, 107--109. \MR{809504}

\bibitem[Mos71]{Moser:TorusKnots}
Louise Moser, \emph{Elementary surgery along a torus knot}, Pacific J. Math.
  \textbf{38} (1971), 737--745.

\bibitem[Mur74]{Murasugi:3Braids}
Kunio Murasugi, \emph{On closed {$3$}-braids}, American Mathematical Society,
  Providence, R.I., 1974, Memoirs of the American Mathmatical Society, No. 151.
  \MR{0356023}

\bibitem[NZ85]{NeumannZagier}
Walter~D. Neumann and Don Zagier, \emph{Volumes of hyperbolic three-manifolds},
  Topology \textbf{24} (1985), no.~3, 307--332. \MR{815482}

\bibitem[OS04]{OSz:HolDisks}
Peter Ozsv\'ath and Zolt\'an Szab\'o, \emph{Holomorphic disks and genus
  bounds}, Geom. Topol. \textbf{8} (2004), 311--334.

\bibitem[Pet10]{Peters:Thesis}
Thomas~David Peters, \emph{Computations of {H}eegaard {F}loer homology: {T}orus
  bundles, {L}-spaces, and correction terms}, ProQuest LLC, Ann Arbor, MI,
  2010, Thesis (Ph.D.)--Columbia University. \MR{2782352}

\bibitem[Prz89]{Przytycki:PositiveKnotsNegativeSignature}
J\'{o}zef~H. Przytycki, \emph{Positive knots have negative signature}, Bull.
  Polish Acad. Sci. Math. \textbf{37} (1989), no.~7-12, 559--562 (1990).
  \MR{1101920}

\bibitem[RR17]{Rasmussen2}
Jacob Rasmussen and Sarah~Dean Rasmussen, \emph{Floer simple manifolds and
  {L}-space intervals}, Adv. Math. \textbf{322} (2017), 738--805.

\bibitem[RSS03]{RobertsShareshianStein}
R.~Roberts, J.~Shareshian, and M.~Stein, \emph{Infinitely many hyperbolic
  3-manifolds which contain no {R}eebless foliation}, J. Amer. Math. Soc.
  \textbf{16} (2003), no.~3, 639--679. \MR{1969207}

\bibitem[Rud93]{Rudolph:QPsliceness}
Lee Rudolph, \emph{Quasipositivity as an obstruction to sliceness}, Bull. Amer.
  Math. Soc. (N.S.) \textbf{29} (1993), no.~1, 51--59.

\bibitem[Sav02]{Saveliev:Homology3Spheres}
Nikolai Saveliev, \emph{Invariants for homology {$3$}-spheres}, Encyclopaedia
  of Mathematical Sciences, vol. 140, Springer-Verlag, Berlin, 2002,
  Low-Dimensional Topology, I. \MR{1941324}

\bibitem[Sch53]{Schubert}
Horst Schubert, \emph{Knoten und {V}ollringe}, Acta Math. \textbf{90} (1953),
  131--286. \MR{72482}

\bibitem[Sch18]{Schultens:SatelliteKnots}
Jennifer Schultens, \emph{Satellite knots}, Concise Encyclopedia of Knot Theory
  (2018).

\bibitem[Sta78]{Stallings:Fibered}
John~R. Stallings, \emph{Constructions of fibred knots and links}, Algebraic
  and geometric topology ({P}roc. {S}ympos. {P}ure {M}ath., {S}tanford {U}niv.,
  {S}tanford, {C}alif., 1976), {P}art 2, Proc. Sympos. Pure Math., XXXII, Amer.
  Math. Soc., Providence, R.I., 1978, pp.~55--60.

\bibitem[Sto03]{Stoimenow:PositiveKnotsJonesPoly}
Alexander Stoimenow, \emph{Positive knots, closed braids and the {J}ones
  polynomial}, Ann. Sc. Norm. Super. Pisa Cl. Sci. (5) \textbf{2} (2003),
  no.~2, 237--285. \MR{2004964}

\bibitem[ST09]{StoimenowTanaka}
Alexander Stoimenow and Toshifumi Tanaka, \emph{Mutation and the colored
  {J}ones polynomial}, J. G\"{o}kova Geom. Topol. GGT \textbf{3} (2009),
  44--78. \MR{2595755}

\bibitem[Tao07]{TerryTao:Amplification}
Terry Tao, \emph{Amplification, arbitrage, and the tensor power trick},
  https://terrytao.wordpress.com/2007/09/05/amplification-arbitrage-and-the-tensor-power-trick/
  (2007).

\bibitem[Thu22]{Thurston:Vol4}
William~P. Thurston, \emph{The geometry and topology of three-manifolds. {V}ol.
  {IV}}, American Mathematical Society, Providence, RI, [2022] \copyright 2022,
  Edited and with a preface by Steven P. Kerckhoff and a chapter by J. W.
  Milnor. \MR{4554426}

\bibitem[Vaf15]{Vafaee:TwistedTorusKnots}
Faramarz Vafaee, \emph{On the knot {F}loer homology of twisted torus knots},
  Int. Math. Res. Not. IMRN (2015), no.~15, 6516--6537.

\bibitem[Zem21]{Zemke:LinkSurgery}
Ian Zemke, \emph{Bordered manifolds with torus boundary and the link surgery
  formula}, 2021.

\end{thebibliography}

\end{document}